\documentclass[11pt,a4paper]{article}

\usepackage[utf8]{inputenc}
\usepackage[T1]{fontenc}
\usepackage{geometry}
\usepackage{lmodern}
\usepackage{amsmath,amssymb,amsthm,amsfonts}
\usepackage{tikz}
\usetikzlibrary{arrows.meta,automata,positioning,calc,fit,backgrounds,bending}
\usepackage{graphicx}
\usepackage{enumitem}
\usepackage{xcolor}
\usepackage{hyperref}

\hypersetup{colorlinks=true,linkcolor=blue,citecolor=blue,urlcolor=blue}

\newtheorem{theorem}{Theorem}[section]
\newtheorem{lemma}[theorem]{Lemma}
\newtheorem{proposition}[theorem]{Proposition}
\newtheorem{corollary}[theorem]{Corollary}
\newtheorem{definition}[theorem]{Definition}
\newtheorem{example}[theorem]{Example}
\newtheorem{remark}[theorem]{Remark}
\newtheorem{notation}[theorem]{Notation}
\newtheorem{problem}[theorem]{Problem}

\newcommand{\C}{\mathcal C}
\newcommand{\W}{\mathcal W}
\newcommand{\D}{\mathcal D}

\newcommand{\Fraw}{\mathcal F^{\mathrm{raw}}}
\newcommand{\Ffor}{\mathcal F}
\newcommand{\N}{\mathbb N}
\newcommand{\Z}{\mathbb Z}
\newcommand{\Q}{\mathbb Q}

\newcommand{\Homeo}{\operatorname{Homeo}}

\newcommand{\Root}{\operatorname{Root}}
\newcommand{\im}{\operatorname{im}}
\newcommand{\Cl}{\operatorname{Cl}}
\newcommand{\supp}{\operatorname{supp}}

\newcommand{\id}{\operatorname{id}}
\newcommand{\ab}{\operatorname{ab}}
\newcommand{\LA}{\mathcal{LA}}
\newcommand{\eps}{\varepsilon}
\newcommand{\Stab}{\operatorname{Stab}}

\title{Higman--Thompson groups $F_n$ all the way down}
\author{Gili Golan}
\date{}

\begin{document}
\maketitle
\begin{abstract}
We prove that for every $n\ge 2$ the Higman--Thompson group $F_n$ has a maximal subgroup of infinite index isomorphic to $F_n$ itself.  In fact, we prove that $F_n$ contains a descending chain of subgroups
\[
        F_n=H_0>H_1>H_2>\cdots,
\]
all isomorphic to $F_n$, such that for every $i$ the only subgroups of $F_n$ containing $H_i$ are $H_i,H_{i-1},\ldots,H_0=F_n$; in particular, each $H_{i+1}$ is maximal in $H_i$.  Moreover, the chain can be chosen to have trivial intersection.  Informally, ``it is $F_n$ all the way down''.

The proof rests on a systematic study of isomorphic copies of Higman--Thompson groups inside one another.  We prove that for all $n\ge m\ge2$, every closed maximal subgroup of $F_m$ isomorphic to $F_n$ is conjugate to the standard copy of $F_n$ by a homeomorphism of the unit interval, which lifts to a homeomorphism from the $n$-ary Cantor space to the $m$-ary Cantor space given by a finite \emph{semi-synchronizing} transducer---a variation of the synchronizing transducers arising in the characterizations of the automorphism groups of the Higman--Thompson groups $G_{n,r}$ and $T_{n,r}$ \cite{AutG,AutT}.  We characterize the homeomorphisms of Cantor spaces conjugating $F_n$ into $F_m$ as precisely the order-preserving or order-reversing rational homeomorphisms whose minimal transducer is semi-synchronizing.  At the heart of the paper is a machinery bridging between the theory of transducers and the theory of Stallings $2$-cores of subgroups of $F_n$: given a closed subgroup $H\le F_n$ and a homeomorphism $\psi$ as above, it produces a tree-automaton defining the conjugate subgroup $H^\psi\cap F_m$; in particular, it reduces the conjugation of finitely generated closed subgroups of $F_n$ by transducer homeomorphisms to an algorithmic procedure.  Among the applications, we prove that Jones' ternary oriented subgroup $\vec F_3\le F_3$ is isomorphic to $F_4$, answering questions of Aiello; in particular, the maximal subgroup of $F_3$ constructed by Aiello and Nagnibeda is isomorphic to $F_4$.

Similarly, using the results and methods of this paper, we prove that all known maximal subgroups of infinite index of Thompson's group $F$ which act minimally on the unit interval are isomorphic to Higman--Thompson groups.  We pose the problem of whether this holds for all maximal subgroups of infinite index of $F$ acting minimally on the unit interval, and describe several results, to appear in forthcoming papers, proved as part of an attempt to solve it.  In particular, a large class of closed subgroups of $F$ are isomorphic to Higman--Thompson groups; this class includes the diagram groups associated with all indecomposable groups generated by geometrically fast sets of $n$ one-bump functions, so that every such group is isomorphic to $F_n$.  This answers, in a strong form, a problem of Brin and Zaremsky \cite{BurilloBuxNucinkis}.  The class also yields, for all $n>m\ge2$, maximal subgroups of $F_m$ isomorphic to $F_n$.
\end{abstract}

\section{Introduction}
\label{sec:introduction}

Recall that R.~Thompson's group $F$ is the group of all piecewise-linear homeomorphisms of the interval $[0,1]$, where all breakpoints are dyadic fractions (i.e., elements of the set $\Z[\frac12]\cap(0,1)$) and all slopes are integer powers of $2$.  The group $F$ is finitely presented, does not contain free non-abelian subgroups, and satisfies many other remarkable properties; see \cite{CannonFloydParry} for a survey.  It was generalized by Brown \cite{Brown} to a family of groups known as the Higman--Thompson groups $F_n$, $n\ge2$: the group $F_n$ is the group of all piecewise-linear homeomorphisms of $[0,1]$ with breakpoints in $\Z[\frac1n]$ and slopes integer powers of $n$, so that $F=F_2$.

Savchuk initiated the study of maximal subgroups of Thompson's group $F$ by proving that for every number $\alpha\in(0,1)$, the stabilizer of $\alpha$ in $F$ (i.e., the subgroup of $F$ of all functions which fix $\alpha$) is a maximal subgroup of $F$ \cite{SavchukGraphs,SavchukSchreier}.  He asked whether these are all the maximal subgroups of infinite index in $F$ (maximal subgroups of $F$ of finite index are in one-to-one correspondence with maximal subgroups of its abelianization $\Z^2$ and are well understood).

In answer to this problem, M.~Sapir and the author constructed an explicit maximal subgroup of infinite index in $F$ which does not fix any number in $(0,1)$ \cite{GolanSapirSubgroupsF}.  Recall that Jones showed that elements of $F$ encode in a natural way all links and knots, and that elements of a subgroup of $F$, denoted $\vec F$ and called the oriented Jones subgroup, encode all oriented links and knots \cite{JonesUnitary,JonesKnotsLinks}.  In \cite{GolanSapirJones} we proved that $\vec F$ is isomorphic to the Higman--Thompson group $F_3$.  The explicit maximal subgroup of $F$ constructed in \cite{GolanSapirSubgroupsF} is the preimage of $\vec F$ under an injective endomorphism from $F$ onto an index-two subgroup of $F$; in particular, it is isomorphic to $F_3$.

In addition to this explicit maximal subgroup, \cite{GolanSapirSubgroupsF} provided a general method for proving the existence of many other maximal subgroups of $F$.  Improving on that method, the author gave in \cite{GolanGeneration} a method for constructing explicit examples and used it to construct three new maximal subgroups of $F$; later, Aiello and Nagnibeda constructed three more explicit examples \cite{AielloNagnibedaThreeColorable}.  In \cite{GolanMaximalF}, the author proved that $F$ has infinitely many pairwise non-isomorphic maximal subgroups of infinite index.  The idea was as follows: for each $n>2$, one can embed $F_n$ into $F$ by replacing, in each tree-diagram of an element of $F_n$, every $n$-ary caret by a binary right vine with $n$ leaves; the image of this embedding is the generalized Jones subgroup $\vec F_{n-1}\cong F_n$ from \cite{GolanSapirJones}.  In \cite{GolanMaximalF}, the author proved that when $n-1$ is prime, $\vec F_{n-1}$ is maximal in a finite-index subgroup of $F$ isomorphic to $F$, which yields a maximal subgroup of $F$ isomorphic to $F_n$.  Later, the author observed (unpublished) that while when $n-1$ is not prime this method does not yield a maximal subgroup, it can be ``tweaked'': replacing each caret of a tree-diagram by a suitable binary tree with $n$ leaves (where not all carets are replaced by the same tree) produces, for every $n>2$, a maximal subgroup of $F$ isomorphic to $F_n$.  Note that this construction produces maximal subgroups of $F$ isomorphic to ``higher'' Higman--Thompson groups.  Thus, a natural problem is whether $F$ has a maximal subgroup isomorphic to $F$ itself.

The search for maximal subgroups naturally extends to the Higman--Thompson groups $F_n$.  Savchuk's proof readily adapts to the general case, showing that for every $\alpha\in(0,1)$, the stabilizer of $\alpha$ in $F_n$ is a maximal subgroup of infinite index.  Aiello and Nagnibeda provided the first example of a maximal subgroup of infinite index of $F_3$ which does not fix any number in $(0,1)$, utilizing the ternary oriented Jones subgroup $\vec F_3\le F_3$, which encodes all oriented links and knots, similarly to $\vec F$.\footnote{This subgroup is denoted $\vec F_3$ in \cite{AielloNagnibedaFThree}; it should not be confused with the subgroup $\vec F_3$ of $F$, defined in \cite{GolanSapirJones} as part of the family $\vec F_n\le F$ generalizing Jones' subgroup $\vec F=\vec F_2$.}  Indeed, they proved that $\vec F_3$ is maximal in an index-two subgroup $G_3\le F_3$ isomorphic to $F_3$ \cite{AielloNagnibedaFThree}, and asked whether for each $n$, the group $F_n$ has a maximal subgroup of infinite index which does not fix any number in $(0,1)$.  Recently, the author and E.~Sapir answered this in the affirmative, by constructing for each $n\ge2$ a maximal subgroup of infinite index of $F_n$ isomorphic to $F_{2n-1}$ \cite{GolanSapirGenerationFn}.  Here too, the maximal subgroups constructed are copies of a ``larger'' Higman--Thompson group inside a ``smaller'' one, and there is a reason for this asymmetry: it is easy to see  that if $F_m$ has a maximal subgroup isomorphic to $F_n$, then necessarily $n\ge m$ (Lemma~\ref{lem:rank-obstruction-maximal-copy}).  The boundary case $n=m$, that is, the problem of constructing a maximal subgroup of $F_n$ isomorphic to $F_n$ itself, and in particular, a maximal subgroup of $F$ isomorphic to $F$, is the subject of the present paper.

The constructions of all known maximal subgroups of $F$ of infinite index, which are not point stabilizers, rely on the Stallings $2$-core of subgroups of $F_n$, a Stallings-type technique for subgroups of diagram groups \cite{GubaSapirDiagramGroups,GubaSapirSubgroups} defined by Guba and Sapir in the 1990s, which first appeared in print in \cite{GolanSapirSubgroupsF}.  In the case of the groups $F_n$, the core of a subgroup $H\le F_n$, denoted $\C(H)$, can be viewed as a \emph{tree-automaton}: a rooted automaton in which every state has $n$ outgoing edges, labeled $0,\dots,n-1$ (see Section~\ref{sec:preliminaries}).  The core of $H$ accepts or rejects elements of $F_n$, and the set of all elements accepted by $\C(H)$ is a subgroup called the \emph{closure} of $H$, denoted $\Cl(H)$.  A subgroup $H$ is \emph{closed} if $H=\Cl(H)$; equivalently, $H\le F_n$ is closed if and only if every piecewise-$H$ function in $F_n$ belongs to $H$.  In \cite{GolanMaximalF}, the author proved that every maximal subgroup of $F$ of infinite index is closed.  It is an open problem whether the same holds for the groups $F_n$, $n>2$ (see \cite{GolanSapirGenerationFn}); note that all known maximal subgroups of infinite index of $F_n$ are closed.

The relation between the core of a subgroup and maximality goes through the generation problem.  In \cite{GolanGeneration,GolanMaximalF}, the author proved that a subgroup $H\le F$ coincides with $F$ if and only if $H[F,F]=F$ and $[F,F]\le\Cl(H)$.  Since a maximal subgroup $M<F$ of infinite index must satisfy $M[F,F]=F$, such $M$ is maximal if and only if whenever an element of $F\setminus M$ is added to $M$, the closure of the resulting subgroup contains $[F,F]$.  For $F_n$ with $n>2$ no such if-and-only-if criterion is currently available, but a sufficient generation criterion was proved in \cite{GolanSapirGenerationFn} (see Theorem~\ref{thm:generation-Fn-prelim}).

The main result of this paper is the following.

\begin{theorem}[see Theorem~\ref{thm:Fn-all-the-way-down-final}]
\label{thm:main-intro}
For every $n\ge2$, there exists a descending chain
\[
        F_n=H_0>H_1>H_2>\cdots
\]
such that:
\begin{enumerate}
    \item $H_i\cong F_n$ for every $i$;
    \item every subgroup of $F_n$ containing $H_i$ is one of
    $H_i,H_{i-1},\ldots,H_0=F_n$;
    \item for each $i$, $H_{i+1}$ is a subgroup of infinite index of $H_i$;
    \item $\bigcap_{i\ge0}H_i=\{1\}$.
\end{enumerate}
\end{theorem}

We now outline the proof, together with several results of independent interest obtained along the way.

The starting point is a study of closed maximal subgroups of $F_m$ isomorphic to some $F_n$.  We first prove that every maximal subgroup of $F_m$ isomorphic to $F_n$ must act minimally on the interval $(0,1)$ (Lemma~\ref{lem:maximal-copy-minimal}).  Since the standard action of $F_n$ is locally moving, a Rubin-type reconstruction theorem of Brum--Matte Bon--Rivas--Triestino \cite{BMBRT} then implies that such a subgroup is conjugate to the standard copy of $F_n$ in $\Homeo_+(0,1)$.  Moreover, a homeomorphism $\phi\in\Homeo_+(0,1)$ conjugating $F_n$ into $F_m$ has an image which is closed in $F_m$ if and only if $\phi$ lifts to a homeomorphism $\psi:\C_n\to\C_m$ of the corresponding Cantor spaces (Proposition~\ref{prop:closed-maximal-copy-realized-by-conjugacy}).  This reduces the study of closed maximal copies of $F_n$ in $F_m$ to the study of Cantor-space homeomorphisms $\psi:\C_n\to\C_m$ satisfying $F_n^\psi\le F_m$.

Such lifts allow us to utilize the rich theory of transducers inducing automorphisms of Thompson groups.  The study of the automorphism groups of the Thompson family has a long history: Brin characterized the automorphism groups of $F$ and $T$ \cite{BrinChameleon}, and with Guzm\'an explored the automorphism groups of the generalized groups $F_{n,r}$ and $T_{n,r}$ \cite{BrinGuzman}, building on the reconstruction theorem of McCleary and Rubin \cite{McClearyRubin}.  Later, Bleak, Cameron, Maissel, Navas and Olukoya characterized the automorphism group of $G_{n,r}$ as the group of rational homeomorphisms of the Cantor space (in the sense of Grigorchuk--Nekrashevych--Sushchanski\u{\i} \cite{GNS}) whose minimal transducer is bi-synchronizing \cite{AutG}, and Olukoya extended this transducer framework to $T_{n,r}$ \cite{AutT}.  Building on these foundations, we characterize the homeomorphisms $\psi:\C_n\to\C_m$ that conjugate $F_n$ into $F_m$.

\begin{theorem}[see Theorem~\ref{thm:conjugator-characterization}]
\label{thm:conjugator-intro}
Let $\psi:\C_n\to\C_m$ be a homeomorphism.  Then $F_n^\psi\le F_m$ if and only if $\psi$ is order-preserving or order-reversing, and $\psi$ is a rational homeomorphism whose minimal transducer is semi-synchronizing.
\end{theorem}

Semi-synchronizing transducers (Definition~\ref{def:semi-synchronizing}) are a variation of the synchronizing transducers of \cite{AutG,AutT}, adapted to the endpoint-preserving groups $F_n$: the two boundary rays are required to eventually stabilize, and synchronization is required only among inner words sharing the same value of the residue invariant $\sigma_n$, the invariant governing branch pairs of elements of $F_n$.  As a corollary, we obtain a transducer description of the automorphism group of $F_n$ (Corollary~\ref{cor:aut-Fn-semisync}), analogous to the descriptions of $\operatorname{Aut}(G_{n,r})$ and $\operatorname{Aut}(T_{n,r})$ in \cite{AutG,AutT}.

The heart of the paper is a machinery bridging between the theory of transducers and the theory of cores of subgroups of $F_n$, which computes how the core of a closed subgroup changes under conjugation.  Given a closed subgroup $H\le F_n$ defined by a full tree-automaton $\mathcal A$, and a homeomorphism $\psi:\C_n\to\C_m$ as above, we construct an $m$-ary tree-automaton defining the closed subgroup $H^\psi\cap F_m$ (see Section~\ref{sec:conjugating-closed-subgroups}).  In particular, when $H^\psi\le F_m$ and $\mathcal A$ is the core of $H$, the construction yields the core of $H^\psi$; since finitely generated closed subgroups of $F_n$ are determined by their cores, this reduces the conjugation of such subgroups by transducer homeomorphisms to an algorithmic procedure.  For the transducers arising in our examples, we develop a geometric form of the construction (Subsection~\ref{subsec:geometric-determinization}) which makes the computations easier to follow.

The machinery has several applications beyond the main theorem.  First, we prove the following.

\begin{theorem}[see Corollary~\ref{thm:jones-F3-is-F4}]
\label{thm:jones-intro}
The ternary oriented Jones subgroup $\vec F_3\le F_3$ is isomorphic to $F_4$.
\end{theorem}

This answers questions of Aiello from \cite{AielloIntro}, asking whether $\vec F_3$ is finitely presented, and whether it is isomorphic to a Higman--Thompson group or to some other known group.  The proof exhibits an explicit semi-synchronizing transducer representing a homeomorphism $\psi_J:\C_4\to\C_3$ with $F_4^{\psi_J}=\vec F_3$.  In particular, the maximal subgroup of $F_3$ constructed by Aiello and Nagnibeda is isomorphic to $F_4$.  Second, we realize the known maximal subgroups of $F$ of infinite index which act minimally on $(0,1)$ as transducer conjugates of Higman--Thompson groups (Subsection~\ref{subsec:known-minimal-maximal-subgroups}).  Consequently, all currently known maximal subgroups of infinite index of Thompson's group $F$ which act minimally on $(0,1)$ are isomorphic to Higman--Thompson groups.  We also realize a copy of the Brin--Navas group $B$ \cite{BrinElementaryAmenable,NavasQuelquesGroupes}, inside $F$, in the form $F^\psi\cap F$ for an explicit transducer homeomorphism $\psi:\C_2\to\C_2$ (Subsection~\ref{subsec:brin-navas-subgroup}), illustrating the constructions in the case where $F^\psi$ is not contained in $F$.

With this machinery in place, the chains are constructed as follows.  For $F=F_2$, we exhibit an explicit order-preserving semi-synchronizing binary transducer with four states, representing a homeomorphism $\varphi:\C_2\to\C_2$ (Example~\ref{ex:binary-semisync-transducer}), and consider the subgroups $H_i=F^{\varphi^i}$, $i\ge0$.  Using our techniques, we compute the cores $\C(H_i)$ and classify their quotient tree-automata; the correspondence between closed overgroups and quotients of the core then implies that the only closed subgroups of $F$ containing $H_i$ are $H_i,H_{i-1},\ldots,H_0=F$.  The generation theorem for $F$ is then used to remove the closedness assumption, and a study of the inverse limit of the cores implies that the chain $(H_i)$ has trivial intersection (Theorem~\ref{thm:intersection-trivial}).

In Section~\ref{sec:standard-chain-and-twisting} we extend the result from $F$ to the general case of $F_n$, $n>2$.  We introduce an operation of \emph{residue inflation}, which turns binary transducers into $n$-ary ones.  Inflating the binary transducer of Section~\ref{sec:binary-descending-chain} yields an $n$-ary semi-synchronizing transducer, representing a homeomorphism $\zeta:\C_n\to\C_n$, and we consider the chain $S_i=F_n^{\zeta^i}$, $i\ge0$.  We prove that every subgroup of $F_n$ containing $S_i$ is one of $S_i,S_{i-1},\ldots,S_0=F_n$ (Theorem~\ref{thm:standard-only-predecessors}).  This  chain, however, does not have trivial intersection.  In the final step we modify it by a \emph{compatible twisting} argument: inner conjugations $H_i=S_i^{k_i}$, with the elements $k_i$ chosen coherently along the chain, in the spirit of stabilizers of rays in coset trees (compare \cite{GK14}).  The twisting preserves the overgroup structure of the chain and forces the intersection to be trivial (Proposition~\ref{prop:compatible-twisting-kills-intersection-general}), which completes the proof of Theorem~\ref{thm:main-intro}.

\subsection*{Towards a classification of maximal subgroups}

By the results of this paper, all known maximal subgroups of infinite index of Thompson's group $F$ which act minimally on $(0,1)$ are isomorphic to Higman--Thompson groups: the members of the family from \cite{GolanMaximalF} are isomorphic to the groups $F_{p+1}$, $p$ prime, by construction; the five known explicit examples from \cite{GolanGeneration,AielloNagnibedaThreeColorable} which act minimally are shown in Subsection~\ref{subsec:known-minimal-maximal-subgroups} to be transducer conjugates of Higman--Thompson groups; and the chains constructed in this paper consist of copies of $F_n$.  The same pattern occurs in the known examples for $F_n$: the maximal subgroups of $F_n$ from \cite{GolanSapirGenerationFn} are isomorphic to $F_{2n-1}$, and the Aiello--Nagnibeda maximal subgroup of $F_3$ is isomorphic to $F_4$ by Theorem~\ref{thm:jones-intro}.  This motivates the following problem.

\begin{problem}
\label{prob:minimal-maximal}
Is every maximal subgroup of infinite index of Thompson's group $F$ which acts minimally on the interval $(0,1)$ isomorphic to a Higman--Thompson group $F_m$, for some $m\ge2$?  More generally, characterize the isomorphism types of the maximal subgroups of infinite index of the Higman--Thompson groups $F_n$.
\end{problem}

The point stabilizers are excluded by the minimality hypothesis.  Problem~\ref{prob:minimal-maximal} is consistent with all examples of maximal subgroups of $F$ which appear in the literature.  The same problem is relevant for the groups $F_n$, $n>2$, possibly under the additional assumption that the maximal subgroup is closed; recall that for $F$, closedness of maximal subgroups of infinite index is automatic \cite{GolanMaximalF}, while for $F_n$ this is an open problem.

Several results obtained by the author as part of an attempt to resolve Problem~\ref{prob:minimal-maximal} will appear in forthcoming papers.  The author has proved that a wide class of closed subgroups of $F$ are isomorphic to Higman--Thompson groups.  These results apply, in particular, to the diagram groups associated with the classes $\mathcal C_n$ of indecomposable geometrically fast groups generated by $n$ one-bump functions; for the terminology of geometrically fast sets, dynamical diagrams, and the classes $\mathcal C_n$, see \cite{BBKMZ,BelkStott}.  Belk and Stott proved that fast one-bump groups are isomorphic to the diagram groups determined by their dynamical diagrams, and used this description to prove that pseudo-$F_4$ is isomorphic to $F_4$ \cite{BelkStott}.  Although the original fast one-bump groups are not, in general, closed subgroups of $F$, the corresponding diagram groups can be realized as closed subgroups of $F$ using methods from \cite{GolanSapirClosedSubgroupsF}.  The author has proved that, for every $n\ge2$, every group in the class $\mathcal C_n$ is isomorphic to $F_n$; the cases $n=2,3$ were previously known, and the case $n=4$ is the theorem of Belk and Stott.  Thus, for the classes $\mathcal C_n$, this answers the strong form of a question of Brin and Zaremsky from \cite{BurilloBuxNucinkis}.  The author has also constructed, for every pair $m,n$ with $n>m\ge2$, a maximal subgroup of $F_m$ isomorphic to $F_n$; in particular, for every $m\geq 2$, Thompson's group $F_m$ has countably many non-isomorphic maximal subgroups of infinite index. 

\subsection*{Organization of the paper}

Section~\ref{sec:preliminaries} contains detailed preliminaries on words and Cantor spaces, tree diagrams and the groups $F_n$, abelianization, tree-automata, cores and closed subgroups, generation criteria, lifts between the interval and Cantor models, transducers, and determinization.  Section~\ref{sec:closed-maximal-copies} proves that closed maximal copies of $F_n$ in $F_m$ are realized by liftable interval conjugacies.  Section~\ref{sec:semi-sync-characterization} proves the characterization of Cantor-space conjugators by semi-synchronization (Theorem~\ref{thm:conjugator-intro}).  Section~\ref{sec:conjugating-closed-subgroups} develops the pullback and forward constructions for conjugating closed subgroups, together with the geometric form of the determinization; it also contains the proof that $\vec F_3\cong F_4$, the realization of the Brin--Navas group, and the transducer realizations of the known minimally-acting maximal subgroups of $F$.  Section~\ref{sec:binary-descending-chain} constructs the descending chain for $F$ and proves that it has trivial intersection.  Section~\ref{sec:standard-chain-and-twisting} introduces residue inflation, constructs the standard chain for $F_n$, and proves the main theorem via the compatible twisting argument.

\section{Preliminaries}
\label{sec:preliminaries}

Throughout the paper, $n,m\ge 2$.  We use left-to-right composition: if $f$ and $g$ are maps, then
\((fg)(x)=g(f(x)).\)
If $H$ is a subgroup of a homeomorphism group and $\psi$ is a homeomorphism, we write
\(H^\psi=\psi^{-1}H\psi .\)

\subsection{Words, prefix codes, Cantor spaces, and intervals}

Let
\[
        X_n=\{0,1,\ldots,n-1\},\qquad
        \W_n=X_n^*,\qquad
        \C_n=X_n^{\N}.
\]
Thus $\W_n$ is the set of all finite words over $X_n$, including the empty word $\eps$, and $\C_n$ is the Cantor space of one-sided infinite words over $X_n$.  If $u,v\in\W_n$ and $\omega\in\C_n$, their concatenations are denoted $uv$ and $u\omega$.  The length of $u$ is denoted $|u|$.

For $u\in\W_n$, the corresponding cylinder is
\[
        U_u=u\C_n=\{u\omega\mid \omega\in\C_n\}.
\]
The cylinder sets are clopen and form a basis for the topology on $\C_n$.

A subset $P\subseteq \W_n$ is a \emph{prefix code} if no element of $P$ is a proper prefix of another element of $P$.  A finite prefix code $P$ is \emph{complete} if the cylinders $\{U_u\mid u\in P\}$ form a partition of $\C_n$.  Equivalently, every infinite word in $\C_n$ has a unique prefix belonging to $P$.  Complete finite prefix codes are exactly the sets of branches of finite full $n$-ary trees, as described below.

We order $X_n$ by $0<1<\cdots<n-1$ and give $\C_n$ the induced lexicographic order.  We use the usual coding map
\[
        \rho_n:\C_n\to[0,1],
        \qquad
        \rho_n(\omega_1\omega_2\cdots)=\sum_{j\ge1}\frac{\omega_j}{n^j}.
\]
The map $\rho_n$ identifies exactly the two base-$n$ expansions of each $n$-adic rational in $(0,1)$.  We write $\omega\sim_{\rho_n}\eta$ if $\rho_n(\omega)=\rho_n(\eta)$.  The nontrivial equivalence classes are precisely
\[
        u i (n-1)^\infty \sim_{\rho_n} u(i+1)0^\infty
        \qquad (u\in\W_n,\ 0\le i<n-1).
\]
All other equivalence classes are singletons.  The points of $\Z[1/n]\cap(0,1)$ will be called the \emph{$n$-adic rationals} in $(0,1)$.

For $u\in\W_n$, we write
\([u]=\rho_n(U_u)\)
for the closed $n$-adic interval associated with $u$.  Thus $[u]$ has endpoints $\rho_n(u0^\infty)$ and $\rho_n(u(n-1)^\infty)$.

If $A\subseteq \C_n$ is nonempty, we let $\Root(A)$ be the longest word $u\in\W_n$ such that $A\subseteq U_u$.  If no nonempty word has this property, then $\Root(A)=\eps$.  Equivalently, $\Root(A)$ is the greatest common prefix of the elements of $A$.

A nonempty word $u\in\W_n$ is called \emph{inner} if it is not of the form $0^k$ and is not of the form $(n-1)^k$.  The empty word will not be called inner.

\subsection{Finite $n$-ary trees and tree diagrams}

A \emph{finite full $n$-ary tree} is a finite rooted planar tree in which every vertex has either no children or exactly $n$ children, ordered from left to right and labelled by $0,1,\ldots,n-1$.  A vertex with no children is a \emph{leaf}.  A vertex with $n$ children is a \emph{father vertex}.  The tree consisting of one father vertex and its $n$ children is an \emph{$n$-caret}.

Every vertex of a finite full $n$-ary tree determines a unique path from the root.  Reading the labels of the edges along this path gives a word in $\W_n$.  If the vertex is a leaf, this word is called a \emph{branch} of the tree.  We usually identify a leaf with its branch label.  The set of branches of a finite full $n$-ary tree is a finite complete prefix code, and every finite complete prefix code arises in this way.

An \emph{$n$-ary tree diagram} is a pair $(T_+,T_-)$ of finite full $n$-ary trees with the same number of leaves.  The tree $T_+$ is the domain tree and $T_-$ is the range tree.  If the branches of $T_+$ are
\(u_1<\cdots<u_k\)
and the branches of $T_-$ are
\(v_1<\cdots<v_k\)
in lexicographic order, then the diagram has \emph{branch pairs}
\[
        u_i\to v_i \qquad (1\le i\le k).
\]
The diagram defines a finite prefix replacement homeomorphism of $\C_n$ by
\[
        u_i\omega\mapsto v_i\omega \qquad (\omega\in\C_n),
\]
and, after applying the quotient map $\rho_n$, a piecewise-linear homeomorphism of $[0,1]$ which maps $[u_i]$ linearly onto $[v_i]$ for each $i$.

If an $n$-caret is attached to the same leaf of both $T_+$ and $T_-$, the represented prefix replacement does not change.  This operation is called an \emph{expansion}.  The inverse operation, removing a common $n$-caret from both trees, is called a \emph{reduction}.  Two tree diagrams are \emph{equivalent} if one can pass from one to the other by a finite sequence of expansions and reductions.  Every tree diagram is equivalent to a unique reduced tree diagram; see Brown's treatment of the groups $F_n$ and the standard binary discussion of Cannon--Floyd--Parry \cite{Brown,CannonFloydParry}. 

One can define a product of two tree diagrams by the common-expansion procedure.  If $(T_+,T_-)$ and $(R_+,R_-)$ are tree diagrams, choose a finite full $n$-ary tree $S$ which is a common expansion of $T_-$ and $R_+$.  Expanding the first diagram gives an equivalent diagram $(T'_+,S)$, and expanding the second gives an equivalent diagram $(S,R'_-)$.  The product is represented by $(T'_+,R'_-)$.  This operation is well defined on equivalence classes and corresponds to composition of the associated prefix replacement maps.

\subsection{The Higman--Thompson groups $F_n$}

The Higman--Thompson group $F_n$ is the group of all increasing piecewise-linear homeomorphisms of $[0,1]$ whose breakpoints lie in $\Z[1/n]$ and whose slopes are integral powers of $n$ \cite{Brown}.  Equivalently, $F_n$ is the group of all order-preserving homeomorphisms of $\C_n$ represented by finite prefix replacements.  Thus $f\in F_n$ if and only if there are finite complete prefix codes
\[
        P=\{u_1<\cdots<u_k\},\qquad Q=\{v_1<\cdots<v_k\},
\]
such that
\[
        f(u_i\omega)=v_i\omega
        \qquad(1\le i\le k,\ \omega\in\C_n).
\]
In this case the induced interval map sends $[u_i]$ linearly onto $[v_i]$ with slope $n^{|u_i|-|v_i|}$.

Equivalently again, $F_n$ is the group of equivalence classes of $n$-ary tree diagrams under expansions and reductions, with multiplication as described in the previous subsection.  If $(T_+,T_-)$ is a tree diagram whose branch pairs are $u_i\to v_i$, then it represents the element of $F_n$ acting by $u_i\omega\mapsto v_i\omega$ on $\C_n$, or linearly from $[u_i]$ to $[v_i]$ in the interval model.  These three models are naturally identified; see \cite{Brown,CannonFloydParry}.

We shall often say that an element $f\in F_n$ \emph{has the branch pair} $u\to v$.  This means that some, not necessarily reduced, tree diagram representing $f$ has $u\to v$ as a branch pair.  Equivalently, $f$ maps the interval $[u]$ linearly onto the interval $[v]$.
  
For a word $u=u_1\cdots u_r\in\W_n$, let
\[
        \sigma_n(u)=\sum_{j=1}^r u_j \pmod {n-1}.
\]
If \(\alpha\in \Z[1/n]\cap(0,1)\) and \(\alpha=\rho_n(u0^\infty)\) for some \(u\in\W_n\), then \(\sigma_n(\alpha):=\sigma_n(u)\) is well defined.  The orbits of the action of \(F_n\) on \(\Z[1/n]\cap(0,1)\) are precisely the sets
\[
        D_{n,i}=\{\alpha\in \Z[1/n]\cap(0,1)\mid \sigma_n(\alpha)=i\},
        \qquad 0\le i\le n-2.
\]
The following lemma records the standard orbit criterion for the action of \(F_n\) on \(n\)-adic intervals; see \cite[Lemma~2.8 and Remark~2.9]{GolanSapirGenerationFn}.

\begin{lemma}[Branch-pair criterion for inner words]\label{lem:branch-pair-criterion-prelim}
Let \(u,v\in\W_n\) be inner words.  There exists an element \(f\in F_n\) having the branch pair \(u\to v\) if and only if
\[
        \sigma_n(u)=\sigma_n(v).
\]

More generally, let
\[
        u_1,\ldots,u_k,\ v_1,\ldots,v_k
\]
be inner words such that the words \(u_1,\ldots,u_k\) are pairwise incomparable, the words \(v_1,\ldots,v_k\) are pairwise incomparable, and the corresponding intervals occur from left to right:
\[
        \max [u_i]\le \min [u_{i+1}]
        \quad\text{and}\quad
        \max [v_i]\le \min [v_{i+1}]
        \qquad(1\le i<k).
\]
Assume further that
\[
        \sigma_n(u_i)=\sigma_n(v_i)
        \qquad(1\le i\le k),
\]
and that adjacency is preserved, that is,
\[
        [u_i]\cap [u_{i+1}]\neq\emptyset
        \quad\Longleftrightarrow\quad
        [v_i]\cap [v_{i+1}]\neq\emptyset
        \qquad(1\le i<k).
\]
Then there exists an element \(f\in F_n\) having the branch pairs
\[
        u_i\to v_i
        \qquad(1\le i\le k).
\]
\end{lemma}

For $F=F_2$ we use the standard generators $x_0,x_1$ of Cannon--Floyd--Parry \cite{CannonFloydParry}, given on the Cantor space by
\[
    x_0:\begin{cases}
        00\eta\mapsto 0\eta,\\
        01\eta\mapsto 10\eta,\\
        1\eta\mapsto 11\eta,
    \end{cases}
    \qquad
    x_1:\begin{cases}
        0\eta\mapsto 0\eta,\\
        100\eta\mapsto 10\eta,\\
        101\eta\mapsto 110\eta,\\
        11\eta\mapsto 111\eta.
    \end{cases}
\]

If $u\in\W_n$, we denote by $F_n[u]$ the subgroup of $F_n$ supported in the interval $[u]$.  It is naturally isomorphic to $F_n$: if $g\in F_n$ has branch pairs $a_j\to b_j$, then the copy $g[u]\in F_n[u]$ has branch pairs $ua_j\to ub_j$, together with identity branch pairs outside $[u]$.

\subsection{Abelianization and finite-index subgroups}
\label{subsec:abelianization}

We use the abelianization of \(F_n\) coming from Brown's infinite presentation
\[
        F_n=
        \left\langle x_0,x_1,x_2,\ldots\ \middle|\ 
        x_k^{-1}x_i x_k=x_{i+n-1}\ \text{for }k<i
        \right\rangle .
\]
      Abelianizing gives
\(F_n/[F_n,F_n]\cong\Z^n.\)
Indeed, in the abelianization,
\[
        [x_i]=[x_{i+n-1}]
        \qquad(i\ge1),
\]
and
\[
        [x_0],[x_1],\ldots,[x_{n-1}]
\]
form a free basis.  See Brown \cite{Brown}; see also the finite
presentation in \cite{GubaSapirDiagramGroups}.

We denote the abelianization map by
\[
        \ab_n:F_n\longrightarrow F_n/[F_n,F_n]\cong\Z^n.
\]
Thus, if
\(e_0,e_1,\ldots,e_{n-1}\)
is the standard basis of \(\Z^n\), then
\(\ab_n(x_0)=e_0,\)
and, for \(i\ge1\),
\(\ab_n(x_i)=e_j,\)
where
\[
        j\in\{1,\ldots,n-1\}
        \qquad\text{and}\qquad
        i\equiv j\pmod{n-1}.
\]

We recall also that for $F=F_2$, one often considers another map as the ``standard'' abelianization map.

Let
\[
        \lambda_0,\lambda_1:F_n\to\Z,
\]
defined by
\[
        \lambda_0(f)=\log_n f'(0^+),
        \qquad
        \lambda_1(f)=\log_n f'(1^-).
\]
Consider the map $$\pi_n\colon F_n\to\mathbb{Z}^2$$ mapping 
$$f\mapsto (\lambda_0(f),\lambda_1(f)).$$
The chain rule  implies that the derived subgroup $[F_n,F_n]$ is contained in the kernel of $\pi_n$. The map $\pi_n$ is onto. Since for $F$, the abelianization is isomorphic to $\mathbb{Z}^2$ and $\mathbb{Z}^2$ is Hopfian, it follows that $\pi_2\colon F\to\mathbb{Z}^2$ has kernel $[F,F]$, and is often used as the abelianization map. 
                        
The derived subgroup \([F_n,F_n]\) is infinite and simple, and every nontrivial normal
subgroup of \(F_n\) contains \([F_n,F_n]\); see Brown \cite{Brown} and the standard
references on the Higman--Thompson groups \cite{CannonFloydParry}.  Consequently every
finite-index subgroup of \(F_n\) contains \([F_n,F_n]\).  Indeed, the normal core (in the
group-theoretic sense, i.e., the intersection of all conjugates) of a finite-index
subgroup is a nontrivial finite-index normal subgroup, and every such normal
subgroup contains \([F_n,F_n]\).

Thus finite-index subgroups of \(F_n\) correspond to finite-index subgroups of the free
abelian group
\(F_n/[F_n,F_n]\cong\Z^n.\)
In particular, a subgroup \(H\le F_n\) is not contained in any proper finite-index subgroup of $F_n$ if and only if
\(H[F_n,F_n]=F_n\); equivalently, if and only if
\(\ab_n(H)=F_n/[F_n,F_n]\).

\subsection{Minimality and locally-moving}
\label{subsec:interval-dynamics-prelim}

Let $I$ be an open interval and let $G\le \Homeo_+(I)$.  The action of $G$ on $I$ is called \emph{minimal} if every $G$-orbit is dense in $I$.  Equivalently, there is no nonempty proper closed subset of $I$ which is invariant under $G$.

For a nonempty open subinterval $J\subseteq I$, let
\[
        G_J=\{g\in G\mid \supp(g)\subseteq J\}.
\]
The group $G$ is called \emph{locally moving} if, for every nonempty open subinterval $J\subseteq I$, the subgroup $G_J$ has no global fixed point in $J$.  The standard action of $F_n$ on $(0,1)$ is locally moving: every point $x$ of a nonempty open subinterval $J\subseteq(0,1)$ lies in the interior of some interval with $n$-adic endpoints $[a,b]$, and some element of $F_n$ supported in that interval moves $x$.
              
\subsection{Tree-automata and accepted diagram groups}
\label{subsec:tree-automata}

We now define the tree-automata used to encode closed subgroups.  The definitions are the $n$-ary analogues of the binary definitions used for subgroups of $F$ in \cite{GolanGeneration,GolanMaximalF}; see also the $F_n$ version in \cite{GolanSapirGenerationFn}.  Some of the cited results are stated in the literature for $F=F_2$.  In the places where we use them below, the proofs depend only on the tree-diagram formalism and pass verbatim from binary carets to $n$-ary carets.

\begin{definition}\label{def:tree-automaton}
A \emph{rooted deterministic $n$-ary automaton} is a triple
\[
        \mathcal A=(Q,\tau,q_0)
\]
where $Q$ is a set, whose elements are called \emph{states}; $q_0\in Q$ is a distinguished state, called the \emph{root} or \emph{initial state}; and
\[
        \tau:Q\times X_n\dashrightarrow Q
\]
is a partially defined transition function.  The word ``deterministic'' refers to the fact that, for fixed $q\in Q$ and $i\in X_n$, there is at most one state $p\in Q$ with $\tau(q,i)=p$.

The transition function extends to finite words in the usual way: $\tau(q,\eps)=q$, and
\[
        \tau(q,ui)=\tau(\tau(q,u),i)
\]
whenever the right-hand side is defined.  Thus $\tau(q,u)$ is defined precisely when the word $u$ can be read from the state $q$. We often use the notation $q\cdot u$ for $\tau(q,u)$, when $\tau(q,u)$ is defined.

An \emph{$n$-ary tree-automaton} is a rooted deterministic $n$-ary automaton satisfying the following two conditions.
\begin{enumerate}[label=(\roman*)]
    \item Every state is reachable from the root: for every $q\in Q$ there is a word $u\in\W_n$ such that $\tau(q_0,u)=q$.
    \item For every $q\in Q$, either $\tau(q,i)$ is undefined for all $i\in X_n$, or $\tau(q,i)$ is defined for every $i\in X_n$.
\end{enumerate}
A state with no outgoing transitions is called a \emph{leaf}.  A state with $n$ outgoing transitions is called a \emph{father state}.  The automaton is \emph{full} if it has no leaves.
\end{definition}

The definition above is completely equivalent to the following graph description.  We draw a directed edge labelled $i$ from $q$ to $p$ whenever $\tau(q,i)=p$.  Since $\tau$ is a function, from a state $q$ there is at most one outgoing edge with a given label.  The two defining conditions for a tree-automaton say that every state is reached by a directed labelled path starting at the root, and that each state has either no outgoing edges or exactly one outgoing edge labelled $i$ for each $i\in X_n$.

\begin{definition}
An $n$-ary tree-automaton $\mathcal A=(Q,\tau,q_0)$ is \emph{folded} if no two distinct father states have exactly the same ordered list of children.  Equivalently, whenever $p,q\in Q$ are father states and
\[
        \tau(p,i)=\tau(q,i)\quad\text{for every }i\in X_n,
\]
then $p=q$.
\end{definition}

\begin{remark}
In \cite{GolanMaximalF}, in the binary case, what we call here a folded tree-automaton is called a rooted tree-automaton.  In the present paper we use the term tree-automaton for the possibly non-folded deterministic object of Definition~\ref{def:tree-automaton}, and folded tree-automaton when the additional foldedness condition holds.  This convention is convenient because non-folded tree-automata will occur several times as intermediate objects before we pass to a folded quotient.  Using the word folded explicitly emphasizes the operation being performed.  Related non-folded objects appear in the $F_n$ generation paper under the terminology of rooted semi tree-automata and semi-cores \cite{GolanSapirGenerationFn}.
\end{remark}

A word $u\in\W_n$ is \emph{readable} in an $n$-ary tree-automaton $\mathcal A=(Q,\tau,q_0)$ if $\tau(q_0,u)$ is defined.  In that case we also write
\(u^+=\tau(q_0,u).\)
A finite full $n$-ary tree is \emph{readable} in $\mathcal A$ if each of its branches is readable.

Let $(T_+,T_-)$ be an $n$-ary tree diagram.  We say that $(T_+,T_-)$ is \emph{readable} in $\mathcal A$ if both $T_+$ and $T_-$ are readable in $\mathcal A$.  If the branch pairs of the diagram are $u_1\to v_1,\ldots,u_k\to v_k$, then the diagram is \emph{accepted} by $\mathcal A$ if it is readable and
\[
        u_i^+=v_i^+
        \qquad\text{for every }i.
\]
We let $\D(\mathcal A)$ be the set of elements of $F_n$ which admit at least one tree diagram accepted by $\mathcal A$.

\begin{lemma}\label{lem:accepted-diagrams-form-subgroup}
Let $\mathcal A$ be a folded $n$-ary tree-automaton.  Then the following hold.
\begin{enumerate}[label=(\roman*)]
    \item If a tree diagram is accepted by $\mathcal A$, then the reduced tree diagram equivalent to it is also accepted by $\mathcal A$.
    \item If a tree diagram is accepted by $\mathcal A$, then an equivalent tree diagram is accepted by $\mathcal A$ if and only if its two trees are readable in $\mathcal A$.
    \item If two elements of $F_n$ admit accepted tree diagrams, then their product admits an accepted tree diagram.
\end{enumerate}
Consequently $\D(\mathcal A)$ is a subgroup of $F_n$.
\end{lemma}

\begin{proof}
For $n=2$, these are \cite[Lemmas~2.7--2.10]{GolanMaximalF}.  The same argument for $n$-ary carets appears in the $F_n$ setting in \cite[Lemma~2.12]{GolanSapirGenerationFn}.
\end{proof}

We next record the basic language of morphisms and quotients, since it will be used in the construction of folded quotients and later for closed overgroups.

\begin{definition}\label{def:morphisms-quotients}
Let $\mathcal A=(Q,\tau,q_0)$ and $\mathcal B=(P,\sigma,p_0)$ be $n$-ary tree-automata.  A \emph{morphism} $\varphi:\mathcal A\to\mathcal B$ is a map $\varphi:Q\to P$ such that $\varphi(q_0)=p_0$ and, whenever $\tau(q,i)$ is defined,
\[
        \sigma(\varphi(q),i) \text{ is defined and }
        \varphi(\tau(q,i))=\sigma(\varphi(q),i).
\]
Equivalently, in the graph model, $\varphi$ sends the root to the root and preserves directed labelled edges.

A morphism is \emph{surjective} if it is onto on states and on directed labelled edges.  That is, every state of $\mathcal B$ is $\varphi(q)$ for some $q\in Q$, and every edge
\[
        p\xrightarrow{\ i\ }p'
        \quad\text{in }\mathcal B
\]
is the image of an edge $q\xrightarrow{\ i\ }q'$ in $\mathcal A$.  A \emph{quotient} of $\mathcal A$ is a tree-automaton $\mathcal B$ equipped with a surjective morphism $\mathcal A\to\mathcal B$.
\end{definition}

Since every state of a tree-automaton is reachable from the root and transitions are deterministic, a morphism between two tree-automata is unique if it exists.

\begin{remark}\label{rem:quotients-congruences}
Let $\mathcal A=(Q,\tau,q_0)$ be an $n$-ary tree-automaton.  Quotients of $\mathcal A$ are equivalently described by equivalence relations on $Q$ compatible with the labelled transitions as follows.

An equivalence relation $\sim$ on $Q$ is an \emph{automaton congruence} if, whenever $q\sim r$ and both $\tau(q,i)$ and $\tau(r,i)$ are defined, we have
\[
        \tau(q,i)\sim\tau(r,i).
\]
Given such a congruence, the quotient $\mathcal A/{\sim}$ has state set $Q/{\sim}$, root $[q_0]$, and a transition
\[
        [q]\xrightarrow{\,i\,}[p]
\]
whenever there exists a representative $q'\in[q]$ such that $\tau(q',i)$ is defined and $p=\tau(q',i)$.  This transition is independent of the chosen representative.  Conversely, if $\varphi:\mathcal A\to\mathcal B$ is a surjective morphism, then the relation $q\sim r$ if and only if $\varphi(q)=\varphi(r)$ is an automaton congruence, and $\mathcal B$ is naturally isomorphic to $\mathcal A/{\sim}$.
\end{remark}
        
\begin{lemma}\label{lem:morphism-implies-inclusion}
If $\varphi:\mathcal A\to\mathcal B$ is a morphism of $n$-ary tree-automata, then
\[
        \D(\mathcal A)\subseteq \D(\mathcal B).
\]
\end{lemma}

\begin{proof}
If a diagram is accepted by $\mathcal A$, then for each branch pair $u\to v$ the words $u$ and $v$ end at the same state of $\mathcal A$.  Applying the morphism, they end at the same state of $\mathcal B$.  Thus the same diagram is accepted by $\mathcal B$.
\end{proof}

\begin{definition}\label{def:folded-quotient}
Let $\mathcal A$ be an $n$-ary tree-automaton.  The \emph{folded quotient} of
$\mathcal A$, denoted $\overline{\mathcal A}$, is the quotient of
$\mathcal A$ by the smallest automaton congruence for which the quotient
automaton is folded.  Equivalently, it is obtained by repeatedly identifying
two father states whose ordered lists of children agree, until no such pair
remains.
\end{definition}

\begin{lemma}
\label{lem:folded-quotient-finite-depth}
Let $\mathcal A=(Q,\tau,q_0)$ be an $n$-ary tree-automaton.  For
$p,q\in Q$, their images in $\overline{\mathcal A}$ are equal if and only if
there exists a finite complete prefix code $P\subseteq\W_n$, readable from
both $p$ and $q$, such that
\[
        p\cdot w=q\cdot w
        \qquad\text{for every }w\in P .
\]
Moreover, folding does not change readability: a word is readable from a state
$p$ in $\mathcal A$ if and only if it is readable from the image of $p$ in
$\overline{\mathcal A}$.

In particular, if $\mathcal A$ is full, then the first condition is equivalent
to the existence of $k\ge0$ such that
\[
        p\cdot w=q\cdot w
        \qquad\text{for every }w\in X_n^k .
\]
\end{lemma}

\begin{proof}
Write \(p\sim q\) if the displayed finite-prefix-code condition holds.  A
standard common-refinement argument shows that \(\sim\) is an equivalence
relation.  It is a congruence: if \(p\sim q\) and both have \(i\)-children,
then either \(p=q\), or the \(i\)-part of a witnessing prefix code witnesses
\(p\cdot i\sim q\cdot i .\)
Also, a leaf is \(\sim\)-equivalent only to itself, since the only complete
prefix code readable from a leaf is \(\{\eps\}\).  Thus the quotient by
\(\sim\) does not create or destroy readable words.

The quotient \(\mathcal A/{\sim}\) is folded.  Indeed, if \(p,q\) are father
states and
\[
        p\cdot i\sim q\cdot i
        \qquad(i\in X_n),
\]
choose witnessing prefix codes \(P_i\) for these \(n\) equivalences.  Then
\[
        \bigcup_{i\in X_n} iP_i
\]
is a finite complete prefix code witnessing \(p\sim q\).

Conversely, let \(\theta\) be any automaton congruence such that
\(\mathcal A/{\theta}\) is folded.  If \(P\) witnesses \(p\sim q\), induction
on the finite full tree with branch set \(P\) shows that
\(p\equiv q\pmod\theta\).  The case \(P=\{\eps\}\) is immediate.  Otherwise
write
\[
        P=\bigcup_{i\in X_n} iP_i .
\]
By induction,
\[
        p\cdot i\equiv q\cdot i\pmod\theta
        \qquad(i\in X_n),
\]
and foldedness of \(\mathcal A/{\theta}\) gives
\(p\equiv q\pmod\theta\).  Hence \(\sim\) is contained in every congruence
whose quotient is folded.  Since \(\mathcal A/{\sim}\) is folded, \(\sim\) is
exactly the folding congruence.

Finally, if \(\mathcal A\) is full, any witnessing finite complete prefix code
can be refined to a level \(X_n^k\), and equality persists under this
refinement.
\end{proof}

\begin{lemma}\label{lem:folding-preserves-D}
For every $n$-ary tree-automaton $\mathcal A$,
\[
        \D(\mathcal A)=\D(\overline{\mathcal A}).
\]
\end{lemma}

\begin{proof}
The inclusion
\[
        \D(\mathcal A)\subseteq\D(\overline{\mathcal A})
\]
follows from Lemma~\ref{lem:morphism-implies-inclusion}.

Conversely, let $f\in\D(\overline{\mathcal A})$, and choose a tree diagram
for $f$ accepted by $\overline{\mathcal A}$, with branch pairs
\[
        u_1\to v_1,\ldots,u_r\to v_r .
\]
By Lemma~\ref{lem:folded-quotient-finite-depth}, folding does not change
readability, so all the words $u_j$ and $v_j$ are readable in $\mathcal A$.
Moreover, since $u_j^+$ and $v_j^+$ have the same image in
$\overline{\mathcal A}$, the same lemma gives a finite complete prefix code
$P_j$ such that
\[
        (u_jw)^+=(v_jw)^+
        \qquad\text{in }\mathcal A
        \qquad(w\in P_j).
\]
Now expand the branch pair $u_j\to v_j$ by attaching the finite full tree with
branch set $P_j$ to both leaves.  Doing this for every $j$ gives an expansion
of the original tree diagram, hence a diagram representing the same element
$f$, whose branch pairs are
\[
        u_jw\to v_jw
        \qquad(1\le j\le r,\ w\in P_j).
\]
All these branch pairs are accepted by $\mathcal A$.  Therefore
$f\in\D(\mathcal A)$, and so
\[
        \D(\overline{\mathcal A})\subseteq\D(\mathcal A).
\]
\end{proof}
\begin{definition}
A subgroup $H\le F_n$ is \emph{closed} if there exists an $n$-ary
tree-automaton $\mathcal A$ such that
\[
        H=\D(\mathcal A).
\]
By Lemma~\ref{lem:folding-preserves-D}, one may equivalently require
$\mathcal A$ to be folded. This agrees with the standard
definition of closed subgroups as diagram groups over rooted tree-automata
\cite{GolanMaximalF,GolanSapirGenerationFn}.
\end{definition}

There is an equivalent interval formulation.  A function $f\in F_n$ is a
\emph{piecewise-$H$ function} if there is a finite subdivision
\(0=a_0<a_1<\cdots<a_k=1\)
of $[0,1]$ into intervals such that, for every $j$, the restriction of $f$ to
$[a_{j-1},a_j]$ agrees with the restriction of some element of $H$.  Since
$f$ and the elements of $H$ have only finitely many $n$-adic breakpoints, this
condition is equivalently witnessed by a finite subdivision whose endpoints
are $n$-adic.  In tree-diagram language, this is equivalent to saying that
$f$ has a not necessarily reduced tree diagram all of whose branch pairs occur
as branch pairs of elements of $H$.

\begin{theorem}
\label{thm:piecewise-closed}
A subgroup $H\le F_n$ is closed if and only if every piecewise-$H$ function
in $F_n$ belongs to $H$.
\end{theorem}

\begin{proof}
For $F=F_2$ this is \cite[Theorem~5.6]{GolanGeneration}; see also
\cite[Section~2.4]{GolanMaximalF}.  The $n$-ary statement is the same
tree-diagram argument; compare the $F_n$ formulation of the closure in
\cite[Lemma~2.17]{GolanSapirGenerationFn}.
\end{proof}

\subsection{The core and closure of a subgroup}
\label{subsec:core-closure}

Let $H\le F_n$ be generated by a set $S$ of reduced tree diagrams.  We
construct the core $\C(H)$ in two stages.

First form a rooted directed edge-labelled graph $\Gamma(S)$ as follows.  For
each reduced diagram $(T_+^s,T_-^s)$ in $S$, view the two trees as directed
edge-labelled graphs, with all edges directed away from the root and labelled
by $0,\ldots,n-1$ from left to right.  Identify the root of $T_+^s$ with the
root of $T_-^s$, and identify each leaf of $T_+^s$ with the corresponding leaf
of $T_-^s$; here corresponding means that the leaves have the same position
in the left-to-right order.  Finally, identify the roots of all the resulting
graphs, over all $s\in S$, to one root.

The graph $\Gamma(S)$ may fail to be deterministic: a vertex may have several
outgoing edges with the same label.  The first step is \emph{determinization}.
Whenever a vertex has two outgoing edges with the same label, identify those
two edges and identify their terminal vertices.  Repeat this operation until
no such pair remains.  Equivalently, take the quotient of the vertex set by
the smallest equivalence relation which makes the outgoing edge with a fixed
label, when it exists, unique.  The resulting rooted labelled graph is an
$n$-ary tree-automaton: every state is reachable from the root, and every
state has either no outgoing edges or exactly one outgoing edge labelled $i$
for each $i\in X_n$.

The second step is \emph{folding}.  Starting from the tree-automaton obtained
after determinization, pass to its folded quotient in the sense of
Definition~\ref{def:folded-quotient}.  Equivalently, identify father states
with identical ordered lists of children, and continue until no such pair
remains.

The folded tree-automaton obtained after these two steps is the \emph{core} of
$H$ and is denoted $\C(H)$.  If the generating set $S$ is infinite, the same
construction can be described as the quotient by the smallest equivalence
relation generated by the determinization and folding requirements above.
The core is independent of the chosen generating set and of the order in which
the identifications are performed, as in the binary case
\cite{GolanSapirSubgroupsF,GolanGeneration,GolanMaximalF}; see also the
$F_n$ construction in \cite{GolanSapirGenerationFn}.

By construction, $\C(H)$ accepts every element of $H$.  The \emph{closure} of
$H$ is
\(\Cl(H)=\D(\C(H)).\)
The closure is the smallest closed subgroup of $F_n$ containing $H$;
equivalently, it is the subgroup of all piecewise-$H$ functions.  In
particular $H\le\Cl(H)$, and $H$ is closed if and only if $H=\Cl(H)$
\cite{GolanGeneration,GolanMaximalF,GolanSapirGenerationFn}.

\begin{definition}\label{def:existence-property}
Let $\mathcal A$ be an $n$-ary tree-automaton.  We say that $\mathcal A$ has
the \emph{existence property} if whenever $u,v\in\W_n$ are readable in
$\mathcal A$ and
\[
        u^+=v^+,
\]
there exists an element of $\D(\mathcal A)$ having the branch pair
\[
        u\to v .
\]
\end{definition}

\begin{lemma}\label{lem:core-existence-property}
Let $H\le F_n$.  If $u,v\in\W_n$ are readable in $\C(H)$, then
\[
        u^+=v^+ \text{ in }\C(H)
\]
if and only if there exists an element of $\Cl(H)$ having the branch pair
$u\to v$.  In particular, $\C(H)$ has the existence property.
\end{lemma}

\begin{proof}
For $F=F_2$ this is \cite[Lemma~2.19]{GolanMaximalF}.  The proof is unchanged
for $n$-ary carets; compare the corresponding $F_n$ core construction in
\cite{GolanSapirGenerationFn}.
\end{proof}

\begin{definition}
Let $\mathcal A$ be an $n$-ary tree-automaton.  An \emph{extension} of
$\mathcal A$ is obtained by attaching a finite or infinite full $n$-ary tree
at each leaf of $\mathcal A$.  A \emph{reduction} is the reverse operation.
The automaton $\mathcal A$ is \emph{reduced} if it is not a nontrivial
extension of another $n$-ary tree-automaton.

A \emph{core automaton} is a folded, reduced $n$-ary tree-automaton with the
existence property.
\end{definition}

A \emph{pre-core} for a closed subgroup $K\le F_n$ is an $n$-ary
tree-automaton $\mathcal A$ such that
$\D(\mathcal A)=K$
and such that $\mathcal A$ has the existence property.  Thus a pre-core need
not be folded or reduced.  By the following lemma, after folding equal father states and deleting hanging trees one obtains
a core automaton for the same closed subgroup.

\begin{lemma}
\label{lem:existence-property-folding}
Let $K\le F_n$ be closed, and let $\mathcal A$ be an $n$-ary tree-automaton
such that
\[
        \D(\mathcal A)=K .
\]
Then $\mathcal A$ has the existence property if and only if its folded quotient
$\overline{\mathcal A}$ has the existence property.

Moreover, the existence property is preserved under extensions and reductions
of hanging trees.  Consequently, if $\mathcal A$ is a pre-core for $K$, then
the reduced automaton obtained from $\overline{\mathcal A}$ by deleting
hanging trees is a core automaton for $K$.
\end{lemma}

\begin{proof}
It is enough to check one elementary folding and one elementary reduction.

First let $\mathcal B$ be obtained from $\mathcal A$ by identifying two father
states $p,q$ with the same ordered list of children.  By
Lemma~\ref{lem:folding-preserves-D},
\[
        \D(\mathcal A)=\D(\mathcal B)=K .
\]
Suppose that $\mathcal A$ has the existence property, and let $u,v$ be
readable words with
\[
        u^+=v^+
        \qquad\text{in }\mathcal B .
\]
If this equality already holds in $\mathcal A$, there is nothing to prove.
Otherwise, in $\mathcal A$ the words $u$ and $v$ end at the two folded states
$p$ and $q$.  Since $p$ and $q$ have the same children,
\[
        (ui)^+=(vi)^+
        \qquad\text{in }\mathcal A
        \qquad(i\in X_n).
\]
For each $i\in X_n$, the existence property of $\mathcal A$ gives an element
of $K$ with branch pair
\(ui\to vi .\)
Patching these elements on the sibling intervals $[ui]$ gives a piecewise-$K$ element
$f$ mapping each $[ui]$ linearly onto $[vi]$.  All these linear pieces have the common
slope $n^{|u|-|v|}$ and agree at the shared endpoints, so $f$ maps $[u]$ linearly onto
$[v]$; that is, $f$ has the branch pair
\(u\to v .\)
Since $K$ is closed, this element belongs to $K=\D(\mathcal B)$.  Thus
$\mathcal B$ has the existence property.

Conversely, if $\mathcal B$ has the existence property and
$u^+=v^+$ in $\mathcal A$, then the same equality holds in $\mathcal B$.
Hence there is an element of
\[
        \D(\mathcal B)=\D(\mathcal A)
\]
with branch pair $u\to v$.  Thus $\mathcal A$ has the existence property.
Iterating over elementary foldings proves the assertion for the folded
quotient.

Now consider an extension $\mathcal A'$ of $\mathcal A$ obtained by attaching a
hanging full tree at a leaf.  Such an extension does not change the accepted
diagram group.  If $\mathcal A$ has the existence property and two words end at
the same state of the added hanging tree, then they have the form
\(ur,\qquad vr,\)
where $u$ and $v$ end at the attaching leaf of $\mathcal A$ and $r$ is the
common suffix inside the hanging tree.  The existence property in
$\mathcal A$ gives an element of $\D(\mathcal A)$ with branch pair $u\to v$;
expanding that branch pair gives the branch pair
\(ur\to vr .\)
Equalities at old states are handled directly in $\mathcal A$.  Hence
$\mathcal A'$ has the existence property.  The converse is immediate, since an
equality in $\mathcal A$ is also an equality in $\mathcal A'$ and
\[
        \D(\mathcal A')=\D(\mathcal A).
\]
Thus extensions and reductions of hanging trees preserve the existence
property.

Finally, if $\mathcal A$ is a pre-core for $K$, then folding preserves both
$\D(\mathcal A)=K$ and the existence property.  Deleting hanging trees also
preserves both.  The resulting automaton is folded, reduced, and has the
existence property; hence it is a core automaton for $K$.
\end{proof}
                                                                         
\begin{proposition}\label{prop:core-automata-characterization}
For every subgroup $H\le F_n$, the core $\C(H)$ is a core automaton.  Conversely, if $\mathcal A$ is a core automaton, then $\mathcal A$ is isomorphic to the core of the closed subgroup $\D(\mathcal A)$.
\end{proposition}

\begin{proof}
For $F=F_2$, this is the characterization of core automata in \cite[Section~6]{GolanMaximalF}.  The proof is the same for $n$-ary tree diagrams.
\end{proof}

\begin{remark}
The idea behind Proposition~\ref{prop:core-automata-characterization} is that the construction of $\C(H)$ starts from reduced tree diagrams, so it does not create irrelevant hanging trees; the folding step makes the automaton folded; and Lemma~\ref{lem:core-existence-property} gives the existence property.  Conversely, for a folded, reduced automaton with the existence property, the equalities between readable path endpoints are exactly the equalities forced by branch pairs of elements accepted by the automaton, so applying the core construction, starting with the set of all reduced tree-diagrams in $\mathcal D(\mathcal A)$, recovers the same automaton.
\end{remark}

For later use we record the standard terminology for states in a core automaton.  Let $\mathcal A$ be a core automaton.  The root is the initial state $q_0$.  A state is a \emph{left state} if it is reached by a word $0^k$ with $k\ge 1$, and a \emph{right state} if it is reached by a word $(n-1)^k$ with $k\ge 1$.  A state is an \emph{inner state} if it is reached by an inner word.  For core automata these classes have the expected disjointness properties, by the existence property: for instance, a left state cannot coincide with an inner state, since no element of $F_n$ has a branch pair $u\to 0^k$ with $u$ inner (such an element would map the left endpoint of $[u]$, which is positive, to $0$).  In the binary case this is the root/left/right/middle terminology of \cite{GolanMaximalF}.

\subsection{Closed overgroups and quotients of cores}
\label{subsec:closed-overgroups-quotients}

Morphisms of tree-automata give inclusions of the accepted diagram groups by Lemma~\ref{lem:morphism-implies-inclusion}.  For cores of closed subgroups, inclusions give morphisms between  the cores.

\begin{lemma}\label{lem:closed-overgroups-quotients-prelim}
Let $K\le L\le F_n$ be closed subgroups.  Then there is a morphism
\[
        \C(K)\longrightarrow \C(L).
\]
If $\C(K)$ is full, this morphism is surjective.
\end{lemma}

\begin{proof}
For $F=F_2$, this is the morphism construction used in \cite[Section~4]{GolanMaximalF}.  We recall the short argument, since it will be used later.  We first note that every word readable in $\C(K)$ is readable in $\C(L)$.  Indeed, the labelled graph $\Gamma(K)$ built, as in Subsection~\ref{subsec:core-closure}, from the reduced diagrams of all elements of $K$ is contained in the corresponding graph $\Gamma(L)$ for $L$, and the composite map $\Gamma(K)\subseteq\Gamma(L)\to\C(L)$ takes values in a deterministic folded automaton; hence it factors through the quotient $\C(K)$ of $\Gamma(K)$, so that every word readable in $\C(K)$ is readable in $\C(L)$.  Now, for a state $q$ of $\C(K)$, choose a readable word $u$ with $u^+=q$ and define
\(\varphi(q)=u^+\)
in $\C(L)$.  This is well-defined.  Indeed, if $u^+=v^+$ in $\C(K)$, then by Lemma~\ref{lem:core-existence-property} there is an element of $\Cl(K)=K$ with branch pair $u\to v$.  Since $K\le L=\Cl(L)$, the same element lies in $\Cl(L)$, and Lemma~\ref{lem:core-existence-property} gives $u^+=v^+$ in $\C(L)$.  The map clearly sends the root to the root and preserves transitions, so it is a morphism.

If $\C(K)$ is full, every word in $\W_n$ is readable in $\C(K)$.  Every state of $\C(L)$ is reached by some word $u$, and then it is the image of the state reached by $u$ in $\C(K)$.  Thus the morphism is onto on states.  It is also onto on edges: if
\[
        p\xrightarrow{\ i\ }p'
\]
is an edge of $\C(L)$, choose a word $u$ reaching $p$ in $\C(L)$.  Let $q$ be the state reached by $u$ in $\C(K)$.  Since $\C(K)$ is full, the edge $q\xrightarrow{\ i\ }\tau(q,i)$ exists, and the morphism sends this edge to $p\xrightarrow{\ i\ }p'$.  Hence the morphism is surjective.
\end{proof}

When $\C(K)$ is full, Lemma~\ref{lem:closed-overgroups-quotients-prelim} allows closed overgroups of $K$ to be studied through quotients of the core $\C(K)$.  This is the point of view used later when the closed overgroups in the descending chain are computed by finite automata.

\subsection{The core of $F_n$}

The following lemma gives the structure of the core of $F_n$, see \cite{GolanSapirGenerationFn}.

\begin{lemma}
\label{lem:explicit-core-of-Fn}
Let \(\mathcal K_n=(Q_n,\tau_n,q_0)\) be the full \(n\)-ary tree-automaton
with state set
\[
        Q_n
        =
        \{q_0,q_L,q_R\}
        \sqcup
        \{q^{\mathrm{in}}_j\mid j\in\Z/(n-1)\Z\}.
\]
Its transitions are
\[
        q_0\cdot 0=q_L,\qquad
        q_0\cdot(n-1)=q_R,\qquad
        q_0\cdot i=q^{\mathrm{in}}_i
        \quad(1\le i\le n-2),
\]
\[
        q_L\cdot 0=q_L,\qquad
        q_L\cdot i=q^{\mathrm{in}}_i
        \quad(1\le i\le n-1),
\]
\[
        q_R\cdot(n-1)=q_R,\qquad
        q_R\cdot i=q^{\mathrm{in}}_i
        \quad(0\le i\le n-2),
\]
and
\[
        q^{\mathrm{in}}_j\cdot i
        =
        q^{\mathrm{in}}_{j+i}
        \qquad
        (j\in\Z/(n-1)\Z,\ i\in X_n),
\]
where the subscripts are taken modulo \(n-1\).  Then
\[
        \C(F_n)\cong \mathcal K_n .
\]

Equivalently, for \(u\in\W_n\),
\[
        u^+=
        \begin{cases}
        q_0, & u=\eps,\\
        q_L, & u=0^k,\ k\ge1,\\
        q_R, & u=(n-1)^k,\ k\ge1,\\
        q^{\mathrm{in}}_{\sigma_n(u)}, & u \text{ is inner.}
        \end{cases}
\]
\end{lemma}

\subsection{Generation criteria}
\label{subsec:generation-criteria}

We shall use the following generation criterion for Thompson's group \(F\) from
\cite{GolanMaximalF}.

\begin{theorem}[Generation theorem for \(F\)]
\label{thm:generation-F-prelim}
Let \(H\le F\).  Then \(H=F\) if and only if
\[
        H[F,F]=F
        \qquad\text{and}\qquad
        [F,F]\le \Cl(H).
\]
Equivalently, if the image of \(H\) in the abelianization of \(F\) is all of \(\Z^2\) and
\[
        [F,F]\le \Cl(H),
\]
then \(H=F\).
\end{theorem}

We shall also use the sufficient generation criterion for \(F_n\) from
\cite[Theorem~3.19]{GolanSapirGenerationFn}.  Since the statement uses the semi-core, we
briefly recall the relevant notation.  If \(X\subseteq F_n\), the semi-core
\(L_{\mathrm{sem}}(X)\)
is obtained from the reduced tree diagrams of the elements of \(X\) by identifying all roots
and identifying the leaves in each branch pair, and then applying the determinization
procedure of Subsection~\ref{subsec:core-closure} only, without folding (in the
terminology of \cite{GolanSapirGenerationFn}, only foldings of type \(1\) are applied).
Thus \(L_{\mathrm{sem}}(X)\) is a tree-automaton, but it need not be folded.  Passing to
its folded quotient gives the core
\(\C(\langle X\rangle).\)

\begin{theorem}[Generation theorem for \(F_n\), \cite{GolanSapirGenerationFn}]
\label{thm:generation-Fn-prelim}
Let \(X\subseteq F_n\), and let
$
        H=\langle X\rangle.
$
Assume that the following conditions hold:
\begin{enumerate}
    \item
    $
            [F_n,F_n]\le \Cl(H).
    $
    
    \item
    $
                H[F_n,F_n]=F_n.
    $
    
    \item For every \(s\in\Z/(n-1)\Z\), there exist
    $
                h_s\in H
            $, $
            \alpha_s\in D_{n,s},
        $
    such that
    \[
            h_s(\alpha_s)=\alpha_s,
            \qquad
            h_s'(\alpha_s^-)=n,
            \qquad
            h_s'(\alpha_s^+)=1.
    \]
    \item There exists an inner word \(u\in\W_n\) such that, for all
    \(v_1,v_2\in\W_n\),
    \[
            \sigma_n(v_1)=\sigma_n(v_2)
            \quad\Longrightarrow\quad
            (uv_1)^+=(uv_2)^+
            \quad
            \text{in }L_{\mathrm{sem}}(X).
    \]
\end{enumerate}
Then
\[
        H=F_n.
\]
\end{theorem}

Finally, we shall use the core criterion for the closure condition from
\cite[Corollary~3.2]{GolanSapirGenerationFn}: for \(H\le F_n\),
\([F_n,F_n]\le \Cl(H)\)
if and only if \(\C(H)\) is full and has exactly \(n-1\) inner states.

\subsection{Lifts between interval and Cantor models}
\label{subsec:lifts-prelim}

The following lemma records precisely when a homeomorphism of Cantor spaces descends to a homeomorphism of the interval.  Related order and equivalence-relation considerations appear in transducer descriptions of automorphism groups of Thompson-like groups; see for instance \cite{AutT}.

\begin{lemma}\label{lem:lift-order-equivalence-prelim}
Let $\psi:\C_n\to\C_m$ be a homeomorphism.  The following are equivalent.
\begin{enumerate}[label=(\arabic*)]
    \item If $\omega\sim_{\rho_n}\eta$, then $\psi(\omega)\sim_{\rho_m}\psi(\eta)$.
    \item If $\alpha\sim_{\rho_m}\beta$, then $\psi^{-1}(\alpha)\sim_{\rho_n}\psi^{-1}(\beta)$.
    \item The map $\psi$ strictly preserves or strictly reverses the lexicographic order.
    \item The map $\psi$ induces a homeomorphism $\bar\psi:[0,1]\to[0,1]$ satisfying, in left-to-right notation,
    \[
        \psi\rho_m=\rho_n\bar\psi.
    \]
\end{enumerate}
If these conditions hold, then $\psi$ carries the two-point $\rho_n$-fibers bijectively to the two-point $\rho_m$-fibers.  Equivalently, $\bar\psi$ maps $n$-adic rationals in $(0,1)$ bijectively to $m$-adic rationals in $(0,1)$.
\end{lemma}

\begin{proof}
Assume (1).  Since $\psi$ is constant on $\rho_n$-fibers after applying $\rho_m$, it descends to a continuous surjection $\bar\psi:[0,1]\to[0,1]$ satisfying $\psi\rho_m=\rho_n\bar\psi$.  We prove that $\bar\psi$ is injective.  Suppose $x_1<x_2$ and $\bar\psi(x_1)=\bar\psi(x_2)=y$.  If $\bar\psi$ is constant on $[x_1,x_2]$, then $\psi$ maps the infinite set $\rho_n^{-1}([x_1,x_2])$ into the finite set $\rho_m^{-1}(y)$, contradicting injectivity of $\psi$.  If $\bar\psi$ is not constant on $[x_1,x_2]$, then by the intermediate value theorem some non-$m$-adic value between $y$ and another value of $\bar\psi$ is attained at least twice.  Its $\rho_m$-fiber is a singleton, again contradicting injectivity of $\psi$.  Hence $\bar\psi$ is injective, so it is a homeomorphism.  This proves (4).  Then (2) follows immediately: if $\rho_m(\alpha)=\rho_m(\beta)$, apply the identity $\psi\rho_m=\rho_n\bar\psi$ to $\psi^{-1}(\alpha)$ and $\psi^{-1}(\beta)$, and use injectivity of $\bar\psi$.

If (4) holds, then $\bar\psi$ is either increasing or decreasing.  The map $\rho_n$ is order-preserving except that it identifies adjacent pairs of words, and the same is true for $\rho_m$.  Thus the lift $\psi$ must respectively preserve or reverse lexicographic order; otherwise two ordered cylinder intervals in $\C_n$ would have images whose projections to $[0,1]$ are ordered incompatibly with the monotonicity of $\bar\psi$.  Hence (4) implies (3).  Finally, an order-preserving or order-reversing bijection of linearly ordered Cantor sets preserves adjacent pairs, and the adjacent pairs are exactly the nontrivial $\rho$-classes.  Hence (3) implies both (1) and (2).

It remains to show that (2) implies the other conditions.  Condition (2) for $\psi$ is condition (1) for $\psi^{-1}$.  By the implications already proved, it yields condition (3) for $\psi^{-1}$; since $\psi$ strictly preserves or strictly reverses the lexicographic order if and only if $\psi^{-1}$ does, condition (3) holds for $\psi$ as well.
\end{proof}

We shall use the same notation for an interval homeomorphism and its lift when no confusion is possible.  If $\phi\in\Homeo_+(0,1)$ maps $n$-adic rationals bijectively to $m$-adic rationals, then it has a unique order-preserving lift $\widetilde\phi:\C_n\to\C_m$ satisfying
\[
        \widetilde\phi\rho_m=\rho_n\phi
\]
in left-to-right notation.

\begin{lemma}\label{lem:closure-conjugacy-prelim}
Let $\psi:\C_n\to\C_m$ be a homeomorphism satisfying the equivalent conditions of Lemma~\ref{lem:lift-order-equivalence-prelim}.  Let $H\le F_n$ and suppose $H^\psi\le F_m$.  Then
\[
        \Cl_{F_m}(H^\psi)=\Cl_{F_n}(H)^\psi.
\]
In particular, if $H$ is closed, then $H^\psi$ is closed in $F_m$.
\end{lemma}

\begin{proof}
Because $\psi$ descends to an interval homeomorphism sending $n$-adic rationals to $m$-adic rationals, the image of a finite $n$-adic subdivision has finitely many $m$-adic endpoints and therefore is refined by a finite $m$-adic subdivision.  The same statement holds for $\psi^{-1}$.

If $f\in\Cl_{F_n}(H)$, then $f$ is piecewise equal to elements of $H$ on a finite $n$-adic partition.  Conjugating by $\psi$ gives a function $f^\psi\in F_m$ which is piecewise equal to elements of $H^\psi$ on a finite $m$-adic refinement of the image partition.  Hence $f^\psi\in\Cl_{F_m}(H^\psi)$.

Conversely, if $g\in\Cl_{F_m}(H^\psi)$, then $g^{\psi^{-1}}$ is piecewise equal to elements of $H$ on a finite $n$-adic refinement of the preimage partition, so $g^{\psi^{-1}}\in\Cl_{F_n}(H)$.  Thus $g\in\Cl_{F_n}(H)^\psi$.
\end{proof}

The next lemma is the form needed later to pass from closed overgroups to arbitrary overgroups in the descending chain.

\begin{lemma}\label{lem:generation-conjugated-copy}
Let $\xi:\C_2\to\C_2$ satisfy the equivalent conditions of Lemma~\ref{lem:lift-order-equivalence-prelim}, and suppose $F^\xi\le F$.  Let $L\le F^\xi$.  If
\[
        \Cl_F(L)=F^\xi
\]
and $L^{\xi^{-1}}\le F$ has full image in the abelianization of $F$, then
\[
        L=F^\xi.
\]
\end{lemma}

\begin{proof}
Apply Lemma~\ref{lem:closure-conjugacy-prelim} to the subgroup $L\le F^{\xi}\le F$ and to the homeomorphism $\xi^{-1}$.  We get
\[
        \Cl_F(L^{\xi^{-1}})=\Cl_F(L)^{\xi^{-1}}=(F^\xi)^{\xi^{-1}}=F.
\]
By assumption, $L^{\xi^{-1}}[F,F]=F$.  Theorem~\ref{thm:generation-F-prelim} gives $L^{\xi^{-1}}=F$, and conjugating by $\xi$ gives $L=F^\xi$.
\end{proof}

\subsection{Germs and type preservation}
\label{subsec:germs-type-prelim}

Let $G\le \Homeo(0,1)$ and let $x\in(0,1)$.  The \emph{stabilizer} of $x$ in $G$ is
\[
        \Stab_G(x)=\{g\in G\mid g(x)=x\}.
\]
Two elements $g,h\in \Stab_G(x)$ have the same \emph{germ at $x$} if they agree on some neighborhood of $x$.  This is an equivalence relation.  The equivalence class of $g$ is denoted $[g]_x$ and is called the germ of $g$ at $x$.  The set of all germs at $x$ forms a group under $[g]_x[h]_x=[gh]_x$ and is denoted $\mathcal G(G,x)$.

For $F_n$, the germ group detects the type of the point:
\[
\mathcal G(F_n,x)\cong
\begin{cases}
\Z^2, & x\in \Z[1/n]\cap(0,1),\\
\Z, & x\in \Q\setminus \Z[1/n],\\
\{1\}, & x\notin \Q.
\end{cases}
\]
Indeed, at an $n$-adic point there are independent left and right slope germs, at a rational non-$n$-adic point there is one periodic slope germ, and at an irrational fixed point every element fixing the point is locally the identity.  The isomorphism type of the germ group is preserved under conjugacy of actions; compare \cite{BleakLanoue,GSstabilizers}.

\begin{lemma}\label{lem:type-preservation-prelim}
Let $\phi\in\Homeo_+(0,1)$ and suppose $H=F_n^\phi\le F_m$ is closed.  Then $\phi$ maps $n$-adic rationals to $m$-adic rationals, rational non-$n$-adic points to rational non-$m$-adic points, and irrational points to irrational points.  Consequently, $\phi$ maps the set of $n$-adic rationals bijectively onto the set of $m$-adic rationals and has an order-preserving lift $\widetilde\phi:\C_n\to\C_m$.
\end{lemma}

\begin{proof}
Conjugation by $\phi$ gives an isomorphism
\[
        \mathcal G(F_n,x)\cong \mathcal G(H,\phi(x))
\]
for every $x\in(0,1)$.  Since $H\le F_m$, the group $\mathcal G(H,\phi(x))$ embeds in $\mathcal G(F_m,\phi(x))$.  Thus an $n$-adic point, whose germ group is $\Z^2$, must map to an $m$-adic point, since only $m$-adic points in $F_m$ have germ group containing $\Z^2$.  Similarly, a rational non-$n$-adic point can map only to a rational point.

It remains to rule out the possibility that an irrational maps to a rational, and that a rational non-$n$-adic point maps to an $m$-adic point.  The group $H\cong F_n$ is finitely generated, so its core is finite.  Also, the standard action of $F_n$ on $(0,1)$ is minimal, and hence the conjugate action of $H$ is minimal.  Therefore the core of $H$ has no leaves: a leaf would determine a nonempty clopen union of $m$-adic cylinders which is invariant under the closed action and would give an interval in which no orbit is dense.

Let $\beta\in(0,1)$ be rational and write one of its eventually periodic $m$-ary expansions as $pv^\infty$.  Reading the prefixes $pv^k$ in the finite full core of $H$, two of them end at the same state.  By the existence property of the core, there is an element of $H$ with a branch pair $pv^{k_1}\to pv^{k_2}$, giving a nontrivial germ at $\beta$.  Thus every rational point has nontrivial germ in $H$, and an irrational point of $F_n$, whose germ group is trivial, cannot map to a rational point.

Now suppose that $x$ is rational but not $n$-adic and that $\phi(x)=\beta$ is $m$-adic.  The point $\beta$ has two $m$-ary expansions, one ending in $0^\infty$ and one ending in $(m-1)^\infty$.  Applying the preceding paragraph to the two one-sided expansions gives elements of $H$ with nontrivial left and right germs at $\beta$.  Since $H$ is closed, we may cut these elements at the $m$-adic point $\beta$: one obtains an element whose germ is nontrivial only from the left and another whose germ is nontrivial only from the right.  These two germs generate a copy of $\Z^2$ in $\mathcal G(H,\beta)$, contradicting $\mathcal G(F_n,x)\cong\Z$.  Hence rational non-$n$-adic points map to rational non-$m$-adic points.  Since $\phi$ is a bijection and the three types partition $(0,1)$, the asserted bijection on $n$-adic rationals follows.  The lift follows from Lemma~\ref{lem:lift-order-equivalence-prelim}.
\end{proof}

\begin{corollary}\label{cor:closed-iff-liftable-prelim}
Let $\phi\in\Homeo_+(0,1)$ and suppose $F_n^\phi\le F_m$.  Then $F_n^\phi$ is closed in $F_m$ if and only if $\phi$ maps $n$-adic rationals bijectively onto $m$-adic rationals.  Equivalently, $\phi$ has an order-preserving lift $\widetilde\phi:\C_n\to\C_m$.
\end{corollary}

\begin{proof}
The forward implication is Lemma~\ref{lem:type-preservation-prelim}.  Conversely, if $\phi$ maps $n$-adic rationals bijectively to $m$-adic rationals, it has an order-preserving lift $\widetilde\phi$.  Since $F_n$ is closed in itself and $F_n^{\widetilde\phi}\le F_m$, Lemma~\ref{lem:closure-conjugacy-prelim} implies that $F_n^{\widetilde\phi}$ is closed in $F_m$.
\end{proof}

\subsection{Local actions}
\label{subsec:local-actions-prelim}

The following definitions are standard in the transducer approach to
Cantor-space dynamics; see Grigorchuk--Nekrashevych--Sushchanski\u{\i} and
later work on automorphisms of Thompson-like groups
\cite{GNS,AutG,AutT}.  We follow the terminology of \cite{AutG}, with minor
changes of notation.

We shall use local actions both for maps defined on all of a Cantor space and
for maps defined on a clopen subset.  Let
\(A\subseteq \C_n\)
be a nonempty clopen subset, and let
\(f:A\to \C_m\)
be a continuous injective map.  If \(u\in\W_n\) satisfies
\(U_u\subseteq A,\)
define
\[
        \theta_f(u)=\Root(f(U_u)).
\]
Since \(f\) is injective, the set \(f(U_u)\) is not a singleton, and hence this
greatest common prefix is a finite word.

The \emph{local action} of \(f\) at \(u\) is the unique continuous injective
map
\(f_u:\C_n\to\C_m\)
such that
\[
        f(u\omega)=\theta_f(u)f_u(\omega)
        \qquad(\omega\in\C_n).
\]
Equivalently, \(f_u\) is obtained by restricting \(f\) to the cylinder \(U_u\)
and then deleting the common prefix \(\theta_f(u)\) from every image sequence.
By construction,
\(\Root(f_u(\C_n))=\eps.\)
If the image \(f(A)\) is clopen in \(\C_m\), then \(f_u(\C_n)\) is clopen in
\(\C_m\).

When \(A=\C_n\), the local action \(f_u\) is defined for every
\(u\in\W_n\).  In this case the set of local actions of \(f\) is
\[
        \LA_f=\{f_u\mid u\in\W_n\}.
\]

The local actions satisfy the cocycle identity.  Namely, if
\(U_u\subseteq A\) and \(v\in\W_n\), then
\(f_{uv}=(f_u)_v .\)
Indeed, put
\[
        a=\theta_f(u),
        \qquad
        b=\theta_{f_u}(v).
\]
Then
\[
        f(uv\omega)
        =
        a f_u(v\omega)
        =
        ab (f_u)_v(\omega).
\]
Since \(b\) is the greatest common prefix of \(f_u(U_v)\), the word \(ab\) is
the greatest common prefix of \(f(U_{uv})\).  Thus
\(\theta_f(uv)=ab,\)
and cancellation of the common prefix gives
\(f_{uv}=(f_u)_v.\)
In particular, if \(f:\C_n\to\C_m\) is continuous and injective and
\(f_u=f_v\), then
\[
        f_{uw}=f_{vw}
        \qquad(w\in\W_n).
\]
                               
\subsection{Transducers}
\label{subsec:transducers-prelim}

An \emph{$(n,m)$-transducer} is a triple
\(T=(S,t,o)\)
where $S$ is a set of \emph{states},
\(t:S\times X_n\to S\)
is the \emph{transition function}, and
\(o:S\times X_n\to\W_m\)
is the \emph{output function}.  If the machine is in state $s$ and reads the input letter $i\in X_n$, it outputs the word $o(s,i)$ and moves to the state $t(s,i)$.

The functions $t$ and $o$ extend to finite words recursively.  We set
\[
        t(s,\eps)=s,\qquad o(s,\eps)=\eps,
\]
and
\[
        t(s,ui)=t(t(s,u),i),
        \qquad
        o(s,ui)=o(s,u)o(t(s,u),i).
\]
For an infinite input $\omega=i_1i_2\cdots\in\C_n$, the output is the infinite concatenation
\[
        o(s,\omega)=o(s,i_1)o(t(s,i_1),i_2)o(t(s,i_1i_2),i_3)\cdots .
\]
We assume throughout that transducers are \emph{nondegenerate}, meaning that this output is an infinite word in $\C_m$ for every state and every infinite input.

An \emph{initial transducer} is a quadruple
\(T_{s_0}=(S,t,o,s_0),\)
where $s_0\in S$ is the initial state.  Each state $s\in S$ induces a continuous map
\[
        h_{T_s}:\C_n\to\C_m,
        \qquad
        h_{T_s}(\omega)=o(s,\omega).
\]
The initial transducer represents the map $h_{T_{s_0}}$.

A state $s$ is called an \emph{injective state} if $h_{T_s}$ is injective, and a \emph{homeomorphism state} if $h_{T_s}:\C_n\to\C_m$ is a homeomorphism.  An initial transducer $T_{s_0}$ is called \emph{injective}, respectively a \emph{homeomorphism transducer}, if its initial state is an injective state, respectively a homeomorphism state.  In this paper we will mostly use homeomorphism transducers.  If $T_{s_0}$ is injective and $s=t(s_0,u)$ for some word $u$, then $s$ is also injective.

Two states $s_1,s_2$ are \emph{$\omega$-equivalent} if $h_{T_{s_1}}=h_{T_{s_2}}$.  For an injective state $s$, we say that $s$ has \emph{incomplete response} if, for some $i\in X_n$, the output $o(s,i)$ is a proper prefix of the actual greatest common prefix of $h_{T_s}(U_i)$.  Equivalently,
\[
        \Root\bigl(h_{T_{t(s,i)}}(\C_n)\bigr)\neq \eps.
\]
Indeed, from
\[
        h_{T_s}(i\omega)=o(s,i)h_{T_{t(s,i)}}(\omega)
\]
it follows that $o(s,i)$ is always a common prefix of $h_{T_s}(U_i)$; incomplete response means that after reading $i$, the machine has not yet output the whole common prefix of that image.

A state $s$ is \emph{accessible} from an initial state $s_0$ if $t(s_0,u)=s$ for some $u\in\W_n$.  An initial transducer is \emph{minimal} if every state is accessible, no state has incomplete response, and no two distinct states are $\omega$-equivalent.

Two initial transducers are called \emph{equivalent} if they represent the same map.  We shall use the standard minimization theorem for homeomorphism transducers: every homeomorphism transducer is equivalent to a unique minimal homeomorphism transducer, up to isomorphism \cite{GNS}.   In such a minimal representative every state is accessible; hence, if the represented map is a homeomorphism, every state is injective by the observation above, and has a clopen image in $\C_m$. 

\begin{definition}
A homeomorphism $\psi:\C_n\to\C_m$ is \emph{rational} if it is represented by a finite initial $(n,m)$-transducer.  If $\psi$ is rational, its unique finite minimal transducer is denoted
\[
        T^\psi=(S_\psi,t_\psi,o_\psi,s_\psi),
\]
where $s_\psi$ is the initial state.
\end{definition}

The following fundamental theorem was proved in \cite{GNS}.

\begin{theorem}[\cite{GNS}]
\label{thm:rational-iff-finite-local-actions-prelim}
Let $\psi:\C_n\to\C_m$ be a homeomorphism.  Then $\psi$ is rational if and only if $\LA_\psi$ is finite.
\end{theorem}

The theorem is best used together with the following explicit correspondence between local actions and states.

\begin{notation}[The state--local-action dictionary]
\label{not:state-local-action-dictionary}
Let $T_{s_0}=(S,t,o,s_0)$ be a minimal homeomorphism transducer, and let $h=h_{T_{s_0}}$.  For $u\in\W_n$, write
\[
        s_u=t(s_0,u),
        \qquad
        O_T(u)=o(s_0,u).
\]
Since the transducer has no incomplete response, an induction on $|u|$ gives
\[
        O_T(u)=\Root(h(U_u))=\theta_h(u).
\]
Moreover,
\[
        h(u\omega)=O_T(u)h_{T_{s_u}}(\omega)
        \qquad(\omega\in\C_n),
\]
and therefore the state $s_u$ induces the local action $h_u$:
\[
        h_u=h_{T_{s_u}}.
\]
Thus, in a minimal homeomorphism transducer, the state reached after reading $u$ induces exactly the local action at $u$.

Conversely, let $\psi:\C_n\to\C_m$ be a homeomorphism, and consider the set of local actions $\LA_\psi$.  One obtains a transducer from the local actions as follows.  The states are the distinct local actions,
\[
        S_\psi=\LA_\psi,
\]
the initial state is $s_\psi=\psi$, and for $f\in\LA_\psi$ and $i\in X_n$ we define
\[
        t_\psi(f,i)=f_i,
        \qquad
        o_\psi(f,i)=\theta_f(i).
\]
The identity
\[
        f(i\omega)=\theta_f(i)f_i(\omega)
\]
shows how the transducer acts from the state $f$.  Iterating this identity along an input word gives exactly the image of that word under $\psi$, so the resulting initial transducer represents $\psi$.  Every state is accessible because every state is of the form $\psi_u$ for some $u\in\W_n$; there is no incomplete response because the output on each letter is defined to be the full common prefix; and distinct states are distinct functions, hence are not $\omega$-equivalent.  Thus this transducer is isomorphic to the minimal transducer $T^\psi$.

For $T^\psi=(S_\psi,t_\psi,o_\psi,s_\psi)$ and $u\in\W_n$ we write
\[
        s^\psi_u=t_\psi(s_\psi,u),
        \qquad
        O^\psi(u)=o_\psi(s_\psi,u).
\]
Then
\[
        O^\psi(u)=\theta_\psi(u),
        \qquad
        h_{T_{s^\psi_u}}=\psi_u,
\]
and hence
\[
        \psi(u\omega)=O^\psi(u)\psi_u(\omega)
        \qquad(\omega\in\C_n).
\]
\end{notation}

\begin{definition}
\label{def:synchronizing-prelim}
Let $T=(S,t,o)$ be a finite $(n,m)$-transducer.  We say that $T$ is \emph{synchronizing at level $\ell$} if there is a map
\[
        \mathfrak s:X_n^\ell\to S
\]
such that, for every word $u\in X_n^\ell$ and every state $s\in S$,
\[
        t(s,u)=\mathfrak s(u).
\]
Equivalently, after reading any input word of length $\ell$, the resulting state depends only on that input word and not on the starting state.  A transducer is \emph{synchronizing} if it is synchronizing at some level.

A rational homeomorphism $\psi:\C_n\to\C_m$ is called \emph{bi-synchronizing} if the minimal transducer $T^\psi$ representing $\psi$ is synchronizing and the minimal transducer $T^{\psi^{-1}}$ representing $\psi^{-1}$ is synchronizing.
\end{definition}

\begin{remark}\label{rem:sync-terminology-prelim}
In \cite{AutG}, the property in Definition~\ref{def:synchronizing-prelim} is called \emph{strong synchronization}.  We use the shorter term \emph{synchronizing}.  This should not be confused with a \emph{synchronous} transducer, where each transition outputs exactly one letter.
\end{remark}

\begin{remark}
\label{rem:AutG-AutT-sync-prelim}
The synchronizing condition was isolated by Bleak--Cameron--Maissel--Navas--Olukoya in their transducer description of automorphisms of the Higman--Thompson groups $G_{n,r}$.  They prove that if $h\in\Homeo(\C_{n,r})$ satisfies
\[
        h^{-1}G_{n,r}h\le G_{n,r},
\]
then $h$ has finitely many local actions \cite[Corollary~6.17]{AutG}; the later synchronization part of their argument shows that the relevant minimal transducers are synchronizing, and their main theorem identifies $\operatorname{Aut}(G_{n,r})$ with the group of rational homeomorphisms of $\C_{n,r}$ represented by finite bi-synchronizing transducers \cite[Theorem~1.1]{AutG}.

Olukoya adapted this framework to the groups $T_{n,r}$.  In particular, he proves the one-sided statement that if $h\in\Homeo(\C_{n,r})$ satisfies
\[
        h^{-1}T_{n,r}h\le T_{n,r},
\]
then $h$ has finitely many local actions and its minimal transducer is synchronizing \cite[Corollary~4.16 and Corollary~5.2]{AutT}.  He then characterizes the normalizer, and hence the automorphism group of $T_{n,r}$, by bi-synchronizing transducers whose induced homeomorphisms respect the cyclic ordering \cite[Theorem~1.1 and Theorem~5.3]{AutT}.

The corresponding endpoint-preserving case for Brown's groups $F_{n,r}$ is not carried out in those papers.  Since, for fixed $n$, the groups $F_{n,r}$ are abstractly isomorphic to the group denoted $F_n$ here, the results of Section~\ref{sec:semi-sync-characterization} give the analogous description for this family as a consequence.
\end{remark}

\subsection{Product and inverse transducers}
\label{subsec:product-inverse-transducers-prelim}

Let $A=(S_A,t_A,o_A)$ be an $(n,m)$-transducer and let $B=(S_B,t_B,o_B)$ be an $(m,\ell)$-transducer.  The \emph{product transducer} $A*B$ is the $(n,\ell)$-transducer with state set $S_A\times S_B$ and with transition and output functions
\[
        t_{A*B}((s_A,s_B),i)=\bigl(t_A(s_A,i),\ t_B(s_B,o_A(s_A,i))\bigr),
\]
\[
        o_{A*B}((s_A,s_B),i)=o_B(s_B,o_A(s_A,i)).
\]
Here $t_B$ and $o_B$ are used in their extended sense on the word $o_A(s_A,i)\in\W_m$.  Thus the output produced by $A$ is immediately fed as input into $B$.  For initial transducers, the initial state is the pair of initial states.  The product represents the composition of the represented maps, in the left-to-right convention.  The product transducer need not be minimal, even if the two factors are minimal.

We shall also use a standard inverse construction.  Let
\(T_{s_0}=(S,t,o,s_0)\)
be a minimal initial $(n,m)$-transducer representing a homeomorphism $h_{T_{s_0}}:\C_n\to\C_m$.  For $s\in S$ and $w\in\W_m$, set
\[
        L_s(w)=\Root\bigl(h_{T_s}^{-1}(U_w)\bigr).
\]
The inverse transducer has states
\[
        S'=
        \{(s,w)\in S\times\W_m\mid L_s(w)=\eps,\ U_w\subseteq \im(h_{T_s})\},
\]
initial state $(s_0,\eps)$, and transition and output defined as follows.  For $(s,w)\in S'$ and $i\in X_m$, let
\(u=L_s(wi).\)
Then
\[
        t'((s,w),i)=\bigl(t(s,u),\ wi-o(s,u)\bigr),
        \qquad
        o'((s,w),i)=u,
\]
where $wi-o(s,u)$ denotes the suffix remaining after deleting the prefix $o(s,u)$ from $wi$.  The cited inverse construction guarantees that $o(s,u)$ is indeed a prefix of $wi$, so this suffix is well defined.

The intuition behind the construction is as follows.  The inverse machine reads letters from the output alphabet $X_m$.  Usually a single output letter is not enough to determine the next input letter of the original machine.  The second coordinate $w$ is a buffer: it records output already read by the inverse machine but not yet matched by a complete input prefix of the original machine.  The condition $L_s(w)=\eps$ says that the current buffer still does not determine a nonempty input prefix.  Once the enlarged buffer $wi$ determines a nonempty input prefix $u$, the inverse machine outputs $u$, moves the original transducer state from $s$ to $t(s,u)$, and removes the matched output $o(s,u)$ from the buffer.

\begin{proposition}[see \cite{GNS}]\label{prop:inverse-transducer-prelim}
The construction above is well defined.  The resulting initial $(m,n)$-transducer represents $h_{T_{s_0}}^{-1}$.  It need not be minimal, but every state is accessible and it has no states of incomplete response.  Moreover, if the original transducer is finite, then the inverse transducer is finite; equivalently, the inverse of a rational homeomorphism is rational.
\end{proposition}

\subsection{Nondeterministic automata and subset determinization}
\label{subsec:subset-construction-prelim}

We shall use the standard subset construction for nondeterministic automata
with $\eps$-edges.  We recall it explicitly because, in the sequel, we shall
apply it to automata whose edges are labelled by finite words rather than by
single letters.  Standard references for finite automata and the subset
construction include \cite{HopcroftMotwaniUllman,Sipser}.

Let $X$ be a finite alphabet.  An \emph{$\eps$-nondeterministic automaton over
$X$} consists of a set of states $P$, an initial state $p_0\in P$, and a
transition function
\[
        \Delta:P\times (X\cup\{\eps\})\to \mathcal P(P),
\]
where $\mathcal P(P)$ denotes the power set of $P$.  Thus, if the automaton is
in state $p$ and reads the symbol $a\in X\cup\{\eps\}$, then $\Delta(p,a)$ is
the set of all states to which the automaton is allowed to move.  If
$\Delta(p,a)=\emptyset$, then there is no such move.  If $q\in\Delta(p,a)$, we
draw an edge
\[
        p \xrightarrow{\ a\ } q .
\]
When $a\in X$, traversing this edge consumes the input letter $a$.  When
$a=\eps$, the edge may be traversed without consuming any input letter.  We do
not include accepting states, since we shall only use the underlying transition
structure.  When $P$ is finite, this is the usual notion of an $\eps$-NFA.

A path in such an automaton is a sequence of edges
\[
        p_0' \xrightarrow{\ a_1\ } p_1'
        \xrightarrow{\ a_2\ } \cdots
        \xrightarrow{\ a_k\ } p_k',
        \qquad a_j\in X\cup\{\eps\}.
\]
The \emph{label} of this path is the word in $X^*$ obtained by concatenating
the labels $a_1,\ldots,a_k$ and then deleting all occurrences of $\eps$.
Thus a path labelled by a word $u\in X^*$ may contain $\eps$-edges before the
first letter of $u$, after the last letter of $u$, or between two consecutive
letters of $u$.

For $A\subseteq P$ and $a\in X$, define
\[
        \Delta(A,a)=\bigcup_{p\in A}\Delta(p,a).
\]
Thus $\Delta(A,a)$ is the set of all states which can be reached from some
state in $A$ by reading the single letter $a$.

For $A\subseteq P$, its \emph{$\eps$-closure} is
\[
        E(A)=
        \{q\in P\mid q\text{ is reachable from some }p\in A
        \text{ by a path consisting only of }\eps\text{-edges}\}.
\]
The path of length zero is allowed, so $A\subseteq E(A)$.

For a word $u\in X^*$, let $R_u$ be the set of all states reachable from the
initial state by a path labelled by $u$:
\[
        R_u=
        \{q\in P\mid
        \text{there exists a path from }p_0\text{ to }q
        \text{ with label }u\}.
\]
Equivalently, the sets $R_u$ are defined recursively by
\(R_\eps=E(\{p_0\}),\)
and, for $a\in X$,
\[
        R_{ua}=E\bigl(\Delta(R_u,a)\bigr).
\]
In words, after reading $u$ one is in the $\eps$-closed set $R_u$; to read one
more letter $a$, one follows all possible $a$-edges out of states in $R_u$ and
then closes again under $\eps$-edges.

The \emph{determinized automaton} obtained by the subset construction has as
states the subsets $R_u\subseteq P$ which arise in this way.  Its initial
state is $R_\eps$, and its transition on the letter $a\in X$ is
\[
        R_u \xrightarrow{\ a\ } R_{ua}.
\]
The empty subset may occur; if it does, it is treated as an ordinary sink state.  This automaton is deterministic: the next state is uniquely determined by the
current subset $R_u$ and the input letter $a$.

We shall also use automata whose edges are labelled by finite words rather
than by single letters.  A \emph{word-labelled automaton over $X$} consists of
a set of states $P$, an initial state $p_0$, and a set of directed edges
\[
        e=(p,w,q),
        \qquad p,q\in P,\quad w\in X^*.
\]
We draw such an edge as
\[
        p \xrightarrow{\ w\ } q .
\]
To apply the subset construction to a word-labelled automaton, we first
replace it by an $\eps$-nondeterministic automaton over $X$.  If $w=\eps$,
then the edge $e=(p,\eps,q)$ is replaced by an $\eps$-edge from $p$ to $q$.
If
\[
        w=b_1b_2\cdots b_k
        \qquad (b_j\in X,\ k\ge 1),
\]
then $e$ is replaced by a path whose edge labels are the individual letters
$b_1,\ldots,b_k$.  More explicitly, if $k=1$, we replace $e$ by the single
edge
\[
        p \xrightarrow{\ b_1\ } q.
\]
If $k\ge 2$, we insert new subdivision states and replace $e$ by
\[
        p \xrightarrow{\ b_1\ } (e,b_2\cdots b_k)
        \xrightarrow{\ b_2\ } (e,b_3\cdots b_k)
        \longrightarrow \cdots \longrightarrow
        (e,b_k) \xrightarrow{\ b_k\ } q .
\]
The notation $(e,\rho)$ means that we are partway through the original edge
$e$, and that the unread suffix of the label of $e$ is $\rho$.  For example,
an edge
\(e=(p,011,q)\)
is replaced by
\[
        p \xrightarrow{\ 0\ } (e,11)
        \xrightarrow{\ 1\ } (e,1)
        \xrightarrow{\ 1\ } q .
\]

After replacing every word-labelled edge in this way, we apply the
$\eps$-closure and subset construction described above.  Thus the states of
the determinized automaton are the reachable $\eps$-closed subsets of the
subdivided automaton.
  
\section{Closed maximal copies and Cantor-space conjugators}
\label{sec:closed-maximal-copies}

In this section we record the general conjugacy framework which motivates the later constructions.  We study closed maximal subgroups of $F_m$ which are abstractly isomorphic to some $F_n$, and show that such subgroups are realized by conjugating the standard action of $F_n$ on the interval.  The results of this section are not needed for the explicit automaton computations later in the paper, but they explain why rational Cantor-space conjugators are the natural objects to study.

\begin{lemma}\label{lem:rank-obstruction-maximal-copy}
Let $m,n\ge 2$.  If $F_m$ has a maximal subgroup isomorphic to $F_n$, then $n\ge m$.
\end{lemma}

\begin{proof}
Let $H\cong F_n$ be a maximal subgroup of $F_m$, and let
\(\pi_m:F_m\to (F_m)_{\ab}\)
be the abelianization map.  By Subsection~\ref{subsec:abelianization}, the abelianization of $F_m$ is free abelian of rank $m$.

If $H$ has finite index in $F_m$, then $\pi_m(H)$ has finite index in $(F_m)_{\ab}$, and hence has rank $m$.  Suppose, then, that $H$ has infinite index.  By maximality, $H$ is not contained in any proper finite-index subgroup of $F_m$.  Therefore $\pi_m(H)$ is not contained in any proper finite-index subgroup of $(F_m)_{\ab}$.  A proper subgroup of a finitely generated free abelian group is always contained in a proper finite-index subgroup, so it follows that $\pi_m(H)=(F_m)_{\ab}$, again of rank $m$.

Thus in all cases $\pi_m(H)$ has rank $m$.  Since $H\cong F_n$ is generated by $n$ elements (namely, by $x_0,\ldots,x_{n-1}$; see Subsection~\ref{subsec:abelianization}), every quotient of $H$ is generated by at most $n$ elements.  The free abelian group of rank $m$ cannot be generated by fewer than $m$ elements, and hence $n\ge m$.
\end{proof}

We shall use the following Rubin-type conjugacy theorem for locally moving groups of homeomorphisms, due to Brum--Matte Bon--Rivas--Triestino.

\begin{theorem}[{\cite[Corollary~4.1.2]{BMBRT}}]\label{thm:rubin-locally-moving-body}
Let $I=(a,b)$ and let $G\le \Homeo_+(I)$ be locally moving.  Let
\[
        \alpha:G\hookrightarrow \Homeo_+(I)
\]
be an injective homomorphism, and assume that the subgroup $\alpha(G)$ acts minimally on $I$.  Assume also that some nontrivial element of $\alpha(G)$ has support bounded away from at least one endpoint of $I$.  Then there exists a homeomorphism $\phi:I\to I$ such that
\[
        \alpha(g)=\phi^{-1}g\phi
        \qquad(g\in G),
\]
where $G$ on the right-hand side is acting on $I$ by the given action.
\end{theorem}

We now prove that the minimality hypothesis in Theorem~\ref{thm:rubin-locally-moving-body} is automatic for maximal copies of $F_n$ in $F_m$.

\begin{lemma}\label{lem:maximal-copy-minimal}
Let $n\ge m\ge 2$, and let $H\le F_m$ be a maximal subgroup of $F_m$ isomorphic to $F_n$. Then the action of $H$ on $(0,1)$ is minimal.
\end{lemma}

\begin{proof}
Assume, toward a contradiction, that the action is not minimal.  Since $H\cong F_n$ is finitely generated, there is a nonempty minimal $H$-invariant closed set
\(\Lambda\subseteq (0,1)\)
for the action of $H$; see \cite[Proposition~2.1.12]{NavasCircle}.  By assumption, $\Lambda\ne(0,1)$.

Let
\[
        K=\{g\in F_m\mid g(\Lambda)=\Lambda\}
\]
be the setwise stabilizer of $\Lambda$ in $F_m$.  Then $H\le K$.  Since the standard action of $F_m$ on $(0,1)$ is minimal, the set $\Lambda$ is not invariant under all of $F_m$, and hence $K<F_m$.  By maximality of $H$, we have $K=H$.

If $\Lambda$ is finite, then all elements of $H$ fix each point of $\Lambda$, because they preserve orientation.  Thus $H$ is contained in the stabilizer in $F_m$ of a point of $(0,1)$, and by maximality it would equal that point stabilizer.  This is impossible: point stabilizers in $F_m$ are not isomorphic to $F_n$; for instance, they contain nontrivial commuting normal subgroups, see also \cite{GSstabilizers} for the binary case.  Hence $\Lambda$ is infinite.

The complement $(0,1)\setminus\Lambda$ has infinitely many connected components.  Indeed, if it had only finitely many components, then their boundary points would form a nonempty finite $H$-invariant subset of $(0,1)$, reducing again to the point-stabilizer contradiction above.

Let
\[
        G=\{h\in H\mid h(x)=x\text{ for every }x\in\Lambda\}
\]
be the pointwise stabilizer of $\Lambda$ in $H$.  Then $G\triangleleft H$.  Let $I$ be a component of $(0,1)\setminus\Lambda$.  Then $I$ contains an $m$-adic interval $[u]$, for some word $u\in\W_m$.  Every element of $F_m[u]$ is supported in $I$ and hence fixes $\Lambda$ pointwise; in particular, it preserves $\Lambda$ setwise, so that $F_m[u]\le K=H$, and therefore $F_m[u]\le G$.  In particular, $G$ is nontrivial.  Since $H\cong F_n$ and every nontrivial normal subgroup of $F_n$ contains the derived subgroup, we have $[H,H]\leq G$.

Now choose two distinct components $I,J$ of $(0,1)\setminus\Lambda$.  Since $[H,H]\le G$ fixes $\Lambda$ pointwise, it fixes the endpoints of each component of $(0,1)\setminus\Lambda$ and hence preserves each component setwise.  It follows that
\[
        N_I=\{g\in [H,H]\mid \supp(g)\subseteq I\}
\]
is a normal subgroup of $[H,H]$, and similarly for $N_J$.  Both are nontrivial: with $[u]\subseteq I$ as above, the subgroup $[F_m[u],F_m[u]]$ is nontrivial, contained in $[H,H]$, and supported in $I$.  The subgroups $N_I$ and $N_J$ have disjoint supports and therefore intersect trivially, contradicting the simplicity of $[H,H]$.
\end{proof}

\begin{proposition}\label{prop:closed-maximal-copy-realized-by-conjugacy}
Let $n\ge m\ge 2$, and let $H\le F_m$ be a closed maximal subgroup isomorphic to $F_n$.  Then there exists a homeomorphism
\[
        \phi\in \Homeo_+(0,1)
\]
such that
\[
        H=F_n^\phi .
\]
Moreover, $\phi$ maps $n$-adic rationals bijectively to $m$-adic rationals and has an order-preserving lift
\[
        \widetilde\phi:\C_n\to\C_m .
\]
\end{proposition}

\begin{proof}
Choose an abstract isomorphism
\[
        \alpha:F_n\to H\le \Homeo_+(0,1).
\]
The homomorphism $\alpha$ is injective, and by Lemma~\ref{lem:maximal-copy-minimal} the subgroup $\alpha(F_n)=H$ acts minimally on $(0,1)$.  By Subsection~\ref{subsec:interval-dynamics-prelim}, the standard action of $F_n$ is locally moving.

It remains only to verify the support hypothesis in Theorem~\ref{thm:rubin-locally-moving-body}.  The group $[H,H]$ is nontrivial, and since $H\le F_m$ we have
\([H,H]\le [F_m,F_m].\)
By the chain rule, every element of $[F_m,F_m]$ fixes a neighborhood of both endpoints $0$ and $1$.  Thus any nontrivial element of $[H,H]$ has support bounded away from the endpoints.
Theorem~\ref{thm:rubin-locally-moving-body} gives a homeomorphism $\phi\in\Homeo(0,1)$ such that $H=F_n^\phi$.

If $\phi$ is orientation-reversing, let
\[
        \iota:[0,1]\to[0,1],
        \qquad
        \iota(x)=1-x.
\]
The map $\iota$ normalizes $F_n$.  Replacing $\phi$ by $\iota\phi$ does not change the subgroup $F_n^\phi$, and the new conjugating homeomorphism is orientation-preserving.  Hence we may assume $\phi\in\Homeo_+(0,1)$.

Since $H=F_n^\phi$ is closed in $F_m$, Lemma~\ref{lem:type-preservation-prelim} and Corollary~\ref{cor:closed-iff-liftable-prelim} imply that $\phi$ maps the $n$-adic rationals bijectively onto the $m$-adic rationals and has an order-preserving lift $\widetilde\phi:\C_n\to\C_m$.
\end{proof}

\section{Cantor-space conjugators and semi-synchronization}
\label{sec:semi-sync-characterization}

In this section we characterize the homeomorphisms $\psi:\C_n\to\C_m$ for which
\(F_n^\psi\le F_m .\)
The first step is to prove that such a $\psi$ induces a homeomorphism of the interval quotients, and hence is either order-preserving or order-reversing on the Cantor spaces.  Once this is known, Lemma~\ref{lem:closure-conjugacy-prelim} implies that $F_n^\psi$ is closed in $F_m$.  Thus the subgroups studied in this section are precisely the lifted Cantor-space versions of interval conjugates of $F_n$ which land as closed subgroups of $F_m$.

The second step is to prove that $\psi$ has only finitely many local actions.  The proof follows the strategy of Bleak--Cameron--Maissel--Navas--Olukoya for $G_{n,r}$ and of Olukoya for $T_{n,r}$; see \cite[Section~6]{AutG} and \cite[Corollary~4.16]{AutT}.  We include the details because the endpoint-preserving groups $F_n$ require two additional modifications: the boundary rays $0^\infty$ and $(n-1)^\infty$ must be treated separately, and branch-pair existence for inner intervals is governed by the congruence invariant $\sigma_n$.

Recall that, for a word $u=u_1\cdots u_r\in\W_n$, we write
\[
        \sigma_n(u)=\sum_{j=1}^r u_j \pmod {n-1}.
\]
We also write
\[
        \W_n^{\mathrm{in}}
        =
        \{u\in\W_n\setminus\{\eps\}\mid u\ne 0^k\text{ and }u\ne (n-1)^k\text{ for all }k\ge1\}
\]
for the set of inner words.  If $u,v\in\W_n$, we write
\([u]\prec [v]\)
when the interval $[u]$ lies strictly to the left of $[v]$, that is, when $\max [u]<\min [v]$.

\begin{lemma}\label{lem:cantor-conjugator-descends}
Let $\psi:\C_n\to\C_m$ be a homeomorphism such that $F_n^\psi\le F_m$.  Then $\psi$ satisfies the equivalent conditions of Lemma~\ref{lem:lift-order-equivalence-prelim}.  In particular, $\psi$ either preserves or reverses the lexicographic order, it descends to a homeomorphism of $[0,1]$, and $F_n^\psi$ is closed in $F_m$.
\end{lemma}

\begin{proof}
The proof is analogous to the quotient-preservation argument for homeomorphisms normalizing the groups $T_{n,r}$; compare the construction of lifts from circle homeomorphisms in \cite[Section~3]{AutT}.  By Lemma~\ref{lem:lift-order-equivalence-prelim}, it is enough to prove that $\psi^{-1}$ maps $\rho_m$-equivalent points to $\rho_n$-equivalent points.

Suppose not.  Then there exist $\alpha,\beta\in\C_m$ such that
\[
        \alpha\sim_{\rho_m}\beta
        \qquad\text{but}\qquad
        \psi^{-1}(\alpha)\not\sim_{\rho_n}\psi^{-1}(\beta).
\]
We first dispose of the cases in which one of the preimages is an endpoint of
$\C_n$, since endpoints are not contained in any inner cylinder.

Suppose first that
\[
        \{\psi^{-1}(\alpha),\psi^{-1}(\beta)\}=\{0^\infty,(n-1)^\infty\}.
\]
Since $\alpha\ne\beta$ and $\alpha\sim_{\rho_m}\beta$, the points $\alpha$ and
$\beta$ are the two $m$-ary expansions of the $m$-adic point
$\rho_m(\alpha)\in(0,1)$; in particular $0^\infty\notin\{\alpha,\beta\}$,
because the $\rho_m$-fiber of $0$ is the singleton $\{0^\infty\}$.  Hence
$x=\psi^{-1}(0^\infty)$ is not an endpoint of $\C_n$, so
$\rho_n(x)\in(0,1)$.  Every $g\in F_m$ fixes $0^\infty$; applying this to
$g=\psi^{-1}f\psi$ for $f\in F_n$ gives
\[
        \psi(f(x))=g(0^\infty)=0^\infty=\psi(x),
\]
so $f(x)=x$ for every $f\in F_n$.  This is impossible, since some element of
$F_n$ moves the point $\rho_n(x)\in(0,1)$, and hence moves $x$.

Suppose next that exactly one of the two preimages is an endpoint; after
interchanging $\alpha$ and $\beta$ if necessary, say
$\psi^{-1}(\alpha)\in\{0^\infty,(n-1)^\infty\}$ while $\psi^{-1}(\beta)$ is
not an endpoint.  Choose an inner word $v\in\W_n^{\mathrm{in}}$ with
$\psi^{-1}(\beta)\in U_v$.  Choose $\gamma\in\C_m$ with
$\rho_m(\gamma)\ne\rho_m(\alpha)$ such that $y=\psi^{-1}(\gamma)$ is not an
endpoint of $\C_n$, and choose a prefix $w$ of $\gamma$ long enough that
$\rho_m(\alpha)\notin[w]$.  By continuity, there is an inner word
$q'\in\W_n^{\mathrm{in}}$ with $y\in U_{q'}$ and $\psi(U_{q'})\subseteq U_w$.
Replacing $q'$ by an inner extension in the appropriate congruence class, we
obtain an inner word $q\in\W_n^{\mathrm{in}}$ with $U_q\subseteq U_{q'}$, so
that $\psi(U_q)\subseteq U_w$, and with $\sigma_n(q)=\sigma_n(v)$.  By
Lemma~\ref{lem:branch-pair-criterion-prelim}, there is $f\in F_n$ with the
branch pair $v\to q$.  Put $g=\psi^{-1}f\psi\in F_m$.  Since $f$ fixes the
endpoints of $\C_n$, we get $g(\alpha)=\alpha$, while
$g(\beta)=\psi(f(\psi^{-1}(\beta)))\in\psi(U_q)\subseteq U_w$, so
\[
        \rho_m(g(\beta))\in[w]
        \qquad\text{while}\qquad
        \rho_m(g(\alpha))=\rho_m(\alpha)\notin[w].
\]
Thus $g(\alpha)\not\sim_{\rho_m}g(\beta)$, contradicting the fact that $g$
preserves $\sim_{\rho_m}$.

We may therefore assume that neither preimage is an endpoint of $\C_n$.
After interchanging $\alpha$ and $\beta$ if necessary, choose inner words $u,v\in\W_n^{\mathrm{in}}$ with $[u]\prec[v]$ such that
\[
        \psi^{-1}(\alpha)\in U_u,
        \qquad
        \psi^{-1}(\beta)\in U_v.
\]

We now choose target branch words for a contradiction.  Since cylinders form a basis and extensions of an inner word can be chosen in any prescribed congruence class modulo $n-1$, we may choose inner words $p,q\in\W_n^{\mathrm{in}}$ such that
\[
        [p]\prec[q],
        \qquad
        \sigma_n(p)=\sigma_n(u),
        \qquad
        \sigma_n(q)=\sigma_n(v),
\]
and such that $\psi(U_p)$ and $\psi(U_q)$ are contained in two $m$-ary cylinders whose projections to $[0,1]$ are disjoint.  Indeed, choose two non-endpoint points of $\C_n$ in increasing order whose images are not $\rho_m$-equivalent, take small source cylinders around them whose images lie in disjoint $m$-ary cylinders, and then refine inside these cylinders to achieve the required congruence classes.

By Lemma~\ref{lem:branch-pair-criterion-prelim}, there exists $f\in F_n$ with branch pairs
\(u\to p, \qquad v\to q.\)
Set
\(g=\psi^{-1}f\psi .\)
By assumption $g\in F_m$.  Since $\psi^{-1}(\alpha)\in U_u$ and $\psi^{-1}(\beta)\in U_v$, we have
\[
        g(\alpha)=\psi(f(\psi^{-1}(\alpha)))\in \psi(U_p),
\]
and similarly
\[
        g(\beta)=\psi(f(\psi^{-1}(\beta)))\in \psi(U_q).
\]
The sets $\psi(U_p)$ and $\psi(U_q)$ are contained in $m$-ary cylinders whose projections to $[0,1]$ are disjoint, so $g(\alpha)$ and $g(\beta)$ are not $\rho_m$-equivalent.  This contradicts the fact that every element of $F_m$ preserves the equivalence relation $\sim_{\rho_m}$.

Therefore $\psi^{-1}$ maps $\rho_m$-equivalent points to $\rho_n$-equivalent points.  Lemma~\ref{lem:lift-order-equivalence-prelim} implies that $\psi$ descends to a homeomorphism of the interval and is either order-preserving or order-reversing.  Since $F_n$ is closed in itself and $F_n^\psi\le F_m$, Lemma~\ref{lem:closure-conjugacy-prelim} implies that $F_n^\psi$ is closed in $F_m$.
\end{proof}

The next two lemmas analyze the local actions of $\psi$ along the two boundary rays $0^\infty$ and $(n-1)^\infty$.  They show that, after sufficiently many repetitions of the initial boundary letter, reading one more boundary letter changes only the finite output prefix and not the induced local action.

\begin{lemma}\label{lem:zero-ray-prefix}
Let $\psi:\C_n\to\C_m$ be a homeomorphism such that $F_n^\psi\le F_m$.  Then there exist $D_0\ge0$ and $k>0$ such that, for all $d\ge D_0$ and all $\omega\in\C_n$, the following hold.
\begin{enumerate}[label=(\arabic*)]
    \item If $\psi$ is order-preserving, then
    \[
        \psi(0^{d+1}\omega)=0^k\psi(0^d\omega).
    \]
    \item If $\psi$ is order-reversing, then
    \[
        \psi(0^{d+1}\omega)=(m-1)^k\psi(0^d\omega).
    \]
\end{enumerate}
Consequently,
\[
        \psi_{0^{d+1}}=\psi_{0^d}
        \qquad(d\ge D_0),
\]
and hence, for all $r\ge D_0$ and all $w\in\W_n$,
\[
        \psi_{0^r w}=\psi_{0^{D_0}w}.
\]
\end{lemma}

\begin{proof}
Let $f\in F_n$ be an element with the branch pair $0\to 00$, so that
\[
        f(0\omega)=00\omega
        \qquad(\omega\in\C_n).
\]
Set $g=\psi^{-1}f\psi\in F_m$.

Assume first that $\psi$ is order-preserving.  Then $\psi(0^\infty)=0^\infty$.  Since $g\in F_m$ fixes $0^\infty$, there exist integers $A,B\ge0$ such that
\[
        g(0^A\eta)=0^B\eta
        \qquad(\eta\in\C_m).
\]
For every $\eta\in U_0\setminus\{0^\infty\}$, the sequence $f^j(\eta)$ converges to $0^\infty$.  Hence $g^j(\psi(\eta))=\psi(f^j(\eta))$ converges to $0^\infty$.  It follows that $B>A$; otherwise the local prefix replacement of $g$ near $0^\infty$ would not move all sufficiently near points toward $0^\infty$.  Let $k=B-A$.

Choose $D_0$ such that $\psi(U_{0^{D_0}})\subseteq U_{0^A}$.  For every $d\ge D_0$ and every $\omega\in\C_n$, we have $\psi(0^d\omega)\in U_{0^A}$, and therefore
\[
        \psi(0^{d+1}\omega)
        =\psi(f(0^d\omega))
        =g(\psi(0^d\omega))
        =0^k\psi(0^d\omega).
\]

The order-reversing case is identical, except that $\psi(0^\infty)=(m-1)^\infty$ and the local prefix replacement of $g$ is taken at the right endpoint.  Thus, for some $A<B$,
\[
        g((m-1)^A\eta)=(m-1)^B\eta,
\]
and the same conjugacy calculation gives
\[
        \psi(0^{d+1}\omega)=(m-1)^{B-A}\psi(0^d\omega)
\]
for all sufficiently large $d$.

It remains to pass from the displayed identities to local actions.  In either case the identity has the form
\[
        \psi(0^{d+1}\omega)=a^k\psi(0^d\omega)
        \qquad(\omega\in\C_n),
\]
where $a$ is either $0$ or $m-1$.  Writing both sides in terms of roots and local actions gives
\[
        \theta_\psi(0^{d+1})\psi_{0^{d+1}}(\omega)
        =a^k\theta_\psi(0^d)\psi_{0^d}(\omega).
\]
Since $\theta_\psi(0^{d+1})$ is the greatest common prefix of the set on the right as $\omega$ varies, and since $\Root(\psi_{0^d}(\C_n))=\eps$, we have
\[
        \theta_\psi(0^{d+1})=a^k\theta_\psi(0^d).
\]
Cancelling this common prefix gives $\psi_{0^{d+1}}=\psi_{0^d}$.  The final assertion follows by repeated use of the cocycle identity for local actions from Subsection~\ref{subsec:local-actions-prelim}.
\end{proof}

\begin{lemma}\label{lem:one-ray-prefix}
Let $\psi:\C_n\to\C_m$ be a homeomorphism such that $F_n^\psi\le F_m$.  Then there exist $D_1\ge0$ and $k>0$ such that, for all $d\ge D_1$ and all $\omega\in\C_n$, the following hold.
\begin{enumerate}[label=(\arabic*)]
    \item If $\psi$ is order-preserving, then
    \[
        \psi((n-1)^{d+1}\omega)=(m-1)^k\psi((n-1)^d\omega).
    \]
    \item If $\psi$ is order-reversing, then
    \[
        \psi((n-1)^{d+1}\omega)=0^k\psi((n-1)^d\omega).
    \]
\end{enumerate}
Consequently,
\[
        \psi_{(n-1)^{d+1}}=\psi_{(n-1)^d}
        \qquad(d\ge D_1),
\]
and hence, for all $r\ge D_1$ and all $w\in\W_n$,
\[
        \psi_{(n-1)^r w}=\psi_{(n-1)^{D_1}w}.
\]
\end{lemma}

\begin{proof}
This is the right-endpoint analogue of Lemma~\ref{lem:zero-ray-prefix}.  One uses an element $f\in F_n$ with branch pair
\((n-1)\to (n-1)(n-1),\)
so that $f((n-1)\omega)=(n-1)(n-1)\omega$, and applies the same endpoint-prefix calculation to $g=\psi^{-1}f\psi\in F_m$.
\end{proof}

To prove that $\psi$ has finitely many local actions, it remains to control local actions at words which do not lie entirely on one of the two boundary rays.

\begin{proposition}\label{prop:uniform-tail-agreement}
Let $\psi:\C_n\to\C_m$ be a homeomorphism such that $F_n^\psi\le F_m$.  For every pair $u,v\in\W_n^{\mathrm{in}}$ with
\[
        \sigma_n(u)=\sigma_n(v),
\]
there exists $k=k(u,v)$ such that
\[
        \psi_{uw}=\psi_{vw}
\]
for every $w\in\W_n$ with $|w|\ge k$.
\end{proposition}

\begin{proof}
This proposition is the analogue, in the endpoint-preserving setting of $F_n$, of the tail-agreement phenomenon used for $G_{n,r}$ in \cite[Proposition~6.6--Corollary~6.16]{AutG}.  We give a direct proof adapted to $F_n$; the branch-pair criterion lets us replace the more elaborate ``almost the same fashion'' machinery by a single element with a prescribed inner branch pair.

By Lemma~\ref{lem:branch-pair-criterion-prelim}, there exists an element $f\in F_n$ with branch pair
\(u\to v.\)
Let
\(g=\psi^{-1}f\psi\in F_m.\)
Choose a finite $m$-ary tree diagram for $g$, and let $P$ be the finite complete prefix code of its domain branches.  The cylinders $\{U_a\mid a\in P\}$ form a finite clopen partition of $\C_m$.  Pulling this partition back by $\psi$ gives a finite clopen partition of $\C_n$.  Intersecting with $U_u$, we obtain a finite clopen partition of $U_u$.

Every finite clopen partition of $U_u$ is refined by a sufficiently deep cylinder partition of $U_u$.  Hence there exists $k\ge0$ such that, for every $w\in\W_n$ with $|w|\ge k$, the set $\psi(U_{uw})$ is contained in a single domain cylinder $U_a$ of the chosen diagram for $g$.

Fix such a word $w$, and let $a\in P$ be such that
\[
        \psi(U_{uw})\subseteq U_a.
\]
Let $a\to b$ be the corresponding branch pair of the diagram of $g$.  Since $\psi(U_{uw})\subseteq U_a$, the root $\theta_\psi(uw)=\Root(\psi(U_{uw}))$ extends $a$; write
\(\theta_\psi(uw)=ar.\)
For every $\omega\in\C_n$, using $\psi g=f\psi$ in left-to-right notation and the fact that $f(uw\omega)=vw\omega$, we get
\[
        g(\psi(uw\omega))=\psi(vw\omega).
\]
On the cylinder $U_a$, the element $g$ acts by replacing the prefix $a$ with $b$.  Therefore
\[
        \psi(vw\omega)
        =g(ar\,\psi_{uw}(\omega))
        =br\,\psi_{uw}(\omega).
\]
The image of the local action $\psi_{uw}$ has empty root, so the root of the set on the right, as $\omega$ varies, is $br$.  Thus
\(\theta_\psi(vw)=br,\)
and after cancelling this common prefix we obtain
\(\psi_{vw}=\psi_{uw}.\)
This proves the proposition.
\end{proof}

\begin{definition}\label{def:boundary-reduced-word}
Let
\[
        D=\max\{D_0,D_1\},
\]
where $D_0$ and $D_1$ are constants satisfying the conclusions of Lemmas~\ref{lem:zero-ray-prefix} and \ref{lem:one-ray-prefix}.  For $u\in\W_n$, the \emph{$D$-boundary-reduced form} of $u$ is obtained as follows.  If $u$ begins with a maximal block $0^r$ with $r>D$, replace this initial block by $0^D$.  If $u$ begins with a maximal block $(n-1)^r$ with $r>D$, replace this initial block by $(n-1)^D$.  If neither case occurs, leave $u$ unchanged.  We denote the resulting word by $\operatorname{red}_D(u)$.
\end{definition}

\begin{remark}\label{rem:boundary-reduction-local-action}
By Lemmas~\ref{lem:zero-ray-prefix} and \ref{lem:one-ray-prefix},
\[
        \psi_u=\psi_{\operatorname{red}_D(u)}
        \qquad(u\in\W_n).
\]
Moreover,
\[
        \sigma_n(u)=\sigma_n(\operatorname{red}_D(u)),
\]
because the digits removed from the initial block are either $0$ or $n-1$, both of which are congruent to $0$ modulo $n-1$.
\end{remark}

\begin{lemma}\label{lem:inner-tail-finite-representatives}
Let $\psi:\C_n\to\C_m$ be a homeomorphism such that $F_n^\psi\le F_m$.  There are a finite set $S\subseteq\W_n^{\mathrm{in}}$ and an integer $k\ge0$ with the following property: for every $u\in\W_n^{\mathrm{in}}$, every $s\in S$ satisfying
\[
        \sigma_n(s)=\sigma_n(u),
\]
and every word $w\in\W_n$ with $|w|\ge k$, one has
\[
        \psi_{uw}=\psi_{sw}.
\]
\end{lemma}

\begin{proof}
The proof mirrors the final bounded-representative step in the proofs of \cite[Corollary~6.17]{AutG} and \cite[Corollary~4.16]{AutT}.  In those settings, arbitrary sufficiently deep local actions are compared with local actions based at a fixed finite antichain.  Here the boundary rays have already been dealt with by Lemmas~\ref{lem:zero-ray-prefix} and \ref{lem:one-ray-prefix}, and Proposition~\ref{prop:uniform-tail-agreement} supplies the corresponding comparison for inner words with the same $\sigma_n$-value.

Increasing $D$ if necessary, assume that $D\ge1$.  Define
\[
\begin{aligned}
S={}&\{a\mid 1\le a\le n-2\} \\
&\cup\{0^r a\mid 1\le r\le D,\ 1\le a\le n-1\} \\
&\cup\{(n-1)^r b\mid 1\le r\le D,\ 0\le b\le n-2\}.
\end{aligned}
\]
Thus $S$ records the possible initial segments at which a boundary-reduced inner word first leaves one of the two boundary rays.  Every $D$-boundary-reduced inner word has a unique prefix in $S$, and $S$ contains a representative of every residue class modulo $n-1$.

Consider the following finite set of ordered pairs of inner words:
\[
\begin{aligned}
P={}&\{(s,t)\in S\times S\mid \sigma_n(s)=\sigma_n(t)\} \\
&\cup\{(sa,t)\mid s,t\in S,\ a\in X_n,\ \sigma_n(sa)=\sigma_n(t)\}.
\end{aligned}
\]
Here $sa$ denotes the concatenation of the word $s$ with the one-letter word $a$.  Since $s$ is inner, the word $sa$ is also inner.  For every pair $(x,y)\in P$, Proposition~\ref{prop:uniform-tail-agreement} gives an integer $k(x,y)$ such that
\[
        \psi_{xz}=\psi_{yz}
        \qquad\text{whenever } |z|\ge k(x,y).
\]
Let $k$ be the maximum of these finitely many integers.

Now let $u\in\W_n^{\mathrm{in}}$ and set $\widetilde u=\operatorname{red}_D(u)$.  By Remark~\ref{rem:boundary-reduction-local-action}, it is enough to prove the assertion for $\widetilde u$.  Write
\[
        \widetilde u=s_0a_1a_2\cdots a_r,
\]
where $s_0\in S$ is the unique prefix of $\widetilde u$ belonging to $S$ and $a_1,\ldots,a_r\in X_n$.  Inductively choose $s_i\in S$ so that
\[
        \sigma_n(s_i)=\sigma_n(s_{i-1}a_i)
        \qquad(1\le i\le r).
\]
This is possible because $S$ contains a representative of every residue class.

Let $w\in\W_n$ with $|w|\ge k$.  Successively replacing $s_{i-1}a_i$ by $s_i$, and applying the definition of $k$ to the remaining suffix $a_{i+1}\cdots a_r w$, gives
\[
\begin{aligned}
        \psi_{\widetilde u w}
        &=\psi_{s_0a_1a_2\cdots a_r w}  \\
        &=\psi_{s_1a_2\cdots a_r w}      \\
        &=\cdots                         \\
        &=\psi_{s_r w}.
\end{aligned}
\]
The word $s_r$ belongs to $S$ and satisfies
\[
        \sigma_n(s_r)=\sigma_n(\widetilde u)=\sigma_n(u).
\]
Let now $s\in S$ be any element with $\sigma_n(s)=\sigma_n(u)$.  Since $(s_r,s)\in P$, the definition of $k$ gives
\[
        \psi_{s_rw}=\psi_{sw}
        \qquad(|w|\ge k).
\]
Combining this with the displayed chain of equalities above gives
\[
        \psi_{uw}=\psi_{\widetilde u w}=\psi_{s_rw}=\psi_{sw}
        \qquad(|w|\ge k),
\]
where the first equality follows from Remark~\ref{rem:boundary-reduction-local-action}.  This proves the strengthened assertion.
\end{proof}

\begin{theorem}\label{thm:finite-local-actions-body}
Let $\psi:\C_n\to\C_m$ be a homeomorphism such that $F_n^\psi\le F_m$.  Then $\LA_\psi$ is finite.  Equivalently, by Theorem~\ref{thm:rational-iff-finite-local-actions-prelim}, $\psi$ is rational.
\end{theorem}

\begin{proof}
Let $D$ be as in Definition~\ref{def:boundary-reduced-word}, and let $S$ and $k$ be as in Lemma~\ref{lem:inner-tail-finite-representatives}.  Put
\[
        L=\max\{D,\ \max\{|s|\mid s\in S\}+k,\ D+k+1\}.
\]
We prove that every local action $\psi_u$ is equal to $\psi_v$ for some word $v\in\W_n$ with $|v|\le L$.

Let $u\in\W_n$ and let $\widetilde u=\operatorname{red}_D(u)$.  By Remark~\ref{rem:boundary-reduction-local-action}, $\psi_u=\psi_{\widetilde u}$.  If $\widetilde u$ is a pure boundary word, meaning that it is of the form $0^r$ or $(n-1)^r$, then boundary reduction gives $r\le D$, and we are done.

Assume that $\widetilde u$ is inner.  If $|\widetilde u|\le D+k+1$, there is nothing to prove.  Otherwise, write
\(\widetilde u=pw\)
where $|w|=k$.  Since $|p|>D+1$ and $\widetilde u$ is boundary-reduced and inner, the prefix $p$ is inner.  Lemma~\ref{lem:inner-tail-finite-representatives} gives an element $s\in S$ such that
\(\psi_{pw}=\psi_{sw}.\)
The word $sw$ has length at most $\max\{|s|\mid s\in S\}+k\le L$.  Thus every local action is represented by a word of length at most $L$.  Since there are only finitely many such words, $\LA_\psi$ is finite.

\end{proof}

\begin{corollary}\label{cor:uniform-inner-synchronization}
Let $\psi:\C_n\to\C_m$ be a homeomorphism such that $F_n^\psi\le F_m$.  Then there exists $k\ge0$ such that, for all $u,v\in\W_n^{\mathrm{in}}$ with $\sigma_n(u)=\sigma_n(v)$ and all $w\in\W_n$ with $|w|\ge k$,
\[
        \psi_{uw}=\psi_{vw}.
\]
\end{corollary}

\begin{proof}
Let $S$ and $k$ be as in Lemma~\ref{lem:inner-tail-finite-representatives}.  Since $S$ contains a representative of every residue class modulo $n-1$, choose $s\in S$ such that
\[
        \sigma_n(s)=\sigma_n(u)=\sigma_n(v).
\]
By Lemma~\ref{lem:inner-tail-finite-representatives}, for every $w\in\W_n$ with $|w|\ge k$ we have
\[
        \psi_{uw}=\psi_{sw}
        \qquad\text{and}\qquad
        \psi_{vw}=\psi_{sw}.
\]
Hence $\psi_{uw}=\psi_{vw}$.
\end{proof}

\begin{definition}\label{def:semi-synchronizing}
Let
\[
        T_{s_0}=(S,t,o,s_0)
\]
be an initial $(n,m)$-transducer.  A state of the form $t(s_0,u)$ is called:
\begin{itemize}
    \item a \emph{root state} if $u=\eps$;
    \item a \emph{left boundary state} if $u=0^r$ for some $r\ge0$;
    \item a \emph{right boundary state} if $u=(n-1)^r$ for some $r\ge0$;
    \item an \emph{inner state} if $u$ is inner.
\end{itemize}
A state may have more than one type.  For instance, the initial state is both a left and a right boundary state.  The set of inner states is closed under transitions: if $s=t(s_0,u)$ with $u$ inner and $a\in X_n$, then $t(s,a)=t(s_0,ua)$, and $ua$ is again inner.

We say that $T_{s_0}$ is \emph{semi-synchronizing} if the following two conditions hold.
\begin{enumerate}[label=(\roman*)]
    \item \emph{Boundary-ray synchronization.}  The two boundary rays eventually stabilize: there are states $s_L,s_R\in S$ such that
    \[
        t(s_L,0)=s_L,
        \qquad
        t(s_R,n-1)=s_R,
    \]
    and $s_L$ and $s_R$ are reached from $s_0$ by words $0^r$ and $(n-1)^r$, respectively.
    \item \emph{Inner synchronization.}  The inner subtransducer synchronizes within each congruence class: there exists $k\ge0$ such that for all inner words $u,v\in\W_n$ with $\sigma_n(u)=\sigma_n(v)$ and all words $w\in X_n^k$,
    \[
        t(s_0,uw)=t(s_0,vw).
    \]
\end{enumerate}
\end{definition}

The characterization below should be compared with the transducer descriptions recalled in Remark~\ref{rem:AutG-AutT-sync-prelim}.  In the $G_{n,r}$ and $T_{n,r}$ settings, it is proved in \cite{AutG} and \cite{AutT} that a homeomorphism whose conjugation sends the Thompson-like group into itself must have finitely many local actions, and its minimal transducer must be synchronizing.  The reverse implication (in the $T_{n,r}$ case, with the assumption that the homeomorphism respects the cyclic order) is not stated in those papers, but it can be obtained by adapting the proof for the $T_{n,r}$ case.  We give instead a combinatorial direct proof for $F_n$ (which can be easily modified to the $T_{n,r}$ and $G_{n,r}$ setting). 

\begin{theorem}[Cantor-space conjugator criterion]\label{thm:conjugator-characterization}
Let $\psi:\C_n\to\C_m$ be a homeomorphism.  Then $F_n^\psi\le F_m$ if and only if the following conditions hold:
\begin{enumerate}[label=(\arabic*)]
    \item $\psi$ is order-preserving or order-reversing;
    \item $\psi$ is rational and its minimal transducer $T^\psi$ is semi-synchronizing.
\end{enumerate}
\end{theorem}

\begin{proof}
Assume first that $F_n^\psi\le F_m$.  By Lemma~\ref{lem:cantor-conjugator-descends}, the map $\psi$ is order-preserving or order-reversing.  By Theorem~\ref{thm:finite-local-actions-body}, it is rational.  Boundary stabilization gives condition (i) in Definition~\ref{def:semi-synchronizing}.  Corollary~\ref{cor:uniform-inner-synchronization} gives a uniform level for condition (ii).  Using the state--local-action dictionary of Notation~\ref{not:state-local-action-dictionary}, this says exactly that the states of $T^\psi$ satisfy the inner synchronization condition.  Thus $T^\psi$ is semi-synchronizing.

Conversely, assume that $\psi$ is order-preserving or order-reversing and that $T^\psi$ is semi-synchronizing.  Let $f\in F_n$ and choose a tree diagram for $f$.  We refine it in two steps.  First refine the boundary branch pairs, if present, so that they extend beyond the boundary-stabilization depth.  Then refine every inner branch pair by attaching a complete $n$-ary tree of height equal to the inner synchronization level.  Let
\[
        u_1\to v_1,\ldots,u_r\to v_r
\]
be the branch pairs of the resulting diagram.

For each $j$, the branch-pair criterion for $F_n$ gives $\sigma_n(u_j)=\sigma_n(v_j)$ whenever $u_j$ and $v_j$ are inner.  Semi-synchronization gives
\(\psi_{u_j}=\psi_{v_j}\)
for inner pairs, and boundary stabilization gives the same equality for the two boundary pairs.  Write
\[
        \psi(u_j\omega)=O^\psi(u_j)\psi_{u_j}(\omega),
        \qquad
        \psi(v_j\omega)=O^\psi(v_j)\psi_{v_j}(\omega).
\]
Let $\theta_j=\psi_{u_j}=\psi_{v_j}$.  The image $\theta_j(\C_n)$ is clopen in $\C_m$, so choose a finite prefix code $P_j\subseteq\W_m$ such that
\[
        \theta_j(\C_n)=\bigsqcup_{p\in P_j}U_p.
\]
For each $p\in P_j$, choose the finite prefix code $Q_{j,p}\subseteq\W_n$ satisfying
\[
        \theta_j^{-1}(U_p)=\bigsqcup_{q\in Q_{j,p}}U_q.
\]
Then on the cylinder $U_{O^\psi(u_j)p}$ the conjugate $f^\psi=\psi^{-1}f\psi$ acts by replacing the prefix $O^\psi(u_j)p$ with $O^\psi(v_j)p$.  Indeed, for $q\in Q_{j,p}$ and $\omega\in\C_n$, if
\[
        \psi(u_jq\omega)=O^\psi(u_j)p\omega'
\]
for some $\omega'\in\C_m$, then
\[
        \psi(v_jq\omega)=O^\psi(v_j)p\omega'.
\]
Thus the finitely many branch pairs
\[
        O^\psi(u_j)p\to O^\psi(v_j)p
        \qquad (1\le j\le r,\ p\in P_j)
\]
represent $f^\psi$, after the finite refinements just described.  Since $\psi$ either preserves or reverses the lexicographic order, and $f$ preserves order, the conjugate $f^\psi$ is order-preserving.  Therefore $f^\psi\in F_m$.  Since $f$ was arbitrary, $F_n^\psi\le F_m$.
\end{proof}

\begin{corollary}\label{cor:aut-Fn-semisync}
The automorphism group $\operatorname{Aut}(F_n)$ can be naturally identified with the normalizer of the standard action of $F_n$ in $\Homeo(0,1)$.  Equivalently, after lifting to the Cantor space, it is identified with the group of homeomorphisms $\psi:\C_n\to\C_n$ which preserve the relation $\sim_{\rho_n}$ and normalize the lifted copy of $F_n$.

Under this identification, $\operatorname{Aut}(F_n)$ consists precisely of those homeomorphisms $\psi:\C_n\to\C_n$ such that both $\psi$ and $\psi^{-1}$ are order-preserving or order-reversing rational homeomorphisms whose minimal transducers are semi-synchronizing.
\end{corollary}

\begin{proof}
The identification of automorphisms with the interval normalizer follows from the reconstruction theorem of McCleary--Rubin, as used by Brin--Guzm\'an in their study of automorphisms of generalized Thompson groups \cite{McClearyRubin,BrinGuzman}.  The equivalence with the Cantor-space normalizer is the liftability statement in Lemma~\ref{lem:lift-order-equivalence-prelim}.  If $\psi$ normalizes $F_n$, then both $F_n^\psi\le F_n$ and $F_n^{\psi^{-1}}\le F_n$, so Theorem~\ref{thm:conjugator-characterization} applies to both $\psi$ and $\psi^{-1}$.  Conversely, applying the theorem to both maps gives the two inclusions $F_n^\psi\le F_n$ and $F_n^{\psi^{-1}}\le F_n$, and hence equality.
\end{proof}

\begin{example}
\label{ex:binary-semisync-transducer}
Let \(\varphi:\C_2\to\C_2\) be the homeomorphism represented by the minimal
initial transducer in Figure~\ref{fig:transducer_semi_sync_corrected}.  This
is the transducer used in the final construction.

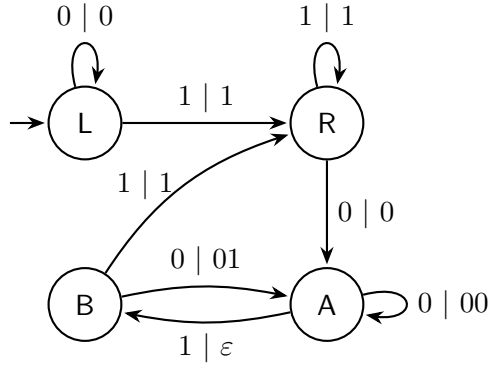
\begin{figure}[h]
\centering
\begin{tikzpicture}[shorten >=1pt, auto, thick, >=Stealth]
    \node[state, initial, initial text=] (L) at (0, 0) {$\mathsf L$};
    \node[state] (R) at (3.2, 0) {$\mathsf R$};
    \node[state] (A) at (3.2, -2.4) {$\mathsf A$};
    \node[state] (B) at (0, -2.4) {$\mathsf B$};

    \path[->]
        (L) edge [loop above] node {0 $|$ 0} (L)
            edge node[above] {1 $|$ 1} (R)
        (R) edge [loop above] node {1 $|$ 1} (R)
            edge node[right] {0 $|$ 0} (A)
        (A) edge [loop right] node {0 $|$ 00} (A)
            edge [bend left=12] node[below] {1 $|$ $\eps$} (B)
        (B) edge [bend left=12] node[above] {0 $|$ 01} (A)
            edge [bend left=20] node[left] {1 $|$ 1} (R);
\end{tikzpicture}
\caption{The minimal binary semi-synchronizing transducer defining
\(\varphi\).}
\label{fig:transducer_semi_sync_corrected}
\end{figure}
\end{example}

\begin{example}\label{ex:conjugating-generators}
For the homeomorphism \(\varphi\) of
Figure~\ref{fig:transducer_semi_sync_corrected}, one computes
\[
    x_0^\varphi:
    \begin{cases}
        00\eta\mapsto 0\eta,\\
        010\eta\mapsto 1000\eta,\\
        0110\eta\mapsto 1001\eta,\\
        0111\eta\mapsto 101\eta,\\
        1\eta\mapsto 11\eta,
    \end{cases}
\]
and
\[
    x_1^\varphi:
    \begin{cases}
        0\eta\mapsto 0\eta,\\
        1000\eta\mapsto 10\eta,\\
        1001\eta\mapsto 11000\eta,\\
        1010\eta\mapsto 11001\eta,\\
        1011\eta\mapsto 1101\eta,\\
        11\eta\mapsto 111\eta.
    \end{cases}
\]
The displayed branch pairs are the reduced tree diagrams for these conjugates.
The direct transducer computation first gives tree diagrams on suitable common
refinements of the standard diagrams for \(x_0\) and \(x_1\).  On each refined
branch pair \(u\to v\), the local actions \(\varphi_u\) and \(\varphi_v\) agree.
Applying the prefix-output formula from
Notation~\ref{not:state-local-action-dictionary}, and then making the finite
refinements coming from the image of this common local action, gives a tree
diagram for the conjugate.  Reducing common carets gives precisely the branch
pairs displayed above.
\end{example}

\section{Conjugating closed subgroups by transducers}
\label{sec:conjugating-closed-subgroups}

In this section we describe how a closed subgroup represented by a full
tree-automaton changes under conjugation by a  homeomorphism of Cantor
spaces which preserves or reverses the lexicographic order.  Throughout this
section
\(\mathcal A=(Q,\tau,q_0)\)
is a full \(n\)-ary tree-automaton,
\(H=\D(\mathcal A)\le F_n,\)
and
\(\psi:\C_n\to\C_m\)
is a  homeomorphism satisfying the equivalent conditions of
Lemma~\ref{lem:lift-order-equivalence-prelim}.  We write
\(\phi=\psi^{-1}.\)
The goal is to construct an \(m\)-ary tree-automaton defining
\(H^\psi\cap F_m .\)
Since \(\psi\) satisfies Lemma~\ref{lem:lift-order-equivalence-prelim}, it maps
\(n\)-adic points bijectively to \(m\)-adic points.  It follows directly from
the piecewise description of closed subgroups that \(H^\psi\cap F_m\) is a
closed subgroup of \(F_m\).  The constructions below realize this intersection as a diagram group (which also proves that it is closed).  If in
addition \(H^\psi\le F_m\), then the same automata define \(H^\psi\) itself.

We give two constructions.  The pullback construction uses the inverse
transducer and gives a direct proof of correctness.  The forward construction
uses the transducer for \(\psi\) itself and is the construction used in the
explicit computations later in the paper.
                                                   
\begin{remark}
\label{rem:full-automata-assumption}
The constructions are stated for full automata.  This is harmless for the main
applications, where the automata are full cores.  In general, one can make an
automaton full by attaching a full rooted \(n\)-ary tree at each leaf; this
does not change the accepted diagram group, since the added trees encode only
identity refinements below leaves.  However, this completion can be infinite
even when the original automaton is finite.
   See also Remark~\ref{rem:finite-stopped-forward-nonfull-core} below for the case where $\mathcal A$ is a finite core automaton and $H^\psi$ is a subgroup of $F_m$. 
\end{remark}

We use the notation of the preliminaries for minimal transducers:
\[
        T^\psi=(S_\psi,t_\psi,o_\psi,s_\psi),
        \qquad
        T^\phi=(S_\phi,t_\phi,o_\phi,s_\phi).
\]
For \(u\in\W_n\) and \(x\in\W_m\), write
\[
        s^\psi_u=t_\psi(s_\psi,u),
        \qquad
        O^\psi(u)=o_\psi(s_\psi,u),
\]
and
\[
        s^\phi_x=t_\phi(s_\phi,x),
        \qquad
        O^\phi(x)=o_\phi(s_\phi,x).
\]
Thus
\[
        \psi(u\omega)=O^\psi(u)\psi_u(\omega),
        \qquad
        \phi(x\omega)=O^\phi(x)\phi_x(\omega).
\]

\subsection{The pullback automaton}
\label{subsec:pullback-automaton}

The pullback construction is the direct one.  Given an \(m\)-ary branch pair
\(x\to y\), we compare the two pulled-back cylinders under
\(\phi=\psi^{-1}.\)
The inverse transducer gives the local descriptions
\[
        \phi(x\omega)=O^\phi(x)\phi_x(\omega),
        \qquad
        \phi(y\omega)=O^\phi(y)\phi_y(\omega).
\]
The condition imposed by the pullback automaton is that the same local action
of \(\phi\) remains on the two sides, and that the two source prefixes
\(O^\phi(x)\) and \(O^\phi(y)\) end at the same state of \(\mathcal A\); when
this condition holds for every branch pair of a tree diagram, the proof below
shows that the diagram represents an element of \(H^\psi\cap F_m\), while the
converse is obtained after allowing refinements of diagrams.  Thus, after
reading a word \(x\in\W_m\), the pullback automaton stores the pair
\[
        \bigl(s^\phi_x,\tau(q_0,O^\phi(x))\bigr):
\]
the state \(s^\phi_x\) of \(T^\phi\), which determines the remaining local
action \(\phi_x\), and the state of \(\mathcal A\) reached by the source prefix
\(O^\phi(x)\).
          
\begin{definition}
\label{def:pullback-automaton}
Let \(\psi:\C_n\to\C_m\) be a  homeomorphism preserving or reversing
lexicographic order, let \(\phi=\psi^{-1}\), and let
\(\mathcal A=(Q,\tau,q_0)\) be a full \(n\)-ary tree-automaton.  The
\emph{pullback automaton} of \(\mathcal A\) by \(\psi\), denoted
\[
        \mathcal P_\psi(\mathcal A),
\]
is the accessible part of the following full \(m\)-ary automaton.  The ambient
state set is \(S_\phi\times Q\), the initial state is \((s_\phi,q_0)\), and
the transition labelled \(a\in X_m\) is
\[
        (s,q)\cdot a
        =
        \bigl(t_\phi(s,a),\tau(q,o_\phi(s,a))\bigr).
\]
Equivalently, after reading \(x\in\W_m\), the reached state is
\[
        \bigl(s^\phi_x,\tau(q_0,O^\phi(x))\bigr).
\]
\end{definition}

\begin{theorem}[The pullback theorem]
\label{thm:pullback-theorem}
Let \(\psi:\C_n\to\C_m\) be a  homeomorphism preserving or reversing
lexicographic order, let \(\phi=\psi^{-1}\), and let
\(\mathcal A=(Q,\tau,q_0)\) be a full \(n\)-ary tree-automaton.  Put
\(H=\D(\mathcal A)\le F_n\).  Then
\[
        \D(\mathcal P_\psi(\mathcal A))
        =
        H^\psi\cap F_m.
\]
\end{theorem}

\begin{proof}
Write \(\mathcal P=\mathcal P_\psi(\mathcal A)\).

First let \(g\in\D(\mathcal P)\), and choose an \(m\)-ary tree diagram for
\(g\) accepted by \(\mathcal P\), with branch pairs
\[
        x_1\to y_1,
        \ldots,
        x_r\to y_r.
\]
Since this is an \(m\)-ary tree diagram, \(g\in F_m\).  We prove that
\(g\in H^\psi\).  Let
\(h=\psi g\phi .\)
Because \(\psi\) either preserves or reverses the lexicographic order and
\(g\) preserves the lexicographic order, the conjugate \(h\) also preserves the
lexicographic order.

For a branch pair \(x_i\to y_i\), acceptance by \(\mathcal P\) gives
\(s^\phi_{x_i}=s^\phi_{y_i}\)
and
\[
        \tau(q_0,O^\phi(x_i))=\tau(q_0,O^\phi(y_i)).
\]
Put
\[
        u_i=O^\phi(x_i),
        \qquad
        v_i=O^\phi(y_i).
\]
For every \(\omega\in\C_m\),
\[
        \phi(x_i\omega)=u_i\phi_{x_i}(\omega),
        \qquad
        \phi(y_i\omega)=v_i\phi_{y_i}(\omega).
\]
Since \(s^\phi_{x_i}=s^\phi_{y_i}\), the state--local-action dictionary gives
\(\phi_{x_i}=\phi_{y_i}\).  Hence, for every
\(\eta\in\phi_{x_i}(\C_m)\),
\(h(u_i\eta)=v_i\eta.\)
Indeed, if \(\eta=\phi_{x_i}(\omega)\), then
\[
\begin{aligned}
        h(u_i\eta)
        &=h(\phi(x_i\omega))        \\
        &=\phi(g(x_i\omega))        \\
        &=\phi(y_i\omega)           \\
        &=v_i\phi_{y_i}(\omega)     \\
        &=v_i\eta .
\end{aligned}
\]
The clopen set \(\phi_{x_i}(\C_m)\) is a finite disjoint union of \(n\)-ary
cylinders.  Refining over those cylinders, \(h\) has an \(n\)-ary tree diagram
all of whose branch pairs over this piece are of the form
\(u_i r\to v_i r.\)
Since \(\tau(q_0,u_i)=\tau(q_0,v_i)\) and \(\mathcal A\) is full, each such
refined pair is accepted by \(\mathcal A\).  Therefore this refined diagram for
\(h\) is accepted by \(\mathcal A\).  Hence \(h\in H\), and so
\(g\in H^\psi\).  Thus
\(g\in H^\psi\cap F_m.\)

Conversely, let \(g\in H^\psi\cap F_m\).  Then
\(h=\psi g\phi\in H.\)
Choose an \(m\)-ary tree diagram for \(g\), and choose an \(n\)-ary tree
diagram for \(h\) accepted by \(\mathcal A\).  Refine the diagram for \(g\) so
that, for every domain branch \(x\), the clopen set \(\phi(U_x)\) lies inside
a single domain cylinder of the accepted diagram for \(h\).  This is possible
because the domain cylinders of the diagram for \(h\) form a finite clopen
partition of \(\C_n\), and their preimages under \(\phi\) form a finite clopen
partition of \(\C_m\).

Let \(x\to y\) be a branch pair of the refined diagram of \(g\).  Let
\(a\to b\) be the branch pair of the accepted diagram for \(h\) whose domain
cylinder contains \(\phi(U_x)\).  Since
\(\phi(U_x)\subseteq U_a,\)
the word \(O^\phi(x)=\Root(\phi(U_x))\) extends \(a\); write
\(O^\phi(x)=ar.\)
On \(U_{ar}\), the element \(h\) acts by
\(ar\eta\mapsto br\eta.\)
Using \(g\phi=\phi h\), for every \(\omega\in\C_m\) we have
\[
        \phi(y\omega)=\phi(g(x\omega))=h(\phi(x\omega)).
\]
Since
\[
        \phi(x\omega)=ar\,\phi_x(\omega),
\]
and \(h\) acts on \(U_{ar}\) by replacing the prefix \(ar\) by \(br\), this
gives the pointwise equality
\[
        \phi(y\omega)=br\,\phi_x(\omega)
        \qquad(\omega\in\C_m).
\]
On the other hand,
\[
        \phi(y\omega)=O^\phi(y)\phi_y(\omega).
\]
The set \(\phi_x(\C_m)\) has empty root, so the root of
\(\{br\,\phi_x(\omega)\mid \omega\in\C_m\}\) is \(br\).  Hence
\(O^\phi(y)=br\), and cancelling this prefix in the displayed pointwise equality
gives
\(\phi_y=\phi_x.\)
Minimality of \(T^\phi\) gives \(s^\phi_x=s^\phi_y\).  Also, the accepted
diagram for \(h\) can be refined at \(a\to b\) to contain the branch pair
\(ar\to br,\)
that is,
\(O^\phi(x)\to O^\phi(y).\)
Since acceptedness is preserved under refinements,
\[
        \tau(q_0,O^\phi(x))=\tau(q_0,O^\phi(y)).
\]
Thus \(x\) and \(y\) reach the same state in \(\mathcal P\).  This holds for
every branch pair of the refined diagram for \(g\), so that diagram is
accepted by \(\mathcal P\).  Hence \(g\in\D(\mathcal P)\).
\end{proof}

\begin{corollary}
\label{cor:pullback-subgroup-case}
Under the hypotheses of Theorem~\ref{thm:pullback-theorem}, if in addition
\(H^\psi\le F_m\), then
\[
        \D(\mathcal P_\psi(\mathcal A))=H^\psi.
\]
\end{corollary}

\begin{lemma}
\label{lem:finite-full-folded-precore-is-core}
Let \(\mathcal B\) be a finite, full, folded pre-core for a closed subgroup
\(K\le F_m\).  Then \(\mathcal B\) is the core \(\C(K)\).
\end{lemma}

\begin{proof}
A finite full tree-automaton has no leaves.  Hence it cannot contain a finite
hanging tree attached at a leaf.  Since it is finite, it also cannot be obtained
from a smaller automaton by attaching an infinite full tree at a leaf.  Thus it
is reduced.  Being folded and having the existence property by the definition
of pre-core, \(\mathcal B\) is a core automaton.  Proposition~\ref{prop:core-automata-characterization}
therefore identifies it with \(\C(\D(\mathcal B))=\C(K)\).
\end{proof}

\begin{proposition}
\label{prop:pullback-precore}
Let \(\psi:\C_n\to\C_m\) be a homeomorphism preserving or reversing
lexicographic order.  Let \(\mathcal A\) be a full \(n\)-ary tree-automaton
with the existence property, and set \(H=\D(\mathcal A)\).  If
\(H^\psi\le F_m\), then \(\mathcal P_\psi(\mathcal A)\) has the existence
property.  Hence it is a pre-core for \(H^\psi\).  In particular, if $\psi$ is rational and
\(\mathcal A=\C(H)\) is finite and full, then folding
\(\mathcal P_\psi(\mathcal A)\) gives \(\C(H^\psi)\); no hanging trees have to
be deleted.
\end{proposition}

\begin{proof}
Suppose \(x,y\in\W_m\) reach the same state in
\(\mathcal P_\psi(\mathcal A)\).  Then
\[
        s^\phi_x=s^\phi_y,
        \qquad
        \tau(q_0,O^\phi(x))=\tau(q_0,O^\phi(y)).
\]
By the existence property of \(\mathcal A\), there is \(h\in H\) with branch
pair
\(O^\phi(x)\to O^\phi(y).\)
Let
\(g=\phi h\psi.\)
Since \(H^\psi\le F_m\), we have \(g\in F_m\).  For every
\(\omega\in\C_m\),
\[
\begin{aligned}
        \phi(g(x\omega))
        &=h(\phi(x\omega))                                      \\
        &=h\bigl(O^\phi(x)\phi_x(\omega)\bigr)                  \\
        &=O^\phi(y)\phi_x(\omega)                               \\
        &=O^\phi(y)\phi_y(\omega)                               \\
        &=\phi(y\omega).
\end{aligned}
\]
Since \(\phi\) is injective, \(g(x\omega)=y\omega\) for all \(\omega\).  Thus
\(g\) has branch pair \(x\to y\), proving the existence property.  The pullback automaton is finite and full when \(\mathcal A\) and \(T^\phi\)
are finite and \(\mathcal A\) is full, and folding preserves fullness.  Hence,
in the finite full case, the folded pullback automaton is a finite full folded
pre-core for \(H^\psi\).  Lemma~\ref{lem:finite-full-folded-precore-is-core}
therefore identifies it with \(\C(H^\psi)\).
\end{proof}

\subsection{The forward automaton}
\label{subsec:forward-automaton}

The forward construction gives the same subgroup while avoiding the inverse
transducer.  Instead of pulling \(m\)-ary cylinders back by
\(\phi=\psi^{-1}\), we run the original transducer \(T^\psi\) forward on
\(n\)-ary input words.  An input prefix \(u\in\W_n\) gives two pieces of
information: the state \(\tau(q_0,u)\) reached in \(\mathcal A\), and the
\(m\)-ary output word \(O^\psi(u)\), which tells us where the cylinder \(U_u\)
begins to land in \(\C_m\).  Thus the output of the transducer is what produces
the \(m\)-ary branches, while the input word is what is tested in
\(\mathcal A\).

The difficulty is that an output prefix \(x\in\W_m\) need not come from a
unique input prefix.  Different input paths, or different stopping points
inside word-labelled edges, may have produced exactly the same output prefix
\(x\).  The raw forward automaton records these forward computations before
determinization.  After subdividing word-labelled edges, the subset \(R_x\)
records all possible positions in the raw computation after the output prefix
\(x\) has been produced.  The forward automaton is the deterministic
\(m\)-ary automaton obtained by remembering these finite sets of possible
forward states.
           
\begin{definition}
\label{def:raw-forward-automaton}
Let
\[
        \psi:\C_n\to\C_m
\]
be a  homeomorphism, and let
\[
        \mathcal A=(Q,\tau,q_0)
\]
be a full \(n\)-ary tree-automaton.

First form the ambient word-labelled automaton with state set
\[
        S_\psi\times Q,
\]
initial state
\[
        (s_\psi,q_0),
\]
and, for every \(s\in S_\psi\), \(q\in Q\), and \(i\in X_n\), an edge
\[
        (s,q)
        \xrightarrow{\ o_\psi(s,i)\ }
        \bigl(t_\psi(s,i),\tau(q,i)\bigr).
\]
Thus reading the input letter \(i\) moves the transducer coordinate from
\(s\) to \(t_\psi(s,i)\), moves the \(\mathcal A\)-coordinate from \(q\) to
\(\tau(q,i)\), and records as edge label the output word produced by
\(T^\psi\) on that input letter.

The \emph{raw forward automaton}
\[
        \Fraw_\psi(\mathcal A)
\]
is the accessible part of this ambient word-labelled automaton: its states are
the states reachable from \((s_\psi,q_0)\), and its edges are the ambient
edges whose source is reachable.  The states of \(\Fraw_\psi(\mathcal A)\)
will be called \emph{principal states}.  This terminology is used to
distinguish them from the subdivision states introduced below.
\end{definition}

\begin{definition}[The subdivided raw automaton]
\label{def:subdivided-raw-forward-automaton}
Let
\[
        \widehat{\Fraw}_\psi(\mathcal A)
\]
be the \(\eps\)-NFA over \(X_m\) obtained from
\(\Fraw_\psi(\mathcal A)\) by subdividing word-labelled edges as in
Subsection~\ref{subsec:subset-construction-prelim}.  Thus an edge labelled
\(\eps\) becomes an \(\eps\)-edge, an edge labelled by a single letter of
\(X_m\) is left unchanged, and an edge
\[
        e:p\xrightarrow{\ b_1\cdots b_k\ }q
        \qquad(k\ge2)
\]
is replaced by the path
\[
        p\xrightarrow{\ b_1\ }(e,b_2\cdots b_k)
        \xrightarrow{\ b_2\ }(e,b_3\cdots b_k)
        \longrightarrow\cdots\longrightarrow
        (e,b_k)\xrightarrow{\ b_k\ }q .
\]
Recall that the notation \((e,\rho)\) means that we are partway through the original raw
edge \(e\), and that \(\rho\) is the unread suffix of the word label of \(e\).
For example, after reading the first letter \(b_1\) of the label
\(b_1\cdots b_k\), the unread suffix is \(b_2\cdots b_k\), so the reached
subdivision state is \((e,b_2\cdots b_k)\).

The original states of \(\Fraw_\psi(\mathcal A)\) are the principal states of
\(\widehat{\Fraw}_\psi(\mathcal A)\), and the newly inserted vertices are its
subdivision states.
\end{definition}

\begin{lemma}
\label{lem:raw-paths}
Let \((s,q)\) be a principal state of \(\Fraw_\psi(\mathcal A)\), and let
\[
        u=i_1\cdots i_r\in\W_n.
\]
The word \(u\) determines a unique raw path in
\(\Fraw_\psi(\mathcal A)\) starting at \((s,q)\), namely the path obtained by
successively following the raw edges corresponding to
\[
        i_1,\ldots,i_r.
\]
This path ends at
\[
        \bigl(t_\psi(s,u),\tau(q,u)\bigr)
\]
and has word label
\[
        o_\psi(s,u).
\]

In particular, from the initial state \((s_\psi,q_0)\), the word \(u\)
determines a raw path ending at
\[
        \bigl(s^\psi_u,\tau(q_0,u)\bigr)
\]
with label
\[
        O^\psi(u).
\]

After subdivision, the same input word \(u\) determines a unique path in
\[
        \widehat{\Fraw}_\psi(\mathcal A)
\]
with the same label \(o_\psi(s,u)\), obtained by replacing each raw edge of
the preceding path by its subdivided path.
\end{lemma}

\begin{proof}
The raw statement follows by induction on \(|u|\).  For one letter it is the
definition of the raw forward automaton.  If
\(u=vi,\)
then by the induction hypothesis the path determined by \(v\) ends at
\[
        \bigl(t_\psi(s,v),\tau(q,v)\bigr)
\]
and has label \(o_\psi(s,v)\).  Following the raw edge corresponding to \(i\)
then gives the terminal state
\[
        \bigl(t_\psi(s,vi),\tau(q,vi)\bigr)
\]
and appends the output word
\(o_\psi(t_\psi(s,v),i).\)
Thus the total label is
\[
        o_\psi(s,v)o_\psi(t_\psi(s,v),i)=o_\psi(s,vi).
\]

Since \((s,q)\) is reachable, every state reached by continuing this path is
also reachable; hence the path lies in the accessible raw automaton.  The
subdivided statement is immediate from the definition of
\(\widehat{\Fraw}_\psi(\mathcal A)\).
\end{proof}

\begin{definition}
\label{def:forward-automaton}
Let \(P\) be the state set of the subdivided raw automaton
\[
        \widehat{\Fraw}_\psi(\mathcal A).
\]
For \(x\in\W_m\), define
\[
        R_x=R_x(\psi,\mathcal A)
\]
to be the set of all states \(p\in P\) such that there exists a path in
\[
        \widehat{\Fraw}_\psi(\mathcal A)
\]
from the initial state \((s_\psi,q_0)\) to \(p\) whose label is \(x\).  Here,
as in Subsection~\ref{subsec:subset-construction-prelim}, the label of a path
is obtained by concatenating the edge labels and deleting all occurrences of
\(\eps\).  Thus paths contributing to \(R_x\) may use \(\eps\)-edges before the
first letter of \(x\), after the last letter of \(x\), or between two
consecutive letters of \(x\).

Equivalently, the subsets \(R_x\) are computed recursively as follows.  For
\(B\subseteq P\), let \(E(B)\) be its \(\eps\)-closure.  Then
\[
        R_\eps=E(\{(s_\psi,q_0)\}),
\]
and, for \(b\in X_m\),
\[
        R_{xb}
        =
        E\left(
        \bigcup_{p\in R_x}\Delta(p,b)
        \right),
\]
where \(\Delta(p,b)\) is the set of \(b\)-labelled successors of \(p\) in
\[
        \widehat{\Fraw}_\psi(\mathcal A).
\]

The \emph{forward automaton}
\[
        \Ffor_\psi(\mathcal A)
\]
is the deterministic \(m\)-ary automaton whose states are the subsets
\[
        R_x\qquad (x\in\W_m),
\]
with initial state \(R_\eps\), and whose transition on \(b\in X_m\) is
\[
        R_x\xrightarrow{\ b\ }R_{xb}.
\]
It is accessible by construction.
\end{definition}

Note that \(R_x\ne\emptyset\) for every \(x\in\W_m\), so that the transitions
above do not leave the displayed state set and \(\Ffor_\psi(\mathcal A)\) is a
full \(m\)-ary tree-automaton.  Indeed, by the final assertion of
Lemma~\ref{lem:forward-state-carried-data} below, applied to any point
\(\eta\in\phi(U_x)\) (a nonempty set, since \(\phi\) is onto), there is a path
labelled \(x\) from the initial state of
\(\widehat{\Fraw}_\psi(\mathcal A)\), and its terminal state belongs to
\(R_x\).

Since \(\Ffor_\psi(\mathcal A)\) is an \(m\)-ary tree-automaton, the accepted
diagram group \(\D(\Ffor_\psi(\mathcal A))\) is defined by the usual
acceptance convention of Subsection~\ref{subsec:tree-automata}.
                                                                                                                                                                         
The pullback and forward constructions encode the same closed subgroup, but
they record the information in different ways.  The pullback automaton records
a single local action of \(\phi=\psi^{-1}\), while the forward automaton
records a finite set of partial forward computations.  The next result says
that these two descriptions become equivalent after passing to a sufficiently
fine common refinement.

We first isolate a compactness argument which will be used twice.

\begin{lemma}
\label{lem:eventual-equality-compactness}
Let \(\mathcal B\) be a deterministic \(m\)-ary automaton.  Let
\(x,y\in\W_m\).  Assume that, for every \(\alpha\in\C_m\), there exists
\(j\ge0\) such that
\[
        x\alpha_1\cdots\alpha_j
        \quad\text{and}\quad
        y\alpha_1\cdots\alpha_j
\]
reach the same state in \(\mathcal B\).  Then there exists \(k\ge0\) such
that
\[
        xw
        \quad\text{and}\quad
        yw
\]
reach the same state in \(\mathcal B\) for every \(w\in X_m^k\).
\end{lemma}

\begin{proof}
For \(j\ge0\), let \(E_j\subseteq\C_m\) be the set of all
\(\alpha\in\C_m\) such that
\[
        x\alpha_1\cdots\alpha_j
        \quad\text{and}\quad
        y\alpha_1\cdots\alpha_j
\]
reach the same state in \(\mathcal B\).  Each \(E_j\) is clopen, because
membership depends only on the first \(j\) letters of \(\alpha\).

The sets \(E_j\) are increasing.  Indeed, once two words reach the same state
in a deterministic automaton, all their equally labelled descendants reach the
same state.  By assumption,
\[
        \C_m=\bigcup_{j\ge0} E_j .
\]
Since \(\C_m\) is compact, finitely many of the \(E_j\)'s cover \(\C_m\).
Because the \(E_j\)'s are increasing, there is \(k\ge0\) such that
\(E_k=\C_m .\)
Now let \(w\in X_m^k\).  Choosing any \(\alpha\in U_w\), the equality of
states at level \(k\) says exactly that \(xw\) and \(yw\) reach the same state.
\end{proof}

We next record the bookkeeping carried by a state of the subdivided raw
forward automaton.  A state \(c\) gives us a remaining output word
\(\rho_c\), a local action \(\chi_c\) of \(\psi\), and a state \(q_c\) of
\(\mathcal A\). 

\begin{lemma}
\label{lem:forward-state-carried-data}
Let \(c\) be a state of the subdivided raw forward automaton
\[
        \widehat{\Fraw}_\psi(\mathcal A).
\]
Then \(c\) determines data
\[
        \rho_c\in\W_m,\qquad
        \chi_c:\C_n\to\C_m,\qquad
        q_c\in Q,
\]
where \(\chi_c\) is a local action of \(\psi\).  In particular,
\(\chi_c\) is continuous and injective, its image \(\chi_c(\C_n)\) is clopen
in \(\C_m\).

These data have the following property.  For every path \(\gamma\) in
\(\widehat{\Fraw}_\psi(\mathcal A)\) from the initial state
\((s_\psi,q_0)\) to \(c\), let
\[
        z=z(\gamma)\in\W_m
\]
be the label of \(\gamma\).  Then there is an associated input word
\[
        a=a(\gamma)\in\W_n
\]
such that
\[
        \psi(a\omega)=z\,\rho_c\,\chi_c(\omega)
        \qquad(\omega\in\C_n),
\]
and
\[
        \tau(q_0,a)=q_c.
\]
The data \(\rho_c,\chi_c,q_c\) depend only on the terminal state \(c\), whereas
the words \(z(\gamma)\) and \(a(\gamma)\)  depend on the chosen path
\(\gamma\).

Moreover, if \(z\in\W_m\) and \(\eta\in\phi(U_z)\), then there is a path
\(\gamma\) in \(\widehat{\Fraw}_\psi(\mathcal A)\) labelled \(z\), ending at
some state \(c\in R_z\), such that
\[
        \eta\in U_{a(\gamma)}.
\]
\end{lemma}

\begin{proof}
We define the data carried by \(c\).

First suppose that \(c\) is a principal state, say
\[
        c=(s,q)\in S_\psi\times Q.
\]
Set
\[
        \rho_c=\eps,\qquad
        \chi_c=h_{T^\psi_s},\qquad
        q_c=q.
\]
Since \(s\) is accessible in the minimal transducer \(T^\psi\),
\(h_{T^\psi_s}\) is one of the local actions of \(\psi\).  Hence
\(\chi_c\) has the stated topological properties.

Let \(\gamma\) be a path in \(\widehat{\Fraw}_\psi(\mathcal A)\) ending at
\(c\), and let
\(z=z(\gamma)\)
be its label.  Let \(a=a(\gamma)\) be the input word read by the corresponding
path in the raw forward automaton.  Since the raw forward automaton is the
product of \(T^\psi\) with \(\mathcal A\), the raw path determined by \(a\)
ends at
\[
        (s^\psi_a,\tau(q_0,a))=(s,q),
\]
and has label
\(O^\psi(a)=z.\)
Therefore
\[
        \psi(a\omega)
        =
        O^\psi(a)\psi_a(\omega)
        =
        z\,h_{T^\psi_s}(\omega)
        =
        z\,\rho_c\,\chi_c(\omega),
\]
and
\(\tau(q_0,a)=q=q_c.\)

Now suppose that \(c\) is a subdivision state.  Thus
\(c=(e,\rho),\)
where \(e\) is a raw edge
\[
        e:(s,q)\xrightarrow{\lambda}(s',q')
\]
and \(\rho\) is the unread suffix of the word label \(\lambda\).  The edge
\(e\) comes from some input letter \(i\in X_n\), so
\[
        s'=t_\psi(s,i),
        \qquad
        q'=\tau(q,i),
        \qquad
        \lambda=o_\psi(s,i).
\]
Define
\[
        \rho_c=\rho,\qquad
        \chi_c=h_{T^\psi_{s'}},\qquad
        q_c=q'.
\]
Again, since \(s'\) is accessible, \(h_{T^\psi_{s'}}\) is a local action of
\(\psi\), and so \(\chi_c\) is a homeomorphism from \(\C_n\) onto the clopen
set \(\chi_c(\C_n)\).

Let \(\gamma\) be a path in \(\widehat{\Fraw}_\psi(\mathcal A)\) ending at
\(c\), and let
\(z=z(\gamma)\)
be its label.  The corresponding raw path reaches the source \((s,q)\) of
\(e\), then reads the input letter \(i\), and then stops partway through the
output word \(\lambda\).  Let \(a_0\) be the input word read before the edge
\(e\), and set
\(a=a_0i.\)
If the part of \(\lambda\) already read is \(\mu\), then
\[
        \lambda=\mu\rho,
        \qquad
        z=O^\psi(a_0)\mu.
\]
Hence
\[
        O^\psi(a)
        =
        O^\psi(a_0)o_\psi(s,i)
        =
        O^\psi(a_0)\mu\rho
        =
        z\rho.
\]
Also
\[
        s^\psi_a=s',
        \qquad
        \tau(q_0,a)=q'.
\]
Therefore
\[
        \psi(a\omega)
        =
        O^\psi(a)\psi_a(\omega)
        =
        z\rho\,h_{T^\psi_{s'}}(\omega)
        =
        z\,\rho_c\,\chi_c(\omega),
\]
and
\(\tau(q_0,a)=q'=q_c.\)

It remains to prove the final assertion.  Let \(\eta\in\phi(U_z)\).  Then
\(\psi(\eta)\in U_z.\)
Run the transducer \(T^\psi\) on the input \(\eta\), and follow the
corresponding path in the raw forward automaton.  After subdividing
word-labelled edges, the output letters are read one at a time.  Since
\(\psi(\eta)\) begins with \(z\), there is a point along the subdivided path at
which exactly the word \(z\) has been read.  Let \(c\) be the state reached at
that point, and let \(\gamma\) be the path just followed.  Then \(c\in R_z\).
If \(a=a(\gamma)\) is the input word consumed up to that point, then the
original input \(\eta\) has prefix \(a\).  Hence
\(\eta\in U_a.\)
\end{proof}

The next lemma is a small continuation principle for the forward automaton.  It
says that if two forward computations have reached the same principal state,
then the subset states obtained by reading a common remaining output word are
the same, provided the relevant cylinders force all computations to pass
through that principal state.

\begin{lemma}
\label{lem:common-principal-continuation}
Let \(a,b\in\W_n\), let \(z,z',\zeta\in\W_m\), and suppose that the raw paths
determined by \(a\) and \(b\) both end at the same principal state \(c_0\) of
\(\Fraw_\psi(\mathcal A)\), with labels \(z\) and \(z'\), respectively.
Assume also that
\[
        \phi(U_{z\zeta})\subseteq U_a,
        \qquad
        \phi(U_{z'\zeta})\subseteq U_b.
\]
Then
\[
        R_{z\zeta}=R_{z'\zeta}.
\]
\end{lemma}

\begin{proof}
Let \(C_\zeta\) be the set of states of
\(\widehat{\Fraw}_\psi(\mathcal A)\) reachable from \(c_0\) by a path labelled
\(\zeta\), allowing \(\eps\)-edges as in the subset construction.

Since the raw path determined by \(a\) reaches \(c_0\) with label \(z\), every
state in \(C_\zeta\) is reachable from the initial state by a path labelled
\(z\zeta\).  Hence
\[
        C_\zeta\subseteq R_{z\zeta}.
\]
Conversely, let \(d\in R_{z\zeta}\).  Choose a path \(\gamma\) in
\(\widehat{\Fraw}_\psi(\mathcal A)\) from the initial state to \(d\), labelled
\(z\zeta\).  By Lemma~\ref{lem:forward-state-carried-data}, this path has an
associated input word \(u=a(\gamma)\) such that
\[
        \psi(U_u)\subseteq U_{z\zeta}.
\]
Equivalently,
\[
        U_u\subseteq \phi(U_{z\zeta})\subseteq U_a.
\]
Thus \(u\) extends \(a\).  The raw path determined by \(u\) therefore factors
through the principal state \(c_0\) after reading the prefix \(a\).  Since the
label accumulated up to that point is \(z\), the remaining part of the
subdivided path has label \(\zeta\).  Hence \(d\in C_\zeta\).  Therefore
\(R_{z\zeta}=C_\zeta.\)

The same argument, using \(b\) and \(z'\), gives
\(R_{z'\zeta}=C_\zeta.\)
Thus
\(R_{z\zeta}=R_{z'\zeta}.\)
\end{proof}

We can now compare the two constructions.  The first implication starts from
equality of forward states.  A common state \(c\) in the two subsets gives the
same remaining output \(\rho_c\), the same local action \(\chi_c\) of
\(\psi\), and the same \(\mathcal A\)-state.  After refining far enough along a
ray, the corresponding local actions of \(\phi=\psi^{-1}\) are therefore the
same.

\begin{theorem}
\label{thm:pullback-forward-eventual-equivalence}
Let
\[
        \psi:\C_n\to\C_m
\]
be a  homeomorphism preserving or reversing the lexicographic order,
let \(\phi=\psi^{-1}\), and let
\[
        \mathcal A=(Q,\tau,q_0)
\]
be a full \(n\)-ary tree-automaton.  Let
\[
        \mathcal P=\mathcal P_\psi(\mathcal A),
        \qquad
        \mathcal F=\Ffor_\psi(\mathcal A)
\]
be the pullback and forward automata.  For \(x\in\W_m\), let \(P_x\) be the
state reached by \(x\) in \(\mathcal P\), and let \(R_x\) be the state reached
by \(x\) in \(\mathcal F\).

Then, for all \(x,y\in\W_m\), the following hold.
\begin{enumerate}[label=(\roman*)]
    \item If \(R_x=R_y\), then there exists \(k\ge0\) such that
    \[
            P_{xw}=P_{yw}
            \qquad\text{for every }w\in X_m^k.
    \]
    \item If \(P_x=P_y\), then there exists \(k\ge0\) such that
    \[
            R_{xw}=R_{yw}
            \qquad\text{for every }w\in X_m^k.
    \]
\end{enumerate}
\end{theorem}

\begin{proof}
We prove (i).  By Lemma~\ref{lem:eventual-equality-compactness}, applied to
the deterministic automaton \(\mathcal P\), it is enough to prove the following
pointwise statement: for every \(\alpha\in\C_m\), there exists \(j\ge0\) such
that
\[
        P_{x\alpha_1\cdots\alpha_j}
        =
        P_{y\alpha_1\cdots\alpha_j}.
\]

Fix \(\alpha\in\C_m\), and put
\(\eta=\phi(x\alpha).\)
By Lemma~\ref{lem:forward-state-carried-data}, choose a path
\(\gamma_x\) in \(\widehat{\Fraw}_\psi(\mathcal A)\), labelled \(x\), ending
at some state \(c\in R_x\), with associated input word
\(a=a(\gamma_x),\)
such that
\(\eta\in U_a.\)
Write
\(\eta=a\beta.\)
The data carried by \(c\) give
\[
        \psi(a\omega)=x\,\rho_c\,\chi_c(\omega)
        \qquad(\omega\in\C_n).
\]
Substituting \(\omega=\beta\), and using \(\psi(\eta)=x\alpha\), gives
\[
        \alpha=\rho_c\,\chi_c(\beta).
\]

Since \(R_x=R_y\), the same state \(c\) also lies in \(R_y\).  Choose a path
\(\gamma_y\) in \(\widehat{\Fraw}_\psi(\mathcal A)\), labelled \(y\), ending
at \(c\), and let
\(b=a(\gamma_y)\)
be its associated input word.  Again by
Lemma~\ref{lem:forward-state-carried-data},
\[
        \psi(b\omega)=y\,\rho_c\,\chi_c(\omega)
        \qquad(\omega\in\C_n),
\]
and
\[
        \tau(q_0,b)=q_c=\tau(q_0,a).
\]
Substituting \(\omega=\beta\), we get
\[
        \psi(b\beta)
        =
        y\,\rho_c\,\chi_c(\beta)
        =
        y\alpha.
\]
Therefore
\(\phi(y\alpha)=b\beta.\)

Choose \(j\) large enough that \(j\ge |\rho_c|\) and
\[
        \phi(U_{x\alpha_1\cdots\alpha_j})\subseteq U_a,
        \qquad
        \phi(U_{y\alpha_1\cdots\alpha_j})\subseteq U_b.
\]
This is possible because the two cylinder neighborhoods shrink to the points
\[
        \phi(x\alpha)=a\beta,
        \qquad
        \phi(y\alpha)=b\beta.
\]
Write
\[
        \alpha_1\cdots\alpha_j=\rho_c\delta.
\]
The first inclusion gives
\[
        U_{x\rho_c\delta}
        =
        U_{x\alpha_1\cdots\alpha_j}
        \subseteq
        \psi(U_a).
\]
Since
\[
        \psi(U_a)=x\rho_c\,\chi_c(\C_n),
\]
deleting the common prefix \(x\rho_c\) gives
\[
        U_\delta\subseteq \chi_c(\C_n).
\]

We now regard
\[
        \chi_c^{-1}:\chi_c(\C_n)\to\C_n
\]
as a map defined on the clopen subset \(\chi_c(\C_n)\subseteq\C_m\).  By the
definition of local actions for maps with clopen domain, the inverse map has a
root and a local action at \(\delta\).  Put
\[
        r=\theta_{\chi_c^{-1}}(\delta)
         =\Root\bigl(\chi_c^{-1}(U_\delta)\bigr),
        \qquad
        \lambda=(\chi_c^{-1})_\delta .
\]
Then
\[
        \chi_c^{-1}(\delta\omega)=r\,\lambda(\omega)
        \qquad(\omega\in\C_m),
\]
where
$
        \lambda:\C_m\to\C_n
$
is continuous and injective and satisfies
$
        \Root(\lambda(\C_m))=\eps.
$

We compute the local action of \(\phi\) on the refined cylinders.  The formula
\[
        \psi(a\xi)=x\rho_c\,\chi_c(\xi)
        \qquad(\xi\in\C_n)
\]
says that, whenever the suffix after \(x\rho_c\) lies in \(\chi_c(\C_n)\),
the inverse map \(\phi\) is obtained by applying \(\chi_c^{-1}\) to that
suffix and then adding the prefix \(a\).  Hence, for every \(\omega\in\C_m\),
\[
\begin{aligned}
        \phi(x\rho_c\delta\omega)
        &=a\,\chi_c^{-1}(\delta\omega)  \\
        &=ar\,\lambda(\omega).
\end{aligned}
\]
Similarly, the formula
\[
        \psi(b\xi)=y\rho_c\,\chi_c(\xi)
        \qquad(\xi\in\C_n)
\]
gives
\[
        \phi(y\rho_c\delta\omega)
        =
        br\,\lambda(\omega)
        \qquad(\omega\in\C_m).
\]
Since
\[
        \Root(\lambda(\C_m))=\eps,
\]
we obtain
\[
        O^\phi(x\alpha_1\cdots\alpha_j)=ar,
        \qquad
        O^\phi(y\alpha_1\cdots\alpha_j)=br,
\]
and the two resulting local actions of \(\phi\) are both \(\lambda\).  Thus
the two words
\[
        x\alpha_1\cdots\alpha_j
        \quad\text{and}\quad
        y\alpha_1\cdots\alpha_j
\]
reach the same transducer state in \(T^\phi\).

For the \(\mathcal A\)-coordinate, we have
\[
\begin{aligned}
        \tau(q_0,ar)
        &=\tau(\tau(q_0,a),r)  \\
        &=\tau(\tau(q_0,b),r)  \\
        &=\tau(q_0,br).
\end{aligned}
\]
Therefore
\[
        P_{x\alpha_1\cdots\alpha_j}
        =
        P_{y\alpha_1\cdots\alpha_j}.
\]
By Lemma~\ref{lem:eventual-equality-compactness}, there exists \(k\ge0\) such
that
\[
        P_{xw}=P_{yw}
        \qquad(w\in X_m^k).
\]
This proves (i).

We now prove (ii).  Here the starting assumption is equality of pullback
states.  Thus the map \(\phi\) has the same local action after \(x\) and after
\(y\), and the two corresponding source prefixes end at the same state of
\(\mathcal A\).  We show that, after refining along any ray, the two forward
computations pass through a common principal state.

Assume
\(P_x=P_y.\)
By Lemma~\ref{lem:eventual-equality-compactness}, applied to the deterministic
automaton \(\mathcal F\), it is enough to prove that for every
\(\alpha=\alpha_1\cdots\alpha_j\cdots\in\C_m\), there exists \(j\ge0\) such that
\[
        R_{x\alpha_1\cdots\alpha_j}
        =
        R_{y\alpha_1\cdots\alpha_j}.
\]

Write
\[
        p=O^\phi(x),
        \qquad
        p'=O^\phi(y).
\]
Since \(P_x=P_y\), we have
\(s^\phi_x=s^\phi_y\)
and
\(\tau(q_0,p)=\tau(q_0,p').\)
By the state--local-action dictionary for the minimal transducer \(T^\phi\),
there is a single local action
\(\kappa:\C_m\to\C_n\)
such that
\(\kappa=\phi_x=\phi_y.\)
Thus
\[
        \phi(x\omega)=p\,\kappa(\omega),
        \qquad
        \phi(y\omega)=p'\,\kappa(\omega)
        \qquad(\omega\in\C_m).
\]

Fix \(\alpha=\alpha_1\cdots\alpha_j\cdots\in\C_m\), and set
\(\beta=\kappa(\alpha).\)
Since \(\kappa\) is a local action of the homeomorphism \(\phi\), its image
\(\kappa(\C_m)\) is clopen in \(\C_n\).  Choose a prefix \(r\) of \(\beta\)
such that
\[
        U_r\subseteq \kappa(\C_m).
\]
We regard
\[
        \kappa^{-1}:\kappa(\C_m)\to\C_m
\]
as a map defined on the clopen subset \(\kappa(\C_m)\subseteq\C_n\).  By the
definition of local actions for maps with clopen domain, put
\[
        \delta=\theta_{\kappa^{-1}}(r)
              =\Root\bigl(\kappa^{-1}(U_r)\bigr),
        \qquad
        \lambda=(\kappa^{-1})_r .
\]
Then
\[
        \kappa^{-1}(r\nu)=\delta\,\lambda(\nu)
        \qquad(\nu\in\C_n),
\]
where
$
        \lambda:\C_n\to\C_m
$
is continuous and injective and satisfies
$
        \Root(\lambda(\C_n))=\eps.
$
Equivalently,
\[
        \kappa(\delta\,\lambda(\nu))=r\nu
        \qquad(\nu\in\C_n).
\]

Using \(\psi=\phi^{-1}\), we obtain, for every \(\nu\in\C_n\),
\[
\begin{aligned}
        \phi(x\delta\,\lambda(\nu))
        &=p\,\kappa(\delta\,\lambda(\nu)) \\
        &=pr\nu,
\end{aligned}
\]
and hence
\[
        \psi(pr\nu)=x\delta\,\lambda(\nu).
\]
Similarly,
\[
        \psi(p'r\nu)=y\delta\,\lambda(\nu).
\]
Since
$
        \Root(\lambda(\C_n))=\eps,
$
the raw paths determined by \(pr\) and \(p'r\) in
\(\Fraw_\psi(\mathcal A)\) have labels \(x\delta\) and \(y\delta\),
respectively.  Moreover, these two raw paths end at the same principal state.
Their transducer coordinates are equal because the displayed formulas show
that the local action of \(\psi\) after both \(pr\) and \(p'r\) is
\(\lambda\).  Their \(\mathcal A\)-coordinates are equal because
\[
\begin{aligned}
        \tau(q_0,pr)
        &=\tau(\tau(q_0,p),r)  \\
        &=\tau(\tau(q_0,p'),r) \\
        &=\tau(q_0,p'r).
\end{aligned}
\]
Call this common principal state \(c_0\).

Choose \(j\) large enough that
\[
        U_{\alpha_1\cdots\alpha_j}\subseteq \kappa^{-1}(U_r).
\]
Since
\[
        \kappa^{-1}(U_r)\subseteq U_\delta,
\]
we can write
\[
        \alpha_1\cdots\alpha_j=\delta\zeta
\]
for some \(\zeta\in\W_m\).  Then
\[
        \kappa(U_{\delta\zeta})\subseteq U_r,
\]
and therefore
\[
        \phi(U_{x\delta\zeta})
        =
        p\,\kappa(U_{\delta\zeta})
        \subseteq
        U_{pr}.
\]
Similarly,
\[
        \phi(U_{y\delta\zeta})
        =
        p'\,\kappa(U_{\delta\zeta})
        \subseteq
        U_{p'r}.
\]
We may therefore apply Lemma~\ref{lem:common-principal-continuation} with
\[
        a=pr,\qquad
        b=p'r,\qquad
        z=x\delta,\qquad
        z'=y\delta.
\]
It gives
\[
        R_{x\delta\zeta}=R_{y\delta\zeta},
\]
that is,
\[
        R_{x\alpha_1\cdots\alpha_j}
        =
        R_{y\alpha_1\cdots\alpha_j}.
\]
By Lemma~\ref{lem:eventual-equality-compactness}, there exists \(k\ge0\) such
that
\[
        R_{xw}=R_{yw}
        \qquad(w\in X_m^k).
\]
This proves (ii).
\end{proof}

\begin{corollary}
\label{cor:folded-pullback-forward-agree}
Under the hypotheses of
Theorem~\ref{thm:pullback-forward-eventual-equivalence}, the folded quotients
\[
        \overline{\mathcal P_\psi(\mathcal A)}
        \qquad\text{and}\qquad
        \overline{\Ffor_\psi(\mathcal A)}
\]
are canonically isomorphic.
\end{corollary}

\begin{proof}
Put
\[
        \mathcal P=\mathcal P_\psi(\mathcal A),
        \qquad
        \mathcal F=\Ffor_\psi(\mathcal A).
\]
For $x\in\W_m$, let $P_x$ and $R_x$ be the states reached by $x$ in
$\mathcal P$ and $\mathcal F$, respectively.  By
Lemma~\ref{lem:folded-quotient-finite-depth} and
Theorem~\ref{thm:pullback-forward-eventual-equivalence}, for all
$x,y\in\W_m$ we have
\[
        P_x=P_y\text{ in }\overline{\mathcal P}
        \quad\Longleftrightarrow\quad
        R_x=R_y\text{ in }\overline{\mathcal F}.
\]
Thus the rule
\[
        \overline{P_x}\longmapsto \overline{R_x}
\]
is well defined and bijective.  It sends the root to the root and commutes
with the $m$ outgoing transitions, since reading $a\in X_m$ sends $x$ to $xa$
on both sides.  Hence it is an isomorphism of folded automata.
\end{proof}

\begin{corollary}
\label{cor:forward-construction-correct}
Let $\psi:\C_n\to\C_m$ be a  homeomorphism preserving or reversing
lexicographic order, let $\mathcal A$ be a full $n$-ary tree-automaton, and
put
\[
        H=\D(\mathcal A)\le F_n .
\]
Then
\[
        \D(\Ffor_\psi(\mathcal A))
        =
        H^\psi\cap F_m .
\]
In particular, if $H^\psi\le F_m$, then
\[
        \D(\Ffor_\psi(\mathcal A))=H^\psi .
\]
\end{corollary}

\begin{proof}
By Corollary~\ref{cor:folded-pullback-forward-agree},
Lemma~\ref{lem:folding-preserves-D}, and
Theorem~\ref{thm:pullback-theorem},
\[
\begin{aligned}
        \D(\Ffor_\psi(\mathcal A))
        &=
        \D(\overline{\Ffor_\psi(\mathcal A)}) \\
        &=
        \D(\overline{\mathcal P_\psi(\mathcal A)}) \\
        &=
        \D(\mathcal P_\psi(\mathcal A)) \\
        &=
        H^\psi\cap F_m .
\end{aligned}
\]
The final assertion follows immediately if $H^\psi\le F_m$.
\end{proof}

\begin{corollary}
\label{cor:forward-precore-subgroup-case}
Assume, in addition to the hypotheses of
Corollary~\ref{cor:forward-construction-correct}, that $\mathcal A$ has the
existence property and that
\[
        H^\psi\le F_m .
\]
Then $\Ffor_\psi(\mathcal A)$ has the existence property.  Hence
$\Ffor_\psi(\mathcal A)$ is a pre-core for $H^\psi$. In particular, the core of $H^\psi$ can be obtained by folding and reducing $\Ffor_\psi(\mathcal A)$.
\end{corollary}

\begin{proof}
By Proposition~\ref{prop:pullback-precore}, the pullback automaton
$\mathcal P_\psi(\mathcal A)$ has the existence property.  Hence
$\overline{\mathcal P_\psi(\mathcal A)}$ has the existence property by
Lemma~\ref{lem:existence-property-folding}.  By
Corollary~\ref{cor:folded-pullback-forward-agree}, the folded quotient
$\overline{\Ffor_\psi(\mathcal A)}$ has the existence property as well.
Applying Lemma~\ref{lem:existence-property-folding} once more, now in the
unfolding direction, shows that $\Ffor_\psi(\mathcal A)$ has the existence
property.

By Corollary~\ref{cor:forward-construction-correct},
\[
        \D(\Ffor_\psi(\mathcal A))=H^\psi .
\]
Thus $\Ffor_\psi(\mathcal A)$ is a pre-core for $H^\psi$.
\end{proof}

\begin{remark}
\label{rem:finite-stopped-forward-nonfull-core}
The forward construction was stated for full automata.  Thus, if
\(\mathcal A=\C(H)\) is finite but not full, one may first replace it by its
full completion \(\widehat{\mathcal A}\); this does not change the accepted
subgroup.  However, the full completion may be infinite.

In the semi-synchronizing case this infinite completion can be replaced by a
finite stopped construction.  Assume that
$
        \psi:\C_n\to\C_m
$
is order-preserving or order-reversing, and that \(T^\psi\) is
semi-synchronizing.  By Theorem~\ref{thm:conjugator-characterization},
\(F_n^\psi\le F_m\), and hence \(H^\psi\le F_m\).

Let \(d\) be a common stabilization depth for the two boundary rays, and let
\(k\) be an inner synchronization level.  Starting from \(\mathcal A\), perform
the following finite extension.  If the left boundary path stops before depth
\(d\), say at a leaf reached by \(0^r\) with \(r<d\), attach a left vine of
length \(d-r\): that is, attach an \(n\)-caret at that leaf, then an
\(n\)-caret at its \(0\)-child, and so on, until the word \(0^d\) is readable.
Do the analogous construction on the right boundary path, using the
\((n-1)\)-child.  Then attach a complete finite \(n\)-ary tree of height \(k\)
at every inner leaf of the automaton obtained.  Denote the resulting finite
extension by \(\mathcal A^\dagger\).

Run the raw forward construction on \(\mathcal A^\dagger\), but stop at the
leaves of \(\mathcal A^\dagger\).  Equivalently, include the usual raw edge
\[
        (s,q)
        \xrightarrow{\ o_\psi(s,i)\ }
        \bigl(t_\psi(s,i),\tau(q,i)\bigr)
\]
only when \(q\) is a father state of \(\mathcal A^\dagger\).

Now let \(c=(s,q)\) be a terminal principal state, with \(q\) a leaf of
\(\mathcal A^\dagger\).  Choose a finite prefix code
\(P_s\subseteq\W_m\) such that
\[
        h_{T^\psi_s}(\C_n)=\bigsqcup_{p\in P_s}U_p .
\]
Attach at \(c\) a finite, not necessarily full, \(m\)-ary prefix tree whose
branches are the words \(p\in P_s\).  The root of this tree is \(c\), and its
terminal vertices are denoted \(c_p\), \(p\in P_s\).  Distinct terminal
principal states receive disjoint copies of these prefix trees.

After subdividing the remaining word-labelled edges and applying the subset
construction, we obtain a finite \(m\)-ary tree-automaton, denoted
$
        \Ffor_{\psi}^{\operatorname{stop}}(\mathcal A).
$
This automaton is a pre-core for \(H^\psi\).  Hence the core
\(\C(H^\psi)\) is obtained from
\(\Ffor_{\psi}^{\operatorname{stop}}(\mathcal A)\) by folding and reducing. We will not make use of this remark in the paper and for the sake of brevity, we do not give a proof.
\end{remark}
                                                                                                                                                                                                                                          
\subsection{The Jones oriented subgroup of \(F_3\)}

Recall that Vaughan Jones introduced a method for constructing unitary representations of
Thompson's groups using planar algebras \cite{JonesPlanarAlgebras,JonesUnitary}.
This construction was later developed in a broader categorical framework which also
applies to Higman--Thompson groups \cite{JonesNoGo}.  These representations gave rise to interesting subgroups of Thompson groups. Indeed, in some of the representations, the stabilizer of a canonical vector, called the vacuum vector, gave rise to fascinating subgroups of Thompson groups. We recall here several Jones subgroups that have
received detailed group-theoretic study.
        
The first and best-known example is Jones's oriented subgroup
\(\vec F=\vec F_2\le F,\)
also called the Jones subgroup of \(F\).  Jones's construction associates knots
and links to elements of \(F\), all links and knots are encoded by elements of $F$. Similarly, elements of  \(\vec F\) can be used to encode all oriented links and knots
\cite{JonesUnitary,JonesKnotsLinks,AielloAlexander}.  In joint work with
Sapir, we proved that
\(\vec F\cong F_3,\)
described \(\vec F\) as the stabilizer of a certain subset of the dyadic
rationals, and showed that \(\vec F\) coincides with its commensurator in
\(F\) \cite{GolanSapirJones}.  We also used \(\vec F\) to construct the first
example of a maximal subgroup of \(F\) of infinite index which does not fix any
point of \((0,1)\) \cite{GolanSapirSubgroupsF}.  That maximal subgroup is the
preimage of \(\vec F\) under an injective endomorphism from \(F\) onto an
index-two subgroup of \(F\).

Jones's 3-colorable subgroup \(\mathcal F\le F\) is another example that has
been studied in detail.  It arises as the stabilizer of the vacuum vector in
one of Jones's representations and has a direct geometric description in terms
of colorings of the planar diagram associated with an element of \(F\)
\cite{JonesNoGo}.  Ren proved that
\(\mathcal F\cong F_4\)
\cite{RenSkein}.  Aiello and Nagnibeda subsequently proved that the associated
quasi-regular representation is irreducible and constructed three explicit
maximal subgroups of infinite index containing the preimage of \(\mathcal F\)
under a suitable injective endomorphism of \(F\)
\cite{AielloNagnibedaThreeColorable}.

Jones later extended his knot and link construction to the ternary group
\(F_3\) and introduced the ternary oriented subgroup
\[
        \overrightarrow{F}_3\le F_3.
\]
Thus every link can be represented by an element of \(F_3\), and every
oriented link can be represented by an element of \(\overrightarrow{F}_3\).
Aiello and Nagnibeda proved that \(\overrightarrow{F}_3\) is finitely
generated, described it as the stabilizer of a certain subset of the triadic
rationals, and showed that it coincides with its commensurator.  They also
proved that \(\overrightarrow{F}_3\) is maximal in an index-two subgroup
\(G_3\cong F_3\), thereby obtaining a maximal subgroup of \(F_3\) isomorphic to
\(\overrightarrow{F}_3\).  At the time, this was the only known maximal
subgroup of infinite index of \(F_3\) which did not fix a point of \((0,1)\)
\cite{AielloNagnibedaFThree}.

In this section we prove that
\[
        \overrightarrow{F}_3\cong F_4 .
\]
This answers questions of Aiello from \cite{AielloIntro}, asking whether
\(\overrightarrow{F}_3\) is finitely presented and whether it is isomorphic to
a Higman--Thompson group, or to some other known group.
   
The idea is to construct an explicit order-preserving homeomorphism
\(\psi_J:\C_4\to \C_3\)
represented by a semi-synchronizing transducer, compute the core of the
conjugated copy
\(F_4^{\psi_J}\le F_3\)
using the forward construction, and then identify the resulting core with the
core of \(\overrightarrow{F}_3\).

We now define the transducer.
      Let \(T_J\) be the initial \((4,3)\)-transducer with state set
\(\{E,a,c,d\},\)
initial state \(E\), and transition-output table below.  In the table, the
entry
$       w:s'
$
in the column labelled \(i\) means that, on input \(i\), the transducer
outputs \(w\in\W_3\) and moves to the state \(s'\).

\[
\begin{array}{c|cccc|c}
        &0&1&2&3&\operatorname{im}(h_s)\\ \hline
E       &0:E    &1:a  &2:c  &22:E  &\C_3\\
a       &00:a   &0:d  &1:E  &2:a   &\C_3\\
c       &0:c    &02:E &10:a &1:d   &U_0\sqcup U_1\\
d       &1:c    &12:E &20:a &2:d   &U_1\sqcup U_2 .
\end{array}
\tag{J1}
\label{eq:jones-F4-to-F3-transducer-table}
\]

Using the table,
  one can check that \(T_J\) represents an order-preserving homeomorphism
\[
        \psi_J:\C_4\longrightarrow \C_3 .
\]
The transducer \(T_J\) is minimal, that is, distinct states induce distinct
local actions: the images of \(h_c\) and \(h_d\) are distinct from each other
and from \(\C_3\), and \(h_E\ne h_a\) because \(h_E(U_1)=U_1\) while
\(h_a(U_1)\subseteq U_0\).  Thus \(T_J\) is the minimal transducer
\(T^{\psi_J}\).

The transducer is semi-synchronizing.  The boundary rays stabilize because
\[
        E\cdot0=E,
        \qquad
        E\cdot3=E,
\]
and the inner synchronization level is \(1\).  Indeed, the states of \(T_J\)
record the residue \(\sigma_4\) modulo \(3\): inner words reach \(E\), \(a\),
or one of \(c,d\) according to whether the residue is \(0\), \(1\) or \(2\),
and the two states \(c,d\) have the same transition function:
\[
\begin{array}{c|cccc}
        &0&1&2&3\\ \hline
c       &c&E&a&d\\
d       &c&E&a&d .
\end{array}
\]
Thus the minimal transducer of \(\psi_J\) is semi-synchronizing, and
Theorem~\ref{thm:conjugator-characterization} gives
\(F_4^{\psi_J}\le F_3 .\)

The transducer is drawn in Figure~\ref{fig:jones-F4-F3-transducer}.

\begin{figure}[htbp]
\centering
\begin{tikzpicture}[
    >=Stealth,
    auto,
    node distance=3.2cm,
    state/.style={
        circle,
        draw,
        thick,
        minimum size=1.1cm,
        font=\small\bfseries,
        inner sep=1pt,
        fill=white
    },
    lab/.style={
        font=\scriptsize,
        inner sep=1pt,
        fill=white
    },
    every loop/.style={looseness=5}
]
    \node[state] (E) at (-2.8, 1.7) {$E$};
    \node[state] (a) at ( 2.8, 1.7) {$a$};
    \node[state] (c) at (-2.8,-1.7) {$c$};
    \node[state] (d) at ( 2.8,-1.7) {$d$};

    \draw[->, thick] (-4.15,1.7)--(E);

    \path[->, thick]
        (E) edge[loop above] node[lab] {$0|0,\ 3|22$} (E)
        (a) edge[loop above] node[lab] {$0|00,\ 3|2$} (a)
        (c) edge[loop below] node[lab] {$0|0$} (c)
        (d) edge[loop below] node[lab] {$3|2$} (d)

        (E) edge[bend left=13] node[lab, above] {$1|1$} (a)
        (a) edge[bend left=13] node[lab, below] {$2|1$} (E)

        (c) edge[bend left=13] node[lab, below] {$3|1$} (d)
        (d) edge[bend left=13] node[lab, above] {$0|1$} (c)

        (E) edge[bend right=13] node[lab, left] {$2|2$} (c)
        (c) edge[bend right=13] node[lab, right] {$1|02$} (E)

        (a) edge[bend left=13] node[lab, right] {$1|0$} (d)
        (d) edge[bend left=13] node[lab, left] {$2|20$} (a)

        (c) edge[bend left=8] node[lab, above, sloped] {$2|10$} (a)
        (d) edge[bend left=8] node[lab, above, sloped] {$1|12$} (E);
\end{tikzpicture}
\caption{The minimal \((4,3)\)-transducer \(T_J\) representing \(\psi_J\).
Edge labels are input--output labels; the full transition-output table is
\eqref{eq:jones-F4-to-F3-transducer-table}.}
\label{fig:jones-F4-F3-transducer}
\end{figure}
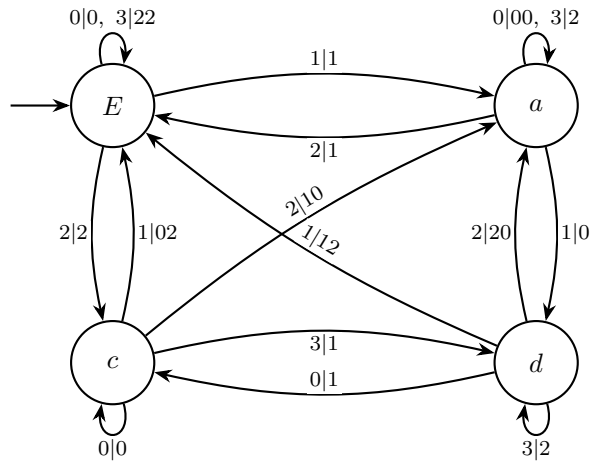

Let \(\mathcal J\) be the following full ternary tree-automaton.  Its states
are
\[
        \rho,\ell,r_0,r_1,
        \zeta_{00},\zeta_{01},\zeta_{10},\zeta_{11},
\]
with initial state \(\rho\), and transition table
\[
\begin{array}{c|ccc}
        &0&1&2\\ \hline
\rho        &\ell        &\zeta_{11} &r_1\\
\ell        &\ell        &\zeta_{11} &\zeta_{01}\\
r_1         &\zeta_{01}  &\zeta_{10} &r_0\\
r_0         &\zeta_{00}  &\zeta_{11} &r_1\\
\zeta_{00}  &\zeta_{00}  &\zeta_{11} &\zeta_{01}\\
\zeta_{01}  &\zeta_{01}  &\zeta_{10} &\zeta_{00}\\
\zeta_{10}  &\zeta_{11}  &\zeta_{01} &\zeta_{10}\\
\zeta_{11}  &\zeta_{10}  &\zeta_{00} &\zeta_{11}.
\end{array}
\tag{J2}
\label{eq:jones-core-table}
\]

\begin{proposition}
\label{prop:core-of-F4psi-is-J}
The core of \(F_4^{\psi_J}\le F_3\) is the automaton \(\mathcal J\).
\end{proposition}

\begin{proof}
We run the forward construction with the core
\(\mathcal K_4=\C(F_4)\) from Lemma~\ref{lem:explicit-core-of-Fn}, and first
identify the raw forward automaton \(\Fraw_{\psi_J}(\mathcal K_4)\).  The
principal state reached by an input word \(u\in\W_4\) is the pair consisting
of the state of \(T_J\) reached by \(u\) and the state of \(\mathcal K_4\)
reached by \(u\).  Since the states of \(T_J\) record the residue \(\sigma_4\)
modulo \(3\), the transducer coordinate is \(E\) when \(u\) is empty, a
boundary word, or an inner word of residue \(0\), and it is \(a\) or one of
\(c,d\) when \(u\) is inner of residue \(1\) or \(2\), respectively; the
\(\mathcal K_4\)-coordinate records which of these cases occurs.  Hence the
accessible principal states are exactly the seven pairs
\[
        (E,q_0),\ (E,q_L),\ (E,q_R),\ (E,q^{\mathrm{in}}_0),\
        (a,q^{\mathrm{in}}_1),\ (c,q^{\mathrm{in}}_2),\ (d,q^{\mathrm{in}}_2),
\]
which we abbreviate, in this order, by
\[
        \rho,\ \ell,\ r,\ b,\ a,\ c,\ d .
\]
(The last three abbreviations are unambiguous, since the
\(\mathcal K_4\)-coordinate of a principal state with transducer coordinate
\(a\), \(c\) or \(d\) is forced.)  Reading the transitions of \(T_J\) and of
\(\mathcal K_4\) simultaneously gives the following table for
\(\Fraw_{\psi_J}(\mathcal K_4)\); the entry \(w:s\) in the column labelled
\(i\) means that the raw edge corresponding to the input letter \(i\) has
word label \(w\) and target \(s\):
\[
\begin{array}{c|cccc}
        &0&1&2&3\\ \hline
\rho    &0:\ell &1:a  &2:c  &22:r\\
\ell    &0:\ell &1:a  &2:c  &22:b\\
r       &0:b    &1:a  &2:c  &22:r\\
b       &0:b    &1:a  &2:c  &22:b\\
a       &00:a   &0:d  &1:b  &2:a\\
c       &0:c    &02:b &10:a &1:d\\
d       &1:c    &12:b &20:a &2:d .
\end{array}
\tag{J3}
\label{eq:jones-raw-forward-table}
\]

Subdivide every raw edge whose output has length \(2\).  If
\[
        s\xrightarrow{\ i\mid uv\ }t
\]
is such an edge, write \([s,i]\), only in this computation, for the
subdivision state reached after the first output letter \(u\) has been read.
There are no \(\eps\)-outputs, so no \(\eps\)-closures occur in the subset
construction.  The reachable subset states are:
\[
\begin{array}{c|c}
x & R_x\\ \hline
\eps & \{\rho\}\\
0    & \{\ell\}\\
1    & \{a\}\\
2    & \{[\rho,3],c\}\\
02   & \{[\ell,3],c\}\\
10   & \{[a,0],d\}\\
11   & \{b\}\\
20   & \{[c,1],c\}\\
21   & \{[c,2],d\}\\
22   & \{r\}\\
101  & \{[d,1],c\}\\
102  & \{[d,2],d\}\\
112  & \{[b,3],c\}\\
222  & \{[r,3],c\}.
\end{array}
\tag{J4}
\label{eq:jones-forward-subsets}
\]
Here \(R_x\) denotes the state reached after reading the ternary word \(x\) in
the determinized forward automaton.

Folding identifies the subset states as follows:
\[
\begin{array}{c|c}
\text{state of }\mathcal J & \text{subset states in the folded class}\\ \hline
\rho       & R_\eps\\
\ell       & R_0\\
\zeta_{11} & R_1\\
r_1        & R_2,\ R_{222}\\
\zeta_{01} & R_{02},\ R_{20},\ R_{101},\ R_{112}\\
\zeta_{10} & R_{10},\ R_{21},\ R_{102}\\
\zeta_{00} & R_{11}\\
r_0        & R_{22}.
\end{array}
\tag{J5}
\label{eq:jones-folding-classes}
\]
A direct check of the three children of each displayed class gives precisely
the transition table \eqref{eq:jones-core-table}.  No two states in
\(\mathcal J\) have the same ordered triple of children, so the quotient is
folded.

Since \(F_4^{\psi_J}\le F_3\), Corollary~\ref{cor:forward-precore-subgroup-case}
shows that the forward automaton obtained from \(\mathcal K_4=\C(F_4)\) is a
pre-core for \(F_4^{\psi_J}\).  Its folded quotient is finite and full, so
there are no hanging trees to delete.  By the computation above, this folded
quotient is \(\mathcal J\).
\end{proof}

It remains to identify this core with the core of the ternary oriented Jones
subgroup $\vec F_3$.

One way to identify the core is to start with the nine generators of
\(\vec F_3\) found by Aiello and Nagnibeda in
\cite[Theorem~2]{AielloNagnibedaFThree} and then compute the core by the usual
core construction.  We use a shorter route.  Aiello and Nagnibeda prove, in
different terminology, that \(\vec F_3\) is the diagram group determined by the
following finite ternary tree-automaton.  This is simply the automaton
translation of their parity criterion for leaves of ternary tree diagrams
\cite[Proposition~1]{AielloNagnibedaFThree}.

Let \(\mathcal A_{\mathrm{AN}}\) be the full ternary tree-automaton with state
set
\[
        \{\eta_{00},\eta_{01},\eta_{10},\eta_{11}\},
\]
initial state \(\eta_{00}\), and transition table
\[
\begin{array}{c|ccc}
        &0&1&2\\ \hline
\eta_{00}  &\eta_{00}&\eta_{11}&\eta_{01}\\
\eta_{01}  &\eta_{01}&\eta_{10}&\eta_{00}\\
\eta_{10}  &\eta_{11}&\eta_{01}&\eta_{10}\\
\eta_{11}  &\eta_{10}&\eta_{00}&\eta_{11}.
\end{array}
\tag{J6}
\label{eq:AN-automaton}
\]
Then
\[
        \D(\mathcal A_{\mathrm{AN}})=\vec F_3 .
\]
The automaton \(\mathcal A_{\mathrm{AN}}\) should not be confused with the core:
it records exactly the Aiello--Nagnibeda acceptance criterion, but it does not
have the existence property.  Applying the core construction, or equivalently
the techniques used for cores in \cite{GolanGeneration,GolanMaximalF}, separates
the root, boundary and inner occurrences of the same states.  The resulting core
is the automaton \(\mathcal J\) computed above.

For completeness, we include the short verification that
\[
        \D(\mathcal J)=\D(\mathcal A_{\mathrm{AN}}).
\]

\begin{lemma}
\label{lem:J-and-AN-accept-same-diagrams}
With \(\mathcal J\) as in \eqref{eq:jones-core-table}, one has
\[
        \D(\mathcal J)=\D(\mathcal A_{\mathrm{AN}}).
\]
Consequently,
\[
        \D(\mathcal J)=\vec F_3 .
\]
\end{lemma}

\begin{proof}
There is a morphism of ternary tree-automata
\[
        \theta:\mathcal J\longrightarrow \mathcal A_{\mathrm{AN}}
\]
defined by
\[
        \theta(\rho)=\eta_{00},
        \qquad
        \theta(\ell)=\eta_{00},
        \qquad
        \theta(r_0)=\eta_{00},
        \qquad
        \theta(r_1)=\eta_{01},
\]
and
\[
        \theta(\zeta_{ij})=\eta_{ij}
        \qquad(i,j\in\{0,1\}).
\]
A direct check of the two transition tables shows that \(\theta\) preserves the
root and all labelled edges.  Hence
\[
        \D(\mathcal J)\subseteq\D(\mathcal A_{\mathrm{AN}}).
\]

Conversely, let \(g\in\D(\mathcal A_{\mathrm{AN}})\), and choose a ternary tree
diagram for \(g\), with branch pairs
\[
        u_0\to v_0,\ldots,u_N\to v_N,
\]
accepted by \(\mathcal A_{\mathrm{AN}}\).  If the diagram is trivial, it is
accepted by \(\mathcal J\).  Otherwise, the first branch in each tree is a left
boundary word:
\[
        u_0=0^a,\qquad v_0=0^b
        \qquad(a,b\ge1),
\]
and both words reach \(\ell\) in \(\mathcal J\).

Similarly, the last branch in each tree is a right boundary word:
\[
        u_N=2^a,\qquad v_N=2^b
        \qquad(a,b\ge1).
\]
In \(\mathcal A_{\mathrm{AN}}\), the word \(2^k\) ends at
\(\eta_{01}\) if \(k\) is odd and at \(\eta_{00}\) if \(k\) is even.  Since
the diagram is accepted by \(\mathcal A_{\mathrm{AN}}\), the exponents \(a\)
and \(b\) have the same parity.  Hence \(u_N\) and \(v_N\) end at the same one
of the two states \(r_0,r_1\) in \(\mathcal J\).

All remaining branch pairs are inner.  For every inner word \(u\), the state
reached by \(u\) in \(\mathcal J\) is one of the four states
\[
        \zeta_{00},\zeta_{01},\zeta_{10},\zeta_{11},
\]
and the restriction of the transition table of \(\mathcal J\) to these four
states is exactly the transition table of \(\mathcal A_{\mathrm{AN}}\), after
replacing \(\zeta_{ij}\) by \(\eta_{ij}\).  Therefore, if two inner words end
at the same state of \(\mathcal A_{\mathrm{AN}}\), they end at the same state
of \(\mathcal J\).

Thus every branch pair of the chosen diagram is accepted by \(\mathcal J\), so
\(g\in\D(\mathcal J).\)
Hence
\[
        \D(\mathcal A_{\mathrm{AN}})\subseteq\D(\mathcal J).
\]
The final assertion follows from
\(\D(\mathcal A_{\mathrm{AN}})=\vec F_3\), which is the automaton formulation
of \cite[Proposition~1]{AielloNagnibedaFThree}.
\end{proof}
                  
\begin{corollary}
\label{thm:jones-F3-is-F4}
The ternary oriented Jones subgroup is equal to the conjugated copy
\[
        F_4^{\psi_J}\le F_3.
\]
In particular,
\[
        \vec F_3\cong F_4 .
\]
\end{corollary}

\begin{proof}
Since \(F_4\) is closed in itself and \(F_4^{\psi_J}\le F_3\),
Lemma~\ref{lem:closure-conjugacy-prelim} shows that \(F_4^{\psi_J}\) is closed
in \(F_3\).  By Proposition~\ref{prop:core-of-F4psi-is-J}, its core is
\(\mathcal J\), so
\(F_4^{\psi_J}=\Cl(F_4^{\psi_J})=\D(\mathcal J).\)
By Lemma~\ref{lem:J-and-AN-accept-same-diagrams},
\(\D(\mathcal J)=\vec F_3\).  Hence \(\vec F_3=F_4^{\psi_J}\cong F_4\).
\end{proof}
                                 
\subsection{Geometric determinization in the finite epsilon-separated case}
\label{subsec:geometric-determinization}

In certain cases, when the structure of the transducer is simple enough, the
subset determinization in the forward construction can be replaced, up to
folding, by a more geometric procedure.  The idea is to remove the output-free
part of the raw forward computation by cloning and contracting \(\eps\)-edges,
and then to remove the remaining nondeterminism by folding one-letter edges
with the same initial vertex and the same label.

Throughout this subsection, \(\psi:\C_n\to\C_m\) is a rational
order-preserving or order-reversing homeomorphism with finite minimal
transducer
\[
        T^\psi=(S_\psi,t_\psi,o_\psi,s_\psi),
\]
and
\(\mathcal A=(Q,\tau,q_0)\)
is a finite full \(n\)-ary tree-automaton.

For \(s\in S_\psi\), define
\[
        I_\eps(s)
        =
        \{(r,i)\in S_\psi\times X_n
          \mid t_\psi(r,i)=s,\ o_\psi(r,i)=\eps\},
\]
and
\[
        I_+(s)
        =
        \{(r,i)\in S_\psi\times X_n
          \mid t_\psi(r,i)=s,\ o_\psi(r,i)\ne\eps\}.
\]
Let
\[
        B_\psi=\{s\in S_\psi\mid I_\eps(s)=\emptyset\}.
\]
The states in \(B_\psi\) are called \emph{entry states} of \(T^\psi\).

\begin{definition}
\label{def:epsilon-separated-transducer}
Let \(\psi:\C_n\to\C_m\) be a rational order-preserving or order-reversing
homeomorphism with finite minimal transducer
\[
        T^\psi=(S_\psi,t_\psi,o_\psi,s_\psi).
\]
We say that \(T^\psi\) is \emph{epsilon-separated} if, for every
\(s\in S_\psi\),
\[
        I_\eps(s)\ne\emptyset
        \quad\Longrightarrow\quad
        I_+(s)=\emptyset.
\]
Thus a state with an incoming \(\eps\)-edge has no incoming non-\(\eps\) edge.

We say that \(T^\psi\) has \emph{homeomorphic entry states} if every state in
\[
        B_\psi=\{s\in S_\psi\mid I_\eps(s)=\emptyset\}
\]
is a homeomorphism state.
\end{definition}

\begin{lemma}
\label{lem:no-epsilon-into-homeomorphism-state}
Let \(\psi:\C_n\to\C_m\) be a rational order-preserving or order-reversing
homeomorphism, and let
\[
        T^\psi=(S_\psi,t_\psi,o_\psi,s_\psi)
\]
be its finite minimal transducer.  Then the following hold.
\begin{enumerate}[label=(\roman*)]
    \item If \(s\in S_\psi\) is a homeomorphism state, then
    \[
            I_\eps(s)=\emptyset.
    \]
    In particular, the initial state \(s_\psi\) belongs to \(B_\psi\).

    \item If \(r\in S_\psi\) and \(u,v\in\W_n\) satisfy
    \[
            o_\psi(r,u)=o_\psi(r,v)
            \qquad\text{and}\qquad
            t_\psi(r,u)=t_\psi(r,v),
    \]
    then
    \[
            u=v.
    \]
\end{enumerate}
\end{lemma}

\begin{proof}
A closely related statement to (i) appears in
\cite[Lemma~8.3(2)]{AutG}.  We recall the short argument.  Suppose that
\[
        t_\psi(r,i)=s,
        \qquad
        o_\psi(r,i)=\eps.
\]
Then, for every \(\omega\in\C_n\),
\[
        h_{T^\psi_r}(i\omega)=h_{T^\psi_s}(\omega).
\]
If \(s\) is a homeomorphism state, then \(h_{T^\psi_s}(\C_n)=\C_m\), and hence
\(h_{T^\psi_r}(U_i)=\C_m.\)
Choose \(j\in X_n\) with \(j\ne i\).  The set \(h_{T^\psi_r}(U_j)\) is
nonempty and is contained in \(\C_m=h_{T^\psi_r}(U_i)\), contradicting
injectivity of \(h_{T^\psi_r}\).  The state \(r\) is injective because it is
accessible from the initial homeomorphism state.  Thus no \(\eps\)-edge enters
\(s\).  Applying this to \(s=s_\psi\) gives \(s_\psi\in B_\psi\).

For (ii), put
\[
        t_\psi(r,u)=t_\psi(r,v)=t
        \qquad\text{and}\qquad
        o_\psi(r,u)=o_\psi(r,v)=\lambda.
\]
Then, for every \(\omega\in\C_n\),
\[
        h_{T^\psi_r}(u\omega)
        =
        \lambda\,h_{T^\psi_t}(\omega)
        =
        h_{T^\psi_r}(v\omega).
\]
Since \(h_{T^\psi_r}\) is injective, we have
\(u\omega=v\omega.\)
Thus \(u=v\).
\end{proof}

Starting from an entry state, we group together the input letters read until
the first nonempty output appears.

\begin{definition}
\label{def:first-output-blocks}
Let \(\psi:\C_n\to\C_m\) be a rational order-preserving or order-reversing
homeomorphism with finite minimal transducer \(T^\psi\), and assume that
\(T^\psi\) is epsilon-separated.  For \(s\in B_\psi\), let
\(\mathcal E(s)\) be the set of all nonempty words \(u\in\W_n\) such that
\(o_\psi(s,u)\ne\eps\), but \(o_\psi(s,v)=\eps\) for every proper prefix
\(v\) of \(u\).  For \(u\in\mathcal E(s)\), put
\[
        \lambda_s(u)=o_\psi(s,u).
\]
Thus \(\lambda_s(u)\) is the first nonempty output word produced when the
transducer starts at \(s\) and reads \(u\).

We call the ordered prefix-code bijection
\[
        \mathcal E(s)\longrightarrow \Lambda(s),
        \qquad
        u\longmapsto \lambda_s(u),
\]
the \emph{first-output cell} at \(s\).  Its individual pairs
\[
        u\mid\lambda_s(u):t_\psi(s,u)
        \qquad(u\in\mathcal E(s))
\]
are called the \emph{first-output blocks}.
\end{definition}

\begin{lemma}
\label{lem:first-output-prefix-codes}
Let \(\psi:\C_n\to\C_m\) be a rational order-preserving or order-reversing
homeomorphism with finite minimal transducer \(T^\psi\).  Assume that
\(T^\psi\) is epsilon-separated and has homeomorphic entry states.  Then, for
every \(s\in B_\psi\), the set \(\mathcal E(s)\) is a finite complete prefix
code in \(\W_n\).  Moreover, for every \(u\in\mathcal E(s)\),
\[
        t_\psi(s,u)\in B_\psi,
\]
and
\[
        \Lambda(s)=\{\lambda_s(u)\mid u\in\mathcal E(s)\}
\]
is a finite complete prefix code in \(\W_m\).
\end{lemma}

\begin{proof}
By definition, no element of \(\mathcal E(s)\) is a proper prefix of another.
Thus the cylinders
\[
        U_u,\qquad u\in\mathcal E(s),
\]
are pairwise disjoint.  It remains to see that they cover \(\C_n\) and that
there are only finitely many of them.

There is no directed cycle all of whose edges have output \(\eps\), since such
a cycle would give an infinite input word with finite output, contradicting
nondegeneracy.  Since \(T^\psi\) is finite, there is a uniform bound on how
long one can read input from \(s\) while producing only \(\eps\).  Hence every
\(\alpha\in\C_n\) has a unique prefix in \(\mathcal E(s)\), and only finitely
many such prefixes can occur.  Therefore \(\mathcal E(s)\) is a finite
complete prefix code.

Let \(u\in\mathcal E(s)\).  The last edge in the path labelled \(u\) has
nonempty output.  Thus \(t_\psi(s,u)\) has an incoming non-\(\eps\) edge.  By
epsilon-separatedness,
\(t_\psi(s,u)\in B_\psi.\)
Since \(s\) and \(t_\psi(s,u)\) are homeomorphism states,
\[
        h_{T^\psi_s}(U_u)
        =
        o_\psi(s,u)\,h_{T^\psi_{t_\psi(s,u)}}(\C_n)
        =
        U_{\lambda_s(u)}.
\]
The cylinders \(U_u\), \(u\in\mathcal E(s)\), partition \(\C_n\), and
\(h_{T^\psi_s}\) is a homeomorphism \(\C_n\to\C_m\).  Therefore the cylinders
\(U_{\lambda_s(u)}\), \(u\in\mathcal E(s)\), partition \(\C_m\).  Thus
\(\Lambda(s)\) is a finite complete prefix code.
\end{proof}

We now describe the clone--contract operation on the raw forward automaton.

\begin{definition}
\label{def:elementary-clone-contract-step}
Let \(\Gamma\) be a finite word-labelled automaton with edge labels in
\(\W_m\), and assume that \(\Gamma\) has no directed \(\eps\)-cycles.  An
\emph{elementary clone--contract step} is performed at a vertex \(v\)
satisfying:
\begin{enumerate}[label=(\roman*)]
    \item \(v\) has at least one incoming \(\eps\)-edge;
    \item every incoming edge of \(v\) is labelled \(\eps\).
\end{enumerate}
Let the incoming \(\eps\)-edges of \(v\) be
\[
        e_1:p_1\xrightarrow{\ \eps\ }v,
        \ldots,
        e_k:p_k\xrightarrow{\ \eps\ }v.
\]
Since \(\Gamma\) has no directed \(\eps\)-cycles and every incoming edge of
\(v\) is labelled \(\eps\), no outgoing edge of \(v\) has terminal vertex
\(v\).  Thus, for every outgoing edge
\[
        v\xrightarrow{\ \lambda\ }w,
        \qquad \lambda\in\W_m,
\]
the terminal vertex \(w\) survives after \(v\) is deleted.

For each \(j\in\{1,\ldots,k\}\), and for each outgoing edge
\[
        v\xrightarrow{\ \lambda\ }w,
\]
add a copied edge
\[
        p_j\xrightarrow{\ \lambda\ }w.
\]
Then delete the edges \(e_1,\ldots,e_k\), delete the vertex \(v\), and delete
all outgoing edges of \(v\).

Equivalently, one may first replace \(v\) by \(k\) copies, one for each
incoming \(\eps\)-edge, copy all outgoing edges of \(v\) to each copy, and
then contract the \(k\) private incoming \(\eps\)-edges.
\end{definition}

\begin{definition}[The clone--contract reduction of \(\Fraw_\psi(\mathcal A)\)]
\label{def:clone-contract-reduction}
Let \(\psi:\C_n\to\C_m\) be a rational order-preserving or order-reversing
homeomorphism with finite minimal transducer \(T^\psi\), and let
\[
        \mathcal A=(Q,\tau,q_0)
\]
be a finite full \(n\)-ary tree-automaton.  Assume that \(T^\psi\) is
epsilon-separated.  Starting with
\[
        \Gamma_0=\Fraw_\psi(\mathcal A),
\]
repeat elementary clone--contract steps as long as there is an
\(\eps\)-edge.  At each stage one may choose any vertex with an incoming
\(\eps\)-edge.

The resulting word-labelled automaton is denoted
\[
        \operatorname{Contr}_\eps\bigl(\Fraw_\psi(\mathcal A)\bigr).
\]
\end{definition}

\begin{lemma}
\label{lem:clone-contract-well-defined}
Let \(\psi:\C_n\to\C_m\) be a rational order-preserving or order-reversing
homeomorphism with finite minimal transducer \(T^\psi\), and let
\[
        \mathcal A=(Q,\tau,q_0)
\]
be a finite full \(n\)-ary tree-automaton.  Assume that \(T^\psi\) is
epsilon-separated.  The process of
Definition~\ref{def:clone-contract-reduction} is well defined and terminates.
Moreover, its final word-labelled automaton is independent, up to canonical
isomorphism, of the order in which the elementary clone--contract steps are
performed.
\end{lemma}

\begin{proof}
First, there are no directed \(\eps\)-cycles in
\(\Fraw_\psi(\mathcal A).\)
Indeed, a directed \(\eps\)-cycle in the raw forward automaton would project to
a directed cycle in \(T^\psi\) whose total output is \(\eps\).  Repeating this
cycle periodically would give an infinite input word whose output is finite,
contradicting nondegeneracy of the transducer.

We next note that, throughout the clone--contract process, the following two
properties are preserved:
\begin{enumerate}[label=(\roman*)]
    \item there are no directed \(\eps\)-cycles;
    \item every vertex with an incoming \(\eps\)-edge has only incoming
    \(\eps\)-edges.
\end{enumerate}
Both properties hold at the start.  The second follows from
epsilon-separatedness of \(T^\psi\): if a transducer state has an incoming
\(\eps\)-edge, then it has no incoming non-\(\eps\) edge.

Suppose the two properties hold before an elementary step at \(v\).  The step
replaces each two-edge path
\[
        p_j\xrightarrow{\ \eps\ }v\xrightarrow{\ \lambda\ }w
\]
by the single edge
\[
        p_j\xrightarrow{\ \lambda\ }w,
\]
and then deletes \(v\).  If a new directed \(\eps\)-cycle were created, it
would have to use one of the new copied \(\eps\)-edges
\[
        p_j\xrightarrow{\ \eps\ }w.
\]
Replacing that copied edge by the old two-edge path
\[
        p_j\xrightarrow{\ \eps\ }v\xrightarrow{\ \eps\ }w
\]
would give a directed \(\eps\)-cycle before the step, a contradiction.  Hence
no directed \(\eps\)-cycle is created.

The second property is also preserved.  The only new incoming edges are copied
edges.  If a copied edge has label \(\eps\), then its target already had an
incoming \(\eps\)-edge before the step, namely the edge copied from \(v\), and
therefore all its incoming edges were already \(\eps\)-edges.  If a copied edge
has nonempty label, then its target already had an incoming non-\(\eps\) edge
before the step, and hence, by the second property before the step, it had no
incoming \(\eps\)-edge.  Thus after the step, every vertex with an incoming
\(\eps\)-edge still has only incoming \(\eps\)-edges.

It follows that every vertex with an incoming \(\eps\)-edge is eligible for an
elementary clone--contract step.  Each step deletes one vertex and creates no
new vertices.  Since the raw automaton is finite, the process terminates.

It remains to identify the final graph.  Every edge appearing in any
intermediate graph represents a path in the original raw forward automaton with
the same total output label.  This is true at the start.  It is preserved by an
elementary step because the new edge
\[
        p_j\xrightarrow{\ \lambda\ }w
\]
represents the concatenation of the two old represented paths
\[
        p_j\xrightarrow{\ \eps\ }v
        \quad\text{and}\quad
        v\xrightarrow{\ \lambda\ }w.
\]

At the end there are no \(\eps\)-edges.  Let
\(p=(s,q)\)
be a remaining accessible state.  Then \(s\in B_\psi\).  Indeed, if
\(s\notin B_\psi\), then by epsilon-separatedness every incoming edge to a
state with transducer coordinate \(s\) is labelled \(\eps\); since \(p\) is
accessible and is not the initial state, it would have been eligible for a
clone--contract step.

An edge in the final contracted automaton starting at \(p=(s,q)\) is
represented by a path in the original raw forward automaton which starts at
\(p\), follows some number of \(\eps\)-edges, and then follows one edge with
nonempty output.  Let \(u\) be the input word read along this represented path.
Then every proper prefix of \(u\) produces empty output from \(s\), while \(u\)
itself produces nonempty output.  Therefore
\(u\in\mathcal E(s).\)
The represented edge has label
\(\lambda_s(u)=o_\psi(s,u)\)
and terminal state
\[
        \bigl(t_\psi(s,u),\tau(q,u)\bigr).
\]
Conversely, every first-output block \(u\in\mathcal E(s)\) determines exactly
such a path starting at \((s,q)\), and hence gives the corresponding edge in
the final graph.

There are no duplicate edges in this intrinsic description.  Indeed, suppose
that \(u,v\in\mathcal E(s)\) give the same edge out of \((s,q)\).  Then
\[
        o_\psi(s,u)=o_\psi(s,v)
        \qquad\text{and}\qquad
        t_\psi(s,u)=t_\psi(s,v).
\]
By Lemma~\ref{lem:no-epsilon-into-homeomorphism-state}(ii), we have
\(u=v.\)
Thus each edge in the intrinsic description occurs once.

Therefore the final graph is characterized intrinsically as follows: its states
are the accessible states in \(B_\psi\times Q\), and its edges are exactly
\[
        (s,q)
        \xrightarrow{\ \lambda_s(u)\ }
        \bigl(t_\psi(s,u),\tau(q,u)\bigr)
        \qquad
        (u\in\mathcal E(s)).
\]
This intrinsic description is independent of the order of the elementary
steps.  Hence the final word-labelled automaton is well defined up to canonical
isomorphism.
\end{proof}

The preceding proof motivates the following direct notation for the graph
obtained after cloning and contracting.

\begin{definition}[The contracted raw forward product]
\label{def:contracted-raw-forward-product}
Let \(\psi:\C_n\to\C_m\) be a rational order-preserving or order-reversing
homeomorphism with finite minimal transducer \(T^\psi\), and let
\[
        \mathcal A=(Q,\tau,q_0)
\]
be a finite full \(n\)-ary tree-automaton.  Assume that \(T^\psi\) is
epsilon-separated and has homeomorphic entry states.  The \emph{contracted raw
forward product}
\[
        \mathcal G^{\mathrm{raw}}_\psi(\mathcal A)
\]
is the word-labelled automaton whose accessible state set is contained in
\[
        B_\psi\times Q,
\]
with initial state
\[
        (s_\psi,q_0),
\]
and whose edges are
\[
        (s,q)
        \xrightarrow{\ \lambda_s(u)\ }
        \bigl(t_\psi(s,u),\tau(q,u)\bigr)
        \qquad
        (u\in\mathcal E(s)).
\]
By Lemma~\ref{lem:clone-contract-well-defined},
\[
        \mathcal G^{\mathrm{raw}}_\psi(\mathcal A)
        \cong
        \operatorname{Contr}_\eps\bigl(\Fraw_\psi(\mathcal A)\bigr).
\]
\end{definition}

By Lemma~\ref{lem:first-output-prefix-codes}, the outgoing labels at each
state of \(\mathcal G^{\mathrm{raw}}_\psi(\mathcal A)\) form a complete
prefix code.

\begin{definition}[The geometric forward automaton]
\label{def:geometric-forward-automaton}
Let \(\psi:\C_n\to\C_m\) be a rational order-preserving or order-reversing
homeomorphism with finite minimal transducer \(T^\psi\), and let
\[
        \mathcal A=(Q,\tau,q_0)
\]
be a finite full \(n\)-ary tree-automaton.  Assume that \(T^\psi\) is
epsilon-separated and has homeomorphic entry states.  The \emph{geometric
forward automaton}
\[
        \mathcal G_\psi(\mathcal A)
\]
is obtained from
\[
        \mathcal G^{\mathrm{raw}}_\psi(\mathcal A)
\]
as follows.  First subdivide every word-labelled edge into one-letter edges.
Then repeatedly fold pairs of one-letter edges with the same initial vertex and
the same label, identifying their terminal vertices, until no such pair
remains.  Since the outgoing labels at each principal state form a prefix code,
these foldings identify exactly the common initial subpaths of the paths
issuing from that principal state.

Equivalently, at each state \(p=(s,q)\), replace the outgoing word-labelled
edges by the prefix tree of the complete prefix code
\[
        \Lambda(s)=\{\lambda_s(u)\mid u\in\mathcal E(s)\}.
\]
Thus, if
\[
        p\xrightarrow{\ \lambda\ }p_\lambda
\]
is an edge of \(\mathcal G^{\mathrm{raw}}_\psi(\mathcal A)\), and
\[
        \lambda=b_1\cdots b_r,
        \qquad b_i\in X_m,
\]
then this edge contributes the path
\[
        p
        \xrightarrow{\ b_1\ }
        (p,b_1)
        \xrightarrow{\ b_2\ }
        (p,b_1b_2)
        \longrightarrow\cdots\longrightarrow
        (p,b_1\cdots b_{r-1})
        \xrightarrow{\ b_r\ }
        p_\lambda,
\]
with common initial subpaths identified.  In this definition, the notation
\((p,\xi)\) records the already-read prefix \(\xi\) in the prefix tree attached
at \(p\).  This is different from the subdivision notation \((e,\rho)\) used
earlier, where \(\rho\) records an unread suffix of the label of an edge.

Since each \(\Lambda(s)\) is a complete prefix code, every nonterminal vertex
of each prefix tree has one outgoing edge labelled \(b\) for every
\(b\in X_m\).  Hence \(\mathcal G_\psi(\mathcal A)\) is a deterministic full
\(m\)-ary tree-automaton.
\end{definition}

We now compare the geometric automaton with the forward automaton
\(\Ffor_\psi(\mathcal A),\)
which was defined by the subset construction on
\[
        \widehat{\Fraw}_\psi(\mathcal A).
\]

The main point is the following simple observation about forward subset states:
once such a subset contains an entry principal state, it is exactly the
\(\eps\)-closure of that state.

\begin{lemma}
\label{lem:forward-subset-containing-entry-state}
Let \(\psi:\C_n\to\C_m\) be a rational order-preserving or order-reversing
homeomorphism with finite minimal transducer \(T^\psi\), and let
\[
        \mathcal A=(Q,\tau,q_0)
\]
be a finite full \(n\)-ary tree-automaton.  Assume that \(T^\psi\) has
homeomorphic entry states.  Let \(x\in\W_m\), and suppose that \(R_x\) contains
a principal state
\[
        p=(s,q)\in B_\psi\times Q.
\]
Then
\[
        R_x=E(\{p\}).
\]
\end{lemma}

\begin{proof}
Since \(p\in R_x\), there is a path in
\[
        \widehat{\Fraw}_\psi(\mathcal A)
\]
from the initial state to \(p\) labelled \(x\).  Let \(a\in\W_n\) be the input
word read along the corresponding raw path.  Then
\[
        O^\psi(a)=x,
        \qquad
        s^\psi_a=s,
        \qquad
        \tau(q_0,a)=q.
\]
Since \(s\in B_\psi\) and the entry states are homeomorphism states,
\(h_{T^\psi_s}\) is onto \(\C_m\).  Hence
\[
        \psi(U_a)=x\,h_{T^\psi_s}(\C_n)=U_x.
\]

Let \(c\in R_x\).  Choose a path \(\gamma\) in
\(\widehat{\Fraw}_\psi(\mathcal A)\) from the initial state to \(c\), labelled
\(x\).  Let \(b\in\W_n\) be the input word corresponding to the path
\(\gamma\).  The image of the cylinder \(U_b\) under \(\psi\) is contained in
\(U_x\).  Since
\(\psi(U_a)=U_x\)
and \(\psi\) is injective, we get
\(U_b\subseteq U_a.\)
Thus \(b=ar\) for some \(r\in\W_n\).  Therefore the raw path determined by
\(b\) factors through the principal state \(p\) after reading the prefix
\(a\).

The prefix path to \(p\) has already produced the whole output word \(x\), and
the full path \(\gamma\) also has label \(x\).  Hence the part of \(\gamma\)
after it passes through \(p\) has label \(\eps\).  Therefore \(c\) is reachable
from \(p\) by an \(\eps\)-labelled path, so
\(c\in E(\{p\}).\)
This proves
\(R_x\subseteq E(\{p\}).\)
The reverse inclusion follows from the fact that \(R_x\) is \(\eps\)-closed
and contains \(p\).  Hence
\(R_x=E(\{p\}).\)
\end{proof}

We next construct the natural map from the geometric automaton to the forward
automaton.

\begin{lemma}
\label{lem:geometric-projection-to-forward}
Let \(\psi:\C_n\to\C_m\) be a rational order-preserving or order-reversing
homeomorphism with finite minimal transducer \(T^\psi\), and let
\[
        \mathcal A=(Q,\tau,q_0)
\]
be a finite full \(n\)-ary tree-automaton.  Assume that \(T^\psi\) is
epsilon-separated and has homeomorphic entry states.  For \(x\in\W_m\), let
\(G_x\) be the state reached by \(x\) in
\[
        \mathcal G_\psi(\mathcal A),
\]
and let \(R_x\) be the state reached by \(x\) in
\[
        \Ffor_\psi(\mathcal A).
\]
If
\[
        G_x=G_y,
\]
then
\[
        R_x=R_y.
\]
Consequently, the rule
\[
        \Pi(G_x)=R_x
\]
defines a well-defined surjective morphism
\[
        \Pi:\mathcal G_\psi(\mathcal A)\longrightarrow
        \Ffor_\psi(\mathcal A).
\]
Moreover, if \(p\) is a principal state of \(\mathcal G_\psi(\mathcal A)\),
then
\[
        \Pi(p)=E(\{p\}).
\]
\end{lemma}

\begin{proof}
We first prove the assertion for principal states.  Suppose
\(G_x=p=(s,q)\)
is a principal state.  By construction of \(\mathcal G_\psi(\mathcal A)\), the
path labelled \(x\) from the root to \(p\) expands to a path in the original
subdivided raw automaton
\[
        \widehat{\Fraw}_\psi(\mathcal A)
\]
with the same label \(x\), ending at the same principal state \(p\).  Hence
\(p\in R_x.\)
Since \(p\in B_\psi\times Q\), Lemma~\ref{lem:forward-subset-containing-entry-state}
gives
\(R_x=E(\{p\}).\)

Now suppose that
\(G_x=G_y.\)
If this common state is a principal state \(p\), then the principal-state case
gives
\(R_x=E(\{p\})=R_y.\)

It remains to consider the case where the common state is a non-principal
vertex in a pasted prefix tree.  Then there are a principal state \(p=(s,q)\),
a nonempty proper prefix \(\xi\) of some word in \(\Lambda(s)\), and
decompositions
\[
        x=x_0\xi,
        \qquad
        y=y_0\xi,
\]
such that
\(G_{x_0}=G_{y_0}=p.\)
Indeed, the common state is the vertex labelled by \(\xi\) in the prefix tree
rooted at \(p\), and the only way to reach it is to first reach \(p\) and then
read \(\xi\) inside that tree.  By the principal-state case,
\(R_{x_0}=E(\{p\})=R_{y_0}.\)
Since the forward automaton is deterministic,
\[
        R_x
        =
        R_{x_0\xi}
        =
        R_{y_0\xi}
        =
        R_y.
\]
Thus \(G_x=G_y\) always implies \(R_x=R_y\), so \(\Pi(G_x)=R_x\) is
well-defined.

The map \(\Pi\) is a morphism because, for every \(b\in X_m\),
\[
        \Pi(G_x\cdot b)
        =
        \Pi(G_{xb})
        =
        R_{xb}
        =
        R_x\cdot b
        =
        \Pi(G_x)\cdot b.
\]
It is surjective because every state of \(\Ffor_\psi(\mathcal A)\) is of the
form \(R_x\) for some \(x\in\W_m\).  The final assertion follows from the
principal-state case.
\end{proof}

We shall also need the following consequence, which says that a non-principal
state of the geometric automaton cannot map to a forward subset containing an
entry principal state.

\begin{lemma}
\label{lem:nonprincipal-does-not-map-to-entry-state}
Let \(\psi:\C_n\to\C_m\) be a rational order-preserving or order-reversing
homeomorphism with finite minimal transducer \(T^\psi\), and let
\[
        \mathcal A=(Q,\tau,q_0)
\]
be a finite full \(n\)-ary tree-automaton.  Assume that \(T^\psi\) is
epsilon-separated and has homeomorphic entry states.  Let \(g\) be a
non-principal state of
\[
        \mathcal G_\psi(\mathcal A).
\]
Then \(\Pi(g)\) contains no principal state in \(B_\psi\times Q\).
\end{lemma}

\begin{proof}
Choose \(x\in\W_m\) such that
\(G_x=g.\)
Since \(g\) is a non-principal state, it lies in a prefix tree rooted at some
principal state \(p=(s,q)\).  Thus there are a word \(x_0\in\W_m\) and a
nonempty proper prefix \(\xi\) of a word in \(\Lambda(s)\) such that
\[
        x=x_0\xi,
        \qquad
        G_{x_0}=p.
\]
By Lemma~\ref{lem:geometric-projection-to-forward},
\(R_{x_0}=E(\{p\}).\)

The word \(\xi\) is a proper prefix of some \(\lambda_s(u)\), with
\(u\in\mathcal E(s)\).  Following the corresponding raw path from \(p\), after
the output prefix \(\xi\) has been read one is at a subdivision state of
\[
        \widehat{\Fraw}_\psi(\mathcal A).
\]
Hence \(R_x=R_{x_0\xi}\) contains at least one subdivision state.

Suppose, toward a contradiction, that \(R_x\) also contains a principal state
\(p'\in B_\psi\times Q.\)
Then Lemma~\ref{lem:forward-subset-containing-entry-state} gives
\(R_x=E(\{p'\}).\)
But \(E(\{p'\})\) consists only of principal states, since \(\eps\)-edges in
\(\widehat{\Fraw}_\psi(\mathcal A)\) come from raw \(\eps\)-labelled edges and
therefore go between principal states.  This contradicts the fact that \(R_x\)
contains a subdivision state.  Thus \(R_x=\Pi(g)\) contains no principal state
in \(B_\psi\times Q\).
\end{proof}

The morphism \(\Pi\) need not be injective before folding.  The next lemma says
that its kernel is killed by the usual folding of tree-automata.

\begin{lemma}
\label{lem:kernel-of-geometric-projection-killed-by-folding}
Let \(\psi:\C_n\to\C_m\) be a rational order-preserving or order-reversing
homeomorphism with finite minimal transducer \(T^\psi\), and let
\[
        \mathcal A=(Q,\tau,q_0)
\]
be a finite full \(n\)-ary tree-automaton.  Assume that \(T^\psi\) is
epsilon-separated and has homeomorphic entry states.  Let \(g,h\) be states of
\[
        \mathcal G_\psi(\mathcal A).
\]
If
\[
        \Pi(g)=\Pi(h),
\]
then \(g\) and \(h\) have the same image in the folded quotient
\[
        \overline{\mathcal G_\psi(\mathcal A)}.
\]
\end{lemma}

\begin{proof}
We first consider principal states.  If \(p,p'\in B_\psi\times Q\) are
principal states and
\(\Pi(p)=\Pi(p'),\)
then, by Lemma~\ref{lem:geometric-projection-to-forward},
\(E(\{p\})=E(\{p'\}).\)
Since \(p,p'\in B_\psi\times Q\), neither \(p\) nor \(p'\) is the terminal
vertex of a nontrivial \(\eps\)-path.  Hence \(p\) belongs to \(E(\{p\})\) and
to no other entry state's \(\eps\)-closure.  Thus \(p=p'\).

Also, a principal state and a non-principal state cannot have the same image
under \(\Pi\).  Indeed, the image of a principal state contains an entry
principal state, while Lemma~\ref{lem:nonprincipal-does-not-map-to-entry-state}
shows that the image of a non-principal state contains no such state.

It remains to consider the case where both \(g\) and \(h\) are non-principal
states.  For such a state \(g\), let \(d(g)\) be the maximum distance from
\(g\) to a principal state below it in the pasted prefix tree.  This is a
positive finite integer.  We prove the claim by induction on
\(d(g)+d(h).\)

Let \(b\in X_m\).  Since \(\Pi\) is a morphism,
\[
        \Pi(g\cdot b)
        =
        \Pi(g)\cdot b
        =
        \Pi(h)\cdot b
        =
        \Pi(h\cdot b).
\]
By the preceding paragraphs, if \(g\cdot b\) and \(h\cdot b\) are principal
states, then they are equal; and it is impossible for one of them to be
principal while the other is non-principal.  If they are both non-principal,
then
\[
        d(g\cdot b)+d(h\cdot b)<d(g)+d(h),
\]
so the induction hypothesis gives that \(g\cdot b\) and \(h\cdot b\) have the
same image in
\[
        \overline{\mathcal G_\psi(\mathcal A)}.
\]
Thus, for every \(b\in X_m\), the \(b\)-children of \(g\) and \(h\) have the
same image in the folded quotient.  Therefore \(g\) and \(h\) have the same
ordered list of children after folding, and hence \(g\) and \(h\) have the same
image in
\[
        \overline{\mathcal G_\psi(\mathcal A)}.
\]
\end{proof}

Hence, we get the following corollary.

\begin{corollary}
\label{thm:geometric-determinization-up-to-folding}
Let \(\psi:\C_n\to\C_m\) be a rational order-preserving or order-reversing
homeomorphism with finite minimal transducer \(T^\psi\), and let
\[
        \mathcal A=(Q,\tau,q_0)
\]
be a finite full \(n\)-ary tree-automaton.  Assume that \(T^\psi\) is
epsilon-separated and has homeomorphic entry states.  Then the folded quotients
\[
        \overline{\mathcal G_\psi(\mathcal A)}
        \qquad\text{and}\qquad
        \overline{\Ffor_\psi(\mathcal A)}
\]
are canonically isomorphic.  Consequently,
\[
        \D(\mathcal G_\psi(\mathcal A))
        =
        \D(\Ffor_\psi(\mathcal A)).
\]
Hence, if \(H=\D(\mathcal A)\), then
\[
        \D(\mathcal G_\psi(\mathcal A))
        =
        H^\psi\cap F_m.
\]
If, in addition, \(H^\psi\le F_m\), then
\[
        \D(\mathcal G_\psi(\mathcal A))=H^\psi.
\]

Moreover, if \(\mathcal A=\C(H)\) is finite and full, has the existence
property, and \(H^\psi\le F_m\), then
\[
        \overline{\mathcal G_\psi(\mathcal A)}
\]
is the core automaton \(\C(H^\psi)\).
\end{corollary}

\begin{proof}
By Lemma~\ref{lem:geometric-projection-to-forward}, there is a surjective
morphism
\[
        \Pi:\mathcal G_\psi(\mathcal A)
        \longrightarrow
        \Ffor_\psi(\mathcal A).
\]
Composing \(\Pi\) with the quotient map
\[
        \Ffor_\psi(\mathcal A)
        \longrightarrow
        \overline{\Ffor_\psi(\mathcal A)}
\]
gives a morphism from \(\mathcal G_\psi(\mathcal A)\) to a folded automaton.
By the universal property of the folded quotient, this morphism factors through
\[
        \overline{\mathcal G_\psi(\mathcal A)}.
\]
Thus \(\Pi\) induces a surjective morphism
\[
        \overline{\Pi}:
        \overline{\mathcal G_\psi(\mathcal A)}
        \longrightarrow
        \overline{\Ffor_\psi(\mathcal A)}.
\]

We prove that \(\overline{\Pi}\) is injective.  Let \(g,h\) be states of
\[
        \mathcal G_\psi(\mathcal A)
\]
whose images under \(\Pi\) have the same image in
\[
        \overline{\Ffor_\psi(\mathcal A)}.
\]
Since the automata are finite, the finite-depth characterization of the folded
quotient gives \(k\ge0\) such that
\[
        \Pi(g\cdot w)=\Pi(h\cdot w)
        \qquad
        \text{for every }w\in X_m^k.
\]
By Lemma~\ref{lem:kernel-of-geometric-projection-killed-by-folding}, for every
\(w\in X_m^k\), the states \(g\cdot w\) and \(h\cdot w\) have the same image in
\[
        \overline{\mathcal G_\psi(\mathcal A)}.
\]
Working upward from level \(k\), it follows that \(g\) and \(h\) have the same
image in
\[
        \overline{\mathcal G_\psi(\mathcal A)}.
\]
Indeed, if the corresponding children of two states have the same images in
the folded quotient, then the two parent states have the same ordered list of
children in the quotient, and hence have the same image there.  Therefore
\(\overline{\Pi}\) is injective, and hence an isomorphism.

Folding preserves accepted diagram groups, so
\[
        \D(\mathcal G_\psi(\mathcal A))
        =
        \D(\overline{\mathcal G_\psi(\mathcal A)})
        =
        \D(\overline{\Ffor_\psi(\mathcal A)})
        =
        \D(\Ffor_\psi(\mathcal A)).
\]
Corollary~\ref{cor:forward-construction-correct} gives
\[
        \D(\Ffor_\psi(\mathcal A))
        =
        H^\psi\cap F_m.
\]
The subgroup case follows from the final assertion of
Corollary~\ref{cor:forward-construction-correct}.

Finally, assume that \(\mathcal A=\C(H)\) is finite and full, has the
existence property, and that \(H^\psi\le F_m\).  By
Corollary~\ref{cor:forward-precore-subgroup-case},
\(\Ffor_\psi(\mathcal A)\)
is a pre-core for \(H^\psi\).  Hence its folded quotient is a finite full
folded pre-core for \(H^\psi\).  By
Lemma~\ref{lem:finite-full-folded-precore-is-core}, it is the core
\(\C(H^\psi).\)
Since
\[
        \overline{\mathcal G_\psi(\mathcal A)}
        \cong
        \overline{\Ffor_\psi(\mathcal A)},
\]
the same holds for \(\overline{\mathcal G_\psi(\mathcal A)}\).
\end{proof}

\begin{remark}
\label{rem:geometric-determinization-computational}
Let \(\psi:\C_n\to\C_m\) be a rational order-preserving or order-reversing
homeomorphism with finite minimal transducer \(T^\psi\), and let
\[
        \mathcal A=(Q,\tau,q_0)
\]
be a finite full \(n\)-ary tree-automaton.  Assume that \(T^\psi\) is
epsilon-separated and has homeomorphic entry states.  By
Corollary~\ref{thm:geometric-determinization-up-to-folding}, one may compute an
automaton equal to the forward automaton up to folding without explicitly
writing down subset states.  The procedure is:

\begin{enumerate}[label=(\arabic*)]
    \item Start with the raw forward automaton
    \[
            \Fraw_\psi(\mathcal A).
    \]
    \item Repeatedly perform elementary clone--contract steps: choose any
    principal state with incoming \(\eps\)-edges, clone it once for each
    incoming \(\eps\)-edge, copy all outgoing edges to each clone, and contract
    the resulting private incoming \(\eps\)-edges.
    \item When no \(\eps\)-edges remain, subdivide the remaining word-labelled
    edges into one-letter edges.
    \item Fold pairs of one-letter edges with the same initial vertex and the
    same label until no such pair remains.
    \item Finally, if desired, pass to the folded quotient in the sense of
    Definition~\ref{def:folded-quotient}.
\end{enumerate}

There is an equivalent tree-pasting form of the same construction.  Once the
first-output blocks are known, one may skip the intermediate word-labelled
edges and paste the prefix trees directly.  The relevant accessible principal
states are the pairs
\[
        (s,q)\in B_\psi\times Q
\]
which are reached in the raw forward automaton, equivalently the pairs of the
form
\[
        (s,q)=\bigl(t_\psi(s_\psi,u),\tau(q_0,u)\bigr)
\]
for some \(u\in\W_n\), with \(s\in B_\psi\).

For each such state \((s,q)\), paste in the finite prefix tree whose branch set
is
\[
        \Lambda(s)=\{\lambda_s(u)\mid u\in\mathcal E(s)\}.
\]
The root of this prefix tree is identified with \((s,q)\), and the leaf
labelled \(\lambda_s(u)\) is identified with
\[
        \bigl(t_\psi(s,u),\tau(q,u)\bigr).
\]
Doing this at every accessible state gives the same automaton as first drawing
the word-labelled edges of
\[
        \mathcal G^{\mathrm{raw}}_\psi(\mathcal A),
\]
then subdividing them and folding common initial subpaths.
\end{remark}

\subsection{Example: The Brin--Navas subgroup}
\label{subsec:brin-navas-subgroup}

In \cite[Section~5]{BrinElementaryAmenable}, Brin defines an elementary
amenable group \(G_1\) of elementary class \(\omega+2\).  The same group was
defined independently, around the same time, by Navas
\cite[Example~6.3]{NavasQuelquesGroupes}.  We shall use the concrete copy of
this group inside Thompson's group \(F\) from
\cite[Theorem~9.1]{GolanGeneration}.  Namely, let
\[
        B=\langle x,y\rangle\le F,
        \qquad
        x=x_0x_1x_2^{-1}x_0^{-1},
        \qquad
        y=x_0x_1^{-2}.
\]
By \cite[Theorem~9.1]{GolanGeneration}, the subgroup \(B\) is closed in \(F\),
and it is maximal inside the normal subgroup
\(K=B[F,F]\triangleleft F,\)
where \(F/K\) is infinite cyclic.

The core of \(B\), computed by applying the core construction to the
generators \(x\) and \(y\) (see \cite[Section~9]{GolanGeneration}), is the
folded binary tree-automaton
\[
\mathcal B=
\begin{array}{c|cc}
        &0&1\\ \hline
\rho   &\ell&r\\
\ell   &\ell&a\\
r      &b&r\\
a      &a&c\\
b      &c&b\\
c      &a&b .
\end{array}
\]
It is drawn in Figure~\ref{fig:brin-navas-core}.  We shall realize this core
as the folded geometric forward automaton of \(F^\psi\cap F\), for a suitable
order-preserving homeomorphism
\(\psi:\C_2\to\C_2 .\)

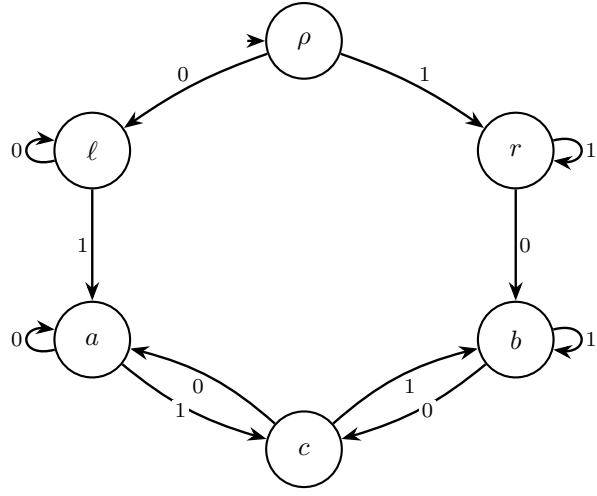
\begin{figure}[htbp]
\centering
\begin{tikzpicture}[
    >=Stealth,
    line width=.9pt,
    state/.style={
        circle,
        draw,
        thick,
        minimum size=1cm,
        font=\small\bfseries,
        inner sep=1pt,
        fill=white
    },
    lab/.style={font=\scriptsize, inner sep=1pt, fill=white}
]
    \node[state] (rho) at (0,2.7) {$\rho$};
    \node[state] (ell) at (-2.8,1.25) {$\ell$};
    \node[state] (r)   at ( 2.8,1.25) {$r$};
    \node[state] (a)   at (-2.8,-1.25) {$a$};
    \node[state] (b)   at ( 2.8,-1.25) {$b$};
    \node[state] (c)   at (0,-2.7) {$c$};

    \draw[->] (-.75,2.7)--(rho);

    \path[->]
        (rho) edge[bend right=8] node[lab,above left] {$0$} (ell)
              edge[bend left=8]  node[lab,above right] {$1$} (r)

        (ell) edge[loop left, looseness=5] node[lab] {$0$} (ell)
              edge node[lab,left] {$1$} (a)

        (r)   edge node[lab,right] {$0$} (b)
              edge[loop right, looseness=5] node[lab] {$1$} (r)

        (a)   edge[loop left, looseness=5] node[lab] {$0$} (a)
              edge[bend right=12] node[lab,left] {$1$} (c)

        (b)   edge[bend left=12] node[lab,right] {$0$} (c)
              edge[loop right, looseness=5] node[lab] {$1$} (b)

        (c)   edge[bend right=14] node[lab,below left] {$0$} (a)
              edge[bend left=14]  node[lab,below right] {$1$} (b);
\end{tikzpicture}
\caption{The core \(\mathcal B=\C(B)\) of the Brin--Navas subgroup \(B\).}
\label{fig:brin-navas-core}
\end{figure}

Define the binary minimal initial transducer \(T^\psi\) with state set
\(\{\rho,\ell,r,a,b,c,d\},\)
initial state \(\rho\), and transition-output table
\[
\begin{array}{c|cc}
        &0&1\\ \hline
\rho    &0:\ell  &1:r\\
\ell    &0:\ell  &1:a\\
r       &0:b     &1:r\\
a       &0:a     &1:c\\
b       &0:c     &1:b\\
c       &00:a    &\eps:d\\
d       &01:c    &1:b .
\end{array}
\]
Here the entry \(w:t\) in the column labelled \(i\) means that, on input
\(i\), the transducer outputs \(w\) and moves to the state \(t\).  The
transducer is drawn in Figure~\ref{fig:brin-navas-transducer}.

\begin{figure}[htbp]
\centering
\begin{tikzpicture}[
    >=Stealth,
    line width=.9pt,
    state/.style={
        circle,
        draw,
        thick,
        minimum size=1cm,
        font=\small\bfseries,
        inner sep=1pt,
        fill=white
    },
    lab/.style={font=\scriptsize, inner sep=1pt, fill=white}
]
    \node[state] (rho) at (0,3.55) {$\rho$};
    \node[state] (ell) at (-3,2) {$\ell$};
    \node[state] (r)   at ( 3,2) {$r$};
    \node[state] (a)   at (-3,-.75) {$a$};
    \node[state] (b)   at ( 3,-.75) {$b$};
    \node[state] (c)   at (0,-2.35) {$c$};
    \node[state] (d)   at (0,-4.15) {$d$};

    \draw[->] (-.75,3.55)--(rho);

    \path[->]
        (rho) edge[bend right=8] node[lab,above left] {$0\mid0$} (ell)
              edge[bend left=8]  node[lab,above right] {$1\mid1$} (r)

        (ell) edge[loop left, looseness=5] node[lab] {$0\mid0$} (ell)
              edge node[lab,left] {$1\mid1$} (a)

        (r)   edge node[lab,right] {$0\mid0$} (b)
              edge[loop right, looseness=5] node[lab] {$1\mid1$} (r)

        (a)   edge[loop left, looseness=5] node[lab] {$0\mid0$} (a)
              edge[bend right=10] node[lab,left] {$1\mid1$} (c)

        (b)   edge[bend left=10] node[lab,right] {$0\mid0$} (c)
              edge[loop right, looseness=5] node[lab] {$1\mid1$} (b)

        (c)   edge[bend right=10] node[lab,left] {$0\mid00$} (a)
              edge[bend left=10] node[lab,right] {$1\mid\eps$} (d)

        (d)   edge[bend left=10] node[lab,left] {$0\mid01$} (c)
              edge[bend right=10] node[lab,right] {$1\mid1$} (b);
\end{tikzpicture}
\caption{The transducer \(T^\psi\).  Edge labels are input-output labels.}
\label{fig:brin-navas-transducer}
\end{figure}
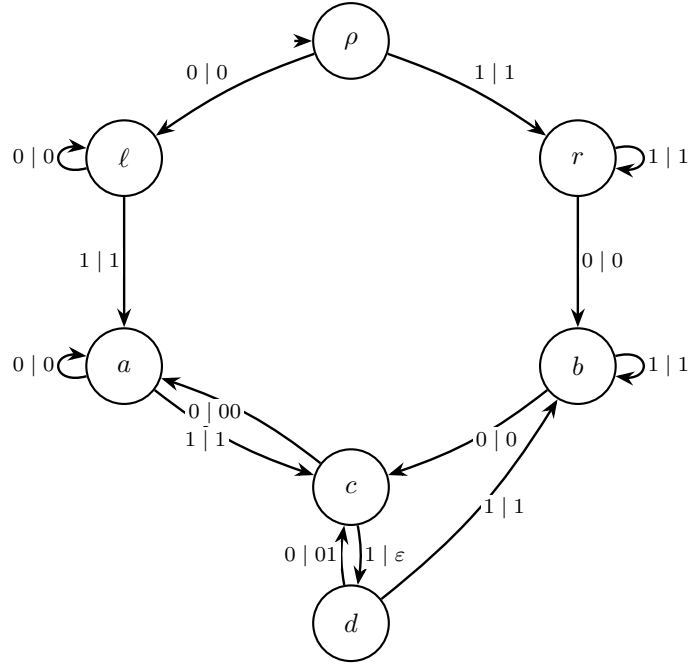

The only output-free edge is
\[
        c\xrightarrow{\,1\mid\eps\,}d,
\]
and no non-\(\eps\)-edge enters \(d\).  Thus \(T^\psi\) is
epsilon-separated, with entry states
\[
        B_\psi=\{\rho,\ell,r,a,b,c\}.
\]
The entry states are homeomorphism states: the first-output partition is
\(\{0,1\}\) at \(\rho,\ell,r,a,b\), while at \(c\) the input code
$
        \{0,10,11\}
$
is sent order-preservingly to the output code
$
        \{00,01,1\}.
$
Hence \(T^\psi\) represents an order-preserving homeomorphism
$
        \psi:\C_2\to\C_2 .
$

Before running the construction, we record a small simplification which
allows us to replace the standard core \(\C(F)\) by the one-state full
automaton in the raw forward construction.  We state it for general \(n\),
as it also applies to one of the examples in the next subsection.

Let \(\mathbf 1_n\) denote the one-state full \(n\)-ary tree-automaton.  Thus
\(\mathbf 1_n\) has one state, denoted \(*\), and
\[
        *\cdot i=*
        \qquad(i\in X_n).
\]
Although \(\mathbf 1_n\) is not the core of \(F_n\), it accepts every
\(n\)-ary tree diagram, and hence
\(\D(\mathbf 1_n)=F_n .\)

\begin{lemma}
\label{lem:delete-forced-Kn-coordinate}
Let
\[
        \psi:\C_n\to\C_m
\]
be an order-preserving or order-reversing homeomorphism, and let
\[
        T^\psi=(S_\psi,t_\psi,o_\psi,s_\psi)
\]
be its minimal \((n,m)\)-transducer.  Let
\[
        \mathcal K_n=\C(F_n)
\]
be the standard core from Lemma~\ref{lem:explicit-core-of-Fn}.  Assume that,
after forgetting outputs, \(T^\psi\) admits a morphism of \(n\)-ary
tree-automata
\[
        \kappa:T^\psi\longrightarrow \mathcal K_n .
\]
Then
\[
        \Fraw_\psi(\mathcal K_n)
        \cong
        \Fraw_\psi(\mathbf 1_n)
\]
via the map
\[
        (s,\kappa(s))\longmapsto (s,*).
\]
Consequently
\[
        \Ffor_\psi(\mathcal K_n)
        \cong
        \Ffor_\psi(\mathbf 1_n),
\]
and their folded quotients are canonically isomorphic.

Under these assumptions, if \(F_n^\psi\le F_m\), then the core of
\(F_n^\psi\) can be computed from the folded quotient of
\[
        \Ffor_\psi(\mathbf 1_n).
\]
\end{lemma}

\begin{proof}
In \(\Fraw_\psi(\mathcal K_n)\), a principal state reached by an input word
\(u\) has the form
\[
        \bigl(t_\psi(s_\psi,u),q_0\cdot u\bigr).
\]
Since \(\kappa\) is a morphism and \(\kappa(s_\psi)=q_0\), the second
coordinate is forced:
\[
        q_0\cdot u=\kappa(t_\psi(s_\psi,u)).
\]
Thus every accessible principal state is of the form \((s,\kappa(s))\), and
replacing the second coordinate by \(*\) preserves all word labels on raw
edges.  Hence the raw automata, the subdivided \(\eps\)-NFAs, and the
determinized forward automata are canonically isomorphic.

If \(F_n^\psi\le F_m\), then
\(\Ffor_\psi(\mathcal K_n)\) is a pre-core for \(F_n^\psi\) by
 Corollary~\ref{cor:forward-precore-subgroup-case}.  The preceding isomorphism
 therefore gives the same folded quotient from
 \(\Ffor_\psi(\mathbf 1_n)\).
\end{proof}

The transducer \(T^\psi\) respects the root, left-boundary, right-boundary,
and inner partition of the standard core \(A_F=\C(F)\): after forgetting
outputs, the map sending \(\rho\), \(\ell\) and \(r\) to the root, left and
right states of \(A_F\), and the four states \(a,b,c,d\) to its unique inner
state, is a morphism \(T^\psi\to A_F\).  Hence, by
Lemma~\ref{lem:delete-forced-Kn-coordinate}, the raw forward automata built
from \(A_F\) and from the one-state full binary automaton \(\mathbf 1\) are
canonically isomorphic, and therefore so are all the automata obtained from
them by the geometric procedure of
Subsection~\ref{subsec:geometric-determinization}.  Thus the folded geometric
construction with \(A_F\) is the same as the folded geometric construction
with \(\mathbf 1\).

\begin{lemma}
\label{lem:brin-navas-core-intersection}
For the order-preserving homeomorphism \(\psi:\C_2\to\C_2\) represented by
\(T^\psi\), the core of \(F^\psi\cap F\) is \(\mathcal B\).
\end{lemma}

\begin{proof}
We run the geometric construction with \(\mathbf 1\).  Thus the raw
word-labelled automaton is obtained from \(T^\psi\) by keeping only the output
labels.  Contract the unique \(\eps\)-edge
\[
        c\xrightarrow{\eps}d
\]
by copying the two outgoing edges of \(d\) to \(c\) and deleting \(d\).  The
contracted output automaton has states
\(\rho,\ell,r,a,b,c\)
and word-labelled transitions
\[
\begin{array}{c|l}
\rho &0:\ell,\quad 1:r\\[1mm]
\ell &0:\ell,\quad 1:a\\[1mm]
r    &0:b,\quad 1:r\\[1mm]
a    &0:a,\quad 1:c\\[1mm]
b    &0:c,\quad 1:b\\[1mm]
c    &00:a,\quad 01:c,\quad 1:b .
\end{array}
\]
All labels have length one except the two labels \(00\) and \(01\) leaving
\(c\).  Subdividing these two edges and folding their common initial
\(0\)-edge creates a temporary state \(z\) with children
\[
        z\cdot0=a,
        \qquad
        z\cdot1=c.
\]
The state \(a\) has the same ordered pair of children, so \(z\) folds with
\(a\).  Therefore the final transitions out of \(c\) are
\[
        c\cdot0=a,
        \qquad
        c\cdot1=b.
\]
The local computation is shown in
Figure~\ref{fig:brin-navas-contraction-folding}.

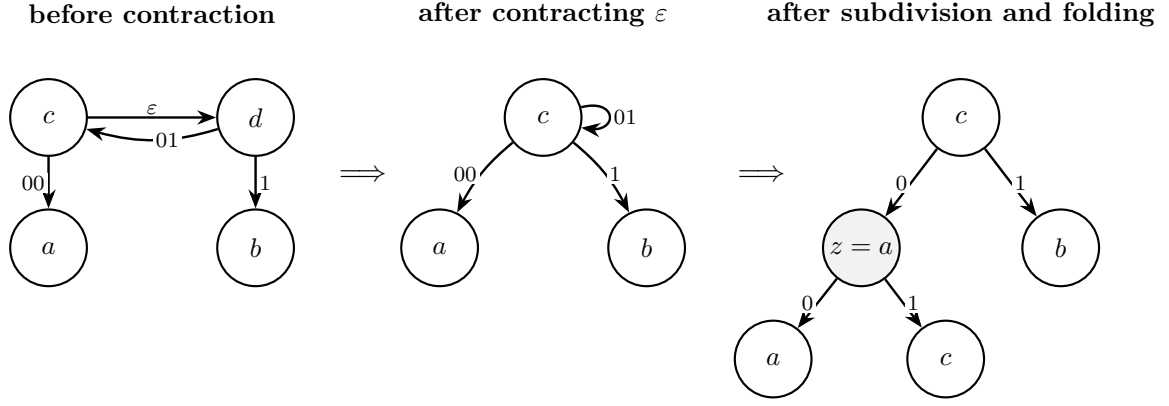
\begin{figure}[htbp]
\centering
\resizebox{0.96\textwidth}{!}{%
\begin{tikzpicture}[
    >=Stealth,
    line width=.9pt,
    state/.style={
        circle,
        draw,
        thick,
        minimum size=1cm,
        font=\small\bfseries,
        inner sep=1pt,
        fill=white
    },
    delay/.style={
        circle,
        draw,
        thick,
        minimum size=1cm,
        font=\small\bfseries,
        inner sep=1pt,
        fill=gray!10
    },
    title/.style={font=\small\bfseries},
    lab/.style={font=\scriptsize, inner sep=1pt, fill=white}
]
     \node[title] at (0,1.35) {before contraction};
    \node[state] (c1) at (-1.35,0) {$c$};
    \node[state] (d1) at ( 1.35,0) {$d$};
    \node[state] (a1) at (-1.35,-1.7) {$a$};
    \node[state] (b1) at ( 1.35,-1.7) {$b$};

    \path[->]
        (c1) edge node[lab,left] {$00$} (a1)
        (c1) edge node[lab,above] {$\eps$} (d1)
        (d1) edge[bend left=17] node[lab,right] {$01$} (c1)
        (d1) edge node[lab,right] {$1$} (b1);

    \node at (2.75,-.75) {$\Longrightarrow$};
     \node[title] at (5.1,1.35) {after contracting \(\eps\)};
    \node[state] (c2) at (5.1,0) {$c$};
    \node[state] (a2) at (3.75,-1.7) {$a$};
    \node[state] (b2) at (6.45,-1.7) {$b$};

    \path[->]
        (c2) edge[bend right=12] node[lab,left] {$00$} (a2)
        (c2) edge[loop right, looseness=5] node[lab] {$01$} (c2)
        (c2) edge[bend left=12] node[lab,right] {$1$} (b2);

    \node at (7.95,-.75) {$\Longrightarrow$};
     \node[title] at (10.55,1.35) {after subdivision and folding};
    \node[state] (c3) at (10.55,0) {$c$};
    \node[delay] (z3) at (9.25,-1.7) {$z=a$};
    \node[state] (b3) at (11.85,-1.7) {$b$};
    \node[state] (a3) at (8.1,-3.15) {$a$};
    \node[state] (c4) at (10.35,-3.15) {$c$};

    \path[->]
        (c3) edge node[lab,left] {$0$} (z3)
        (c3) edge node[lab,right] {$1$} (b3)
        (z3) edge node[lab,left] {$0$} (a3)
        (z3) edge node[lab,right] {$1$} (c4);
\end{tikzpicture}%
}
\caption{The only nontrivial geometric step.  Contracting the unique
\(\eps\)-edge gives the word-labelled fan \(00,01,1\) at \(c\).  Subdivision
folds the common initial \(0\)-edge, and the temporary state has the same
children as \(a\).}
\label{fig:brin-navas-contraction-folding}
\end{figure}

All other transitions are read directly from the contracted automaton.  Hence
the folded geometric automaton is
\[
\begin{array}{c|cc}
        &0&1\\ \hline
\rho   &\ell&r\\
\ell   &\ell&a\\
r      &b&r\\
a      &a&c\\
b      &c&b\\
c      &a&b,
\end{array}
\]
which is exactly \(\mathcal B\).  By
Corollary~\ref{thm:geometric-determinization-up-to-folding},
\[
        \D(\mathcal B)=F^\psi\cap F.
\]
Since \(\mathcal B=\C(B)\) is a core automaton, this identifies
\(\mathcal B\) with the core of \(F^\psi\cap F\).
\end{proof}

\begin{corollary}
\label{cor:brin-navas-intersection}
\label{prop:brin-navas-intersection}
For the interval homeomorphism induced by \(\psi\), one has
\[
        B=F\cap F^\psi
\]
inside \(\Homeo_+(0,1)\).  In particular, the Brin--Navas group is the
intersection of two copies of Thompson's group \(F\).
\end{corollary}
                                                                                                                                                                                                                                                                                                                                                                                                                                                                                   \tikzset{
    autstate/.style={circle,draw,thick,minimum size=7.5mm,inner sep=1pt,
        font=\small\bfseries,fill=white},
    autlabel/.style={font=\scriptsize,inner sep=1.5pt,fill=white,
        text height=1.45ex,text depth=.25ex},
    every loop/.style={looseness=7}
}

\subsection{Known maximal subgroups of \(F\) acting minimally}
\label{subsec:known-minimal-maximal-subgroups}

\tikzset{
    >=Stealth,
    autstate/.style={
        circle,
        draw,
        thick,
        minimum size=8mm,
        inner sep=1pt,
        font=\small\bfseries,
        fill=white
    },
    autlabel/.style={
        font=\scriptsize,
        inner sep=1pt,
        fill=white
    },
    autedge/.style={->,thick}
}

Apart from the point stabilizers \(\Stab_F(t)\), \(t\in(0,1)\), the known
maximal subgroups of \(F\) of infinite index fall into two families of
examples.  The first consists of the maximal subgroups isomorphic to
\(F_{p+1}\), for prime \(p\), constructed in \cite{GolanMaximalF}; the case
\(p=2\) is the maximal subgroup arising from Jones's oriented subgroup
\(\vec F\), which was constructed in \cite{GolanSapirSubgroupsF} and shown to
be isomorphic to \(F_3\) in \cite{GolanSapirJones}.  The second consists of
six explicit additional examples: three examples from
\cite[Section~10.3]{GolanGeneration}, and three examples constructed by
Aiello--Nagnibeda from Jones's \(3\)-colorable subgroup \(\mathcal F\le F\)
\cite{AielloNagnibedaThreeColorable}.

Five of these six additional examples act minimally on \((0,1)\).  The
remaining example is the subgroup \(K_2\) from
\cite[Example~10.13]{GolanGeneration}; its core shows that its action is not
minimal.  In this subsection we record the cores of the five minimally acting
examples and give transducers realizing them by the forward construction with the appropriate core of $F_n$.
In each case the final deduction is the same, and we spell it out once, for
the whole family: the displayed transducer is minimal, order-preserving and
semi-synchronizing, so Theorem~\ref{thm:conjugator-characterization} gives
\(F_n^\psi\le F\) for the represented homeomorphism \(\psi\); the conjugate
\(F_n^\psi\) is closed in \(F\) (Lemma~\ref{lem:closure-conjugacy-prelim})
with core the folded quotient of the forward automaton
(Corollary~\ref{cor:forward-precore-subgroup-case}), which the computation
identifies with the core \(\C(K)\) of the maximal subgroup \(K\le F\) in
question; and \(K\) is closed, since it is maximal of infinite index, hence
\(K=\Cl(K)=\D(\C(K))=F_n^\psi\).
Consequently, all currently known maximal subgroups of \(F\) of infinite index
which act minimally on \((0,1)\) are isomorphic to Higman--Thompson groups.

\subsubsection*{The examples from the generation problem in $F$ paper}

We start with the two relevant examples from \cite[Section~10.3]
{GolanGeneration}, since  their cores
are easier to visualize.

\paragraph{The subgroup \(K_1\).}
The subgroup from \cite[Example~10.12]{GolanGeneration} is
\[
        K_1=
        \langle x_0,\ x_1x_2x_1^{-1},\ x_1^2x_2^{-1}\rangle .
\]
The core \(\mathcal L_1=\C(K_1)\) and an initial
\((3,2)\)-transducer \(T_{K_1}\) are displayed in
Figure~\ref{fig:golan-K1-core-transducer}.  In the tables, the core is given by its transition table, and an entry \(w:\mathsf S\) in a transducer table means that the transducer outputs \(w\) and moves to the state \(\mathsf S\).  In the transducer figure, an edge
labelled \(i\mid w\) means that on input \(i\), the transducer outputs \(w\)
and moves along that edge.

The transducer \(T_{K_1}\) is minimal, order-preserving and
semi-synchronizing.  The forward construction applies to
\((T_{K_1},\C(F_3))\), and its folded quotient is \(\mathcal L_1\).  Hence
\[
        K_1=F_3^{\psi_{K_1}}\cong F_3,
\]
where \(\psi_{K_1}:\C_3\to\C_2\) is the homeomorphism represented by
\(T_{K_1}\).

\begin{figure}[htbp]
\centering
 \begin{minipage}[t]{0.46\textwidth}
\centering
\scriptsize
\[
\mathcal L_1=
\begin{array}{c|cc}
        &0&1\\ \hline
 e&f&g\\
 f&f&h\\
 g&h&g\\
 h&k&\ell\\
 k&h&k\\
 \ell&\ell&h
\end{array}
\]
\end{minipage}
\hfill
\begin{minipage}[t]{0.50\textwidth}
\centering
\scriptsize
\[
T_{K_1}=\begin{array}{c|ccc|c}
        &0&1&2&\operatorname{im} \\ \hline
\mathsf E&0:\mathsf A&100:\mathsf K&1:\mathsf G&\C_2\\
\mathsf A&0:\mathsf A&10:\mathsf K&11:\mathsf A&\C_2\\
\mathsf K&00:\mathsf K&01:\mathsf A&1:\mathsf K&\C_2\\
\mathsf G&01:\mathsf A&100:\mathsf K&1:\mathsf G&U_{01}\sqcup U_1
\end{array}
\]
\end{minipage}

\medskip
\begin{minipage}[t]{0.46\textwidth}
\centering
\begin{tikzpicture}[scale=.96,transform shape]
\node[autstate] (e) at (0,3.35) {$e$};
\node[autstate] (f) at (-1.85,2.15) {$f$};
\node[autstate] (g) at (1.85,2.15) {$g$};
\node[autstate] (h) at (0,0.95) {$h$};
\node[autstate] (k) at (-1.65,-0.75) {$k$};
\node[autstate] (l) at (1.65,-0.75) {$\ell$};

\path[autedge]
(e) edge node[autlabel,left] {$0$} (f)
    edge node[autlabel,right] {$1$} (g)
(f) edge[loop left] node[autlabel] {$0$} (f)
    edge node[autlabel,above] {$1$} (h)
(g) edge node[autlabel,above] {$0$} (h)
    edge[loop right] node[autlabel] {$1$} (g)
(h) edge[bend left=20] node[autlabel,above left,pos=.42] {$0$} (k)
    edge[bend right=20] node[autlabel,above right,pos=.42] {$1$} (l)
(k) edge[bend left=20] node[autlabel,below right,pos=.62] {$0$} (h)
    edge[loop left] node[autlabel] {$1$} (k)
(l) edge[loop right] node[autlabel] {$0$} (l)
    edge[bend right=20] node[autlabel,below left,pos=.62] {$1$} (h);
\end{tikzpicture}

\medskip
\(\mathcal L_1=\C(K_1)\)
\end{minipage}
\hfill
\begin{minipage}[t]{0.50\textwidth}
\centering
\begin{tikzpicture}[scale=.93,transform shape]
\node[autstate] (E) at (0,3.25) {$\mathsf E$};
\node[autstate] (A) at (-2.65,1.25) {$\mathsf A$};
\node[autstate] (K) at (2.65,1.25) {$\mathsf K$};
\node[autstate] (G) at (0,-1.25) {$\mathsf G$};

\draw[->,thick] (-.75,3.25) -- (E);

\path[autedge]
(E) edge[bend right=10] node[autlabel,above left] {$0\mid0$} (A)
    edge[bend left=10] node[autlabel,above right] {$2\mid1$} (G)
    edge node[autlabel,right,pos=.55] {$1\mid100$} (K)

(A) edge[loop left] node[autlabel] {$0\mid0,\ 2\mid11$} (A)
    edge[bend left=18] node[autlabel,above] {$1\mid10$} (K)

(K) edge[loop right] node[autlabel] {$0\mid00,\ 2\mid1$} (K)
    edge[bend left=18] node[autlabel,below] {$1\mid01$} (A)

(G) edge[loop below] node[autlabel] {$2\mid1$} (G)
    edge[bend right=12] node[autlabel,left,pos=.55] {$0\mid01$} (A)
    edge[bend left=12] node[autlabel,right,pos=.55] {$1\mid100$} (K);
\end{tikzpicture}

\medskip
\(T_{K_1}\)
\end{minipage}
\label{fig:golan-K1-core-transducer}
\caption{Applying the forward construction to \((T_{K_1},\C(F_3))\), and then folding,
gives the core \(\mathcal L_1\).  Hence \(K_1=F_3^{\psi_{K_1}}\cong F_3\).}
\end{figure}
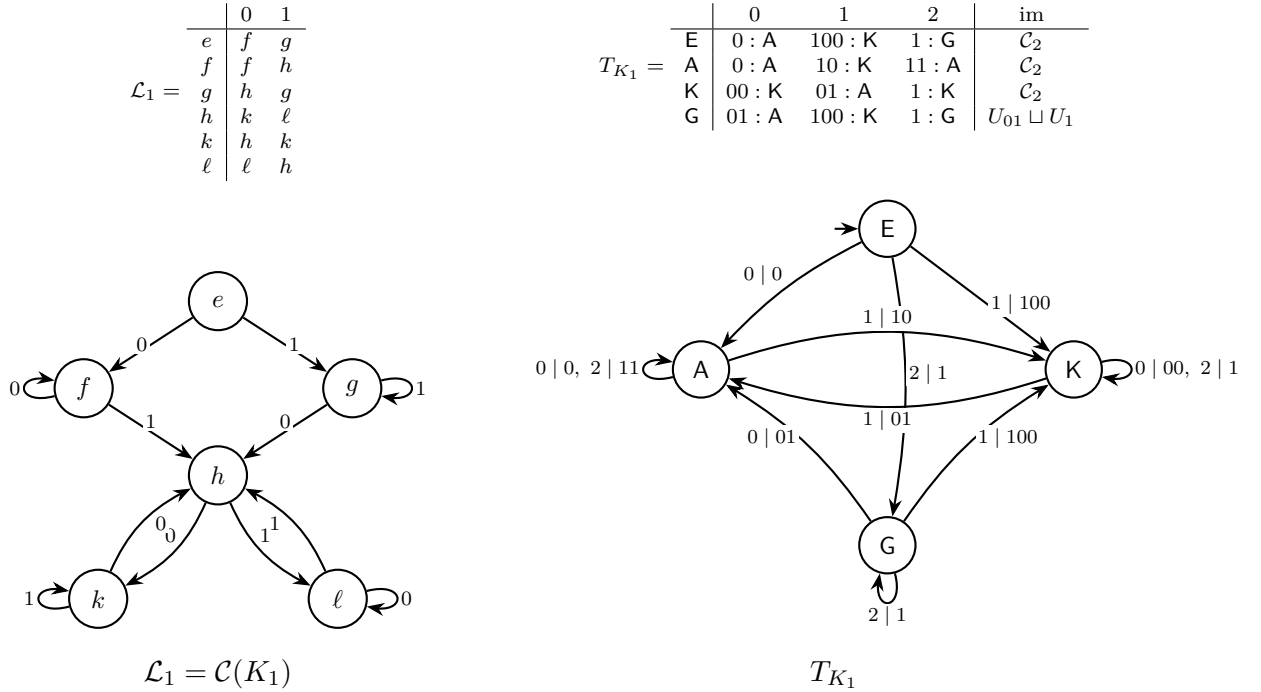

\paragraph{The transitive subgroup \(K_3\).}
The subgroup from \cite[Proposition~10.15]{GolanGeneration} is
\[
        K_3=
        \langle x_0,\ x_1x_2x_1^{-3},\
        x_1x_2x_3x_2^{-3}x_1^{-1}\rangle .
\]
It is a maximal subgroup of \(F\) of infinite index and acts transitively on
the set of finite dyadic fractions.  The core \(\mathcal L_3=\C(K_3)\) and a
binary transducer \(T_{K_3}\) are displayed in
Figure~\ref{fig:golan-K3-core-transducer}.

The transducer \(T_{K_3}\) is minimal, order-preserving and
semi-synchronizing.  In this case, by
Lemma~\ref{lem:delete-forced-Kn-coordinate}, the geometric construction can be
carried out using the one-state full binary automaton \(\mathbf 1_2\), and its folded
quotient is \(\mathcal L_3\).  Hence
\[
        K_3=F^{T_{K_3}}\cong F .
\]
That also explains why the action of $K_3$ on the dyadics is transitive. Indeed, $T_{K_3}$ is order preserving and as such maps the dyadics bijectively onto the dyadics. In fact, similarly, if $H$ is a maximal subgroup of $F$ isomorphic to $F_n$, its number of orbits on the dyadics must be $n-1$.

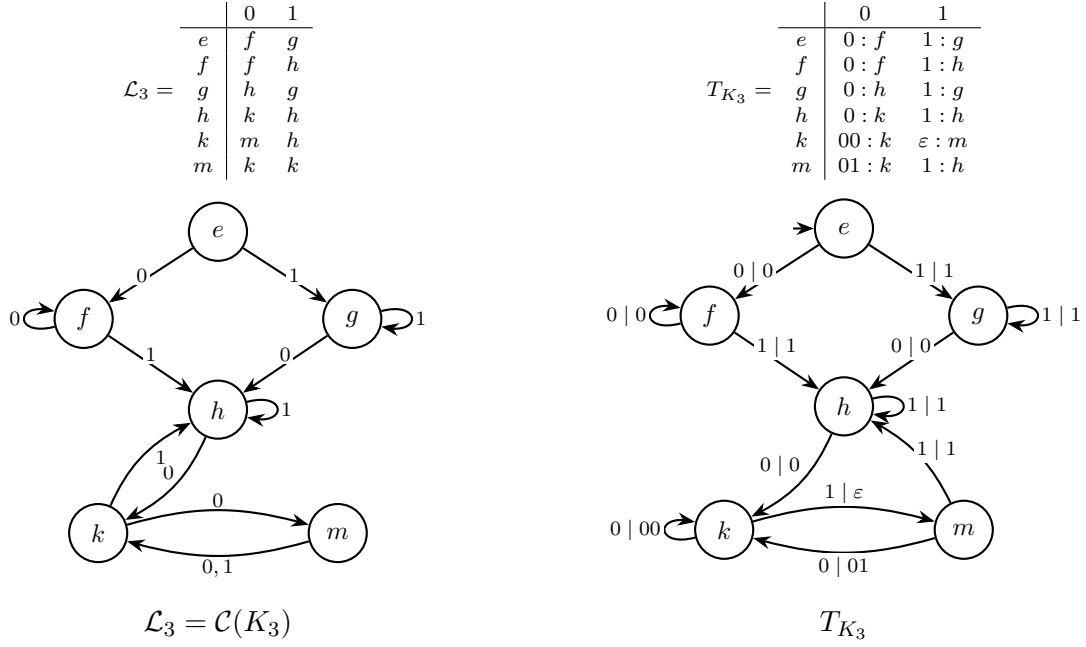
\begin{figure}[htbp]
\centering
 \begin{minipage}[t]{0.46\textwidth}
\centering
\scriptsize
\[
\mathcal L_3=
\begin{array}{c|cc}
        &0&1\\ \hline
 e&f&g\\
 f&f&h\\
 g&h&g\\
 h&k&h\\
 k&m&h\\
 m&k&k
\end{array}
\]
\end{minipage}
\hfill
\begin{minipage}[t]{0.50\textwidth}
\centering
\scriptsize
\[
T_{K_3}=\begin{array}{c|cc}
        &0&1\\ \hline
 e&0:f&1:g\\
 f&0:f&1:h\\
 g&0:h&1:g\\
 h&0:k&1:h\\
 k&00:k&\eps:m\\
 m&01:k&1:h
\end{array}
\]
\end{minipage}

\medskip
\begin{minipage}[t]{0.46\textwidth}
\centering
\begin{tikzpicture}[scale=.96,transform shape]
\node[autstate] (e) at (0,3.55) {$e$};
\node[autstate] (f) at (-1.85,2.35) {$f$};
\node[autstate] (g) at (1.85,2.35) {$g$};
\node[autstate] (h) at (0,1.10) {$h$};
\node[autstate] (k) at (-1.65,-0.60) {$k$};
\node[autstate] (m) at (1.65,-0.60) {$m$};

\path[autedge]
(e) edge node[autlabel,left] {$0$} (f)
    edge node[autlabel,right] {$1$} (g)
(f) edge[loop left] node[autlabel] {$0$} (f)
    edge node[autlabel,above] {$1$} (h)
(g) edge node[autlabel,above] {$0$} (h)
    edge[loop right] node[autlabel] {$1$} (g)
(h) edge[bend left=20] node[autlabel,above left,pos=.45] {$0$} (k)
    edge[loop right] node[autlabel] {$1$} (h)
(k) edge[bend left=20] node[autlabel,below right,pos=.62] {$1$} (h)
    edge[bend left=16] node[autlabel,above,pos=.50] {$0$} (m)
(m) edge[bend left=16] node[autlabel,below,pos=.50] {$0,1$} (k);
\end{tikzpicture}

\medskip
\(\mathcal L_3=\C(K_3)\)
\end{minipage}
\hfill
\begin{minipage}[t]{0.50\textwidth}
\centering
\begin{tikzpicture}[scale=.96,transform shape]
\node[autstate] (e) at (0,3.55) {$e$};
\node[autstate] (f) at (-1.85,2.35) {$f$};
\node[autstate] (g) at (1.85,2.35) {$g$};
\node[autstate] (h) at (0,1.10) {$h$};
\node[autstate] (k) at (-1.65,-0.60) {$k$};
\node[autstate] (m) at (1.65,-0.60) {$m$};

\draw[->,thick] (-.7,3.55) -- (e);

\path[autedge]
(e) edge node[autlabel,left] {$0\mid0$} (f)
    edge node[autlabel,right] {$1\mid1$} (g)
(f) edge[loop left] node[autlabel] {$0\mid0$} (f)
    edge node[autlabel,above] {$1\mid1$} (h)
(g) edge node[autlabel,above] {$0\mid0$} (h)
    edge[loop right] node[autlabel] {$1\mid1$} (g)
(h) edge[bend left=20] node[autlabel,above left,pos=.45] {$0\mid0$} (k)
    edge[loop right] node[autlabel] {$1\mid1$} (h)
(k) edge[loop left] node[autlabel] {$0\mid00$} (k)
    edge[bend left=16] node[autlabel,above,pos=.50] {$1\mid\varepsilon$} (m)
(m) edge[bend left=16] node[autlabel,below,pos=.50] {$0\mid01$} (k)
    edge[bend right=18] node[autlabel,right,pos=.56] {$1\mid1$} (h);
\end{tikzpicture}

\medskip
\(T_{K_3}\)
\end{minipage}
\caption{The core of \(K_3\) from \cite[Proposition~10.15]{GolanGeneration} and a binary transducer realizing it by the geometric construction with \(\mathbf 1_2\).}
\label{fig:golan-K3-core-transducer}
\end{figure}

\subsubsection*{The Aiello--Nagnibeda examples}

Aiello and Nagnibeda work inside the rectangular subgroup
\[
        K_{(2,2)}=
        \{f\in F\mid \log_2 f'(0),\log_2 f'(1)\in 2\Z\}\cong F .
\]
They define
\[
        M_0=\langle \mathcal F,x_0^2\rangle,\qquad
        M_1=\langle \mathcal F,x_1^2\rangle,\qquad
        M_2=\langle \mathcal F,\sigma(x_1)^2\rangle,
\]
where \(\sigma\) in their paper is the automorphism induced by the order-reversing linear
homeomorphism \(t\mapsto 1-t\).  They prove that these are maximal subgroups
of infinite index in \(K_{(2,2)}\).  Taking preimages under the isomorphism
\(F\cong K_{(2,2)}\) gives maximal subgroups of infinite index in \(F\)
\cite{AielloNagnibedaThreeColorable}.  We work with the isomorphic copies
\(M_i\le K_{(2,2)}\).

The subgroup \(M_2\) is the image of \(M_1\) under \(\sigma\), so \(M_1\) and
\(M_2\) are isomorphic.  Thus it is enough to display the data for \(M_0\) and
\(M_1\).

\paragraph{The subgroup \(M_0\).}
The core \(\mathcal A_0=\C(M_0)\), read from the labelled diagrams in
\cite[Figure~10]{AielloNagnibedaThreeColorable}, and a corresponding initial
\((3,2)\)-transducer \(T_{M_0}\) are displayed in
Figure~\ref{fig:AN-M0-core-transducer}.  The transducer \(T_{M_0}\) is
minimal, order-preserving and semi-synchronizing.  Applying the forward
construction to \((T_{M_0},\C(F_3))\), and then folding, gives
\(\mathcal A_0\).  Hence
\[
        M_0=F_3^{\psi_0}\cong F_3 .
\]

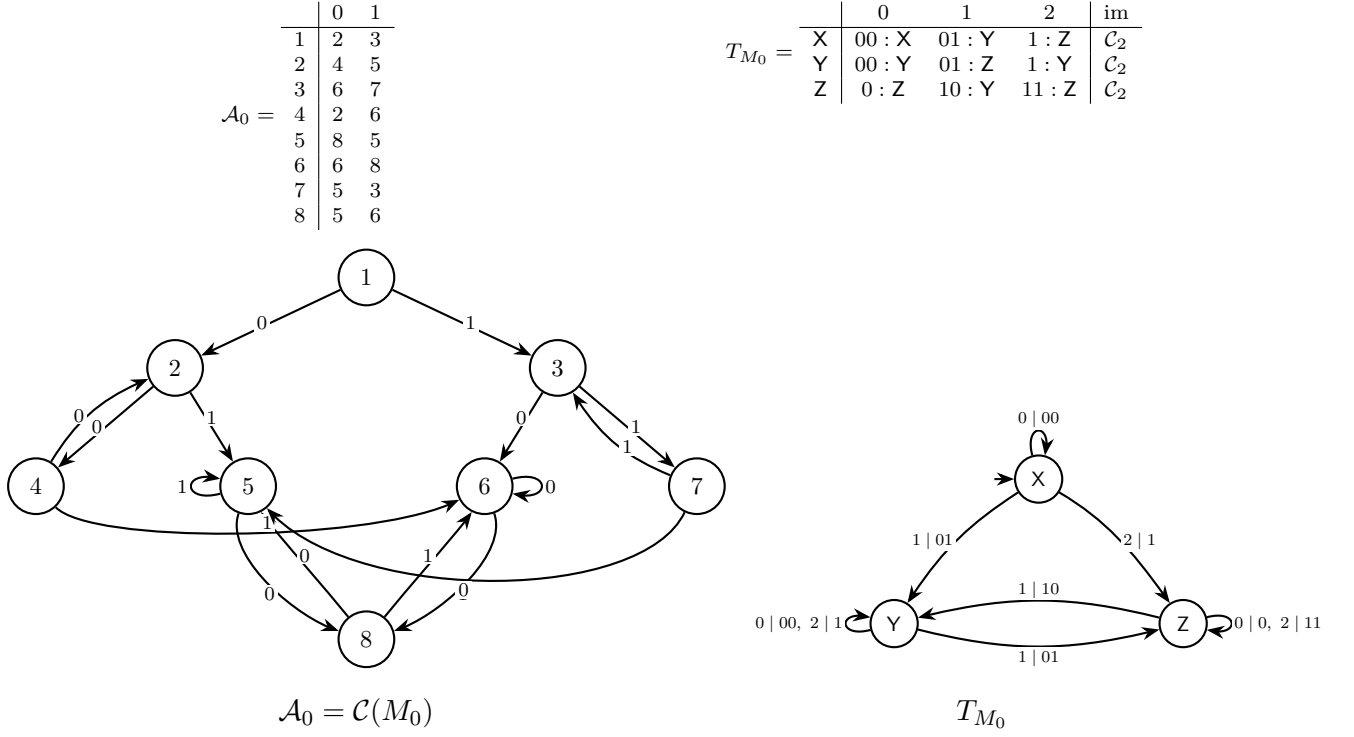
\begin{figure}[htbp]
\centering
 \begin{minipage}[t]{0.50\textwidth}
\centering
\scriptsize
\[
\mathcal A_0=
\begin{array}{c|cc}
        &0&1\\ \hline
1&2&3\\
2&4&5\\
3&6&7\\
4&2&6\\
5&8&5\\
6&6&8\\
7&5&3\\
8&5&6
\end{array}
\]
\end{minipage}
\hfill
\begin{minipage}[t]{0.46\textwidth}
\centering
\scriptsize
\[
T_{M_0}=\begin{array}{c|ccc|c}
        &0&1&2&\operatorname{im}\\ \hline
\mathsf X&00:\mathsf X&01:\mathsf Y&1:\mathsf Z&\C_2\\
\mathsf Y&00:\mathsf Y&01:\mathsf Z&1:\mathsf Y&\C_2\\
\mathsf Z&0:\mathsf Z&10:\mathsf Y&11:\mathsf Z&\C_2
\end{array}
\]
\end{minipage}

\medskip
\begin{minipage}[t]{0.58\textwidth}
\centering
\begin{tikzpicture}[scale=.92,transform shape]
\node[autstate] (s1) at (0,4.15) {$1$};
\node[autstate] (s2) at (-2.75,2.85) {$2$};
\node[autstate] (s3) at (2.75,2.85) {$3$};
\node[autstate] (s4) at (-4.75,1.15) {$4$};
\node[autstate] (s5) at (-1.70,1.15) {$5$};
\node[autstate] (s6) at (1.70,1.15) {$6$};
\node[autstate] (s7) at (4.75,1.15) {$7$};
\node[autstate] (s8) at (0,-1.05) {$8$};

\path[autedge]
(s1) edge node[autlabel,left] {$0$} (s2)
     edge node[autlabel,right] {$1$} (s3)
(s2) edge node[autlabel,left] {$0$} (s4)
     edge node[autlabel,above] {$1$} (s5)
(s3) edge node[autlabel,above] {$0$} (s6)
     edge node[autlabel,right] {$1$} (s7)
(s4) edge[bend left=18] node[autlabel,left,pos=.45] {$0$} (s2)
(s7) edge[bend left=18] node[autlabel,right,pos=.45] {$1$} (s3)
(s5) edge[loop left] node[autlabel] {$1$} (s5)
(s6) edge[loop right] node[autlabel] {$0$} (s6);

\draw[->,thick] (s5) .. controls (-2.05,.15) and (-1.00,-.75) .. node[autlabel,left,pos=.58] {$0$} (s8);
\draw[->,thick] (s8) .. controls (-.65,-.25) and (-1.45,.65) .. node[autlabel,right,pos=.50] {$0$} (s5);
\draw[->,thick] (s6) .. controls (2.05,.15) and (1.00,-.75) .. node[autlabel,right,pos=.58] {$1$} (s8);
\draw[->,thick] (s8) .. controls (.65,-.25) and (1.45,.65) .. node[autlabel,left,pos=.50] {$1$} (s6);
\draw[->,thick] (s4) .. controls (-3.95,.30) and (.20,.35) .. node[autlabel,above,pos=.55] {$1$} (s6);
\draw[->,thick] (s7) .. controls (3.95,-.55) and (-.25,-.55) .. node[autlabel,below,pos=.55] {$0$} (s5);
\end{tikzpicture}

\medskip
\(\mathcal A_0=\C(M_0)\)
\end{minipage}
\hfill
\begin{minipage}[t]{0.38\textwidth}
\centering
\begin{tikzpicture}[scale=.78,transform shape]
\node[autstate] (X) at (0,2.80) {$\mathsf X$};
\node[autstate] (Y) at (-2.45,0.35) {$\mathsf Y$};
\node[autstate] (Z) at (2.45,0.35) {$\mathsf Z$};

\draw[->,thick] (-.75,2.80) -- (X);

\path[autedge]
(X) edge[loop above] node[autlabel] {$0\mid00$} (X)
    edge[bend right=12] node[autlabel,left] {$1\mid01$} (Y)
    edge[bend left=12] node[autlabel,right] {$2\mid1$} (Z)
(Y) edge[loop left] node[autlabel] {$0\mid00,\ 2\mid1$} (Y)
    edge[bend right=14] node[autlabel,below] {$1\mid01$} (Z)
(Z) edge[loop right] node[autlabel] {$0\mid0,\ 2\mid11$} (Z)
    edge[bend right=14] node[autlabel,above] {$1\mid10$} (Y);
\end{tikzpicture}

\medskip
\(T_{M_0}\)
\end{minipage}
\caption{The core of \(M_0\) and a minimal \((3,2)\)-transducer realizing it by the forward construction with \(\C(F_3)\).}
\label{fig:AN-M0-core-transducer}
\end{figure}

\paragraph{The subgroup \(M_1\).}
The core \(\mathcal A_1=\C(M_1)\), read from
\cite[Figure~11]{AielloNagnibedaThreeColorable}, and a corresponding initial
\((3,2)\)-transducer \(T_{M_1}\) are displayed in
Figure~\ref{fig:AN-M1-core-transducer}.  The transducer \(T_{M_1}\) is
minimal, order-preserving and semi-synchronizing.  Applying the forward
construction to \((T_{M_1},\C(F_3))\), and then folding, gives
\(\mathcal A_1\).  Hence
\[
        M_1=F_3^{\psi_1}\cong F_3 .
\]
Since \(M_2=\sigma(M_1)\), we also have
\[
        M_2\cong F_3 .
\]

\begin{figure}[htbp]
\centering
 \begin{minipage}[t]{0.46\textwidth}
\centering
\scriptsize
\[
\mathcal A_1=
\begin{array}{c|cc}
        &0&1\\ \hline
1&2&3\\
2&4&5\\
3&6&7\\
4&2&5\\
5&6&8\\
6&5&6\\
7&8&3\\
8&8&5
\end{array}
\]
\end{minipage}
\hfill
\begin{minipage}[t]{0.50\textwidth}
\centering
\scriptsize
\[
T_{M_1}=\begin{array}{c|ccc|c}
        &0&1&2&\operatorname{im}\\ \hline
\mathsf E&0:\mathsf L&10:\mathsf A&11:\mathsf B&\C_2\\
\mathsf L&00:\mathsf L&010:\mathsf A&\eps:\mathsf O&\C_2\\
\mathsf A&00:\mathsf A&01:\mathsf B&1:\mathsf A&\C_2\\
\mathsf B&0:\mathsf B&10:\mathsf A&11:\mathsf B&\C_2\\
\mathsf O&011:\mathsf B&10:\mathsf A&11:\mathsf B&U_{011}\sqcup U_{10}\sqcup U_{11}
\end{array}
\]
\end{minipage}

\medskip
\begin{minipage}[t]{0.48\textwidth}
\centering
\begin{tikzpicture}[scale=.73,transform shape]
\node[autstate] (s1) at (0,4.15) {$1$};
\node[autstate] (s2) at (-2.75,2.85) {$2$};
\node[autstate] (s3) at (2.75,2.85) {$3$};
\node[autstate] (s4) at (-4.75,1.15) {$4$};
\node[autstate] (s5) at (-1.70,1.15) {$5$};
\node[autstate] (s6) at (1.70,1.15) {$6$};
\node[autstate] (s7) at (4.75,1.15) {$7$};
\node[autstate] (s8) at (0,-1.05) {$8$};

\path[autedge]
(s1) edge node[autlabel,left] {$0$} (s2)
     edge node[autlabel,right] {$1$} (s3)
(s2) edge node[autlabel,left] {$0$} (s4)
     edge node[autlabel,above] {$1$} (s5)
(s3) edge node[autlabel,above] {$0$} (s6)
     edge node[autlabel,right] {$1$} (s7)
(s4) edge[bend left=18] node[autlabel,left,pos=.45] {$0$} (s2)
     edge[bend left=8] node[autlabel,above,pos=.58] {$1$} (s5)
(s5) edge[bend left=10] node[autlabel,above,pos=.54] {$0$} (s6)
     edge[bend left=12] node[autlabel,left,pos=.55] {$1$} (s8)
(s6) edge[bend left=10] node[autlabel,below,pos=.54] {$0$} (s5)
     edge[loop right] node[autlabel] {$1$} (s6)
(s7) edge[bend left=18] node[autlabel,right,pos=.45] {$1$} (s3)
(s8) edge[loop below] node[autlabel] {$0$} (s8)
     edge[bend left=12] node[autlabel,below,pos=.52] {$1$} (s5);

\draw[->,thick] (s7) .. controls (3.55,-.15) and (1.25,-.55) .. node[autlabel,below,pos=.55] {$0$} (s8);
\end{tikzpicture}

\medskip
\(\mathcal A_1=\C(M_1)\)
\end{minipage}
\hfill
\begin{minipage}[t]{0.48\textwidth}
\centering
\begin{tikzpicture}[scale=.70,transform shape]
\node[autstate] (E) at (0,4.45) {$\mathsf E$};
\node[autstate] (L) at (-5.00,2.55) {$\mathsf L$};
\node[autstate] (B) at (5.00,2.55) {$\mathsf B$};
\node[autstate] (A) at (0,1.20) {$\mathsf A$};
\node[autstate] (O) at (-3.20,-1.55) {$\mathsf O$};

\draw[->,thick] (-.85,4.45) -- (E);

\path[autedge]
(E) edge[bend right=5] node[autlabel,above left] {$0\mid0$} (L)
    edge node[autlabel,right,pos=.55] {$1\mid10$} (A)
    edge[bend left=5] node[autlabel,above right] {$2\mid11$} (B)
(L) edge[loop left] node[autlabel] {$0\mid00$} (L)
    edge[bend left=7] node[autlabel,above,pos=.50] {$1\mid010$} (A)
    edge node[autlabel,left] {$2\mid\varepsilon$} (O)
(A) edge[loop below] node[autlabel] {$0\mid00,\ 2\mid1$} (A)
    edge[bend left=12] node[autlabel,below right,pos=.55] {$1\mid01$} (B)
(B) edge[loop right] node[autlabel] {$0\mid0,\ 2\mid11$} (B)
    edge[bend left=12] node[autlabel,above left,pos=.55] {$1\mid10$} (A)
(O) edge[bend right=8] node[autlabel,right,pos=.52] {$1\mid10$} (A);

\draw[->,thick] (O) .. controls (-.55,-2.60) and (4.10,.60) .. node[autlabel,below,pos=.58] {$0\mid011,\ 2\mid11$} (B);
\end{tikzpicture}

\medskip
\(T_{M_1}\)
\end{minipage}
\caption{The core of \(M_1\) and a minimal \((3,2)\)-transducer realizing it by the forward construction with \(\C(F_3)\).  The subgroup \(M_2\) is obtained from \(M_1\) by reflection.}
\label{fig:AN-M1-core-transducer}
\end{figure}
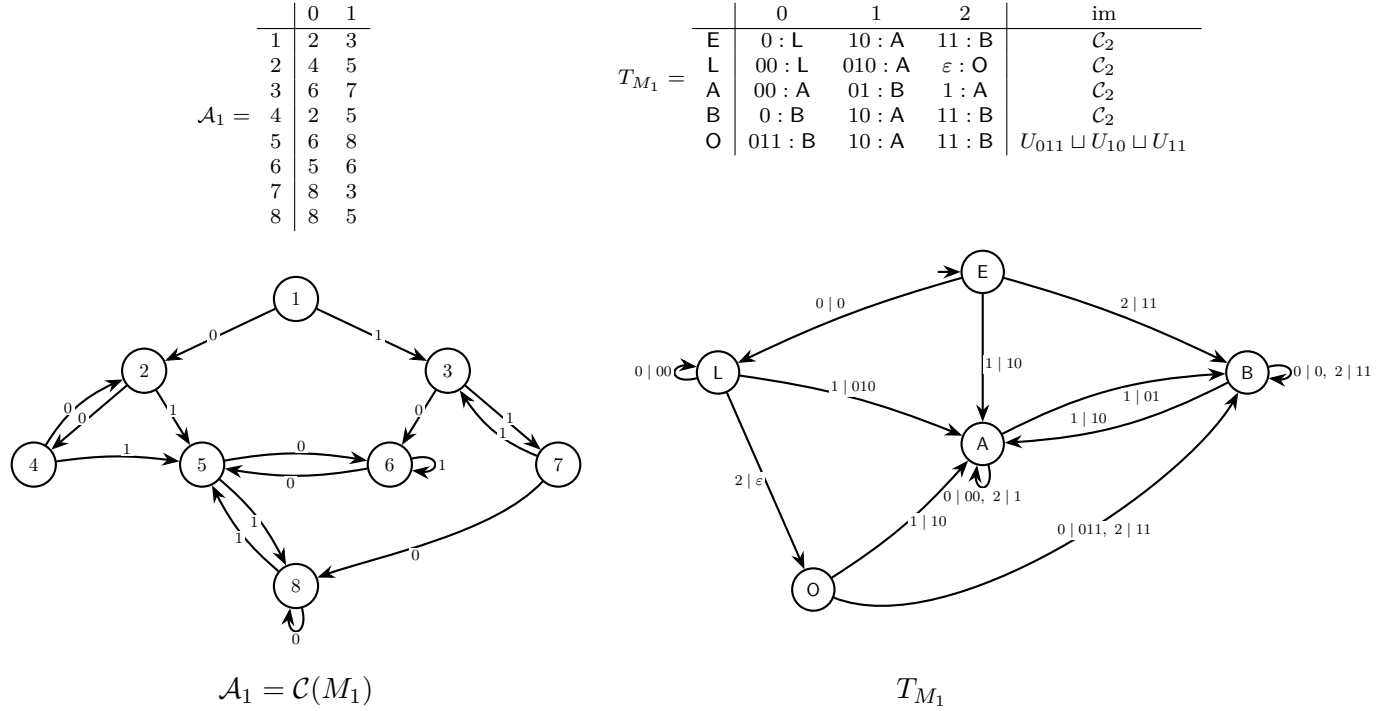

Combining the two examples from \cite{GolanGeneration}, the
Aiello--Nagnibeda examples, and the previously known Jones-oriented examples,
all currently known maximal subgroups of \(F\) of infinite index which act
minimally on \((0,1)\) are isomorphic to Higman--Thompson groups.

\begin{remark}
For higher Higman--Thompson groups, the only known maximal subgroups of infinite index which are not point-stabilizers are the maximal subgroup of $F_3$, constructed by Aiello and Nagnibeda that is isomorphic to Jones' subgroup $\vec{F}_3$ \cite{AielloNagnibedaFThree} and our recent construction with Eytan Sapir of a maximal subgroup of $F_n$ isomorphic to $F_{2n-1}$, for each $n$ \cite{GolanSapirGenerationFn}. Hence, all known examples of maximal subgroups of Higman--Thompson groups that act minimally on the interval $(0,1)$ are isomorphic to Higman--Thompson groups. Moreover, while the action of the subgroup \(K_2\) from \cite[Example~10.13]{GolanGeneration}  on \((0,1)\) is not minimal (so by
Lemma~\ref{lem:maximal-copy-minimal}  \(K_2\) cannot be
isomorphic to \(F_n\) for any \(n\)), extending the transducer methods in this paper to the case of semi-conjugacy  can be used to show that \(K_2\)
is isomorphic to a permutational wreath product of \(F\) with itself.  This
can also be proved more easily by diagram-group methods.  In fact, all the isomorphism
results in this subsection admit alternative diagram-group proofs.
We will expand on this method in future papers and we believe it will be very interesting to connect the diagram-groups methods to the transducer-methods in this paper. This section also raises the natural problem of whether all maximal subgroups of infinite index of Thompson's groups which are not point stabilizers are isomorphic to Higman--Thompson groups. 
  \end{remark}
  
\section{A descending chain of copies of \(F\) with trivial intersection}
\label{sec:binary-descending-chain}

In this section we apply the binary transducer
\(\varphi:\C_2\to\C_2\) from
Example~\ref{ex:binary-semisync-transducer}.  We use the left-to-right
composition convention fixed in the preliminaries, and put
\[
        \psi_i=\varphi^i,
        \qquad
        H_i=F^{\psi_i}
        \qquad(i\ge0).
\]
Thus \(H_0=F\).  Since the transducer \(\varphi\) is order-preserving and
semi-synchronizing, Theorem~\ref{thm:conjugator-characterization} gives
\(F^\varphi\le F.\)
It follows inductively that
\[
        H_{i+1}=(F^\varphi)^{\psi_i}\le F^{\psi_i}=H_i
        \qquad(i\ge0).
\]
Each \(H_i\) is isomorphic to \(F\), since it is conjugate to \(F\) by the
homeomorphism \(\psi_i\).

We shall also need that these subgroups are closed.

\begin{lemma}
\label{lem:Hi-closed}
For every \(i\ge0\), the subgroup \(H_i\) is closed in \(F\).
\end{lemma}

\begin{proof}
We have already observed that \(H_i=F^{\varphi^i}\le F\).  Since
\(\varphi^i\) is order-preserving, it satisfies the equivalent conditions of
Lemma~\ref{lem:lift-order-equivalence-prelim}.  Applying
Lemma~\ref{lem:closure-conjugacy-prelim} to \(F^{\varphi^i}\le F\), we get
\[
        \Cl_F(H_i)
        =
        \Cl_F(F^{\varphi^i})
        =
        \Cl_F(F)^{\varphi^i}
        =
        F^{\varphi^i}
        =
        H_i.
\]
Thus \(H_i\) is closed.
\end{proof}

We shall compute the cores of the subgroups \(H_i\).  We then use the quotient
description of closed overgroups, together with the generation theorem for
\(F\), to prove that for each \(i\), the only subgroups of \(F\) containing
\(H_i\) are
\(H_i,H_{i-1},\ldots,H_0=F.\)
Finally we prove that
\(\bigcap_{i\ge0}H_i=\{1\}.\)

\subsection{The core automata of the chain}
\label{subsec:cores-of-binary-chain}

Let \(\nu_2(N)\) denote the exponent of the largest power of \(2\) dividing
the positive integer \(N\).  For an odd integer \(t\ge1\), put
\(k(t)=2^{\nu_2(t+1)}-2 .\)
Thus, if \(0\le t<2^i\) and \(t\) is odd, then
\(1\le \nu_2(t+1)\le i,\)
and hence
\(0\le k(t)\le 2^i-2.\)

We define a sequence of finite full binary tree-automata
\[
        \mathcal A_i
        \qquad(i\ge0).
\]
The automaton \(\mathcal A_0\), the standard core of \(F\), has states
\(\rho,\ell,r,c\)
and transitions
\[
\begin{array}{c|cc}
        &0&1\\ \hline
\rho   &\ell&r\\
\ell   &\ell&c\\
r      &c&r\\
c      &c&c .
\end{array}
\]
For \(i\ge1\), the automaton \(\mathcal A_i\) has states
\[
        \rho,\ell,r,c,a_0,a_1,\ldots,a_{2^i-1},
\]
with transitions
\[
\begin{array}{c|cc}
        &0&1\\ \hline
\rho   &\ell&r\\
\ell   &\ell&c\\
r      &a_0&r\\
c      &a_0&c,
\end{array}
\]
and, for \(0\le t<2^i\),
\[
        a_t\cdot0=a_{t+1\pmod {2^i}},
\]
while
\[
        a_t\cdot1=
        \begin{cases}
        c, & t\text{ even},\\[2mm]
        a_{k(t)}, & t\text{ odd}.
        \end{cases}
\]
The state \(a_0\) is reached from the root by the word \(10\), and the
\(0\)-transitions cycle through all \(a_t\)'s.  Thus every state listed above
is reachable.  The displayed transition table also shows that every state has
two children and that no two father states have the same ordered pair of
children; hence each \(\mathcal A_i\) is full and folded.  We shall prove that
\[
        \C(H_i)\cong \mathcal A_i
        \qquad(i\ge0).
\]

\subsubsection{The first two cores: a worked geometric construction}
\label{subsubsec:first-two-cores-worked}

We begin by computing the first two cores explicitly.  The point of this
computation is not only to identify the automata, but also to show exactly how
the geometric construction of
Subsection~\ref{subsec:geometric-determinization} is being applied and to help
the reader follow the inductive proof for the general case.

The minimal transducer \(T^\varphi\) from
Example~\ref{ex:binary-semisync-transducer} satisfies the hypotheses of the
geometric construction.  It is epsilon-separated: the only output-free edge is
\[
        \mathsf A\xrightarrow{\ 1\mid\eps\ }\mathsf B,
\]
and the state \(\mathsf B\) has no incoming non-\(\eps\) edge.  Thus the
entry states are
\[
        B_\varphi=\{\mathsf L,\mathsf R,\mathsf A\}.
\]
The first-output input/output/state blocks from these entry states are:
\[
\begin{array}{c|c}
\mathsf L
    & 0\mid0:\mathsf L,\quad 1\mid1:\mathsf R\\[1mm]
\mathsf R
    & 0\mid0:\mathsf A,\quad 1\mid1:\mathsf R\\[1mm]
\mathsf A
    & 0\mid00:\mathsf A,\quad 10\mid01:\mathsf A,\quad
      11\mid1:\mathsf R .
\end{array}
\]
For each entry state, the output words in the corresponding row form a
complete binary prefix code.  Hence the entry states are homeomorphism states,
and Corollary~\ref{thm:geometric-determinization-up-to-folding} applies.  For
\(\mathsf L\) and \(\mathsf R\), the output code is the one-letter code
\(\{0,1\}\).  The state \(\mathsf A\) is the only entry state whose output
code is not made of one-letter words; its output code is
\(\{00,01,1\}.\)

We shall use the following concrete form of the contracted raw product.  If
\(\mathcal B=(Q,\tau,q_0)\) is the input automaton and
\[
        (s,q)\in B_\varphi\times Q
\]
is an accessible principal state, then every first-output block
\[
        u\mid\lambda:t_\varphi(s,u)
\]
gives a word-labelled edge
\[
        (s,q)
        \xrightarrow{\ \lambda\ }
        \bigl(t_\varphi(s,u),\,q\cdot u\bigr)
\]
in
\[
        \mathcal G^{\mathrm{raw}}_\varphi(\mathcal B).
\]
The geometric automaton \(\mathcal G_\varphi(\mathcal B)\) is then obtained by
subdividing these word-labelled edges into one-letter edges and folding common
initial subpaths.

At the entry state \(\mathsf A\), this local operation is shown in
Figure~\ref{fig:a-fan-subdivision}.  The left side is the word-labelled fan in
the contracted raw product.  The right side is what remains after subdividing
the labels \(00\) and \(01\) and folding their common first \(0\)-edge.  The
other entry states also contribute their one-letter \(0\)- and \(1\)-edges;
the point of the figure is that \(\mathsf A\) is the only place where an
intermediate state is created.

\begin{figure}[htbp]
\centering
\resizebox{0.86\textwidth}{!}{%
\begin{tikzpicture}[
    >=Stealth,
    line width=1.05pt,
    node distance=1cm,
    pstate/.style={
        circle, draw=blue!55!black, fill=blue!5,
        thick, minimum size=1.25cm, font=\large\bfseries, align=center
    },
    dstate/.style={
        circle, draw=orange!85!black, fill=orange!8,
        thick, minimum size=1.25cm, font=\large\bfseries, align=center
    },
    tstate/.style={
        circle, draw=green!45!black, fill=green!6,
        thick, minimum size=1.35cm, font=\large\bfseries, align=center
    },
    title/.style={font=\Large\bfseries},
    lab/.style={font=\large\bfseries, inner sep=1pt},
    smalllab/.style={font=\normalsize\bfseries, align=center, inner sep=1pt}
]
    \node[title] at (-4.2,1.1) {word-labelled fan};
    \node[pstate] (Lroot) at (-4.2,0) {$(\mathsf A,q)$};
    \node[tstate] (L00) at (-7,-2.2) {$(\mathsf A,q\!\cdot\!0)$};
    \node[tstate] (L01) at (-4.2,-2.75) {$(\mathsf A,q\!\cdot\!10)$};
    \node[tstate] (L1)  at (-1.35,-2.2) {$(\mathsf R,q\!\cdot\!11)$};

    \draw[->, blue!70!black]
        (Lroot) to[bend right=13]
        node[lab, above left] {$00$} (L00);
    \draw[->, orange!90!black]
        (Lroot) to
        node[lab, right] {$01$} (L01);
    \draw[->, violet!80!black]
        (Lroot) to[bend left=13]
        node[lab, above right] {$1$} (L1);

    \node[smalllab] at (-7,-3.25) {$0\mid00$};
    \node[smalllab] at (-4.2,-3.85) {$10\mid01$};
    \node[smalllab] at (-1.35,-3.25) {$11\mid1$};

    \draw[dashed, thick] (-0.05,1.0) -- (-0.05,-3.6);

    \node[title] at (4.2,1.1) {after subdivision and folding};
    \node[pstate] (Rroot) at (4.2,0) {$(\mathsf A,q)$};
    \node[dstate] (D) at (2.35,-1.55) {$D_q$\\[-1mm]\scriptsize$((\mathsf A,q),0)$};
    \node[tstate] (R00) at (0.55,-3.25) {$(\mathsf A,q\!\cdot\!0)$};
    \node[tstate] (R01) at (4.2,-3.25) {$(\mathsf A,q\!\cdot\!10)$};
    \node[tstate] (R1)  at (7.35,-1.55) {$(\mathsf R,q\!\cdot\!11)$};

    \draw[->, blue!70!black]
        (Rroot) to[bend right=8]
        node[lab, above left] {$0$} (D);
    \draw[->, violet!80!black]
        (Rroot) to[bend left=8]
        node[lab, above] {$1$} (R1);
    \draw[->, blue!70!black]
        (D) to[bend right=8]
        node[lab, left] {$0$} (R00);
    \draw[->, orange!90!black]
        (D) to[bend left=8]
        node[lab, right] {$1$} (R01);
\end{tikzpicture}%
}
\caption{The local operation at a principal state whose transducer coordinate
is \(\mathsf A\).  The two word-labelled edges with labels \(00\) and \(01\)
have the same first output letter, so after subdivision their initial
\(0\)-edges are folded into one delay state.}
\label{fig:a-fan-subdivision}
\end{figure}
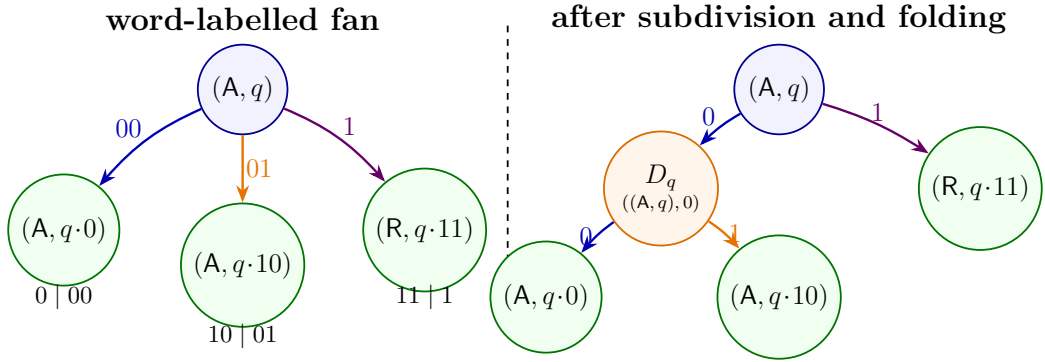

\paragraph{Constructing \(\C(H_1)\).}

We begin with
\(\mathcal A_0=\C(F),\)
whose states are
\(\rho,\ell,r,c\)
and whose transition table is
\[
\begin{array}{c|cc}
        &0&1\\ \hline
\rho   &\ell&r\\
\ell   &\ell&c\\
r      &c&r\\
c      &c&c .
\end{array}
\]
We form
\[
        \mathcal G^{\mathrm{raw}}_\varphi(\mathcal A_0).
\]

The accessible principal states are exactly
\[
        \widehat\rho=(\mathsf L,\rho),\qquad
        \widehat\ell=(\mathsf L,\ell),\qquad
        \widehat r=(\mathsf R,r),\qquad
        \widehat c=(\mathsf R,c),\qquad
        P=(\mathsf A,c).
\]
Indeed, starting at \((\mathsf L,\rho)\), the two first-output blocks from
\(\mathsf L\) give
\[
        (\mathsf L,\rho)\xrightarrow{\ 0\ }(\mathsf L,\ell),
        \qquad
        (\mathsf L,\rho)\xrightarrow{\ 1\ }(\mathsf R,r).
\]
From \((\mathsf L,\ell)\), the first-output blocks from \(\mathsf L\) give
\[
        (\mathsf L,\ell)\xrightarrow{\ 0\ }(\mathsf L,\ell),
        \qquad
        (\mathsf L,\ell)\xrightarrow{\ 1\ }(\mathsf R,c).
\]
From \((\mathsf R,r)\), the first-output blocks from \(\mathsf R\) give
\[
        (\mathsf R,r)\xrightarrow{\ 0\ }(\mathsf A,r\cdot0)
        =
        (\mathsf A,c),
        \qquad
        (\mathsf R,r)\xrightarrow{\ 1\ }(\mathsf R,r).
\]
From \((\mathsf R,c)\), they give
\[
        (\mathsf R,c)\xrightarrow{\ 0\ }(\mathsf A,c\cdot0)
        =
        (\mathsf A,c),
        \qquad
        (\mathsf R,c)\xrightarrow{\ 1\ }(\mathsf R,c).
\]
Thus \(P=(\mathsf A,c)\) is accessible.  Finally, from
\(P=(\mathsf A,c)\), the three first-output blocks from \(\mathsf A\) give
terminal principal states
\[
        0\mid00:\mathsf A
        \quad\Longrightarrow\quad
        (\mathsf A,c\cdot0)=(\mathsf A,c)=P,
\]
\[
        10\mid01:\mathsf A
        \quad\Longrightarrow\quad
        (\mathsf A,c\cdot10)=(\mathsf A,c)=P,
\]
and
\[
        11\mid1:\mathsf R
        \quad\Longrightarrow\quad
        (\mathsf R,c\cdot11)=(\mathsf R,c)=\widehat c.
\]
No other principal state is therefore reachable.

Consequently, the contracted raw product has the word-labelled transition
table
\[
\begin{array}{c|c}
\widehat\rho
    & 0:\widehat\ell,\quad 1:\widehat r\\[1mm]
\widehat\ell
    & 0:\widehat\ell,\quad 1:\widehat c\\[1mm]
\widehat r
    & 0:P,\quad 1:\widehat r\\[1mm]
\widehat c
    & 0:P,\quad 1:\widehat c\\[1mm]
P
    & 00:P,\quad 01:P,\quad 1:\widehat c .
\end{array}
\]
This word-labelled product is drawn in
Figure~\ref{fig:A0-to-A1-raw-revised}.  The labels in this figure are output
words.

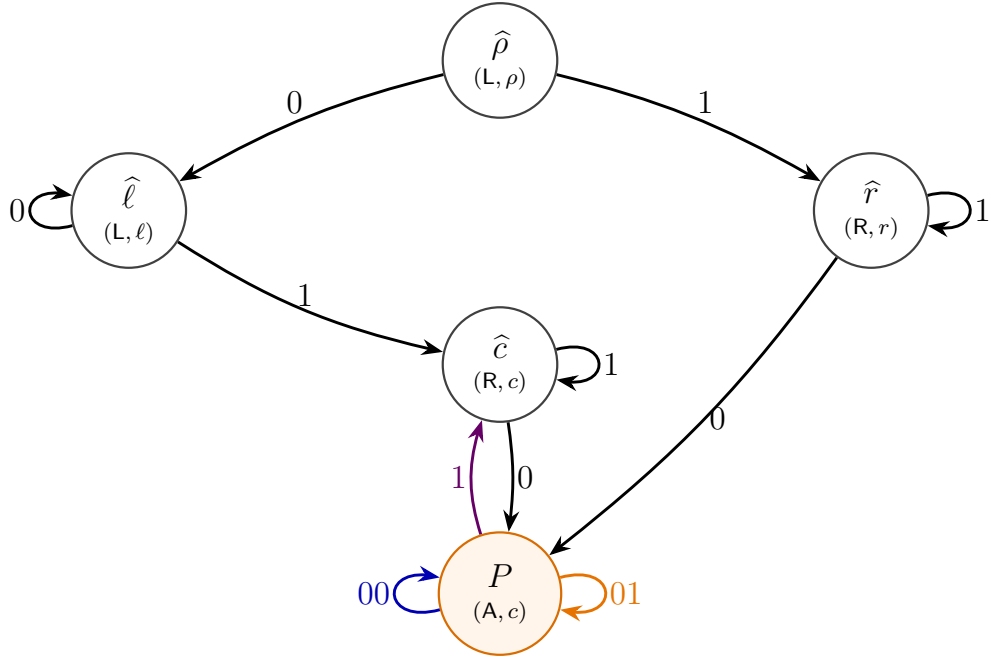
\begin{figure}[htbp]
\centering
\resizebox{0.82\textwidth}{!}{%
\begin{tikzpicture}[
    >=Stealth,
    line width=1.05pt,
    pstate/.style={
        circle, draw=black!75, fill=white,
        thick, minimum size=1.45cm, font=\large\bfseries, align=center
    },
    astate/.style={
        circle, draw=orange!85!black, fill=orange!8,
        thick, minimum size=1.55cm, font=\large\bfseries, align=center
    },
    lab/.style={font=\large\bfseries, inner sep=1pt}
]
    \node[pstate] (rho) at (0,4.5) {$\widehat\rho$\\[-1mm]{\scriptsize $(\mathsf L,\rho)$}};
    \node[pstate] (ell) at (-4.7,2.6) {$\widehat\ell$\\[-1mm]{\scriptsize $(\mathsf L,\ell)$}};
    \node[pstate] (r)   at (4.7,2.6) {$\widehat r$\\[-1mm]{\scriptsize $(\mathsf R,r)$}};
    \node[pstate] (c)   at (0,0.65) {$\widehat c$\\[-1mm]{\scriptsize $(\mathsf R,c)$}};
    \node[astate] (P)   at (0,-2.25) {$P$\\[-1mm]{\scriptsize $(\mathsf A,c)$}};

    \draw[->] (rho) to[bend right=8]
        node[lab, above left] {$0$} (ell);
    \draw[->] (rho) to[bend left=8]
        node[lab, above right] {$1$} (r);

    \draw[->] (ell) edge[loop left, looseness=5]
        node[lab] {$0$} (ell);
    \draw[->] (ell) to[bend right=10]
        node[lab, above] {$1$} (c);

    \draw[->] (r) edge[loop right, looseness=5]
        node[lab] {$1$} (r);
    \draw[->] (r) to[bend left=8]
        node[lab, right] {$0$} (P);

    \draw[->] (c) edge[loop right, looseness=5]
        node[lab] {$1$} (c);
    \draw[->] (c) to[bend left=8]
        node[lab, right] {$0$} (P);

    \draw[->, blue!70!black] (P) edge[loop left, looseness=5]
        node[lab] {$00$} (P);
    \draw[->, orange!90!black] (P) edge[loop right, looseness=5]
        node[lab] {$01$} (P);
    \draw[->, violet!80!black] (P) to[bend left=18]
        node[lab, left] {$1$} (c);
\end{tikzpicture}%
}
\caption{The contracted raw product
\(\mathcal G^{\mathrm{raw}}_\varphi(\mathcal A_0)\).  Edge labels are output
words.}
\label{fig:A0-to-A1-raw-revised}
\end{figure}

We now pass from the contracted raw product to the geometric automaton by
subdividing word-labelled edges and folding common initial subpaths.  At
\[
        \widehat\rho,\widehat\ell,\widehat r,\widehat c
\]
all output labels have length \(1\), so those edges remain unchanged.  At
\(P=(\mathsf A,c),\)
the two word-labelled edges
\[
        P\xrightarrow{\ 00\ }P,
        \qquad
        P\xrightarrow{\ 01\ }P
\]
are first subdivided into two paths of length \(2\).  Their first edges both
start at \(P\) and both have label \(0\), so these first edges are folded.  We
obtain one delay state
\[
        D=(P,0)=((\mathsf A,c),0),
\]
and the three word-labelled edges from \(P\) become
\[
        P\xrightarrow{\ 0\ }D,
        \qquad
        P\xrightarrow{\ 1\ }\widehat c,
\]
and
\[
        D\xrightarrow{\ 0\ }P,
        \qquad
        D\xrightarrow{\ 1\ }P.
\]
Thus the geometric automaton has states
\[
        \widehat\rho,\widehat\ell,\widehat r,\widehat c,P,D.
\]
Its ordered pairs of children are
\[
\begin{array}{c|c}
\widehat\rho & (\widehat\ell,\widehat r)\\
\widehat\ell & (\widehat\ell,\widehat c)\\
\widehat r & (P,\widehat r)\\
\widehat c & (P,\widehat c)\\
P & (D,\widehat c)\\
D & (P,P).
\end{array}
\]
These ordered pairs are all distinct, so no further folding occurs.

Finally we rename
\[
        a_0=P=(\mathsf A,c),
        \qquad
        a_1=D=((\mathsf A,c),0),
\]
and remove the hats from the four boundary states.  The resulting automaton is
\[
\begin{array}{c|cc}
        &0&1\\ \hline
\rho   &\ell&r\\
\ell   &\ell&c\\
r      &a_0&r\\
c      &a_0&c\\
a_0    &a_1&c\\
a_1    &a_0&a_0 .
\end{array}
\]
This is precisely \(\mathcal A_1\).  It is drawn in
Figure~\ref{fig:A1-final-revised}.

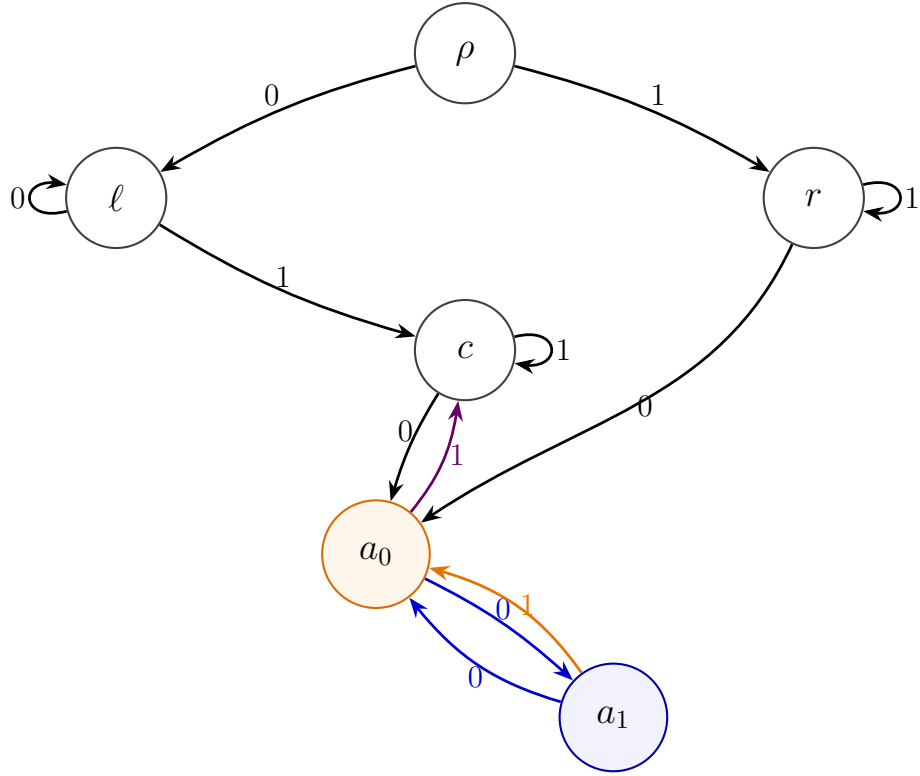
\begin{figure}[htbp]
\centering
\resizebox{0.76\textwidth}{!}{%
\begin{tikzpicture}[
    >=Stealth,
    line width=1.05pt,
    bstate/.style={
        circle, draw=black!75, fill=white,
        thick, minimum size=1.35cm, font=\Large\bfseries, align=center
    },
    astate/.style={
        circle, draw=orange!85!black, fill=orange!8,
        thick, minimum size=1.45cm, font=\Large\bfseries, align=center
    },
    dstate/.style={
        circle, draw=blue!65!black, fill=blue!5,
        thick, minimum size=1.45cm, font=\Large\bfseries, align=center
    },
    lab/.style={font=\large\bfseries, inner sep=1pt}
]
    \node[bstate] (rho) at (0,5.0) {$\rho$};
    \node[bstate] (ell) at (-4.7,3.05) {$\ell$};
    \node[bstate] (r)   at (4.7,3.05) {$r$};
    \node[bstate] (c)   at (0,1.0) {$c$};
    \node[astate] (a0) at (-1.2,-1.75) {$a_0$};
    \node[dstate] (a1) at (2.0,-3.95) {$a_1$};

    \draw[->] (rho) to[bend right=8]
        node[lab, above left] {$0$} (ell);
    \draw[->] (rho) to[bend left=8]
        node[lab, above right] {$1$} (r);

    \draw[->] (ell) edge[loop left, looseness=5]
        node[lab] {$0$} (ell);
    \draw[->] (ell) to[bend right=8]
        node[lab, above] {$1$} (c);

    \draw[->] (r) edge[loop right, looseness=5]
        node[lab] {$1$} (r);
    \draw[->] (r) to[out=245,in=35]
        node[lab, right] {$0$} (a0);

    \draw[->] (c) edge[loop right, looseness=5]
        node[lab] {$1$} (c);
    \draw[->] (c) to[bend right=8]
        node[lab, left, pos=.35] {$0$} (a0);

    \draw[->, blue!80!black] (a0) to[bend left=8]
        node[lab, above] {$0$} (a1);
    \draw[->, violet!80!black] (a0) to[bend right=16]
        node[lab, right, pos=.55] {$1$} (c);

    \draw[->, blue!80!black] (a1) to[bend left=18]
        node[lab, below] {$0$} (a0);
    \draw[->, orange!90!black] (a1) to[bend right=20]
        node[lab, right] {$1$} (a0);
\end{tikzpicture}%
}
\caption{The geometric automaton obtained from
\(\mathcal G^{\mathrm{raw}}_\varphi(\mathcal A_0)\) after subdivision and
folding common initial subpaths.  This is \(\mathcal A_1\).}
\label{fig:A1-final-revised}
\end{figure}

Since \(\mathcal A_0=\C(F)\), since \(F^\varphi=H_1\le F\), and since the
geometric automaton above is already folded, Corollary
\ref{thm:geometric-determinization-up-to-folding} gives
\[
        \C(H_1)=\C(F^\varphi)\cong\mathcal A_1.
\]

\paragraph{Constructing \(\C(H_2)\).}

We now repeat the same construction with input automaton
\(\C(H_1)\cong\mathcal A_1.\)
To distinguish the old \(a\)-states of the input automaton from the new states
we are about to construct, temporarily write them as
\[
        a^{(1)}_0,\qquad a^{(1)}_1.
\]
For reference, we display side by side the first-output blocks of the
transducer and the transition table of the input automaton \(\mathcal A_1\):
\[
\begin{array}{c|c}
\text{entry state of }T^\varphi & \text{first-output blocks}\\ \hline
\mathsf L
    & 0\mid0:\mathsf L,\quad 1\mid1:\mathsf R\\[1mm]
\mathsf R
    & 0\mid0:\mathsf A,\quad 1\mid1:\mathsf R\\[1mm]
\mathsf A
    & 0\mid00:\mathsf A,\quad 10\mid01:\mathsf A,\quad
      11\mid1:\mathsf R
\end{array}
\qquad
\begin{array}{c|cc}
\mathcal A_1 &0&1\\ \hline
\rho   &\ell&r\\
\ell   &\ell&c\\
r      &a^{(1)}_0&r\\
c      &a^{(1)}_0&c\\
a^{(1)}_0&a^{(1)}_1&c\\
a^{(1)}_1&a^{(1)}_0&a^{(1)}_0 .
\end{array}
\]
We form
\[
        \mathcal G^{\mathrm{raw}}_\varphi(\mathcal A_1).
\]
The accessible principal states are exactly
\[
        \widehat\rho=(\mathsf L,\rho),\qquad
        \widehat\ell=(\mathsf L,\ell),\qquad
        \widehat r=(\mathsf R,r),\qquad
        \widehat c=(\mathsf R,c),
\]
together with
\[
        P_0=(\mathsf A,a^{(1)}_0),
        \qquad
        P_1=(\mathsf A,a^{(1)}_1).
\]

The first four states are reached as follows:
\[
        \widehat\rho\xrightarrow{\ 0\ }\widehat\ell,
        \qquad
        \widehat\rho\xrightarrow{\ 1\ }\widehat r,
\]
\[
        \widehat\ell\xrightarrow{\ 0\ }\widehat\ell,
        \qquad
        \widehat\ell\xrightarrow{\ 1\ }\widehat c.
\]
Using the row for \(\mathsf R\) in the transducer table, the remaining
one-letter first-output blocks give
\[
        \widehat r\xrightarrow{\ 0\ }P_0,
        \qquad
        \widehat r\xrightarrow{\ 1\ }\widehat r,
\]
and
\[
        \widehat c\xrightarrow{\ 0\ }P_0,
        \qquad
        \widehat c\xrightarrow{\ 1\ }\widehat c.
\]
Thus \(P_0\) is accessible.

Now start at \(P_0=(\mathsf A,a^{(1)}_0)\).  The three first-output blocks
from the transducer state \(\mathsf A\), together with the table of
\(\mathcal A_1\), give
\[
        0\mid00:\mathsf A
        \quad\Longrightarrow\quad
        (\mathsf A,a^{(1)}_0\cdot0)
        =
        (\mathsf A,a^{(1)}_1)
        =
        P_1,
\]
\[
        10\mid01:\mathsf A
        \quad\Longrightarrow\quad
        (\mathsf A,a^{(1)}_0\cdot10)
        =
        (\mathsf A,a^{(1)}_0)
        =
        P_0,
\]
and
\[
        11\mid1:\mathsf R
        \quad\Longrightarrow\quad
        (\mathsf R,a^{(1)}_0\cdot11)
        =
        (\mathsf R,c)
        =
        \widehat c.
\]
Thus \(P_1\) is accessible.

Next start at \(P_1=(\mathsf A,a^{(1)}_1)\).  Again using the three
first-output blocks from \(\mathsf A\), we get
\[
        0\mid00:\mathsf A
        \quad\Longrightarrow\quad
        (\mathsf A,a^{(1)}_1\cdot0)
        =
        (\mathsf A,a^{(1)}_0)
        =
        P_0,
\]
\[
        10\mid01:\mathsf A
        \quad\Longrightarrow\quad
        (\mathsf A,a^{(1)}_1\cdot10)
        =
        (\mathsf A,a^{(1)}_1)
        =
        P_1,
\]
and
\[
        11\mid1:\mathsf R
        \quad\Longrightarrow\quad
        (\mathsf R,a^{(1)}_1\cdot11)
        =
        (\mathsf R,c)
        =
        \widehat c.
\]
Hence no further principal states appear.

Therefore the contracted raw product has word-labelled transition table
\[
\begin{array}{c|c}
\widehat\rho
    & 0:\widehat\ell,\quad 1:\widehat r\\[1mm]
\widehat\ell
    & 0:\widehat\ell,\quad 1:\widehat c\\[1mm]
\widehat r
    & 0:P_0,\quad 1:\widehat r\\[1mm]
\widehat c
    & 0:P_0,\quad 1:\widehat c\\[1mm]
P_0
    & 00:P_1,\quad 01:P_0,\quad 1:\widehat c\\[1mm]
P_1
    & 00:P_0,\quad 01:P_1,\quad 1:\widehat c .
\end{array}
\]
This word-labelled product is drawn in
Figure~\ref{fig:A1-to-A2-raw-revised}.  Edge labels are output words.

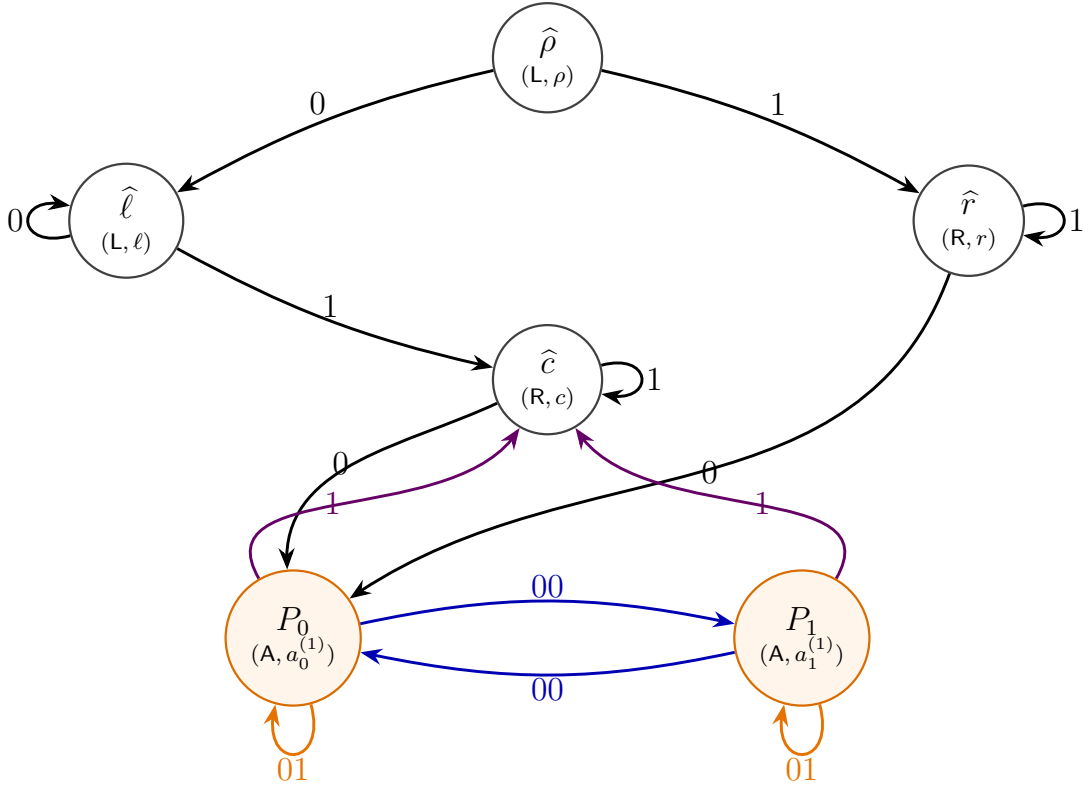
\begin{figure}[htbp]
\centering
\resizebox{0.90\textwidth}{!}{%
\begin{tikzpicture}[
    >=Stealth,
    line width=1.05pt,
    pstate/.style={
        circle, draw=black!75, fill=white,
        thick, minimum size=1.35cm, font=\large\bfseries, align=center
    },
    astate/.style={
        circle, draw=orange!85!black, fill=orange!8,
        thick, minimum size=1.7cm, font=\large\bfseries, align=center
    },
    lab/.style={font=\large\bfseries, inner sep=1pt}
]
    \node[pstate] (rho) at (0,5.15) {$\widehat\rho$\\[-1mm]{\scriptsize $(\mathsf L,\rho)$}};
    \node[pstate] (ell) at (-5.3,3.1) {$\widehat\ell$\\[-1mm]{\scriptsize $(\mathsf L,\ell)$}};
    \node[pstate] (r)   at (5.3,3.1) {$\widehat r$\\[-1mm]{\scriptsize $(\mathsf R,r)$}};
    \node[pstate] (c)   at (0,1.1) {$\widehat c$\\[-1mm]{\scriptsize $(\mathsf R,c)$}};

    \node[astate] (P0) at (-3.2,-2.15) {$P_0$\\[-1mm]{\scriptsize $(\mathsf A,a^{(1)}_0)$}};
    \node[astate] (P1) at (3.2,-2.15) {$P_1$\\[-1mm]{\scriptsize $(\mathsf A,a^{(1)}_1)$}};

    \draw[->] (rho) to[bend right=8]
        node[lab, above left] {$0$} (ell);
    \draw[->] (rho) to[bend left=8]
        node[lab, above right] {$1$} (r);

    \draw[->] (ell) edge[loop left, looseness=5]
        node[lab] {$0$} (ell);
    \draw[->] (ell) to[bend right=8]
        node[lab, above] {$1$} (c);

    \draw[->] (r) edge[loop right, looseness=5]
        node[lab] {$1$} (r);
    \draw[->] (r) to[out=250,in=35]
        node[lab, right] {$0$} (P0);

    \draw[->] (c) edge[loop right, looseness=5]
        node[lab] {$1$} (c);
    \draw[->] (c) to[out=205,in=95]
        node[lab, left] {$0$} (P0);

    \draw[->, blue!70!black] (P0) to[bend left=12]
        node[lab, above] {$00$} (P1);
    \draw[->, blue!70!black] (P1) to[bend left=12]
        node[lab, below] {$00$} (P0);

    \draw[->, orange!90!black] (P0) edge[loop below, looseness=5]
        node[lab] {$01$} (P0);
    \draw[->, orange!90!black] (P1) edge[loop below, looseness=5]
        node[lab] {$01$} (P1);

    \draw[->, violet!80!black] (P0) to[out=120,in=240]
        node[lab, left] {$1$} (c);
    \draw[->, violet!80!black] (P1) to[out=60,in=300]
        node[lab, right] {$1$} (c);
\end{tikzpicture}%
}
\caption{The contracted raw product
\(\mathcal G^{\mathrm{raw}}_\varphi(\mathcal A_1)\).  Edge labels are output
words.  The loops labelled \(01\) will share their initial \(0\)-edges with
the edges labelled \(00\) after subdivision.}
\label{fig:A1-to-A2-raw-revised}
\end{figure}

We now subdivide and fold the word-labelled edges.  At \(P_0\), the two
word-labelled edges
\[
        P_0\xrightarrow{\ 00\ }P_1,
        \qquad
        P_0\xrightarrow{\ 01\ }P_0
\]
share the first output letter \(0\).  After subdivision, their first
\(0\)-edges are folded into a single edge
\[
        P_0\xrightarrow{\ 0\ }D_0,
        \qquad
        D_0=(P_0,0).
\]
The two second letters give
\[
        D_0\xrightarrow{\ 0\ }P_1,
        \qquad
        D_0\xrightarrow{\ 1\ }P_0.
\]
The remaining edge from \(P_0\) is already one-letter:
\[
        P_0\xrightarrow{\ 1\ }\widehat c.
\]

Similarly, at \(P_1\), the two word-labelled edges
\[
        P_1\xrightarrow{\ 00\ }P_0,
        \qquad
        P_1\xrightarrow{\ 01\ }P_1
\]
produce one delay state
\(D_1=(P_1,0),\)
with
\[
        P_1\xrightarrow{\ 0\ }D_1,
        \qquad
        P_1\xrightarrow{\ 1\ }\widehat c,
\]
and
\[
        D_1\xrightarrow{\ 0\ }P_0,
        \qquad
        D_1\xrightarrow{\ 1\ }P_1.
\]
This local subdivision is shown in Figure~\ref{fig:A2-subdivision-local}.

\begin{figure}[htbp]
\centering
\resizebox{0.82\textwidth}{!}{%
\begin{tikzpicture}[
    >=Stealth,
    line width=1.05pt,
    astate/.style={
        circle, draw=orange!85!black, fill=orange!8,
        thick, minimum size=1.55cm, font=\Large\bfseries, align=center
    },
    delay/.style={
        circle, draw=blue!70!black, fill=blue!5,
        thick, minimum size=1.55cm, font=\Large\bfseries, align=center
    },
    cstate/.style={
        circle, draw=black!80, fill=white,
        thick, minimum size=1.35cm, font=\Large\bfseries, align=center
    },
    lab/.style={font=\large\bfseries, inner sep=1pt}
]
    \node[cstate] (c) at (0,2.4) {$\widehat c$};

    \node[astate] (P0) at (-4.4,0) {$P_0$};
    \node[delay]  (D0) at (-4.4,-2.45) {$D_0$};

    \node[astate] (P1) at (4.4,0) {$P_1$};
    \node[delay]  (D1) at (4.4,-2.45) {$D_1$};

    \draw[->, blue!70!black] (P0) to[bend right=8]
        node[lab, right] {$0$} (D0);
    \draw[->, violet!80!black] (P0) to[out=75,in=220]
        node[lab, above left] {$1$} (c);
    \draw[->, blue!70!black] (D0) to[out=330,in=240]
        node[lab, below] {$0$} (P1);
    \draw[->, orange!90!black] (D0) to[bend right=24]
        node[lab, left] {$1$} (P0);

    \draw[->, blue!70!black] (P1) to[bend left=8]
        node[lab, left] {$0$} (D1);
    \draw[->, violet!80!black] (P1) to[out=105,in=320]
        node[lab, above right] {$1$} (c);
    \draw[->, blue!70!black] (D1) to[out=210,in=300]
        node[lab, below] {$0$} (P0);
    \draw[->, orange!90!black] (D1) to[bend left=24]
        node[lab, right] {$1$} (P1);
\end{tikzpicture}%
}
\caption{The local subdivision step in the construction of \(\mathcal A_2\).
The pair of edges \(00,01\) out of \(P_0\) folds through \(D_0\), and the pair
of edges \(00,01\) out of \(P_1\) folds through \(D_1\).}
\label{fig:A2-subdivision-local}
\end{figure}
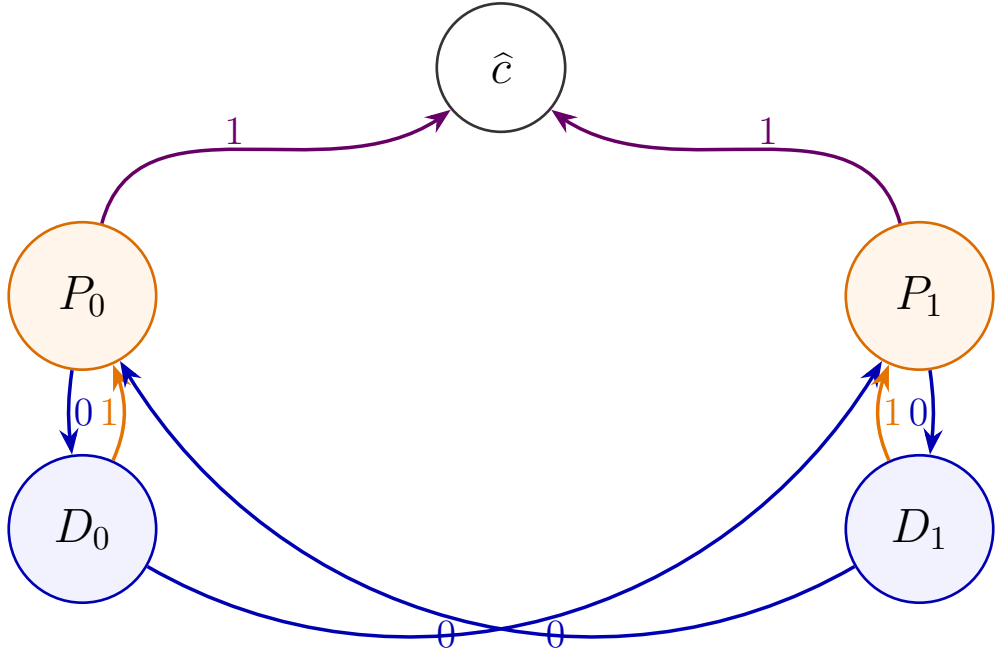

Thus the geometric automaton has states
\[
        \widehat\rho,\widehat\ell,\widehat r,\widehat c,
        P_0,D_0,P_1,D_1.
\]
The ordered pairs of children are
\[
\begin{array}{c|c}
\widehat\rho & (\widehat\ell,\widehat r)\\
\widehat\ell & (\widehat\ell,\widehat c)\\
\widehat r & (P_0,\widehat r)\\
\widehat c & (P_0,\widehat c)\\
P_0 & (D_0,\widehat c)\\
D_0 & (P_1,P_0)\\
P_1 & (D_1,\widehat c)\\
D_1 & (P_0,P_1).
\end{array}
\]
These ordered pairs are all distinct, so the geometric automaton is already
folded.

Finally rename
\[
        a_0=P_0,\qquad
        a_1=D_0,\qquad
        a_2=P_1,\qquad
        a_3=D_1,
\]
and remove the hats from the four boundary states.  We obtain
\[
\begin{array}{c|cc}
        &0&1\\ \hline
\rho   &\ell&r\\
\ell   &\ell&c\\
r      &a_0&r\\
c      &a_0&c\\
a_0    &a_1&c\\
a_1    &a_2&a_0\\
a_2    &a_3&c\\
a_3    &a_0&a_2 .
\end{array}
\]
This automaton is precisely \(\mathcal A_2\), drawn in
Figure~\ref{fig:A2-final-revised}.  The larger labels
\(P_0,D_0,P_1,D_1\) show the construction; the smaller labels inside the same
states show the final names \(a_0,a_1,a_2,a_3\).

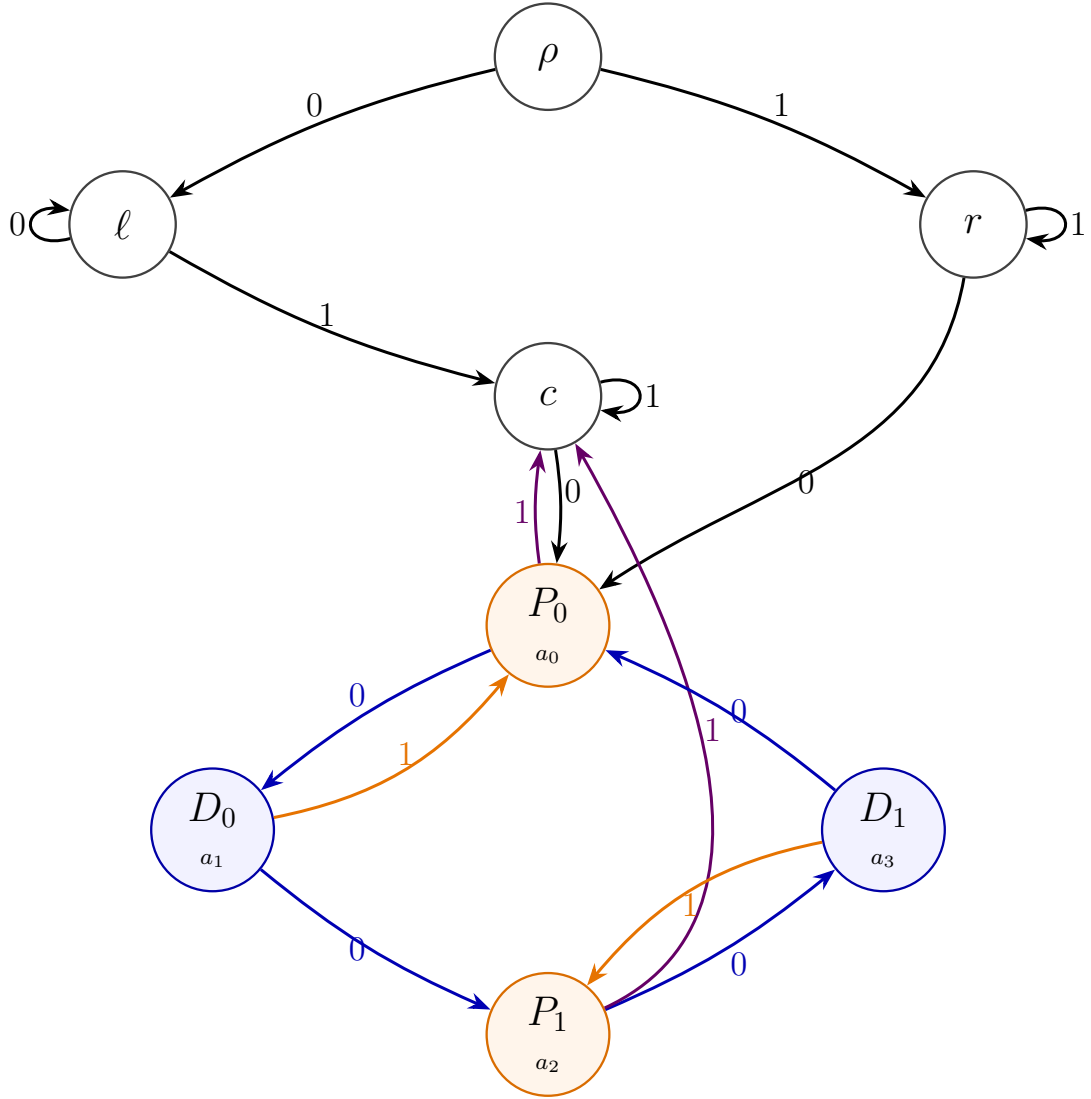
\begin{figure}[htbp]
\centering
\resizebox{0.90\textwidth}{!}{%
\begin{tikzpicture}[
    >=Stealth,
    line width=1.05pt,
    bstate/.style={
        circle, draw=black!75, fill=white,
        thick, minimum size=1.3cm, font=\Large\bfseries, align=center
    },
    pstate/.style={
        circle, draw=orange!85!black, fill=orange!8,
        thick, minimum size=1.5cm, font=\Large\bfseries, align=center
    },
    dstate/.style={
        circle, draw=blue!65!black, fill=blue!5,
        thick, minimum size=1.5cm, font=\Large\bfseries, align=center
    },
    lab/.style={font=\large\bfseries, inner sep=1pt}
]
    \node[bstate] (rho) at (0,6.2) {$\rho$};
    \node[bstate] (ell) at (-5.2,4.15) {$\ell$};
    \node[bstate] (r)   at (5.2,4.15) {$r$};
    \node[bstate] (c)   at (0,2.05) {$c$};

    \node[pstate] (P0) at (0,-0.75) {$P_0$\\[-1mm]{\scriptsize $a_0$}};
    \node[dstate] (D0) at (-4.1,-3.25) {$D_0$\\[-1mm]{\scriptsize $a_1$}};
    \node[pstate] (P1) at (0,-5.75) {$P_1$\\[-1mm]{\scriptsize $a_2$}};
    \node[dstate] (D1) at (4.1,-3.25) {$D_1$\\[-1mm]{\scriptsize $a_3$}};

    \draw[->] (rho) to[bend right=8]
        node[lab, above left] {$0$} (ell);
    \draw[->] (rho) to[bend left=8]
        node[lab, above right] {$1$} (r);

    \draw[->] (ell) edge[loop left, looseness=5]
        node[lab] {$0$} (ell);
    \draw[->] (ell) to[bend right=8]
        node[lab, above] {$1$} (c);

    \draw[->] (r) edge[loop right, looseness=5]
        node[lab] {$1$} (r);
    \draw[->] (r) to[out=260,in=35]
        node[lab, right, pos=.58] {$0$} (P0);

    \draw[->] (c) edge[loop right, looseness=5]
        node[lab] {$1$} (c);
    \draw[->] (c) to[bend left=8]
        node[lab, right, pos=.35] {$0$} (P0);

    \draw[->, blue!70!black] (P0) to[bend right=8]
        node[lab, above left] {$0$} (D0);
    \draw[->, violet!80!black] (P0) to[bend left=8]
        node[lab, left, pos=.45] {$1$} (c);

    \draw[->, blue!70!black] (D0) to[bend right=8]
        node[lab, left] {$0$} (P1);
    \draw[->, orange!90!black] (D0) to[bend right=20]
        node[lab, above] {$1$} (P0);

    \draw[->, blue!70!black] (P1) to[bend right=8]
        node[lab, below right] {$0$} (D1);
    \draw[->, violet!80!black] (P1) to[out=25,in=300,looseness=1.0]
        node[lab, right, pos=.55] {$1$} (c);

    \draw[->, blue!70!black] (D1) to[bend right=8]
        node[lab, right] {$0$} (P0);
    \draw[->, orange!90!black] (D1) to[bend right=20]
        node[lab, below] {$1$} (P1);
\end{tikzpicture}%
}
\caption{The geometric automaton obtained from
\(\mathcal G^{\mathrm{raw}}_\varphi(\mathcal A_1)\) after subdividing the
word-labelled edges and folding common initial subpaths.  The small labels
record the final renaming
\(P_0=a_0\), \(D_0=a_1\), \(P_1=a_2\), \(D_1=a_3\).}
\label{fig:A2-final-revised}
\end{figure}

Since \(\mathcal A_1\cong\C(H_1)\), since \(H_1^\varphi=H_2\le F\), and
since the geometric automaton above is already folded, Corollary
\ref{thm:geometric-determinization-up-to-folding} gives
\(\C(H_2)\cong\mathcal A_2.\)

\subsubsection{The general core computation}
\label{subsubsec:general-core-computation}

The computations of
\[
        \C(H_1)\cong\mathcal A_1
        \qquad\text{and}\qquad
        \C(H_2)\cong\mathcal A_2
\]
are the first two instances of an inductive construction.  We now prove the
general statement in the same language used above.

\begin{proposition}
\label{prop:cores-of-Hi}
For every \(i\ge0\),
\[
        \C(H_i)\cong\mathcal A_i.
\]
\end{proposition}

\begin{proof}
The case \(i=0\) is the standard core of \(F\), and the worked construction
above gives
\(\C(H_1)\cong\mathcal A_1.\)
We prove the induction step for \(i\ge1\).  Assume that
\(\C(H_i)\cong\mathcal A_i.\)
We shall prove that
\[
        \C(H_{i+1})\cong\mathcal A_{i+1}.
\]

Recall the first-output table of \(T^\varphi\):
\[
\begin{array}{c|c}
\mathsf L
    & 0\mid0:\mathsf L,\quad 1\mid1:\mathsf R\\[1mm]
\mathsf R
    & 0\mid0:\mathsf A,\quad 1\mid1:\mathsf R\\[1mm]
\mathsf A
    & 0\mid00:\mathsf A,\quad 10\mid01:\mathsf A,\quad
      11\mid1:\mathsf R .
\end{array}
\]
The transducer is epsilon-separated and all entry states are homeomorphism
states, as checked above.  Also
\[
        H_i^\varphi
        =
        F^{\varphi^i\varphi}
        =
        F^{\varphi^{i+1}}
        =
        H_{i+1}
        \le F.
\]
Therefore Corollary~\ref{thm:geometric-determinization-up-to-folding} applies
to the pair
\((T^\varphi,\mathcal A_i).\)
It says that \(\C(H_{i+1})\) is the folded quotient of the geometric forward
automaton obtained from this pair.  We shall compute that automaton and show
that it is already folded.

Recall that, for odd \(t\ge1\),
\(k(t)=2^{\nu_2(t+1)}-2 .\)
Let
\(M=2^i.\)
Since \(i\ge1\), write
\[
        a^{(i)}_0,\ldots,a^{(i)}_{M-1}
\]
for the \(a\)-states of \(\mathcal A_i\).  Thus the relevant part of
\(\mathcal A_i\) is
\[
\begin{array}{c|cc}
        &0&1\\ \hline
\rho   &\ell&r\\
\ell   &\ell&c\\
r      &a^{(i)}_0&r\\
c      &a^{(i)}_0&c\\
a^{(i)}_t
       &a^{(i)}_{t+1\pmod M}
       &
        \begin{cases}
        c, & t\text{ even},\\
        a^{(i)}_{k(t)}, & t\text{ odd}.
        \end{cases}
\end{array}
\]
We shall use two immediate consequences of this table.  Define
\[
        \mu_i(t)=
        \begin{cases}
        0, & t\text{ even},\\[2mm]
        k(t)+1, & t\text{ odd}.
        \end{cases}
\]
Then, for \(0\le t<M\),
\[
        a^{(i)}_t\cdot10=a^{(i)}_{\mu_i(t)}
        \qquad\text{and}\qquad
        a^{(i)}_t\cdot11=c.
\]
Indeed, if \(t\) is even, then
\[
        a^{(i)}_t\cdot1=c,
        \qquad
        c\cdot0=a^{(i)}_0=a^{(i)}_{\mu_i(t)},
        \qquad
        c\cdot1=c.
\]
If \(t\) is odd, then
\[
        a^{(i)}_t\cdot1=a^{(i)}_{k(t)}.
\]
The index \(k(t)\) is even, so
\[
        a^{(i)}_t\cdot10
        =
        a^{(i)}_{k(t)}\cdot0
        =
        a^{(i)}_{k(t)+1},
\]
and
\[
        a^{(i)}_t\cdot11
        =
        a^{(i)}_{k(t)}\cdot1
        =
        c.
\]

We now compute the contracted raw product
\[
        \mathcal G^{\mathrm{raw}}_\varphi(\mathcal A_i).
\]
Its accessible principal states are exactly
\[
        \widehat\rho=(\mathsf L,\rho),\qquad
        \widehat\ell=(\mathsf L,\ell),\qquad
        \widehat r=(\mathsf R,r),\qquad
        \widehat c=(\mathsf R,c),
\]
together with
\[
        P_t=(\mathsf A,a^{(i)}_t)
        \qquad(0\le t<M).
\]

Let us verify this.  Starting from the initial principal state
\((\mathsf L,\rho)\), the first-output blocks from \(\mathsf L\) give
\[
        (\mathsf L,\rho)\xrightarrow{\ 0\ }(\mathsf L,\ell),
        \qquad
        (\mathsf L,\rho)\xrightarrow{\ 1\ }(\mathsf R,r).
\]
From \((\mathsf L,\ell)\), the first-output blocks from \(\mathsf L\) give
\[
        (\mathsf L,\ell)\xrightarrow{\ 0\ }(\mathsf L,\ell),
        \qquad
        (\mathsf L,\ell)\xrightarrow{\ 1\ }(\mathsf R,c).
\]
From \((\mathsf R,r)\), the first-output blocks from \(\mathsf R\) give
\[
        (\mathsf R,r)\xrightarrow{\ 0\ }(\mathsf A,r\cdot0)
        =
        (\mathsf A,a^{(i)}_0)
        =
        P_0,
\]
and
\[
        (\mathsf R,r)\xrightarrow{\ 1\ }(\mathsf R,r).
\]
Similarly, from \((\mathsf R,c)\), the first-output blocks from \(\mathsf R\)
give
\[
        (\mathsf R,c)\xrightarrow{\ 0\ }(\mathsf A,c\cdot0)
        =
        (\mathsf A,a^{(i)}_0)
        =
        P_0,
\]
and
\[
        (\mathsf R,c)\xrightarrow{\ 1\ }(\mathsf R,c).
\]
Thus \(P_0\) is accessible.

Now suppose \(P_t=(\mathsf A,a^{(i)}_t)\) is accessible.  The three
first-output blocks from the transducer state \(\mathsf A\) are
\[
        0\mid00:\mathsf A,
        \qquad
        10\mid01:\mathsf A,
        \qquad
        11\mid1:\mathsf R.
\]
Therefore the terminal principal states are
\[
        (\mathsf A,a^{(i)}_t\cdot0)
        =
        (\mathsf A,a^{(i)}_{t+1\pmod M})
        =
        P_{t+1\pmod M},
\]
\[
        (\mathsf A,a^{(i)}_t\cdot10)
        =
        (\mathsf A,a^{(i)}_{\mu_i(t)})
        =
        P_{\mu_i(t)},
\]
and
\[
        (\mathsf R,a^{(i)}_t\cdot11)
        =
        (\mathsf R,c)
        =
        \widehat c.
\]
The first of these shows, by cycling through the \(0\)-edges, that all states
\(P_t\), \(0\le t<M\), are accessible.  The displayed list also shows that no
other principal states can appear.

Consequently, the contracted raw product has the following word-labelled
transition table:
\[
\begin{array}{c|c}
\widehat\rho
    & 0:\widehat\ell,\quad 1:\widehat r\\[1mm]
\widehat\ell
    & 0:\widehat\ell,\quad 1:\widehat c\\[1mm]
\widehat r
    & 0:P_0,\quad 1:\widehat r\\[1mm]
\widehat c
    & 0:P_0,\quad 1:\widehat c\\[1mm]
P_t
    & 00:P_{t+1\pmod M},\quad
      01:P_{\mu_i(t)},\quad
      1:\widehat c
      \qquad(0\le t<M).
\end{array}
\]

We now subdivide the word-labelled edges and fold common initial subpaths.  At
each principal state \(P_t\), the output labels \(00\) and \(01\) share their
first output letter \(0\).  After subdivision, their first edges are folded
into one edge
\[
        P_t\xrightarrow{\ 0\ }D_t.
\]
The output label \(1\) remains the edge
\[
        P_t\xrightarrow{\ 1\ }\widehat c.
\]
The second letters of \(00\) and \(01\) give
\[
        D_t\xrightarrow{\ 0\ }P_{t+1\pmod M},
        \qquad
        D_t\xrightarrow{\ 1\ }P_{\mu_i(t)}.
\]

Thus the geometric automaton has states
\[
        \widehat\rho,\widehat\ell,\widehat r,\widehat c,
        P_0,D_0,P_1,D_1,\ldots,P_{M-1},D_{M-1},
\]
and transitions
\[
        \widehat\rho\cdot0=\widehat\ell,
        \qquad
        \widehat\rho\cdot1=\widehat r,
\]
\[
        \widehat\ell\cdot0=\widehat\ell,
        \qquad
        \widehat\ell\cdot1=\widehat c,
\]
\[
        \widehat r\cdot0=P_0,
        \qquad
        \widehat r\cdot1=\widehat r,
\]
\[
        \widehat c\cdot0=P_0,
        \qquad
        \widehat c\cdot1=\widehat c,
\]
and, for \(0\le t<M\),
\[
        P_t\cdot0=D_t,
        \qquad
        P_t\cdot1=\widehat c,
\]
\[
        D_t\cdot0=P_{t+1\pmod M},
        \qquad
        D_t\cdot1=P_{\mu_i(t)}.
\]

Rename the states by
\[
        a_{2t}=P_t,
        \qquad
        a_{2t+1}=D_t
        \qquad(0\le t<M),
\]
and remove the hats from
\[
        \widehat\rho,\widehat\ell,\widehat r,\widehat c.
\]
After this renaming, the \(0\)-edges among the \(a\)-states form one cycle:
\[
        a_j\cdot0=a_{j+1\pmod {2M}}
        \qquad(0\le j<2M).
\]
The \(1\)-edges from the even-indexed states are
\(a_{2t}\cdot1=c.\)
For the odd-indexed states, write
\(j=2t+1.\)
If \(t\) is even, then
\(D_t\cdot1=P_0,\)
so
\(a_{2t+1}\cdot1=a_0.\)
Since \(t\) is even, \(t+1\) is odd, and therefore
\(\nu_2(2t+2)=1.\)
Hence
\[
        k(2t+1)=2^{\nu_2(2t+2)}-2=0,
\]
so this is exactly
\[
        a_{2t+1}\cdot1=a_{k(2t+1)}.
\]

If \(t\) is odd, then
\(D_t\cdot1=P_{k(t)+1}.\)
Writing \(d=\nu_2(t+1)\), this gives
\[
        a_{2t+1}\cdot1
        =
        a_{2(k(t)+1)}
        =
        a_{2(2^d-1)}
        =
        a_{2^{d+1}-2}.
\]
But
\(\nu_2(2t+2)=d+1,\)
so again
\[
        a_{2t+1}\cdot1=a_{k(2t+1)}.
\]
Thus, for every odd \(j<2M\),
\(a_j\cdot1=a_{k(j)}.\)

We have obtained exactly the transition table defining
\(\mathcal A_{i+1}.\)
Since the automata \(\mathcal A_{i+1}\) were observed above to be folded, no
further tree-automaton folding is needed.  Hence the folded geometric
automaton is \(\mathcal A_{i+1}\).  By
Corollary~\ref{thm:geometric-determinization-up-to-folding},
\[
        \C(H_{i+1})\cong\mathcal A_{i+1}.
\]
This completes the induction.
\end{proof}

\subsection{Closed overgroups}

For \(0\le j\le i\), the inclusion
\(H_i\le H_j\)
is already known from the construction of the descending chain.  Since
\(H_i\) and \(H_j\) are closed, and since
\(\C(H_i)\cong\mathcal A_i\)
is full, Lemma~\ref{lem:closed-overgroups-quotients-prelim} gives a unique
surjective morphism
\[
        \mathcal A_i\longrightarrow\mathcal A_j .
\]
We denote this morphism by
\[
        \pi_{i,j}:\mathcal A_i\to\mathcal A_j.
\]
By the construction in Lemma~\ref{lem:closed-overgroups-quotients-prelim}, it
is the map sending the state reached by a word \(u\) in \(\mathcal A_i\) to
the state reached by the same word \(u\) in \(\mathcal A_j\).

Explicitly, \(\pi_{i,j}\) fixes the states
\(\rho,\ell,r,c.\)
If \(j=0\), every \(a_t\) is sent to the inner state \(c\) of
\(\mathcal A_0\).  If \(j\ge1\), then
\[
        a_t\mapsto a_{t\pmod {2^j}}.
\]
Indeed, the words
\[
        \eps,\quad 0,\quad 1,\quad 01,\quad 10^{t+1}
\]
reach respectively
\[
        \rho,\quad \ell,\quad r,\quad c,\quad a_t
\]
in \(\mathcal A_i\), and they reach the corresponding states described above
in \(\mathcal A_j\).

We record the direct congruence check, since the same check will be used for
\(\mathcal A_\infty\).  The transitions from
\(\rho,\ell,r,c\) are immediate.  For \(j=0\), the image of every \(a_t\) is
\(c\), and
\(c\cdot0=c\cdot1=c\)
in \(\mathcal A_0\), so the transition relations are preserved.

Now assume \(j\ge1\).  The \(0\)-transitions are preserved because reduction
modulo \(2^j\) commutes with adding \(1\).  For the \(1\)-transitions, parity
is preserved modulo \(2^j\).  If \(t\) is even, both \(a_t\) and its image
have \(1\)-child \(c\).  If \(t\) is odd and \(\nu_2(t+1)<j\), then
\(k(t)=2^{\nu_2(t+1)}-2\)
is unchanged modulo \(2^j\).  If \(\nu_2(t+1)\ge j\), then
\(t\equiv 2^j-1\pmod {2^j},\)
and
\[
        k(t)=2^{\nu_2(t+1)}-2\equiv 2^j-2\pmod {2^j},
\]
which is exactly the \(1\)-child of \(a_{2^j-1}\) in \(\mathcal A_j\).
Therefore \(\pi_{i,j}\) is the surjective morphism induced by the inclusion
\(H_i\le H_j\).

We next prove that these are the only closed overgroups.

\begin{lemma}
\label{lem:quotients-of-Ai}
Let \(i\ge0\), and let \(\mathcal B\) be a tree automaton which projects onto $\mathcal A_0$. Assume that $\mathcal B$  is a quotient
of \(\mathcal A_i\).  Then
\[
        \mathcal B\cong \mathcal A_j
\]
for some \(j\) with \(0\le j\le i\).  Under this isomorphism the quotient map
is \(\pi_{i,j}\).
\end{lemma}

\begin{proof}
The case \(i=0\) is immediate, so assume \(i\ge1\).  Let
\[
        \theta:\mathcal A_i\to\mathcal B
\]
be a surjective morphism.  Since $\mathcal B$ projects onto $\mathcal A_0=\mathcal C(F)$, 
it has a clear partition of states into: root, left boundary states, right boundary states, and inner  states.  Thus the images of
\(\rho,\ell,r\)
are distinct and are not inner states.  The states
\(c,a_0,\ldots,a_{2^i-1}\)
have inner images.  Write
\[
        C=\theta(c),
        \qquad
        A_t=\theta(a_t).
\]

Suppose first that
\(C=A_t\)
for some \(t\).  If \(t\) is odd, then applying the transition labelled \(1\)
gives
\(C=A_{k(t)},\)
and \(k(t)\) is even.  Thus we may assume that \(t\) is even.  Applying the
transition labelled \(0\) to \(C=A_t\) gives
\(A_0=A_{t+1}.\)
Since \(t+1\) is odd, it is relatively prime to \(2^i\).  The equality
\(A_0=A_{t+1}\), together with invariance under the cyclic \(0\)-transition,
therefore forces
\(A_0=A_1=\cdots=A_{2^i-1}.\)
Together with \(C=A_t\), all inner states have the same image.  The quotient
is therefore the automaton \(\mathcal A_0\).

Now assume that
\[
        C\ne A_t
        \qquad(0\le t<2^i).
\]
The relation
\(A_s=A_t\)
is invariant under adding \(1\) to the indices modulo \(2^i\), because
\[
        a_t\cdot0=a_{t+1\pmod {2^i}}.
\]
Hence the equivalence relation on
\(\{0,\ldots,2^i-1\}\)
is congruence modulo a divisor \(d\) of \(2^i\).  Thus
\(d=2^j\)
for some \(0\le j\le i\).  Since \(C\) is distinct from the \(A_t\)'s, we
cannot have \(d=1\).  Indeed, if all \(A_t\)'s were equal, then comparing the
\(1\)-children of \(A_0\) and \(A_1\) would force
\(C=A_0.\)
Hence \(j\ge1\).

The quotient therefore identifies the \(a\)-states exactly modulo \(2^j\), and
does not identify them with \(c\).  The quotient of \(\mathcal A_i\) by this
congruence is precisely the quotient already described by \(\pi_{i,j}\).  As
checked above, that quotient is \(\mathcal A_j\).  Hence
\[
        \mathcal B\cong\mathcal A_j.
\]
\end{proof}

\begin{theorem}
\label{thm:closed-overgroups-Hi}
Let \(i\ge0\).  If \(L\le F\) is closed and
\[
        H_i\le L\le F,
\]
then
\[
        L=H_j
\]
for some \(0\le j\le i\).
\end{theorem}

\begin{proof}
By Proposition~\ref{prop:cores-of-Hi},
\(\C(H_i)\cong\mathcal A_i,\)
which is full.  Since \(H_i\le L\) and both groups are closed,
Lemma~\ref{lem:closed-overgroups-quotients-prelim} gives a surjective morphism
\[
        \mathcal A_i\longrightarrow \C(L).
\]
Since $L\leq F$, the core $\mathcal C(L)$ projects onto $\mathcal A_0=\mathcal C(F)$. Hence, by Lemma~\ref{lem:quotients-of-Ai}, the core \(\C(L)\) is isomorphic to
\(\mathcal A_j\) for some \(j\le i\).  Hence
\[
        L=\D(\C(L))=\D(\mathcal A_j)=H_j.
\]
\end{proof}

\subsection{Arbitrary overgroups}

We now remove the closedness assumption.  First we record the abelianization
calculation needed for the generation theorem.

\begin{lemma}
\label{lem:Hi-full-abelianization}
For every \(i\ge0\),
\[
        \pi_{\ab}(H_i)=\Z^2.
\]
\end{lemma}

\begin{proof}
For every \(k\ge1\), the transducer \(\varphi\) maps the left and right
boundary cylinders onto themselves:
\[
        \varphi(U_{0^k})=U_{0^k},
        \qquad
        \varphi(U_{1^k})=U_{1^k}.
\]
The same is true for every power
\(\psi_i=\varphi^i.\)
Let \(f\in F\).  Suppose that near the left endpoint \(f\) has branch
replacement
\(0^p\eta\mapsto 0^q\eta .\)
Equivalently, for all sufficiently large \(n\),
\(f(U_{0^n})=U_{0^{n-p+q}}.\)
Since \(\psi_i\) preserves every cylinder \(U_{0^n}\), the conjugate
\(f^{\psi_i}\) also sends
\[
        U_{0^n}\quad\text{onto}\quad U_{0^{n-p+q}}
\]
for all sufficiently large \(n\).  Thus the left endpoint slope character is
unchanged by conjugation by \(\psi_i\).  The same argument at the right
endpoint, using the cylinders \(U_{1^n}\), shows that the right endpoint slope
is also unchanged.

Therefore the standard abelianization map of $F$ satisfies:
\[
        \pi_{2}(f^{\psi_i})=\pi_{2}(f)
        \qquad(f\in F).
\]
Since
\(H_i=F^{\psi_i},\)
it follows that
\[
        \pi_{2}(H_i)=\pi_{2}(F)=\Z^2.
\]
\end{proof}

\begin{theorem}
\label{thm:all-overgroups-Hi}
Let \(i\ge0\).  If \(K\le F\) and
\[
        H_i\le K\le F,
\]
then
\[
        K=H_j
\]
for some \(0\le j\le i\).
\end{theorem}

\begin{proof}
Let
\(M=\Cl_F(K).\)
Since \(H_i\) is closed and \(H_i\le K\), we have
\(H_i\le M\le F.\)
By Theorem~\ref{thm:closed-overgroups-Hi}, there exists \(j\le i\) such that
\(M=H_j=F^{\psi_j}.\)
Thus
\[
        K\le F^{\psi_j}
        \qquad\text{and}\qquad
        \Cl_F(K)=F^{\psi_j}.
\]

It remains only to check the abelianization hypothesis in
Lemma~\ref{lem:generation-conjugated-copy}.  Since \(K\) contains \(H_i\), the
subgroup \(K^{\psi_j^{-1}}\le F\) contains
\[
        H_i^{\psi_j^{-1}}
        =
        F^{\psi_i\psi_j^{-1}}
        =
        F^{\varphi^{i-j}}
        =
        H_{i-j}.
\]
By Lemma~\ref{lem:Hi-full-abelianization}, \(H_{i-j}\) has full image in the
abelianization of \(F\).  Therefore \(K^{\psi_j^{-1}}\) also has full image in
the abelianization of \(F\).  Lemma~\ref{lem:generation-conjugated-copy} gives
\(K=F^{\psi_j}=H_j.\)
\end{proof}

In particular, \(H_i\) is contained in no subgroups of \(F\) except
\(H_i,H_{i-1},\ldots,H_0=F.\)
Since the cores \(\mathcal A_i\) have different numbers of states, the groups
\(H_i\) are pairwise distinct.  Hence each inclusion
\(H_{i+1}<H_i\)
is maximal.

\subsection{The intersection is trivial}

It remains to prove that the descending chain has trivial intersection.

Define an infinite binary tree-automaton \(\mathcal A_\infty\) with states
\[
        \rho,\ell,r,c,a_0,a_1,a_2,\ldots
\]
and transitions
\[
\begin{array}{c|cc}
        &0&1\\ \hline
\rho   &\ell&r\\
\ell   &\ell&c\\
r      &a_0&r\\
c      &a_0&c,
\end{array}
\]
and, for \(t\ge0\),
\(a_t\cdot0=a_{t+1},\)
while
\[
        a_t\cdot1=
        \begin{cases}
        c, & t\text{ even},\\[2mm]
        a_{k(t)}, & t\text{ odd}.
        \end{cases}
\]

For \(i\ge1\), let
\[
        \rho_i:\mathcal A_\infty\to \mathcal A_i
\]
be the map fixing
\(\rho,\ell,r,c\)
and sending
\[
        a_t\mapsto a_{t\pmod {2^i}}.
\]
Let
\[
        \rho_0:\mathcal A_\infty\to\mathcal A_0
\]
be the map fixing \(\rho,\ell,r\) and sending
\(c,a_0,a_1,\ldots\)
to the inner state \(c\) of \(\mathcal A_0\).  The same congruence check used
for the maps \(\pi_{i,j}\) shows that every \(\rho_i\) is a surjective
morphism, and for \(0\le j\le i\) we have
\(\rho_i\pi_{i,j}=\rho_j.\)
Here composition is left-to-right.

Thus \(\mathcal A_\infty\) is the inverse limit of the finite automata
\(\mathcal A_i\) in the following concrete sense.  For a finite word \(u\),
write \(u_\infty^+\) for the state reached by \(u\) in
\(\mathcal A_\infty\), and write \(u_i^+\) for the state reached by \(u\) in
\(\mathcal A_i\).  Since \(\rho_i\) is a morphism,
\[
        \rho_i(u_\infty^+)=u_i^+
        \qquad(i\ge0).
\]
Therefore
\[
        u_\infty^+=v_\infty^+
        \quad\Longrightarrow\quad
        u_i^+=v_i^+
        \quad\text{for every }i\ge0.
\]

Conversely, suppose that
\[
        u_i^+=v_i^+
        \qquad\text{for every }i\ge0.
\]
Put
\[
        p=u_\infty^+,
        \qquad
        q=v_\infty^+.
\]
If one of \(p,q\) is among
\(\rho,\ell,r,c,\)
then equality of their images in \(\mathcal A_1\) forces \(p=q\), since these
four states remain distinct from one another and from all \(a\)-states in
\(\mathcal A_1\).  If
\(p=a_s, \qquad q=a_t,\)
then equality of their images in every \(\mathcal A_i\), \(i\ge1\), gives
\[
        s\equiv t\pmod {2^i}
        \qquad(i\ge1),
\]
and hence \(s=t\).  Thus \(p=q\).  Hence, for finite words \(u,v\),
\[
        u_\infty^+=v_\infty^+
        \quad\Longleftrightarrow\quad
        u_i^+=v_i^+
        \text{ for every }i\ge0.
\]

\begin{lemma}
\label{lem:Ainfty-rigidity}
Let \((T_+,T_-)\) be a finite binary tree diagram with branch pairs
\[
        u_1\to v_1,\ldots,u_m\to v_m
\]
listed from left to right.  Suppose that
\[
        u_j^+=v_j^+
        \qquad\text{in }\mathcal A_\infty
        \qquad(1\le j\le m).
\]
Then
\[
        T_+=T_-.
\]
Equivalently,
\[
        u_j=v_j
        \qquad(1\le j\le m).
\]
\end{lemma}

\begin{proof}
We prove the following stronger statement.  For every state \(q\) of
\(\mathcal A_\infty\), the ordered list of states reached by the leaves of a
finite full binary tree \(T\), when \(T\) is read starting at \(q\), determines
the tree \(T\).

For a state \(q\) and a finite full binary tree \(T\), denote this ordered
state list by
\(\Lambda_q(T).\)
If \(T\) is the one-leaf tree, then
\(\Lambda_q(T)=(q).\)
If \(T=T_0\vee T_1\), then
\[
        \Lambda_q(T)
        =
        \Lambda_{q\cdot0}(T_0)\,
        \Lambda_{q\cdot1}(T_1).
\]

Call an adjacent pair of states formally collapsible if it is equal to
\((p\cdot0,p\cdot1)\)
for some state \(p\) of \(\mathcal A_\infty\).  The formally collapsible pairs
are exactly
\[
        (\ell,r),\qquad
        (\ell,c),\qquad
        (a_0,r),\qquad
        (a_0,c),
\]
together with
\[
        (a_{t+1},c)
        \qquad(t\ge0,\ t\text{ even}),
\]
and
\[
        (a_{t+1},a_{k(t)})
        \qquad(t\ge0,\ t\text{ odd}).
\]
Each such pair has a unique parent, since \(\mathcal A_\infty\) is folded.

We do not collapse an arbitrary formally collapsible adjacent pair.  We first
identify the leftmost such pair.  Suppose \(T\) has more than one leaf.  The
two leaves of every exposed caret of \(T\) give a formally collapsible adjacent
pair in \(\Lambda_q(T)\).  We claim that the leftmost formally collapsible
adjacent pair is precisely the pair of leaves belonging to the leftmost
exposed caret of \(T\).

Let \(u<v\) be adjacent leaves of \(T\) which are not siblings.  Write
\[
        u=w0\,1^p,
        \qquad
        v=w1\,0^h,
\]
where \(w\) is their greatest common prefix.  Since \(u\) and \(v\) are not
siblings, at least one of \(p,h\) is positive.

If \(p>0\), then the subtree below \(w0\) has more than one leaf.  Hence it
contains an exposed caret, and that exposed caret occurs strictly to the left
of \(v\).  Thus \(u,v\) cannot give the leftmost formally collapsible adjacent
pair.

It remains to consider the case \(p=0\), so \(h>0\).  Let \(s\) be the state
reached by \(w\), starting at the state \(q\).  The states reached by \(u\) and
\(v\) are
\[
        s\cdot0
        \qquad\text{and}\qquad
        (s\cdot1)\cdot0^h.
\]
We check, using the displayed list of formally collapsible pairs, that this
ordered pair is never formally collapsible.

If \(s=\rho\) or \(s=\ell\), then \(s\cdot0=\ell\), while
\((s\cdot1)\cdot0^h\)
is an \(a\)-state.  The only formally collapsible pairs with first coordinate
\(\ell\) are
\[
        (\ell,r)
        \qquad\text{and}\qquad
        (\ell,c),
\]
so the pair is not formally collapsible.

If \(s=r\) or \(s=c\), then \(s\cdot0=a_0\), while
\((s\cdot1)\cdot0^h\)
is an \(a\)-state.  The only formally collapsible pairs with first coordinate
\(a_0\) are
\[
        (a_0,r)
        \qquad\text{and}\qquad
        (a_0,c).
\]
Thus the pair is not formally collapsible.

Finally suppose \(s=a_t\).  Then
\(s\cdot0=a_{t+1}.\)
If \(t\) is even, then \(s\cdot1=c\), so
\[
        (s\cdot1)\cdot0^h=a_{h-1},
\]
whereas the only formally collapsible pair with first coordinate \(a_{t+1}\)
is
\((a_{t+1},c).\)
Since \(a_{h-1}\ne c\), the pair is not formally collapsible.  If \(t\) is
odd, then
\(s\cdot1=a_{k(t)},\)
so
\[
        (s\cdot1)\cdot0^h=a_{k(t)+h},
\]
whereas the only formally collapsible pair with first coordinate \(a_{t+1}\)
is
\((a_{t+1},a_{k(t)}).\)
Since \(h>0\), the pair is not formally collapsible.

This proves the claim.  Therefore \(\Lambda_q(T)\) determines the leftmost
exposed caret of \(T\) and the state labeling its parent.  Collapsing that
actual exposed caret in \(T\) corresponds to replacing the corresponding
adjacent pair in \(\Lambda_q(T)\) by its unique parent.  Repeating the
procedure reconstructs \(T\) uniquely, because every finite full binary tree
can be reduced to the one-leaf tree by repeatedly collapsing exposed carets.

Thus \(\Lambda_q(T)\) determines \(T\) for every state \(q\).  Applying this
with \(q=\rho\), the hypothesis
\[
        u_j^+=v_j^+
        \qquad(1\le j\le m)
\]
says exactly that
\[
        \Lambda_\rho(T_+)=\Lambda_\rho(T_-).
\]
Hence \(T_+=T_-\).  Since the branch pairs are listed from left to right, we
get
\[
        u_j=v_j
        \qquad(1\le j\le m).
\]
\end{proof}

\begin{theorem}[Trivial intersection]
\label{thm:intersection-trivial}
The chain
\[
        F=H_0>H_1>H_2>\cdots
\]
has trivial intersection:
\[
        \bigcap_{i\ge0}H_i=\{1\}.
\]
\end{theorem}

\begin{proof}
Let
\(g\in\bigcap_{i\ge0}H_i.\)
Choose the reduced tree diagram of \(g\), with branch pairs
\[
        u_1\to v_1,\ldots,u_m\to v_m,
        \qquad
        u_1<\cdots<u_m,\quad v_1<\cdots<v_m.
\]
For every \(i\), the element \(g\) belongs to
\(H_i=\D(\mathcal A_i).\)
Since \(\mathcal A_i\) is folded and full, Lemma
\ref{lem:accepted-diagrams-form-subgroup} implies that the reduced diagram of
\(g\) itself is accepted by \(\mathcal A_i\).  Therefore, for each branch pair
\(u_k\to v_k,\)
the words \(u_k\) and \(v_k\) reach the same state in \(\mathcal A_i\), for
every \(i\).

By the inverse-limit property of \(\mathcal A_\infty\), the words \(u_k\) and
\(v_k\) reach the same state in \(\mathcal A_\infty\).  By
Lemma~\ref{lem:Ainfty-rigidity}, applied to the reduced tree diagram of
\(g\), we get
\[
        u_k=v_k
        \qquad(1\le k\le m).
\]
Thus the reduced diagram represents the identity.  Therefore
\(g=1.\)
\end{proof}

Combining Theorems~\ref{thm:all-overgroups-Hi} and
\ref{thm:intersection-trivial}, we obtain the desired chain.

\begin{corollary}
\label{cor:final-binary-chain}
There is a descending chain
\[
        F=H_0>H_1>H_2>\cdots
\]
such that every \(H_i\) is isomorphic to \(F\), every overgroup of \(H_i\) in
\(F\) is one of
\[
        H_i,H_{i-1},\ldots,H_0,
\]
and
\[
        \bigcap_{i\ge0}H_i=\{1\}.
\]
In particular, each inclusion
\[
        H_{i+1}<H_i
\]
is maximal.
\end{corollary}
   
\section{A chain of subgroups for $F_n$}
\label{sec:standard-chain-and-twisting}

In this section we construct the chain of subgroups for $F_n$, $n\geq 2$.  The construction has
two steps.  First, we construct a standard descending chain
\[
        F_n=S_0>S_1>S_2>\cdots
\]
such that every \(S_i\cong F_n\) and every subgroup of \(F_n\) containing \(S_i\) is one of
\[
        S_i,S_{i-1},\ldots,S_0=F_n.
\]
This standard chain is obtained from powers of a single transducer.  It is the natural
\(n\)-ary generalization of the binary transducer from
Figure~\ref{fig:transducer_semi_sync_corrected}: the binary transducer is inserted into the
residue-zero state of \(\C(F_n)\), while the other residue states only route by the invariant
\(\sigma_n\).

The standard chain need not have trivial intersection.  In the second step we apply a
compatible inner-twisting argument.  This preserves the overgroup structure and forces the
intersection to be trivial.

Throughout the section we let \(m=n-1\).
When \(m=1\), the constructions below reduce to the binary constructions of
Section~\ref{sec:binary-descending-chain}.  Thus the new material is only
needed for \(m\ge2\).

\subsection{The residue-inflated transducer}
\label{subsec:residue-inflated-transducer}

\subsubsection{Inflating binary trees and tree-pair cells}

We first record a simple way to turn binary trees and binary tree-pair cells
into \(n\)-ary ones.

\begin{definition}[\(n\)-inflation of binary trees]
\label{def:arity-inflation-of-binary-trees}
Let \(T\) be a finite full binary tree.  Its \(n\)-inflation
\[
        T^{[n]}
\]
is the finite full \(n\)-ary tree obtained by replacing each binary caret by
an \(n\)-caret, placing the binary left edge at the \(0\)-edge, placing the
binary right edge at the \(m\)-edge, and inserting the middle children
\[
        1,\ldots,m-1
\]
as leaves.
\end{definition}

If \(p\in\{0,1\}^*\), write
\[
        \widehat p\in\{0,m\}^*
\]
for the word obtained from \(p\) by replacing every binary letter \(1\) by
\(m\).  Conversely, for \(w\in\{0,m\}^*\), write
\[
        \check w\in\{0,1\}^*
\]
for the word obtained from \(w\) by replacing every letter \(m\) by \(1\).
Thus \(p\mapsto\widehat p\) is a bijection from \(\{0,1\}^*\) onto
\(\{0,m\}^*\), with inverse \(w\mapsto\check w\), and it preserves lengths,
concatenation, the prefix order, and the lexicographic order.  If
\(P\subseteq\{0,1\}^*\) is the branch set of a finite binary tree, write
\[
        P^{[n]}
\]
for the branch set of the \(n\)-inflated tree.

The leaves of \(P^{[n]}\) are of two types.  The old binary leaves are
\[
        \widehat p
        \qquad(p\in P).
\]
The remaining leaves are middle leaves.  Each middle leaf has the form
\[
        \widehat u d,
        \qquad
        u\in\{0,1\}^*,\quad 1\le d\le m-1,
\]
where \(u\) is a father vertex of the original binary tree.  Since
\(\widehat u\in\{0,m\}^*\), we have \(\sigma_n(\widehat u)=0\).  Thus old
binary leaves satisfy
\(\sigma_n(\widehat p)=0,\)
while the middle leaf \(\widehat u d\) satisfies
\(\sigma_n(\widehat u d)=d.\)
Here and below, nonzero residues modulo \(m\) are identified with their
representatives in \(\{1,\ldots,m-1\}\).

Figure~\ref{fig:binary-tree-3-inflation-example} shows the case \(n=3\).

\begin{figure}[htbp]
\centering
\begin{tikzpicture}[
    >=Stealth,
    line width=.9pt,
    leaf/.style={circle,draw,inner sep=1.5pt,fill=white},
    lab/.style={font=\small,inner sep=1pt},
    every node/.style={font=\small}
]
     \node at (0,1.05) {binary tree};
    \node[leaf] (br) at (0,0) {};
    \node[leaf] (b0) at (-.85,-.9) {};
    \node[leaf] (b1) at (.85,-.9) {};
    \node[leaf] (b00) at (-1.35,-1.8) {};
    \node[leaf] (b01) at (-.35,-1.8) {};

    \draw (br)--node[lab,left] {$0$}(b0);
    \draw (br)--node[lab,right] {$1$}(b1);
    \draw (b0)--node[lab,left] {$0$}(b00);
    \draw (b0)--node[lab,right] {$1$}(b01);

    \node[lab] at (-1.35,-2.18) {$00$};
    \node[lab] at (-.35,-2.18) {$01$};
    \node[lab] at (.85,-1.28) {$1$};

    \node at (2.4,-.75) {$\rightsquigarrow$};
     \node at (5.2,1.05) {3-inflation};
    \node[leaf] (tr) at (5.2,0) {};
    \node[leaf] (t0) at (4.0,-.9) {};
    \node[leaf] (t1) at (5.2,-.9) {};
    \node[leaf] (t2) at (6.4,-.9) {};
    \node[leaf] (t00) at (3.35,-1.8) {};
    \node[leaf] (t01) at (4.0,-1.8) {};
    \node[leaf] (t02) at (4.65,-1.8) {};

    \draw (tr)--node[lab,left] {$0$}(t0);
    \draw (tr)--node[lab,right] {$1$}(t1);
    \draw (tr)--node[lab,right] {$2$}(t2);

    \draw (t0)--node[lab,left] {$0$}(t00);
    \draw (t0)--node[lab,right] {$1$}(t01);
    \draw (t0)--node[lab,right] {$2$}(t02);

    \node[lab] at (3.35,-2.18) {$00$};
    \node[lab] at (4.0,-2.18) {$01$};
    \node[lab] at (4.65,-2.18) {$02$};
    \node[lab] at (5.2,-1.28) {$1$};
    \node[lab] at (6.4,-1.28) {$2$};
\end{tikzpicture}
\caption{The binary tree with branch set \(\{00,01,1\}\) inflates, for
\(n=3\), to the ternary tree with branch set \(\{00,01,02,1,2\}\).}
\label{fig:binary-tree-3-inflation-example}
\end{figure}
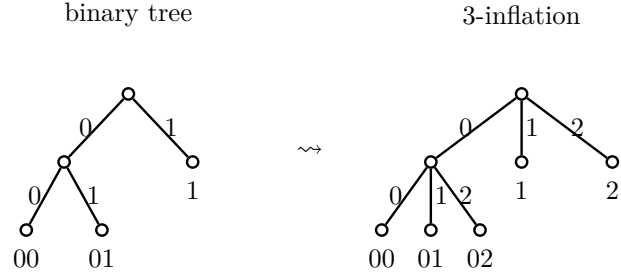

We also inflate cells.  By a \emph{binary cell} we mean an order-preserving
bijection
\[
        \theta:P\longrightarrow Q
\]
between two finite complete binary prefix codes.  Equivalently, \(\theta\) is
the increasing pairing of the leaves of a pair \((T_+,T_-)\) of finite full
binary trees with the same number of leaves.  Since
\[
        |P^{[n]}|=|P|+(m-1)(|P|-1),
\]
and similarly for \(Q\), the inflated branch sets \(P^{[n]}\) and
\(Q^{[n]}\) have the same number of elements.  The \emph{\(n\)-inflation} of
the cell \(\theta\) is the unique order-preserving bijection
\[
        \theta^{[n]}:P^{[n]}\longrightarrow Q^{[n]}.
\]
Equivalently, \(\theta^{[n]}\) is the increasing pairing of the leaves of the
inflated pair \((T_+^{[n]},T_-^{[n]})\).

\begin{remark}
\label{rem:residues-in-inflated-cells}
Let
\[
        (T_+,T_-)
\]
be a binary tree-pair cell, with leaves paired in increasing order.  Then
\[
        (T_+^{[n]},T_-^{[n]})
\]
is an \(n\)-ary tree diagram, and hence represents an element of \(F_n\).
Therefore corresponding branches
\[
        u\longrightarrow v
\]
in the inflated cell satisfy
\[
        \sigma_n(u)=\sigma_n(v)\pmod m.
\]
In other words, if \(\theta:P\to Q\) is a binary cell, then
\[
        \sigma_n\bigl(\theta^{[n]}(u)\bigr)=\sigma_n(u)
        \qquad(u\in P^{[n]}).
\]
\end{remark}

\begin{lemma}
\label{lem:inflated-cells-types-residues}
Let \(P\) and \(Q\) be finite complete binary prefix codes, and let
\(\theta:P\to Q\) be a binary cell.  Then the following hold.
\begin{enumerate}[label=(\roman*)]
    \item The father vertices of the prefix tree of \(P^{[n]}\) are exactly
    the words \(\widehat\xi\), where \(\xi\) is a father vertex of the prefix
    tree of \(P\).  For such \(\xi\), the \(0\)-child and the \(m\)-child of
    \(\widehat\xi\) are
    \[
        \widehat\xi\,0=\widehat{\xi0},
        \qquad
        \widehat\xi\,m=\widehat{\xi1},
    \]
    and, for \(1\le d\le m-1\), the \(d\)-child of \(\widehat\xi\) is the
    middle leaf \(\widehat\xi\,d\), of residue
    \(\sigma_n(\widehat\xi\,d)=d\).
    \item For every \(u\in P\),
    \[
        \theta^{[n]}(\widehat u)=\widehat{\theta(u)} .
    \]
    \item For every \(1\le d\le m-1\), the bijection \(\theta^{[n]}\) maps
    the middle leaves of \(P^{[n]}\) of residue \(d\) onto the middle leaves
    of \(Q^{[n]}\) of residue \(d\).
\end{enumerate}
\end{lemma}

\begin{proof}
(i)  By Definition~\ref{def:arity-inflation-of-binary-trees}, the tree of
\(P^{[n]}\) is obtained from the tree of \(P\) by replacing every binary
caret by an \(n\)-caret, the former \(0\)- and \(1\)-edges becoming the
\(0\)- and \(m\)-edges, and the middle children being inserted as new leaves.
Hence the father vertices of the inflated tree are exactly the images
\(\widehat\xi\) of the father vertices \(\xi\) of the tree of \(P\), and the
children are as displayed.  Since \(\widehat\xi\in\{0,m\}^*\), we have
\(\sigma_n(\widehat\xi)=0\), and hence
\(\sigma_n(\widehat\xi\,d)=d\).

(ii) and (iii)  By Remark~\ref{rem:residues-in-inflated-cells},
corresponding leaves of the inflated cell \(\theta^{[n]}\) have equal
residues.  In an inflated code, the leaves of residue \(0\) are exactly the
old leaves, and the leaves of residue \(d\ne0\) are exactly the middle leaves
of residue \(d\).  This proves (iii), and shows that \(\theta^{[n]}\) maps
old leaves onto old leaves.  For (ii), note that \(\theta\) and the hatting
map are order-preserving, so both
\[
        \widehat u\longmapsto\theta^{[n]}(\widehat u)
        \qquad\text{and}\qquad
        \widehat u\longmapsto\widehat{\theta(u)}
\]
are order-preserving bijections from the old leaves of \(P^{[n]}\) onto the
old leaves of \(Q^{[n]}\).  Two order-preserving bijections between the same
finite linearly ordered sets coincide, which proves (ii).
\end{proof}

\subsubsection{Residue inflation of transducers}

We now define the residue inflation of the binary transducers needed below.
The definition uses the existing notions of entry states and first-output
blocks from Subsection~\ref{subsec:geometric-determinization}.

\begin{definition}
\label{def:inflatable-binary-transducer}
Let
\[
        T=(S,t,o,s_0)
\]
be a finite minimal binary transducer representing an order-preserving
homeomorphism of \(\C_2\).  Let
$B_T$
be its set of entry states, that is, the set of states with no incoming
\(\eps\)-edge.

We say that \(T\) is \emph{inflatable} if:
\begin{enumerate}
    \item every state outside \(B_T\) has a unique incoming edge, and this edge
    has output \(\eps\);
    \item every state in \(B_T\) is a homeomorphism state.
\end{enumerate}
In particular, an inflatable transducer is epsilon-separated and has
homeomorphic entry states.
\end{definition}

For an inflatable transducer, the incoming \(\eps\)-edges form a finite forest whose trees are 
rooted at the entry states.  Thus, for every state \(s\in S\), there are
well-defined data
\[
        e(s)\in B_T,
        \qquad
        u(s)\in\{0,1\}^*,
\]
where \(e(s)\) is the root of the \(\eps\)-component containing \(s\), and
\(u(s)\) is the input word labelling the path from \(e(s)\) to \(s\).  Thus
\[
        t(e(s),u(s))=s,
        \qquad
        o(e(s),u(s))=\eps.
\]
For \(s\in B_T\), this means
\[
        e(s)=s,
        \qquad
        u(s)=\eps.
\]

For \(e\in B_T\), we write
\[
        \mathcal E_T(e)
\]
for the first-output input code at \(e\), and
\[
        \lambda_e:\mathcal E_T(e)\to\W_2
\]
for the first-output label map.  By
Lemma~\ref{lem:first-output-prefix-codes}, \(\mathcal E_T(e)\) and
\[
        \Lambda_T(e)=\{\lambda_e(p)\mid p\in\mathcal E_T(e)\}
\]
are finite complete binary prefix codes, and
\[
        \lambda_e:\mathcal E_T(e)\longrightarrow\Lambda_T(e)
\]
is an order-preserving bijection, that is, a binary cell; we call it the
\emph{first-output cell} at \(e\).  Its \(n\)-inflation is denoted
\[
        \theta_e^{[n]}:
        \mathcal E_T(e)^{[n]}\longrightarrow \Lambda_T(e)^{[n]}.
\]

\begin{definition}[Residue inflation of an inflatable transducer]
\label{def:ordinary-residue-inflation-transducer}
Let \(T=(S,t,o,s_0)\) be an inflatable binary transducer, and let
\[
        z\in B_T
\]
be an entry state, called the residue-zero return state.

The residue-\(m\) inflation of \(T\), with return state \(z\), is the
\(n\)-ary transducer
\[
        \operatorname{Res}_m(T,z)
\]
defined as follows.  Its state set is
\[
        S\sqcup\{\mathsf O_1,\ldots,\mathsf O_{m-1}\},
\]
with initial state \(s_0\).  It will be convenient to denote \(z\) also by
\(\mathsf O_0\); subscripts of the states \(\mathsf O_r\) are then read
modulo \(m\).

For old states \(s\in S\), the extreme-letter transitions are copied from
\(T\):
\[
        s\xrightarrow{\ 0\mid\widehat{o(s,0)}\ }t(s,0),
        \qquad
        s\xrightarrow{\ m\mid\widehat{o(s,1)}\ }t(s,1).
\]

For middle letters \(1\le d\le m-1\), let
\[
        e=e(s),
        \qquad
        u=u(s).
\]
The word \(\widehat u d\) is a middle leaf of the \(n\)-inflated input tree
\(\mathcal E_T(e)^{[n]}\).  Define
\[
        v_{s,d}=\theta_e^{[n]}(\widehat u d),
\]
and set
\[
        s\xrightarrow{\ d\mid v_{s,d}\ }\mathsf O_d
        \qquad(1\le d\le m-1).
\]

Finally, the routing states copy letters and track the digit sum modulo
\(m\):
\[
        \mathsf O_r\xrightarrow{\ d\mid d\ }\mathsf O_{r+d}
        \qquad(1\le r\le m-1,\ d\in X_n),
\]
with subscripts modulo \(m\).  In particular, the machine returns to
\(\mathsf O_0=z\) exactly when the digit sum returns to \(0\) modulo \(m\).
If \(m=1\), there are no routing states and this construction is just \(T\).
\end{definition}

The following is immediate from the definition. 

\begin{lemma}
\label{lem:first-output-of-ordinary-residue-inflation}
Let \(T\) be an inflatable binary transducer, let \(z\in B_T\), and put
\[
        T'=\operatorname{Res}_m(T,z).
\]
Then \(T'\) is epsilon-separated.  Its entry states are
\[
        B_T\sqcup\{\mathsf O_1,\ldots,\mathsf O_{m-1}\}.
\]
For each old entry state \(e\in B_T\), the first-output cell of \(T'\) at
\(e\) is the \(n\)-inflation
\[
        \theta_e^{[n]}:
        \mathcal E_T(e)^{[n]}
        \longrightarrow
        \Lambda_T(e)^{[n]}
\]
of the first-output cell of \(T\) at \(e\).  The old leaf \(\widehat p\),
with \(p\in\mathcal E_T(e)\), terminates at \(t(e,p)\).  A middle leaf of
residue \(r\) terminates at \(\mathsf O_r\).

At a routing state \(\mathsf O_r\), the first-output cell is the one-letter
cell
\[
        d\mid d:\mathsf O_{r+d}
        \qquad(d\in X_n),
\]
with subscripts modulo \(m\).
\end{lemma}

The next lemma records the state reached by an arbitrary input word in the
residue inflation.  For \(w\in\W_n\), let
\(\operatorname{tail}(w)\)
denote the maximal suffix of \(w\) belonging to \(\{0,m\}^*\).

\begin{lemma}
\label{lem:states-in-residue-inflation}
Let \(T'=\operatorname{Res}_m(T,z)\), with transition function \(t'\).  Then
the following hold.
\begin{enumerate}[label=(\roman*)]
    \item For every old state \(s\in S\) and every \(x\in\{0,m\}^*\),
    \[
        t'(s,x)=t(s,\check x).
    \]
    \item For every \(w\in\W_n\) with \(\sigma_n(w)\ne0\),
    \[
        t'(s_0,w)=\mathsf O_{\sigma_n(w)}.
    \]
    \item For every \(w\in\W_n\) with \(\sigma_n(w)=0\),
    \[
        t'(s_0,w)=
        \begin{cases}
        t(s_0,\check w), & w\in\{0,m\}^*,\\[1mm]
        t\bigl(z,\check{\operatorname{tail}(w)}\bigr), & w\notin\{0,m\}^* .
        \end{cases}
    \]
\end{enumerate}
\end{lemma}

\begin{proof}
(i)  The extreme-letter transitions of \(T'\) at old states are copied from
\(T\); in particular they lead again to old states.  The claim follows by
induction on \(|x|\).

We prove (ii) and (iii) by simultaneous induction on \(|w|\).  For
\(w=\eps\), we have \(\sigma_n(w)=0\) and \(w\in\{0,m\}^*\), and both sides
of (iii) equal \(s_0\).  Let \(w=w'a\) with \(a\in X_n\), and assume (ii)
and (iii) for \(w'\).

Suppose first that \(\sigma_n(w')=0\), so that \(t'(s_0,w')\) is the old
state given by (iii).  If \(a\in\{0,m\}\), then, by (i) applied to the one
letter \(a\), the state after \(w\) is obtained by reading \(\check a\) in
\(T\) from that old state.  When \(w'\in\{0,m\}^*\), also \(w\in\{0,m\}^*\)
and
\[
        t'(s_0,w)
        =
        t\bigl(t(s_0,\check{w'}),\check a\bigr)
        =
        t(s_0,\check w);
\]
when \(w'\notin\{0,m\}^*\), we have
\(\operatorname{tail}(w)=\operatorname{tail}(w')\,a\)
and
\[
        t'(s_0,w)
        =
        t\bigl(t(z,\check{\operatorname{tail}(w')}),\check a\bigr)
        =
        t\bigl(z,\check{\operatorname{tail}(w)}\bigr).
\]
Since \(\sigma_n(w)=0\), this proves (iii) for \(w\).  If \(a=d\) is a middle
letter, then reading \(d\) from an old state leads to \(\mathsf O_d\), and
\(\sigma_n(w)=d\ne0\), proving (ii) for \(w\).

Suppose now that \(\sigma_n(w')=s\ne0\), so that
\(t'(s_0,w')=\mathsf O_s\).  Reading \(a\) leads to \(\mathsf O_{s+a}\),
with subscripts modulo \(m\).  If \(s+a\not\equiv0\pmod m\), then
\(\sigma_n(w)=s+a\ne0\), proving (ii) for \(w\).  If \(s+a\equiv0\pmod m\),
then \(a\notin\{0,m\}\), since the letters \(0\) and \(m\) do not change the
digit sum modulo \(m\); hence \(a\) is a middle letter, so
\(w\notin\{0,m\}^*\) and \(\operatorname{tail}(w)=\eps\), while the new state
is
\[
        \mathsf O_0=z=t(z,\eps),
\]
proving (iii) for \(w\).
\end{proof}
                            
\begin{lemma}
\label{lem:ordinary-residue-inflation-homeomorphism}
Let \(T=(S,t,o,s_0)\) be an inflatable binary transducer, let
\(z\in B_T\), and put
\[
        T'=\operatorname{Res}_m(T,z).
\]
Then \(T'\) is accessible and has no states of incomplete response.  Moreover,
\(T'\) represents an order-preserving rational
homeomorphism of \(\C_n\).

Assume, in addition, that the binary state \(z\) is not the identity state,
that is,
\[
        h_{T_z}\ne \id_{\C_2}.
\]
Then \(T'\) is minimal.
\end{lemma}

\begin{proof}
If \(m=1\), then \(T'=T\), and the assertions follow from the minimality of
\(T\).  Assume \(m\ge2\).

We first prove accessibility.  Since \(T\) is accessible, every old state
\(s\in S\) satisfies \(s=t(s_0,\xi)\) for some binary word \(\xi\), and then
\[
        s=t'(s_0,\widehat\xi)
\]
by Lemma~\ref{lem:states-in-residue-inflation}(i).  Each routing state is
accessible since
\(t'(s_0,d)=\mathsf O_d\)
for every middle letter \(d\).
         
We next prove that the entry states represent order-preserving homeomorphisms.
By Lemma~\ref{lem:first-output-of-ordinary-residue-inflation}, every entry
state of \(T'\) has a first-output cell which is an order-preserving bijection
between finite complete \(n\)-ary prefix codes.  The terminal states of these
cells are again entry states.  Hence, along every infinite input ray, the
successive first-output cells determine a nested sequence of domain cylinders
and a nested sequence of image cylinders.  That gives existence of an
image point for every input point and existence of a preimage for every output
point; the prefix-code bijections give uniqueness.  Therefore each entry state
of \(T'\) represents an order-preserving homeomorphism of \(\C_n\).  In
particular, the initial state represents an order-preserving homeomorphism.
Since \(T'\) is finite, this homeomorphism is rational.

We now prove that there is no state of incomplete response.  It is enough to show that
the image of every state of \(T'\) has empty root.  This is clear for the
entry states, because entry states are homeomorphism states.

Let \(s\in S\) be an old non-entry state.  Let
$e=e(s)$ and $u=u(s)$
so that
\[
        t(e,u)=s,\qquad o(e,u)=\eps.
\]
Consider the first-output cell at \(e\).  Let \(P_s\) be the collection of
leaves of \(\mathcal E_T(e)\) below the vertex \(u\), and let \(Q_s\) be the
corresponding collection of leaves in \(\Lambda_T(e)\).  Then
\[
        h_{T_s}(\C_2)=\bigsqcup_{q\in Q_s}U_q .
\]
Since \(T\) is minimal, it has no incomplete response.  Hence
\[
        \Root\bigl(h_{T_s}(\C_2)\bigr)=\eps,
\]
or equivalently \(\Root(Q_s)=\eps\).

In \(T'\), the corresponding part of the first-output cell is obtained by
\(n\)-inflating this binary cell.  Thus the image of the old state \(s\) is
the union of the \(n\)-ary cylinders determined by the \(n\)-inflation
\(Q_s^{[n]}\).  Since \(\Root(Q_s)=\eps\), also
\[
        \Root(Q_s^{[n]})=\eps.
\]
Therefore
\[
        \Root\bigl(h_{T'_s}(\C_n)\bigr)=\eps .
\]

Finally, each routing state \(\mathsf O_r\) is an entry state, hence a
homeomorphism state, so its image is all of \(\C_n\) and has empty root.
Thus every state of \(T'\) has image with empty root.  By the incomplete
response criterion from Subsection~\ref{subsec:transducers-prelim}, \(T'\) has
no states of incomplete response.

It remains to prove minimality under the additional assumption
\(h_{T_z}\ne\id_{\C_2}\).  We already proved accessibility and no incomplete
response, so it remains only to separate states.

Let
\[
        E=\{0,m\}^{\mathbb N}\subseteq \C_n
\]
be the extreme Cantor subset.  For \(\alpha\in\C_2\), write
\(\widehat\alpha\in E\) for the sequence obtained by replacing every binary
letter \(1\) by \(m\).  For old states \(s\in S\), the restriction of
\(h_{T'_s}\) to \(E\) is the inflated binary action:
\[
        h_{T'_s}(\widehat\alpha)=\widehat{h_{T_s}(\alpha)}
        \qquad(\alpha\in\C_2).
\]
For every routing state \(\mathsf O_r\), the restriction to \(E\) is the
identity:
\[
        h_{T'_{\mathsf O_r}}(\widehat\alpha)=\widehat\alpha .
\]

If \(s,t\in S\) are distinct old states, then they are not
\(\omega\)-equivalent in the minimal binary transducer \(T\).  Hence
\(h_{T_s}\ne h_{T_t}\).  Restricting to \(E\), the displayed formula gives
\[
        h_{T'_s}\ne h_{T'_t}.
\]
Thus old states remain distinct.

Next, \(z\) is not equivalent to any routing state.  Indeed, by assumption
\(h_{T_z}\ne\id_{\C_2}\), so the restriction of \(h_{T'_z}\) to \(E\) is not
the identity, while every routing state restricts to the identity on \(E\).

We now separate the routing states from one another.  Suppose
\(\mathsf O_r\) and \(\mathsf O_s\) are routing states with \(r\ne s\).  Choose
\[
        d\equiv -r\pmod m,
        \qquad
        1\le d\le m-1.
\]
Then
\[
        \mathsf O_r\cdot d=\mathsf O_0=z,
        \qquad
        \mathsf O_s\cdot d=\mathsf O_{s-r}.
\]
If \(\mathsf O_r\) and \(\mathsf O_s\) were \(\omega\)-equivalent, then after
reading the common first letter \(d\) the resulting local actions would also
be \(\omega\)-equivalent.  Thus \(z\) would be equivalent to
\(\mathsf O_{s-r}\), contradicting the preceding paragraph.  Hence the routing
states are pairwise distinct.

It remains to separate old states from routing states.  Let \(s\in S\).  If
\(h_{T_s}\ne \id_{\C_2}\), then the restriction of \(h_{T'_s}\) to \(E\) is
not the identity, while every routing state restricts to the identity on
\(E\).  Hence \(s\) is not equivalent to a routing state.

Suppose now that \(h_{T_s}=\id_{\C_2}\).  Then \(s\) is an entry state and its
binary first-output cell is the identity cell
\[
        0\mid0:s,\qquad 1\mid1:s.
\]
Therefore, in the residue-inflated transducer,
\[
        s\xrightarrow{\ d\mid d}\mathsf O_d
        \qquad(1\le d\le m-1).
\]
Let \(\mathsf O_r\) be a routing state and choose
\[
        d\equiv -r\pmod m,
        \qquad
        1\le d\le m-1.
\]
If \(s\) were equivalent to \(\mathsf O_r\), then after reading \(d\) we would
get
\[
        \mathsf O_d\sim z.
\]
But we already proved that \(z\) is not equivalent to any routing state.  This
is impossible.  Hence no old state is equivalent to a routing state.

Thus no two distinct states of \(T'\) are \(\omega\)-equivalent.  Since \(T'\)
is accessible and has no incomplete response, it is minimal.
\end{proof}

\begin{lemma}
\label{lem:ordinary-residue-inflation-semisync}
Let \(T=(S,t,o,s_0)\) be an inflatable binary transducer which is
semi-synchronizing.  Let \(z\in B_T\) be an inner state, and put
\[
        T'=\operatorname{Res}_m(T,z).
\]
Then \(T'\) is semi-synchronizing.
\end{lemma}

\begin{proof}
If \(m=1\), then \(T'=T\), so there is nothing to prove.  Assume \(m\ge2\).

The boundary-ray condition is immediate.  The \(0\)-ray in \(T'\) is copied
from the \(0\)-ray of \(T\), and the \(m\)-ray in \(T'\) is copied from the
\(1\)-ray of \(T\).  Hence the eventual stabilization of the two boundary
rays for \(T\) gives the eventual stabilization of the two boundary rays for
\(T'\).

It remains to check inner synchronization.  Let \(k\) be an inner
synchronization level for \(T\); note that for \(n=2\) the congruence
condition in the definition of inner synchronization is vacuous, so
\[
        t(s_0,aw')=t(s_0,bw')
\]
for all binary inner words \(a,b\) and all \(w'\in X_2^k\).  We prove that
the same \(k\) works for \(T'\).

Let \(u,v\in\W_n\) be inner words with
\[
        \sigma_n(u)=\sigma_n(v),
\]
and let \(w\in X_n^k\).  Write \(s\) for the common residue.

If \(s\ne0\), then by Lemma~\ref{lem:states-in-residue-inflation}(ii),
\[
        t'(s_0,u)=\mathsf O_s=t'(s_0,v),
\]
and hence
\(t'(s_0,uw)=t'(s_0,vw)\)
by determinism.

Assume now that \(s=0\).  By
Lemma~\ref{lem:states-in-residue-inflation}(iii), the states
\(t'(s_0,u)\) and \(t'(s_0,v)\) are old states, and each of them is reached
in \(T\) by a binary inner word.  Indeed, if \(u\in\{0,m\}^*\), then
\(t'(s_0,u)=t(s_0,\check u)\), and \(\check u\) is a binary inner word,
because \(u\) is inner and hatting matches the words \(0^j\) and \(1^j\) with
the words \(0^j\) and \(m^j\).  If \(u\notin\{0,m\}^*\), then
\(t'(s_0,u)=t(z,\check{\operatorname{tail}(u)})\); choosing a binary inner
word \(b\) with \(t(s_0,b)=z\), which is possible since \(z\) is an inner
state of \(T\), we get
\[
        t'(s_0,u)=t\bigl(s_0,b\,\check{\operatorname{tail}(u)}\bigr),
\]
and \(b\,\check{\operatorname{tail}(u)}\) is inner, being an extension of an
inner word.  Thus there are binary inner words \(a_u,a_v\) with
\[
        t'(s_0,u)=t(s_0,a_u),
        \qquad
        t'(s_0,v)=t(s_0,a_v).
\]

Now consider \(w\).  If \(w\in\{0,m\}^k\), then by
Lemma~\ref{lem:states-in-residue-inflation}(i),
\[
        t'(s_0,uw)=t(s_0,a_u\check w),
        \qquad
        t'(s_0,vw)=t(s_0,a_v\check w),
\]
and since \(|\check w|=k\), the inner synchronization of \(T\) gives
\(t'(s_0,uw)=t'(s_0,vw)\).  If \(w\) contains a middle letter, write
\(w=x\,d\,y\) with \(x\in\{0,m\}^*\) and \(d\) a middle letter.  By
Lemma~\ref{lem:states-in-residue-inflation}(i), reading \(x\) keeps both
computations at old states, and reading \(d\) from an old state leads to
\(\mathsf O_d\), independently of that state.  The remaining suffix \(y\) is
common, so
\(t'(s_0,uw)=t'(s_0,vw)\).
\end{proof}

\subsubsection{Residue inflation of binary automata}

We use the same notation for the corresponding operation on binary
tree-automata.

\begin{definition}[Residue inflation of pointed binary automata]
\label{def:residue-inflation-automata}
Let
\[
        \mathcal A=(Q,\tau,\rho)
\]
be a full binary tree-automaton, and let
\[
        c\in Q
\]
be a distinguished state.  Define
\[
        \operatorname{Res}_m(\mathcal A,c)
\]
to be the following \(n\)-ary tree-automaton.  Its state set is
\[
        Q\sqcup\{c_1,\ldots,c_{m-1}\},
\]
with root \(\rho\).  It will be convenient to denote \(c\) also by \(c_0\);
subscripts of the residue states \(c_s\) are then read modulo \(m\).

For \(q\in Q\),
\[
        q\cdot0=(q\cdot0)_{\mathcal A},
        \qquad
        q\cdot m=(q\cdot1)_{\mathcal A},
\]
and
\[
        q\cdot d=c_d
        \qquad(1\le d\le m-1).
\]
For the residue states,
\[
        c_s\cdot d=c_{s+d}
        \qquad(1\le s\le m-1,\ d\in X_n),
\]
with subscripts modulo \(m\); in particular
\(c_s\cdot d=c\)
when \(s+d\equiv0\pmod m\).  If \(m=1\), this is just \(\mathcal A\).
\end{definition}

\begin{lemma}
\label{lem:residue-inflation-preserves-full-folded}
If \(\mathcal A\) is full and folded, then
\[
        \operatorname{Res}_m(\mathcal A,c)
\]
is full and folded.
\end{lemma}

\begin{proof}
Fullness is immediate.  If two old states have the same ordered \(n\)-tuple of
children after inflation, then in particular they have the same \(0\)-child
and the same \(m\)-child.  Thus they have the same \(0\)- and \(1\)-children
in \(\mathcal A\), and foldedness of \(\mathcal A\) makes them equal.

The new residue states are distinguished by their \(0\)-children:
\[
        c_s\cdot0=c_s.
\]
No old state has \(0\)-child equal to one of the new residue states.  Hence no
new residue state folds with an old state or with another residue state.
\end{proof}

\subsubsection{The raw product after inflation}
\label{subsubsec:raw-product-after-inflation}

We now combine the two inflations.  Throughout this subsubsection and the
next,
\[
        T=(S,t,o,s_0)
\]
is an inflatable binary transducer with set of entry states \(B_T\), the
state \(z\in B_T\) is a fixed entry state, and
\[
        \mathcal A=(Q,\tau,\rho)
\]
is a full binary tree-automaton with a distinguished state \(c\in Q\).  We
assume \(m\ge2\); for \(m=1\) every residue inflation is the identity
operation.  We put
\[
        T'=\operatorname{Res}_m(T,z),
        \qquad
        \mathcal A'=\operatorname{Res}_m(\mathcal A,c).
\]

By Lemmas~\ref{lem:ordinary-residue-inflation-homeomorphism}
and~\ref{lem:first-output-of-ordinary-residue-inflation}, the transducer
\(T'\) represents an order-preserving rational homeomorphism of \(\C_n\), it
is epsilon-separated, and its entry states
\[
        B_{T'}=B_T\sqcup\{\mathsf O_1,\ldots,\mathsf O_{m-1}\}
\]
are homeomorphism states.  Hence the geometric construction of
Subsection~\ref{subsec:geometric-determinization} applies to the pair
\((T',\mathcal A')\), and we write
\[
        \Gamma_n=\mathcal G^{\mathrm{raw}}_{T'}(\mathcal A').
\]
Every principal state of \(\Gamma_n\) is a pair in
\[
        \bigl(B_T\sqcup\{\mathsf O_1,\ldots,\mathsf O_{m-1}\}\bigr)
        \times
        \bigl(Q\sqcup\{c_1,\ldots,c_{m-1}\}\bigr).
\]
A principal state of \(\Gamma_n\) lying in \(B_T\times Q\) is called a
\emph{binary principal state}.  We put
\[
        C_r=(\mathsf O_r,c_r)
        \qquad(0\le r\le m-1),
\]
with subscripts read modulo \(m\), and call
\(C_1,\ldots,C_{m-1}\)
the \emph{routing principal states}; the state
\[
        C_0=(\mathsf O_0,c_0)=(z,c)
\]
is a binary principal state.

\begin{lemma}
\label{lem:raw-product-after-inflation}
The root of \(\Gamma_n\) is the binary principal state \((s_0,\rho)\).
Moreover, the following hold.
\begin{enumerate}[label=(\roman*)]
    \item Let \((e,q)\) be an accessible binary principal state of
    \(\Gamma_n\).  The edges of \(\Gamma_n\) at \((e,q)\) are:
    \[
        (e,q)
        \xrightarrow{\ \widehat{\lambda_e(u)}\ }
        \bigl(t(e,u),\,q\cdot u\bigr)
        \qquad(u\in\mathcal E_T(e)),
        \leqno{\mathrm{(E1)}}
    \]
    whose targets are again binary principal states, and
    \[
        (e,q)
        \xrightarrow{\ \theta_e^{[n]}(\widehat{u_0}\,d)\ }
        C_d
        \qquad
        \bigl(u_0\ \text{a father vertex of the tree of }
        \mathcal E_T(e),\ 1\le d\le m-1\bigr).
        \leqno{\mathrm{(E2)}}
    \]
    By Lemma~\ref{lem:inflated-cells-types-residues}, the labels of the
    edges \(\mathrm{(E1)}\) are exactly the old leaves of
    \(\Lambda_T(e)^{[n]}\), and, for each \(d\), the labels of the edges
    \(\mathrm{(E2)}\) with target \(C_d\) are exactly the middle leaves of
    \(\Lambda_T(e)^{[n]}\) of residue \(d\).
    \item Let \(C_r\), \(1\le r\le m-1\), be accessible in \(\Gamma_n\).  The
    edges of \(\Gamma_n\) at \(C_r\) are:
    \[
        C_r
        \xrightarrow{\ d\ }
        C_{r+d}
        \qquad(d\in X_n),
        \leqno{\mathrm{(E3)}}
    \]
    with subscripts modulo \(m\).
\end{enumerate}
\end{lemma}

\begin{proof}
The root of \(\Gamma_n\) is the pair of initial states, that is,
\((s_0,\rho)\).

(i)  By Definition~\ref{def:contracted-raw-forward-product}, the edges of
\(\Gamma_n\) at \((e,q)\) are indexed by the leaves \(\ell\) of the
first-output input code of \(T'\) at \(e\); the edge indexed by \(\ell\) has
label equal to the first-output label of \(\ell\), transducer coordinate
equal to the terminal state of the first-output block, and
\(\mathcal A'\)-coordinate \(q\cdot\ell\).  By
Lemma~\ref{lem:first-output-of-ordinary-residue-inflation}, the first-output
cell of \(T'\) at \(e\) is the \(n\)-inflation
\[
        \theta_e^{[n]}:
        \mathcal E_T(e)^{[n]}\longrightarrow\Lambda_T(e)^{[n]}
\]
of the first-output cell of \(T\) at \(e\); the old leaf \(\widehat u\)
terminates at \(t(e,u)\in B_T\), and a middle leaf of residue \(d\)
terminates at \(\mathsf O_d\).  For an old leaf \(\widehat u\), the label is
\[
        \theta_e^{[n]}(\widehat u)=\widehat{\lambda_e(u)}
\]
by Lemma~\ref{lem:inflated-cells-types-residues}(ii), applied to the cell
\(\lambda_e\), and the \(\mathcal A'\)-coordinate is
\[
        q\cdot\widehat u=q\cdot u\in Q,
\]
because the transitions of \(\mathcal A'\) on the extreme letters \(0,m\)
copy the binary transitions of \(\mathcal A\).  This gives \(\mathrm{(E1)}\).
For a middle leaf \(\widehat{u_0}\,d\), the \(\mathcal A'\)-coordinate is
\[
        q\cdot(\widehat{u_0}\,d)
        =
        (q\cdot u_0)\cdot d
        =
        c_d,
\]
because every state of \(Q\) has \(d\)-child \(c_d\) in \(\mathcal A'\).
This gives \(\mathrm{(E2)}\).  The description of the labels follows from
Lemma~\ref{lem:inflated-cells-types-residues}(ii) and (iii).

(ii)  By Lemma~\ref{lem:first-output-of-ordinary-residue-inflation}, the
first-output cell of \(T'\) at \(\mathsf O_r\) is the one-letter routing cell
\[
        d\mid d:\mathsf O_{r+d}
        \qquad(d\in X_n),
\]
and in \(\mathcal A'\) we have
\(c_r\cdot d=c_{r+d}\),
with subscripts modulo \(m\).  This gives \(\mathrm{(E3)}\).
\end{proof}

\subsubsection{Inflation and the geometric construction}
\label{subsubsec:inflation-and-geometric}

The goal of this subsubsection is to prove that residue inflation commutes with
the geometric construction: inflating the transducer and the input automaton
and then applying the geometric construction yields the same automaton as
applying the binary geometric construction first and inflating the result
(Theorem~\ref{thm:ordinary-residue-inflation-commutes-with-geometric} below).
This is the mechanism behind the computation of the cores of the standard
chain in Subsection~\ref{subsec:standard-chain-cores}: it reduces every
\(n\)-ary core computation to the corresponding binary computation of
Section~\ref{sec:binary-descending-chain}.

One hypothesis deserves emphasis before we begin.  The right-hand side of the
theorem is the residue inflation of the binary geometric automaton at the pair
formed by the return state of the transducer and the distinguished state of
the input automaton.  Since the return state is an entry state, this pair is a
principal state of the binary contracted raw product, \emph{provided it is
accessible there}.  If the pair were not accessible, it would not appear in
the binary geometric automaton at all, and the inflation on the right-hand
side would not even be defined.  Thus the accessibility hypothesis in the
theorem is necessary.
In the application to the standard chain it is verified directly; see the
proof of Proposition~\ref{prop:standard-chain-cores}.

The idea of the proof is the following.  In both automata, the middle letters
\(1,\ldots,m-1\) are inert: reading a middle letter drops the computation into
a routing layer of \(m-1\) states, which records nothing but the digit sum
modulo \(m\), and the computation re-enters the binary layer, at the return
pair, exactly when the digit sum returns to \(0\); on the extreme letters
\(0\) and \(m\), the binary layer reproduces the binary geometric automaton,
with hatted words.  Consequently, in both automata every state is reached
either by the hatted version \(\widehat\xi\) of a binary word \(\xi\), or by a
single middle letter.  The isomorphism is then forced: it must send the state
reached by such a representative word in one automaton to the state reached by
the same word in the other.  Verifying that this assignment is well defined,
bijective, and compatible with the transitions amounts to verifying that two
representative words reach the same state in one automaton if and only if
they do so in the other, together with the three reading rules just described
(Lemma~\ref{lem:reading-inflated-geometric}).

We keep the setting and notation of
Subsubsection~\ref{subsubsec:raw-product-after-inflation}.  In addition, note
that by Definition~\ref{def:inflatable-binary-transducer} the transducer
\(T\) itself is epsilon-separated and its entry states are homeomorphism
states, so the geometric construction applies to the pair
\((T,\mathcal A)\) as well.  We write
\[
        \Gamma_2=\mathcal G^{\mathrm{raw}}_T(\mathcal A),
        \qquad
        \mathcal G=\mathcal G_T(\mathcal A),
        \qquad
        \mathcal G'=\mathcal G_{T'}(\mathcal A').
\]

We can now state the theorem.

\begin{theorem}[Residue inflation commutes with the geometric construction]
\label{thm:ordinary-residue-inflation-commutes-with-geometric}
Let \(T\) be an inflatable binary transducer, let \(z\in B_T\), and let
\((\mathcal A,c)\) be a pointed full binary tree-automaton.  Assume that
\[
        (z,c)
\]
is accessible in the contracted raw product
\[
        \mathcal G^{\mathrm{raw}}_T(\mathcal A);
\]
since \(z\) is an entry state, \((z,c)\) is then a principal state of
\(\mathcal G^{\mathrm{raw}}_T(\mathcal A)\), and hence a state of
\(\mathcal G_T(\mathcal A)\), so that the right-hand side below is defined.
Then there is a canonical isomorphism
\[
        \mathcal G_{\operatorname{Res}_m(T,z)}
        \bigl(\operatorname{Res}_m(\mathcal A,c)\bigr)
        \cong
        \operatorname{Res}_m
        \bigl(\mathcal G_T(\mathcal A),(z,c)\bigr).
\]
\end{theorem}

Before turning to the proof, we record an observation about reading words in
an arbitrary geometric automaton.  It is clear from
Definition~\ref{def:geometric-forward-automaton}; we note it explicitly
because it will be used repeatedly below.

\begin{remark}
\label{rem:reading-geometric}
Let \(\psi\) and \(\mathcal B\) be as in
Definition~\ref{def:geometric-forward-automaton}, and put
\[
        \Gamma=\mathcal G^{\mathrm{raw}}_\psi(\mathcal B),
        \qquad
        \mathcal H=\mathcal G_\psi(\mathcal B).
\]
Recall that at each principal state \(p=(s,q)\) of \(\Gamma\), distinct
outgoing edges carry distinct labels, and the set of these labels is the
finite complete prefix code \(\Lambda(s)\) of nonempty words
(Lemma~\ref{lem:first-output-prefix-codes} and the proof of
Lemma~\ref{lem:clone-contract-well-defined}).  Then the following hold.
\begin{enumerate}[label=(\roman*)]
    \item The states of \(\mathcal H\) are the principal states \(p\) of
    \(\Gamma\), together with the delay states \((p,\xi)\), where \(p=(s,q)\)
    is a principal state and \(\xi\) is a nonempty father vertex of the
    prefix tree of \(\Lambda(s)\).  The foldings in the construction identify
    only common initial subpaths of edges with the same source, so vertices
    pasted at distinct principal states are never identified; hence, with the
    convention \((p,\eps)=p\), two such states are equal only if they have
    the same principal component and the same prefix component.
    \item Let \(p=(s,q)\) be a principal state, let \(\xi\) be a father
    vertex of the prefix tree of \(\Lambda(s)\), possibly empty, and let
    \(b\in X_m\).  Reading \(b\) from \((p,\xi)\) moves one step down the
    pasted prefix tree at \(p\):
    \[
        (p,\xi)\cdot b
        =
        \begin{cases}
        (p,\xi b),
        & \text{if }\xi b\text{ is a father vertex of the tree of }
          \Lambda(s),\\[1mm]
        \text{the target of the edge of }\Gamma\text{ at }p
        \text{ labelled }\xi b,
        & \text{if }\xi b\in\Lambda(s).
        \end{cases}
    \]
    Exactly one of the two cases occurs, since \(\xi b\) is comparable with
    some codeword of the complete prefix code \(\Lambda(s)\), and a codeword
    cannot be a father vertex.
    \item Applying (ii) repeatedly: if \(x\) is the concatenation of the
    labels along a path of \(\Gamma\) from the root to a principal state
    \(p\), then
    \[
        x^+=p
        \qquad\text{in }\mathcal H .
    \]
\end{enumerate}
\end{remark}

\paragraph{The two automata to be compared.}
Since \(\mathcal A\) is full, the binary geometric automaton \(\mathcal G\)
is a full binary tree-automaton
(Definition~\ref{def:geometric-forward-automaton}), and by the accessibility
hypothesis \((z,c)\) is one of its states.  Hence the residue inflation
\[
        \mathcal R
        =
        \operatorname{Res}_m\bigl(\mathcal G,(z,c)\bigr)
\]
is defined.  To keep the three residue layers apart, we denote the residue
states of \(\mathcal R\) by
\[
        \mathfrak o_1,\ldots,\mathfrak o_{m-1},
\]
and accordingly
\(\mathfrak o_0=(z,c)\).
Thus, by Definition~\ref{def:residue-inflation-automata}, the state set of
\(\mathcal R\) is the disjoint union of the state set of \(\mathcal G\) and
\(\{\mathfrak o_1,\ldots,\mathfrak o_{m-1}\}\), and the transitions of
\(\mathcal R\) are: for a state \(g\) of \(\mathcal G\),
\[
        g\cdot0=(g\cdot0)_{\mathcal G},
        \qquad
        g\cdot m=(g\cdot1)_{\mathcal G},
        \qquad
        g\cdot d=\mathfrak o_d
        \quad(1\le d\le m-1),
\]
and, for the residue states,
\[
        \mathfrak o_s\cdot a=\mathfrak o_{s+a}
        \qquad(a\in X_n),
\]
with subscripts modulo \(m\).

Since \((z,c)\) is accessible in \(\Gamma_2\), there is a path of
\(\Gamma_2\) from the root to \((z,c)\); we fix, once and for all, a binary
word
\[
        \pi\in\{0,1\}^*
\]
equal to the concatenation of the labels along such a path.  By
Remark~\ref{rem:reading-geometric}(iii),
\[
        \pi^+=(z,c)
        \qquad\text{in }\mathcal G .
\]

\begin{lemma}
\label{lem:reading-inflated-geometric}
The following hold in \(\mathcal G'\), where \(\xi,\zeta\) range over binary
words, \(d\) over middle letters, \(a\) over \(X_n\), and \(r,s\) over
\(\{1,\ldots,m-1\}\).
\begin{enumerate}[label=(\roman*)]
    \item If
    \(\xi^+=\bigl((e,q),\eta\bigr)\)
    in \(\mathcal G\), where \((e,q)\) is a principal state of \(\Gamma_2\)
    and \(\eta\) is a father vertex of the prefix tree of
    \(\Lambda_T(e)\), possibly empty, then
    \[
        \widehat\xi^{\,+}=\bigl((e,q),\widehat\eta\bigr)
        \qquad\text{in }\mathcal G' .
    \]
    Consequently,
    \[
        \widehat\xi^{\,+}=\widehat\zeta^{\,+}\ \text{in }\mathcal G'
        \quad\Longleftrightarrow\quad
        \xi^+=\zeta^+\ \text{in }\mathcal G .
    \]
    \item \(\widehat\xi^{\,+}\cdot d=C_d\).  In particular, taking
    \(\xi=\eps\), each \(C_d\) is accessible in \(\Gamma_n\) and
    \(d^+=C_d\).
    \item \(C_s\cdot a=C_{s+a}\), with subscripts modulo \(m\).
    \item The states \(C_1,\ldots,C_{m-1}\) are pairwise distinct, and
    \(C_s\ne\widehat\xi^{\,+}\) for every binary word \(\xi\).
    \item Every state of \(\mathcal G'\) is of the form
    \(\widehat\xi^{\,+}\) for some binary word \(\xi\), or \(C_s\) for some
    \(s\).
\end{enumerate}
\end{lemma}

\begin{proof}
Throughout, we use Remark~\ref{rem:reading-geometric} for the pairs
\((T,\mathcal A)\) and \((T',\mathcal A')\).

(i)  We argue by induction on \(|\xi|\).  For \(\xi=\eps\), both sides are
the respective roots, which equal \((s_0,\rho)\) by
Lemma~\ref{lem:raw-product-after-inflation}.  Let \(\xi'=\xi a\) with
\(a\in\{0,1\}\), write \(\xi^+=\bigl((e,q),\eta\bigr)\) in \(\mathcal G\),
and assume
\(\widehat\xi^{\,+}=\bigl((e,q),\widehat\eta\bigr)\)
in \(\mathcal G'\).  By Remark~\ref{rem:reading-geometric}(ii), applied in
\(\mathcal G\), exactly one of the following occurs: either \(\eta a\) is a
father vertex of the tree of \(\Lambda_T(e)\), and
\[
        \xi'^{\,+}=\bigl((e,q),\eta a\bigr)
        \qquad\text{in }\mathcal G;
\]
or \(\eta a=\lambda_e(u)\) for a unique \(u\in\mathcal E_T(e)\), and
\[
        \xi'^{\,+}=\bigl(t(e,u),\,q\cdot u\bigr)
        \qquad\text{in }\mathcal G .
\]
Now read the letter \(\widehat a\in\{0,m\}\) from
\(\bigl((e,q),\widehat\eta\bigr)\) in \(\mathcal G'\).  By
Lemma~\ref{lem:inflated-cells-types-residues}(i), applied to the code
\(\Lambda_T(e)\), the \(\widehat a\)-child of
\(\widehat\eta\) in the prefix tree of \(\Lambda_T(e)^{[n]}\) is
\(\widehat{\eta a}\); it is a father vertex of that tree if and only if
\(\eta a\) is a father vertex of the tree of \(\Lambda_T(e)\), and in the
codeword case
\[
        \widehat{\eta a}
        =
        \widehat{\lambda_e(u)}
\]
is an old codeword, whose unique edge at \((e,q)\) is the edge
\(\mathrm{(E1)}\) determined by \(u\), with target
\(\bigl(t(e,u),\,q\cdot u\bigr)\), by the label description in
Lemma~\ref{lem:raw-product-after-inflation}(i).  By
Remark~\ref{rem:reading-geometric}(ii), applied in \(\mathcal G'\), in both
cases the claim holds for \(\xi'\).

For the consequence: by Remark~\ref{rem:reading-geometric}(i), applied in
\(\mathcal G\) and in \(\mathcal G'\), each of the two equalities holds if
and only if the two states have the same principal component and the same
prefix component; the displayed formula transforms the data by hatting the
prefix component, and hatting is injective.

(ii)  Write \(\widehat\xi^{\,+}=\bigl((e,q),\widehat\eta\bigr)\) as in (i).
By Lemma~\ref{lem:inflated-cells-types-residues}(i), the \(d\)-child of
\(\widehat\eta\) in the prefix tree of \(\Lambda_T(e)^{[n]}\) is the middle
codeword \(\widehat\eta\,d\), of residue \(d\).  By the label description in
Lemma~\ref{lem:raw-product-after-inflation}(i), its unique edge at
\((e,q)\) is an edge \(\mathrm{(E2)}\) with target \(C_d\).  By
Remark~\ref{rem:reading-geometric}(ii),
\(\widehat\xi^{\,+}\cdot d=C_d\).
Taking \(\xi=\eps\) gives \(d^+=C_d\); in particular \(C_d\) is a state of
\(\mathcal G'\), hence accessible in \(\Gamma_n\).

(iii)  By Lemma~\ref{lem:raw-product-after-inflation}(ii), the labels of the
edges of \(\Gamma_n\) at \(C_s\) form the one-letter code \(X_n\); reading
\(a\) completes the codeword \(a\), and by
Remark~\ref{rem:reading-geometric}(ii) arrives at the target \(C_{s+a}\) of
the edge \(\mathrm{(E3)}\) labelled \(a\).

(iv)  If \(C_r=C_s\), then comparing the first coordinates gives
\(\mathsf O_r=\mathsf O_s\), and hence \(r=s\).  For the second assertion, by
(i) the state \(\widehat\xi^{\,+}\) is either a binary principal state, whose
transducer coordinate lies in \(B_T\subseteq S\), or a delay state.  The
routing principal state \(C_s\) has transducer coordinate
\(\mathsf O_s\notin S\), so it is not a binary principal state, and by
Remark~\ref{rem:reading-geometric}(i) it is not a delay state.

(v)  Since \(\mathcal G'\) is a tree-automaton, every state is of the form
\(w^+\) for some \(w\in\W_n\).  We argue by induction on \(|w|\).  For
\(w=\eps\), the root is \(\widehat\eps^{\,+}\).  Let \(w=w'a\).  If
\(w'^+=\widehat\xi^{\,+}\), then for \(a\in\{0,m\}\) we get
\(w^+=(\widehat\xi\,a)^+\)
by determinism, and \(\widehat\xi\,a\in\{0,m\}^*\) is again a hatted word;
while for \(a=d\) a middle letter, \(w^+=C_d\) by (ii).  If
\(w'^+=C_s\), then \(w^+=C_{s+a}\) by (iii); if \(s+a\not\equiv0\pmod m\)
this is a routing principal state, and if \(s+a\equiv0\pmod m\) then
\[
        w^+
        =
        C_0
        =
        (z,c)
        =
        \widehat\pi^{\,+},
\]
where the last equality is (i) applied to \(\xi=\pi\), using
\(\pi^+=(z,c)\) in \(\mathcal G\).
\end{proof}

\begin{proof}[Proof of
Theorem~\ref{thm:ordinary-residue-inflation-commutes-with-geometric}]
We first record that \(\mathcal R\) is a full \(n\)-ary tree-automaton.  It
is full because \(\mathcal G\) is full and, by
Definition~\ref{def:residue-inflation-automata}, every state of
\(\mathcal R\) has all \(n\) children.  Every state is reachable from the
root: the transitions of \(\mathcal R\) on states of \(\mathcal G\) and
extreme letters copy the binary transitions of \(\mathcal G\), so
\[
        \widehat\xi^{\,+}=\xi^+_{\mathcal G}
        \qquad\text{in }\mathcal R
        \qquad(\xi\in\{0,1\}^*),
        \tag{\(\flat\)}
\]
and every state of \(\mathcal G\) is of the form \(\xi^+_{\mathcal G}\);
moreover, for \(1\le s\le m-1\),
\[
        s^+=\mathfrak o_s
        \qquad\text{in }\mathcal R,
        \tag{\(\sharp\)}
\]
since the root is a state of \(\mathcal G\) and its \(s\)-child is
\(\mathfrak o_s\).

Define a map \(\Phi\) from the states of \(\mathcal R\) to the states of
\(\mathcal G'\) as follows.  For a state \(g\) of \(\mathcal G\), choose a
binary word \(\xi\) with \(\xi^+=g\) in \(\mathcal G\), and set
\[
        \Phi(g)=\widehat\xi^{\,+}
        \qquad\text{in }\mathcal G' ;
\]
by Lemma~\ref{lem:reading-inflated-geometric}(i), this does not depend on the
choice of \(\xi\), and \(\Phi\) is injective on the states of
\(\mathcal G\).  For the residue states, set
\[
        \Phi(\mathfrak o_s)=C_s
        \qquad(1\le s\le m-1).
\]
By Lemma~\ref{lem:reading-inflated-geometric}(iv), \(\Phi\) is injective on
the residue states, and its values there are distinct from its values on the
states of \(\mathcal G\); hence \(\Phi\) is injective.  By
Lemma~\ref{lem:reading-inflated-geometric}(v), \(\Phi\) is surjective.
Taking \(\xi=\eps\) shows that \(\Phi\) sends the root to the root.  Note
also that
\[
        \Phi\bigl((z,c)\bigr)
        =
        \widehat\pi^{\,+}
        =
        (z,c)
        =
        C_0,
\]
so \(\Phi(\mathfrak o_s)=C_s\) holds for all \(0\le s\le m-1\).

It remains to check that \(\Phi\) commutes with the \(n\) labelled
transitions; we use the description of the transitions of \(\mathcal R\)
given above.

Let \(g=\xi^+_{\mathcal G}\) be a state of \(\mathcal G\).

For the extreme
letters, by determinism in \(\mathcal G\), in \(\mathcal G'\) and in
\(\mathcal R\), together with \((\flat)\),
\[
        \Phi(g\cdot0)
        =
        \Phi\bigl((\xi0)^+_{\mathcal G}\bigr)
        =
        \widehat{\xi0}^{\,+}
        =
        \bigl(\widehat\xi\,0\bigr)^+
        =
        \widehat\xi^{\,+}\cdot0
        =
        \Phi(g)\cdot0
        \qquad\text{in }\mathcal G',
\]
and likewise
\[
        \Phi(g\cdot m)=\Phi(g)\cdot m,
\]
using \(\widehat\xi\,m=\widehat{\xi1}\).  For a middle letter \(d\),
\[
        \Phi(g\cdot d)
        =
        \Phi(\mathfrak o_d)
        =
        C_d
        =
        \widehat\xi^{\,+}\cdot d
        =
        \Phi(g)\cdot d
\]
by Lemma~\ref{lem:reading-inflated-geometric}(ii).  Finally, for a residue
state \(\mathfrak o_s\), \(1\le s\le m-1\), and \(a\in X_n\),
\[
        \Phi(\mathfrak o_s\cdot a)
        =
        \Phi(\mathfrak o_{s+a})
        =
        C_{s+a}
        =
        C_s\cdot a
        =
        \Phi(\mathfrak o_s)\cdot a
\]
by Lemma~\ref{lem:reading-inflated-geometric}(iii), where the second equality
uses \(\Phi(\mathfrak o_0)=C_0\) in the case \(s+a\equiv0\pmod m\).

Thus \(\Phi\) is a bijective map sending the root to the root and commuting
with all labelled transitions, that is, an isomorphism
\[
        \operatorname{Res}_m\bigl(\mathcal G_T(\mathcal A),(z,c)\bigr)
        \cong
        \mathcal G_{\operatorname{Res}_m(T,z)}
        \bigl(\operatorname{Res}_m(\mathcal A,c)\bigr).
        \qedhere
\]
\end{proof}

\subsection{The standard chain}
\label{subsec:standard-chain}
 
Let \(T^\varphi\) be the binary transducer from
Section~\ref{sec:binary-descending-chain}.  Its states are
\[
        \mathsf L,\mathsf R,\mathsf A,\mathsf B,
\]
with initial state \(\mathsf L\), and transition-output table
\[
\begin{array}{c|cc}
        &0&1\\ \hline
\mathsf L
        &0:\mathsf L&1:\mathsf R\\[1mm]
\mathsf R
        &0:\mathsf A&1:\mathsf R\\[1mm]
\mathsf A
        &00:\mathsf A&\eps:\mathsf B\\[1mm]
\mathsf B
        &01:\mathsf A&1:\mathsf R .
\end{array}
\]
This transducer is inflatable.  The only state with an incoming
\(\eps\)-edge is \(\mathsf B\), and its unique incoming edge is
\[
        \mathsf A\xrightarrow{\ 1\mid\eps\ }\mathsf B.
\]
The entry states are
\[
        \mathsf L,\mathsf R,\mathsf A,
\]
and they are homeomorphism states.

Define the \(n\)-ary transducer
\[
         T^\zeta=\operatorname{Res}_m(T^\varphi,\mathsf R).
\]
Note that by Lemma \ref{lem:ordinary-residue-inflation-homeomorphism}, $T^\zeta$ is a minimal transducer representing an order-preserving homeomorphism. Since $T^\varphi$ is semi-synchronizing, the transducer $T^\zeta$ is semi-synchronizing as well.
Let
\[
        \zeta:\C_n\to\C_n
\]
be the homeomorphism represented by $T^\zeta$.  Explicitly,
\( T^\zeta\) has initial state \(\mathsf L\), old states
\[
        \mathsf L,\mathsf R,\mathsf A,\mathsf B,
\]
and, for \(m\ge2\), routing states
\[
        \mathsf O_1,\ldots,\mathsf O_{m-1}.
\]
As in Definition~\ref{def:ordinary-residue-inflation-transducer}, we denote
the return state \(\mathsf R\) also by \(\mathsf O_0\), and subscripts of the
routing states \(\mathsf O_r\) are read modulo \(m\).
Its transition-output table is
\[
\begin{array}{c|ccc}
        &0&1\le d\le m-1&m\\ \hline
\mathsf L
        &0:\mathsf L&d:\mathsf O_d&m:\mathsf R\\[1mm]
\mathsf R
        &0:\mathsf A&d:\mathsf O_d&m:\mathsf R\\[1mm]
\mathsf A
        &00:\mathsf A&0d:\mathsf O_d&\eps:\mathsf B\\[1mm]
\mathsf B
        &0m:\mathsf A&d:\mathsf O_d&m:\mathsf R .
\end{array}
\]
For routing states,
\[
        \mathsf O_s\xrightarrow{\ d\mid d\ }\mathsf O_{s+d}
        \qquad(1\le s\le m-1,\ d\in X_n),
\]
with subscripts modulo \(m\); in particular, the routing states return to
\(\mathsf O_0=\mathsf R\) exactly when the accumulated digit sum returns to
\(0\) modulo \(m\).

Now, define
\[
        S_i=F_n^{\zeta^i}
        \qquad(i\ge0).
\]
Then \(S_0=F_n\), and since $T^\zeta$ is semi-synchronizing,
\[
        S_{i+1}
        =
        F_n^{\zeta^{i+1}}
        =
        (F_n^\zeta)^{\zeta^i}
        \le
        F_n^{\zeta^i}
        =
        S_i.
\]
Each \(S_i\) is isomorphic to \(F_n\).  Moreover, each \(S_i\) is closed in
\(F_n\).  Indeed, \(\zeta^i\) is order-preserving and
\(S_i=F_n^{\zeta^i}\le F_n\), so Lemma~\ref{lem:closure-conjugacy-prelim}
gives
\[
        \Cl_{F_n}(S_i)
        =
        \Cl_{F_n}(F_n)^{\zeta^i}
        =
        F_n^{\zeta^i}
        =
        S_i.
\]

\subsection{The core automata}
\label{subsec:standard-chain-cores}

For every \(i\ge0\), define
\[
        \mathcal B_i=\operatorname{Res}_m(\mathcal A_i,c),
\]
where \(\mathcal A_i\) is the \(i\)-th binary chain automaton from
Section~\ref{sec:binary-descending-chain}, and \(c\) is its distinguished
 binary inner state. 
 Note that  \[
         \mathcal B_0\cong\C(F_n).
 \]
                                                                            
\begin{lemma}
\label{lem:standard-Bi-full-folded}
For every \(i\ge0\), the automaton \(\mathcal B_i\) is full and folded.
\end{lemma}

\begin{proof}
For \(i=0\), this is Lemma~\ref{lem:explicit-core-of-Fn}.  For \(i\ge1\),
it follows from Lemma~\ref{lem:residue-inflation-preserves-full-folded},
because
\[
        \mathcal B_i=\operatorname{Res}_m(\mathcal A_i,c)
\]
and \(\mathcal A_i\) is full and folded.
\end{proof}

\begin{proposition}
\label{prop:standard-chain-cores}
For every \(i\ge0\),
\[
        \C(S_i)\cong\mathcal B_i.
\]
\end{proposition}

\begin{proof}
The case \(i=0\) is Lemma~\ref{lem:explicit-core-of-Fn}, since
\[
        \mathcal B_0=\operatorname{Res}_m(\mathcal A_0,c)=\C(F_n).
\]

Assume
\[
        \C(S_i)\cong\mathcal B_i
        =
        \operatorname{Res}_m(\mathcal A_i,c).
\]
Since
\[
        S_{i+1}=S_i^\zeta\le F_n,
\]
Corollary~\ref{thm:geometric-determinization-up-to-folding} gives
\[
        \C(S_{i+1})
        \cong
        \overline{\mathcal G_\zeta(\mathcal B_i)}.
\]
Now
\[
        T^\zeta=\operatorname{Res}_m(T^\varphi,\mathsf R),
\]
and the state
\[
        (\mathsf R,c)
\]
is accessible in
\[
        \mathcal G^{\mathrm{raw}}_{T^\varphi}(\mathcal A_i).
\]
Indeed, \((\mathsf L,\ell)\) is accessible, and the first-output block \(1\)
from \(\mathsf L\) sends
\[
        (\mathsf L,\ell)
        \longmapsto
        (\mathsf R,c).
\]
Theorem~\ref{thm:ordinary-residue-inflation-commutes-with-geometric} gives
\[
        \mathcal G_\zeta(\mathcal B_i)
        \cong
        \operatorname{Res}_m
        \bigl(\mathcal G_\varphi(\mathcal A_i),(\mathsf R,c)\bigr).
\]
The binary core computation in Section~\ref{sec:binary-descending-chain}
shows that the geometric automaton \(\mathcal G_\varphi(\mathcal A_i)\) is
already full and folded, and that
\[
        \mathcal G_\varphi(\mathcal A_i)
        \cong
        \mathcal A_{i+1}.
\]
Under this isomorphism the principal state
\[
        (\mathsf R,c)
\]
is the distinguished state \(c\) of \(\mathcal A_{i+1}\).  Hence
\[
        \mathcal G_\zeta(\mathcal B_i)
        \cong
        \operatorname{Res}_m(\mathcal A_{i+1},c)
        =
        \mathcal B_{i+1}.
\]
By Lemma~\ref{lem:standard-Bi-full-folded}, \(\mathcal B_{i+1}\) is already
full and folded, so the folded quotient is unchanged.  Therefore
\[
        \C(S_{i+1})\cong\mathcal B_{i+1}.
\]
\end{proof}

In particular, for each $i$, $S_{i+1}$ is a strict subgroup of $S_i$ since their cores are not isomorphic. 

\subsection{Closed overgroups of the standard chain}
\label{subsec:standard-chain-closed-overgroups}

We now compute the closed overgroups of the subgroups \(S_i\).  The only
input from the binary construction is the corresponding quotient computation
for the binary cores \(\mathcal A_i\).

Recall that
\[
        \mathcal A_0=\C(F)
\]
is the standard binary core, and that \(c\) denotes its unique inner state.
We use the same letter \(c\) for the distinguished inner state of each
\(\mathcal A_i\).  The binary quotient computation from
Section~\ref{sec:binary-descending-chain} implies the following: if
\[
        \theta:\mathcal A_i\longrightarrow\mathcal D
\]
is a quotient of binary tree-automata such that the unique morphism from $\mathcal A_i$ to $\mathcal A_0$  factors through $\theta$ (i.e., such that $\theta$ preserves the partition of root, left state, right state and inner states from $\mathcal A_i$), then 
\[
        \mathcal D\cong \mathcal A_j
\]
for some \(0\le j\le i\).
Moreover, under this isomorphism the image
\(\theta(c)\) of the distinguished state is the distinguished state \(c\) of
\(\mathcal A_j\).

We first record that morphisms inflate.

\begin{lemma}
\label{lem:inflating-morphisms}
Let
\[
        \alpha:\mathcal A\longrightarrow\mathcal D
\]
be a morphism of full binary tree-automata, and let \(c\in\mathcal A\).  Put
\[
        \bar c=\alpha(c).
\]
Then \(\alpha\) induces a morphism
\[
        \alpha^{[n]}:
        \operatorname{Res}_m(\mathcal A,c)
        \longrightarrow
        \operatorname{Res}_m(\mathcal D,\bar c).
\]
It is defined by
\[
        \alpha^{[n]}(q)=\alpha(q)
\]
on old states, and by
\[
        \alpha^{[n]}(c_s)=\bar c_s
        \qquad(1\le s\le m-1)
\]
on the new residue states.  If \(\alpha\) is surjective, then
\(\alpha^{[n]}\) is surjective.
\end{lemma}

\begin{proof}
On old states and on the extreme letters \(0\) and \(m\), this is exactly the
fact that \(\alpha\) preserves the binary \(0\)- and \(1\)-transitions.  On an
old state \(q\) and a middle letter \(1\le d\le m-1\), both sides go to the
new residue state of residue \(d\):
\[
        \alpha^{[n]}(q\cdot d)
        =
        \alpha^{[n]}(c_d)
        =
        \bar c_d
        =
        \alpha(q)\cdot d .
\]
Finally, on the residue states, both inflated automata use the same
residue-routing rule: writing \(c_0=c\) and \(\bar c_0=\bar c\), we have
\[
        c_s\cdot d=c_{s+d}
        \qquad(1\le s\le m-1,\ d\in X_n),
\]
with subscripts read modulo \(m\), and similarly for the states
\(\bar c_s\).  Thus \(\alpha^{[n]}\) is a
morphism.  If \(\alpha\) is onto, then every old state of
\(\operatorname{Res}_m(\mathcal D,\bar c)\) is hit, and every new residue
state \(\bar c_s\) is the image of \(c_s\).  Hence \(\alpha^{[n]}\) is onto.
\end{proof}

We shall use the canonical identification
\[
        \operatorname{Res}_m(\mathcal A_0,c)\cong \C(F_n).
\]
Indeed, under this identification the old root, left state and right state of
\(\mathcal A_0\) become respectively
\[
        q_0,\ q_L,\ q_R
\]
in \(\C(F_n)\), and the residue state \(c_s\), where \(c_0=c\) is the inner
state of \(\mathcal A_0\), becomes
\[
        q^{\mathrm{in}}_s
        \qquad(0\le s\le m-1).
\]

\begin{proposition}
\label{prop:quotients-of-Bi}
Let \(i\ge0\), and let
\[
        \Theta:\mathcal B_i\longrightarrow\mathcal E
\]
be a quotient of \(n\)-ary tree-automata.  Suppose that \(\mathcal E\) admits
a surjective morphism
\[
        \mathcal E\longrightarrow \C(F_n).
\]
Then
\[
        \mathcal E\cong\mathcal B_j
\]
for some \(0\le j\le i\).
\end{proposition}

\begin{proof}
Let
\[
        \pi_i:\mathcal A_i\longrightarrow\mathcal A_0=\C(F)
\]
be the natural binary morphism.  It sends the distinguished state \(c\) of
\(\mathcal A_i\) to the unique inner state \(c\) of \(\mathcal A_0\).
By Lemma~\ref{lem:inflating-morphisms}, it inflates to a morphism
\[
        \pi_i^{[n]}:
        \mathcal B_i
        =
        \operatorname{Res}_m(\mathcal A_i,c)
        \longrightarrow
        \operatorname{Res}_m(\mathcal A_0,c)
        =
        \C(F_n).
\]

Let
\[
        \mu:\mathcal E\longrightarrow\C(F_n)
\]
be a surjective morphism.  Since rooted deterministic tree-automata admit at
most one morphism to a fixed target, the two morphisms
\[
        \Theta\mu,\ \pi_i^{[n]}:
        \mathcal B_i\longrightarrow\C(F_n)
\]
are equal.  Hence the kernel of \(\Theta\) refines the fibers of
\(\pi_i^{[n]}\).

In particular, no nonzero residue state of \(\mathcal B_i\) can be identified
with any other state.  Indeed, for \(1\le s\le m-1\), the state \(c_s\) is the
unique state of \(\mathcal B_i\) mapped by \(\pi_i^{[n]}\) to
\(q^{\mathrm{in}}_s\).  Old states map only to
\[
        q_0,\ q_L,\ q_R,\ q^{\mathrm{in}}_0,
\]
and the nonzero residue states map to the pairwise distinct states
\[
        q^{\mathrm{in}}_s
        \qquad(1\le s\le m-1).
\]
Therefore all identifications made by \(\Theta\) occur among the old states of
\(\mathcal A_i\).

Restrict the kernel of \(\Theta\) to the old states.  This gives a binary
automaton congruence on \(\mathcal A_i\), because the old \(0\)- and
\(m\)-transitions in \(\mathcal B_i\) are precisely the binary \(0\)- and
\(1\)-transitions in \(\mathcal A_i\).  Let
        $\bar{\mathcal A}$
be the corresponding binary quotient of \(\mathcal A_i\), and let
$\bar c$
be the image of the distinguished state \(c\).

We claim that
\[
        \mathcal E
        \cong
        \operatorname{Res}_m(\bar{\mathcal A},\bar c).
        \tag{\(*\)}
\]
Indeed, the old states of \(\mathcal E\) are exactly the quotient classes of
old states of \(\mathcal A_i\), and the nonzero residue states remain
singleton states.  The transitions are the residue-inflated transitions: on
old states, the \(0\)- and \(m\)-edges are induced from the binary quotient,
all middle \(d\)-edges go to the residue state of residue \(d\), and the
residue states follow the usual residue-routing rule, returning to
\(\bar c_0=\bar c\) when the residue becomes zero.  This is precisely
\(\operatorname{Res}_m(\bar{\mathcal A},\bar c)\).

Moreover, since
\[
        \Theta\mu=\pi_i^{[n]},
\]
the binary morphism
\[
        \pi_i:\mathcal A_i\longrightarrow\mathcal A_0
\]
 factors through a surjective morphism
\[
        \bar{\mathcal A}\longrightarrow\mathcal A_0.
\]
By the binary quotient computation recalled at the beginning of this
subsection,
\[
        \bar{\mathcal A}\cong\mathcal A_j
\]
for some \(0\le j\le i\), and under this isomorphism \(\bar c\) corresponds to
the distinguished state \(c\) of \(\mathcal A_j\).  Therefore
\[
        \operatorname{Res}_m(\bar{\mathcal A},\bar c)
        \cong
        \operatorname{Res}_m(\mathcal A_j,c)
        =
        \mathcal B_j.
\]
Together with \((*)\), this gives
\[
        \mathcal E\cong\mathcal B_j.
\]
\end{proof}

\begin{theorem}
\label{thm:closed-overgroups-Si}
Let \(i\ge0\).  If \(L\le F_n\) is closed and
\[
        S_i\le L\le F_n,
\]
then
\[
        L=S_j
\]
for some \(0\le j\le i\).  Thus the closed overgroups of \(S_i\) in \(F_n\)
are exactly
\[
        S_0,S_1,\ldots,S_i.
\]
\end{theorem}

\begin{proof}
Let \(L\le F_n\) be closed and suppose
\[
        S_i\le L\le F_n.
\]
By Proposition~\ref{prop:standard-chain-cores},
\[
        \C(S_i)\cong\mathcal B_i.
\]
The automaton \(\mathcal B_i\) is full by
Lemma~\ref{lem:standard-Bi-full-folded}.  Since \(S_i\le L\) and both
subgroups are closed, Lemma~\ref{lem:closed-overgroups-quotients-prelim} gives
a surjective morphism
\[
        \mathcal B_i\longrightarrow \C(L).
\]
Also \(L\le F_n\), so the same lemma gives a morphism
\[
        \C(L)\longrightarrow \C(F_n).
\]
The automaton \(\C(L)\) is a quotient of the full automaton \(\mathcal B_i\),
so it is full.  Therefore the morphism
\[
        \C(L)\longrightarrow\C(F_n)
\]
is surjective.

Thus \(\C(L)\) is a quotient of \(\mathcal B_i\) which admits a surjective
morphism onto \(\C(F_n)\).  By
Proposition~\ref{prop:quotients-of-Bi},
\[
        \C(L)\cong\mathcal B_j
\]
for some \(0\le j\le i\).  But
\[
        \C(S_j)\cong\mathcal B_j.
\]
Since closed subgroups are determined by their cores, we get
\[
        L=S_j.
\]
\end{proof}

\subsection{Abelianization and normal subgroups in the standard chain}
\label{subsec:standard-chain-abelianization-normal}

In this section we prove that the image of each group $S_i$ in the abelianization of $F_n$ is $F_n[F_n,F_n]$. 

We write
\[
        e_0,e_1,\ldots,e_m
\]
for the standard basis of
\[
        F_n/[F_n,F_n]\cong\mathbb Z^{m+1}
\]
from Subsection~\ref{subsec:abelianization}.  Thus
\[
        \ab_n(x_0)=e_0,
\]
and, for \(j\ge1\),
\[
        \ab_n(x_j)=e_r
        \quad\text{where}\quad
        r\in\{1,\ldots,m\}
        \quad\text{and}\quad
        j\equiv r\pmod m.
\]
In particular,
\[
        \ab_n(x_m)=e_m.
\]

For the binary group \(F=F_2\), we shall denote the Brown basis of
\[
        F/[F,F]\cong\mathbb Z^2
\]
by
\[
        e_0^F,e_1^F.
\]
Thus
\[
        e_0^F=\ab_2(x_0^F),
        \qquad
        e_1^F=\ab_2(x_1^F).
\]
We use the superscript \(F\) only to avoid confusing the generators of $F$ and of $F/[F,F]$
with those of \(F_n\) and $F_n/[F_n,F_n]$.

We shall use the following \(n\)-ary inflation of binary diagrams.  If
\(g\in F\) is represented by a binary tree diagram, let
\[
        \operatorname{Infl}_n(g)\in F_n
\]
be the element represented by the \(n\)-inflation of that diagram, in the
sense of Definition~\ref{def:arity-inflation-of-binary-trees}. 

This gives an injective homomorphism
\[
        \operatorname{Infl}_n:F\longrightarrow F_n.
\]
On the standard generators it satisfies
\[
        \operatorname{Infl}_n(x_0^F)=x_0,
        \qquad
        \operatorname{Infl}_n(x_1^F)=x_m.
\]
Consequently, on abelianizations,
\[
        e_0^F\longmapsto e_0,
        \qquad
        e_1^F\longmapsto e_m.
\]

\begin{lemma}
\label{lem:inflated-binary-copy-gives-end-generators}
For every \(i\ge0\), one has
\[
        e_0,e_m\in \ab_n(S_i).
\]
\end{lemma}

\begin{proof}
Fix \(i\ge0\).  Recall that
\[
        H_i=F^{\varphi^i}\le F
\]
is the \(i\)-th subgroup in the binary chain.  By the binary computation, the
image of \(H_i\) in the abelianization of \(F\) is all of
\(F/[F,F]\cong\mathbb Z^2\).  Equivalently, there exist elements
\(h_0,h_1\in H_i\) such that
\[
        \ab_2(h_0)=e_0^F,
        \qquad
        \ab_2(h_1)=e_1^F.
\]

We claim that
\[
        \operatorname{Infl}_n(H_i)\le S_i.
\]
Indeed, the core of \(H_i\) is the binary automaton \(\mathcal A_i\), while
the core of \(S_i\) is
\[
        \mathcal B_i=\operatorname{Res}_m(\mathcal A_i,c)
\]
by Proposition~\ref{prop:standard-chain-cores}.  If a binary diagram is
accepted by \(\mathcal A_i\), then its \(n\)-inflation is accepted by
\(\operatorname{Res}_m(\mathcal A_i,c)\).  On the old binary branches this is
immediate from the definition of residue inflation.  On each newly inserted
middle branch both sides go to the same residue-routing state.  Hence every
inflated diagram of an element of \(H_i\) is accepted by \(\mathcal B_i\), and
therefore belongs to \(S_i=\D(\mathcal B_i)\).

Now apply \(\operatorname{Infl}_n\) to \(h_0\) and \(h_1\).  Since the induced
map on abelianizations sends
\[
        e_0^F\mapsto e_0,
        \qquad
        e_1^F\mapsto e_m,
\]
we get
\[
        \ab_n(\operatorname{Infl}_n(h_0))=e_0,
        \qquad
        \ab_n(\operatorname{Infl}_n(h_1))=e_m.
\]
    \end{proof}

We get the remaining generators of the abelianization from the following lemma.

\begin{lemma}
\label{lem:adjacent-differences-in-Si}
Assume \(m\ge2\).  For every \(i\ge0\) and every
\[
        0\le j\le m-2,
\]
the element
\[
        \delta_j=x_jx_{j+1}^{-1}
\]
belongs to \(S_i\).  Moreover,
\[
        \ab_n(\delta_j)=e_j-e_{j+1}.
\]
\end{lemma}

\begin{proof}
The abelianization computation is immediate from the definition of
\(\ab_n\):
\[
        \ab_n(\delta_j)
        =
        \ab_n(x_j)-\ab_n(x_{j+1})
        =
        e_j-e_{j+1}.
\]

It remains to prove that \(\delta_j\in S_i\).  If \(i=0\), this is immediate
because
\[
        S_0=F_n.
\]
Assume from now on that \(i\ge1\).  We use the core
\[
        \C(S_i)=\mathcal B_i=\operatorname{Res}_m(\mathcal A_i,c).
\]

Since \(\mathcal B_i\) is the residue inflation of \(\mathcal A_i\), we shall
only need the following consequences of the definition.  The old states of
\(\mathcal A_i\) keep their names in \(\mathcal B_i\).  Thus the old root,
left state, right state and distinguished inner state are
\[
        \rho,\ell,r,c,
\]
and the residue states are
\[
        c_s
        \qquad(0\le s\le m-1),
\]
where, as in Definition~\ref{def:residue-inflation-automata}, we denote the
distinguished state \(c\) also by \(c_0\) and read subscripts modulo \(m\).
The relevant transitions are
\[
        \rho\cdot0=\ell,
        \qquad
        \rho\cdot m=r,
        \qquad
        \rho\cdot d=c_d
        \quad(1\le d\le m-1),
\]
\[
        \ell\cdot0=\ell,
        \qquad
        \ell\cdot m=c_0,
        \qquad
        \ell\cdot d=c_d
        \quad(1\le d\le m-1),
\]
and, for every old state \(q\) of \(\mathcal A_i\),
\[
        q\cdot d=c_d
        \qquad(1\le d\le m-1).
\]
Finally, for \(1\le s\le m-1\),
\[
        c_s\cdot d=c_{s+d}
        \qquad(d\in X_n),
\]
where subscripts are read modulo \(m\); in particular \(c_s\cdot d=c_0\) when
\(s+d\equiv0\pmod m\).

Fix
\[
        0\le j\le m-2.
\]
The standard reduced tree diagram of
\[
        \delta_j=x_jx_{j+1}^{-1}
\]
has branch pairs
\[
        \delta_j:\left\{
        \begin{array}{ll}
        k\to k, & 0\le k<j,\\[1mm]
        j0\to j, & \\[1mm]
        j(a+1)\to (j+1)a, & 0\le a\le m-1,\\[1mm]
        j+1\to (j+1)m, & \\[1mm]
        k\to k, & j+2\le k\le m.
        \end{array}
        \right.
\]
           
The first and last lines are identity branch pairs, hence are accepted
automatically.

We check the remaining branch pairs.  First consider
\[
        j0\to j.
\]
If \(j=0\), then
\[
        (00)^+=\ell\cdot0=\ell=0^+.
\]
If \(j\ge1\), then
\[
        (j0)^+=c_j\cdot0=c_j=j^+.
\]
Thus \(j0\to j\) is accepted in all cases, including the endpoint case
\(j=0\).

Next consider
\[
        j(a+1)\to (j+1)a,
        \qquad 0\le a\le m-1.
\]
If \(j=0\), then the left-hand branch ends at
\[
        (0(a+1))^+
        =\ell\cdot(a+1)
        =c_{(a+1)},
\]
using \(\ell\cdot d=c_d\) for \(1\le d\le m-1\) and \(\ell\cdot m=c_0\); the
right-hand branch ends at
\[
        (1a)^+
        =\rho\cdot1\cdot a
        =c_1\cdot a
        =c_{(1+a)}.
\]
Hence, the two terminal states agree.

If \(j\ge1\), then, since \(0\le j\le m-2\), both \(j\) and \(j+1\) satisfy
\(1\le j,\,j+1\le m-1\), so \(\rho\cdot j=c_j\) and \(\rho\cdot(j+1)=c_{j+1}\).
The left-hand branch ends at
\[
        (j(a+1))^+
        =c_j\cdot(a+1)
        =c_{(j+a+1)},
\]
and the right-hand branch ends at
\[
        ((j+1)a)^+
        =c_{j+1}\cdot a
        =c_{(j+1+a)}.
\]
Hence, the two terminal states agree in this case as well.

Thus the branch pair \(j(a+1)\to(j+1)a\) is accepted for every
\(0\le a\le m-1\).

For the branch pair
\[
        j+1\to (j+1)m,
\]
we have
\[
        (j+1)^+=c_{j+1},
\]
and
\[
        ((j+1)m)^+
        =c_{j+1}\cdot m
        =c_{(j+1+m)}
        =c_{j+1},
\]
since \(1\le j+1\le m-1\).  Thus this branch pair is accepted.

Every branch pair in the displayed diagram is accepted by \(\mathcal B_i\).
Since
\[
        S_i=\D(\mathcal B_i),
\]
we get that for all $j\in\{0,\dots,m-2\}$
\[
        \delta_j\in S_i.
\]
\end{proof}

Since $\{e_0,e_0-e_1,e_1-e_2,\dots,e_{m-2}-e_{m-1},e_m\}$ generates $\mathbb{Z}^{m+1}$ we get the following.

\begin{corollary}
\label{prop:Si-full-abelianization}
For every \(i\ge0\), the image of \(S_i\) in the abelianization of \(F_n\) is
all of
\[
        F_n/[F_n,F_n].
\]
Equivalently,
\[
        S_i[F_n,F_n]=F_n.
\]
\end{corollary}
                                                                                                                                        
\subsection{One-sided slopes in the standard chain}
\label{subsec:standard-one-sided-slopes}

\begin{lemma}
\label{lem:standard-chain-one-sided}
For every \(i\ge0\) and every \(s\in\Z/m\Z\), there exist
\[
        h_s\in S_i,
        \qquad
        \alpha_s\in D_{n,s},
\]
such that
\[
        h_s(\alpha_s)=\alpha_s,
        \qquad
        h_s'(\alpha_s^-)=n,
        \qquad
        h_s'(\alpha_s^+)=1.
\]
\end{lemma}

\begin{proof}
In the core \(\mathcal B_i=\C(S_i)\), for every residue \(s\) there is a state
over residue \(s-1\) with an \(m\)-loop, namely the state \(c_{s-1}\), with
subscripts read modulo \(m\) and \(c_0=c\).  Indeed, if
\(s-1\not\equiv0\pmod m\), the residue-routing rule gives
\(c_{s-1}\cdot m=c_{s-1},\)
while the residue-zero state \(c_0=c\) satisfies
\(c_0\cdot m=c_0,\)
since \(c\cdot1=c\) in \(\mathcal A_i\).
Choose an inner word \(u_s\) ending at this state.  Then
\[
        (u_sm)^+=u_s^+
        \qquad
        \text{in }\mathcal B_i.
\]
Since \(S_i\) is closed, the core-existence property gives an element of \(S_i\) with branch
pair
\(u_sm\to u_s.\)
Let
\[
        \alpha_s=\rho_n(u_sm^\infty).
\]
This is the right endpoint of \([u_s]\).  Its terminating expansion has residue
\(\sigma_n(u_s)+1=s,\)
so
\(\alpha_s\in D_{n,s}.\)
The branch pair
\(u_sm\to u_s\)
fixes \(\alpha_s\) and has left derivative \(n\) at \(\alpha_s\).  Taking the component on
the right of \(\alpha_s\) to be the identity, using closedness of \(S_i\), gives an element
\(h_s\in S_i\) with
\[
        h_s(\alpha_s)=\alpha_s,
        \qquad
        h_s'(\alpha_s^-)=n,
        \qquad
        h_s'(\alpha_s^+)=1.
\]
\end{proof}

\subsection{Arbitrary overgroups of the standard chain}
\label{subsec:standard-arbitrary-overgroups}

\begin{lemma}
\label{lem:standard-dense-overgroups}
Let \(i\ge1\), and let
\[
        S_i\le K\le F_n.
\]
If
\[
        \Cl(K)=F_n,
\]
then
\[
        K=F_n.
\]
\end{lemma}

\begin{proof}
We apply Theorem~\ref{thm:generation-Fn-prelim} to the generating set \(X=K\).

The closure condition holds because
\(\Cl(K)=F_n.\)
Thus
\([F_n,F_n]\le \Cl(K).\)

The abelianization condition holds because
\(S_i\le K\)
and \(S_i\) has full image in the abelianization of $F_n$.

By Lemma~\ref{lem:standard-chain-one-sided}, the one-sided
\(D_{n,s}\)-condition holds for the subgroup \(S_i\), and hence for \(K\).

It remains to verify the semi-core condition.  Since \(S_i\) is closed, the semi-core built
from the full generating set \(S_i\) is \(\mathcal B_i\).  
Indeed, since \(S_i\) is accepted by \(\mathcal B_i=\C(S_i)\), every
branch-pair identification imposed in the semi-core construction is already
present in \(\mathcal B_i\).  Conversely, suppose that two words \(u,v\) end
at the same state of \(\mathcal B_i\).  By the existence property of the core,
and since \(S_i=\Cl(S_i)\), some element of \(S_i\) has the branch pair
\(u\to v\).  This branch pair need not occur in the reduced diagram, but it is
obtained from a reduced branch pair \(p\to q\) by appending a common suffix
\(w\), so that \(u=pw\) and \(v=qw\).  The reduced pair \(p\to q\) is glued in
the semi-core, and reading \(w\) from the identified state \(p^+=q^+\) implies that
 \(u^+\) and \(v^+\) are identified.  Hence the semi-core from the full generating
set \(S_i\) makes exactly the identifications present in \(\mathcal B_i\), that
is,
\[
        L_{\mathrm{sem}}(S_i)=\mathcal B_i .
\]

Since \(S_i\le K\), the semi-core
\(L_{\mathrm{sem}}(K)\) is a quotient of \(\mathcal B_i\).  Since
\(\Cl(K)=F_n,\)
the folded quotient of this semi-core is
\(\C(F_n)=\mathcal B_0.\)
In particular, \(L_{\mathrm{sem}}(K)\) admits a surjective morphism onto \(\mathcal B_0\).
By Proposition~\ref{prop:quotients-of-Bi},
\[
        L_{\mathrm{sem}}(K)\cong\mathcal B_j
\]
for some \(0\le j\le i\).  If \(j>0\), then the folded quotient would still be
\(\mathcal B_j\), since \(\mathcal B_j\) is already folded.  This contradicts
\(\Cl(K)=F_n\).  Hence
\[
        L_{\mathrm{sem}}(K)\cong\mathcal B_0=\C(F_n).
\]
In \(\C(F_n)\), the terminal state of an inner word depends only on its
\(\sigma_n\)-value.  Therefore the semi-core condition in
Theorem~\ref{thm:generation-Fn-prelim} holds.

All hypotheses of Theorem~\ref{thm:generation-Fn-prelim} are satisfied, and hence
\(K=F_n.\)
\end{proof}

\begin{theorem}[Only predecessors occur in the standard chain]
\label{thm:standard-only-predecessors}
Let
\[
        S_i< K\le F_n.
\]
Then
\[
        K=S_j
\]
for some
\[
        0\le j< i.
\]
\end{theorem}

\begin{proof}
Let
\(L=\Cl(K).\)
Since \(S_i\) is closed and \(S_i< K\), we have
\(S_i< L\le F_n.\)
By Theorem~\ref{thm:closed-overgroups-Si},
\(L=S_j\)
for some \(0\le j<i\).

Thus
\(K\le S_j=F_n^{\zeta^j}.\)
Conjugate by \(\zeta^{-j}\), and put
\(K'=K^{\zeta^{-j}}.\)
Then
\(S_{i-j}< K'\le F_n.\)
We claim that
\(\Cl(K')=F_n.\)
Indeed, since \(K'\) strictly contains \(S_{i-j}\), Theorem~\ref{thm:closed-overgroups-Si}
gives
\(\Cl(K')=S_r\)
for some \(0\le r< i-j\).  If \(r>0\), then conjugating back gives
\(K\le S_{j+r}<S_j,\)
contradicting
\(\Cl(K)=S_j.\)
Therefore
\(\Cl(K')=S_0=F_n.\)
Hence, Lemma~\ref{lem:standard-dense-overgroups} gives
\(K'=F_n.\)
Therefore
\(K=S_j.\)
\end{proof}

\begin{corollary}
\label{cor:standard-chain}
The subgroups
\[
        S_i=F_n^{\zeta^i}
\]
form a strictly descending chain
\[
        F_n=S_0>S_1>S_2>\cdots
\]
such that every \(S_i\cong F_n\), and every subgroup of \(F_n\) containing \(S_i\) is one
of
\[
        S_i,S_{i-1},\ldots,S_0.
\]
In particular, every adjacent inclusion
\[
        S_{i+1}<S_i
\]
is maximal.
\end{corollary}
       
\subsection{Compatible twisting}
\label{subsec:compatible-twisting}

We now modify the standard chain by inner conjugations.  The modification preserves the
overgroup property but kills the intersection. We will use the following lemma and proposition.

\begin{lemma}
\label{lem:avoid-one-element}
Let \(G\) be a group, and let
\[
        L< H\le G.
\]
Assume that \(L\) contains no nontrivial normal subgroup of \(H\). Then
for every \(g\in G\setminus\{1\}\) there exists \(h\in H\) such that
\[
        g\notin L^h.
\]
\end{lemma}

\begin{proof}
If \(g\notin H\), then we may take \(h=1\), since \(L< H\), and hence
\(g\notin L\).

Now suppose that \(g\in H\). We claim that there exists \(h\in H\) such that
\(hgh^{-1}\notin L.\)
Indeed, if this were false, then
\[
        hgh^{-1}\in L
        \qquad\text{for every }h\in H.
\]
Therefore the normal closure of \(g\) in \(H\)
   would be contained in \(L\). This normal closure is nontrivial, since it contains
\(g\neq 1\), and it is normal in \(H\). Thus \(L\) would contain a nontrivial
normal subgroup of \(H\), contradicting the hypothesis.
   \end{proof}

The construction in the following proposition can also be viewed in the language of coset trees:
a compatible choice of cosets \(S_i k_i\) gives a ray in the associated coset
tree, and the corresponding stabilizers are the conjugates
\(S_i^{k_i}=k_i^{-1}S_i k_i\). Thus the intersection of these conjugates is
the stabilizer of the ray. For a related use of this viewpoint, compare
Grigorchuk--Kravchenko~\cite[Definition~2.11 and proof of Theorem~4.2]{GK14}.

\begin{proposition}
\label{prop:compatible-twisting-kills-intersection-general}
Let \(G\) be a countable group, and suppose that
\[
        G=S_0>S_1>S_2>\cdots
\]
is a strictly descending chain of subgroups satisfying the following two
conditions.

\begin{enumerate}
    \item For every \(i\ge 0\), the interval of subgroups between \(S_i\) and
    \(G\) is exactly
    \[
        \{K\mid S_i\le K\le G\}
        =
        \{S_i,S_{i-1},\ldots,S_0\}.
    \]

    \item For every \(i\ge 0\), the subgroup \(S_{i+1}\) contains no nontrivial
    normal subgroup of \(S_i\).
\end{enumerate}

Then there exists a strictly descending chain
\[
        G=H_0>H_1>H_2>\cdots
\]
such that:

\begin{enumerate}
    \item for every \(i\ge 0\), \(H_i\) is conjugate to \(S_i\). In particular,
    \(H_i\cong S_i\). 

    \item for every \(i\ge 0\), the interval of subgroups between \(H_i\) and
    \(G\) is exactly
    \[
        \{K\mid H_i\le K\le G\}
        =
        \{H_i,H_{i-1},\ldots,H_0\};
    \]

    \item
    \[
        \bigcap_{i\ge 0}H_i=\{1\}.
    \]
\end{enumerate}
\end{proposition}

\begin{proof}
The idea is to enumerate the nontrivial elements of \(G\), and at the \(i\)-th
step choose the next subgroup so that it misses the \(i\)-th element. The point
is that we do this by conjugating inside the current subgroup. This keeps all
previously constructed levels unchanged, and therefore preserves the interval
structure.

Since \(G\) is countable, choose a sequence
\(g_0,g_1,g_2,\ldots\)
of all elements in \(G\setminus\{1\}\). 

We shall construct elements \(k_i\in G\) and subgroups
\(H_i=S_i^{k_i}\)
with the following compatibility  property:
\[
        H_j=S_j^{k_i}
        \qquad\text{for every }0\le j\le i.
        \tag{\(\ast_i\)}
\]
This means that, at stage \(i\), the finite initial segment
\(H_0,H_1,\ldots,H_i\)
is obtained from the corresponding initial segment
\(S_0,S_1,\ldots,S_i\)
by one common conjugation.

Set
$k_0=1,$  $H_0=S_0=G.$
Then the compatibility condition \((\ast_0)\) holds.

Now suppose that \(k_i\) and
\(H_0>H_1>\cdots>H_i\)
have been constructed and satisfy \((\ast_i)\). Define
\(L_i=S_{i+1}^{k_i}.\)
By \((\ast_i)\), we have
\(H_i=S_i^{k_i},\)
and therefore
\[
        L_i=S_{i+1}^{k_i}< S_i^{k_i}=H_i.
\]

We first observe that \(L_i\) contains no nontrivial normal subgroup of \(H_i\).
Indeed, suppose \(N\le L_i\) is normal in \(H_i\). Conjugating by \(k_i\), we get
\(k_iNk_i^{-1}\le S_{i+1},\)
and
\(k_iNk_i^{-1}\unlhd S_i\),  contradicting the second assumption in the theorem.

Applying Lemma~\ref{lem:avoid-one-element} to
\(L_i< H_i\le G\)
and to the element \(g_i\), we may choose \(h_i\in H_i\) such that
\(g_i\notin L_i^{h_i}.\)

Now define
\(k_{i+1}=k_i h_i\)
and
\(H_{i+1}=L_i^{h_i}.\)
Then
\[
        H_{i+1}
        =
        L_i^{h_i}
        =
        \left(S_{i+1}^{k_i}\right)^{h_i}
        =
        S_{i+1}^{k_i h_i}
        =
        S_{i+1}^{k_{i+1}}.
\]
Moreover,
\[
        H_{i+1}=L_i^{h_i}<H_i^{h_i}=H_i,
\]
because \(h_i\in H_i\). 

By our choice of \(h_i\), we also have
\(g_i\notin H_{i+1}.\)

It remains to check that the compatibility  condition is preserved. Let
\(0\le j\le i\). Since the chain is descending, we have
\(H_i\le H_j.\)
Because \(h_i\in H_i\), it follows that \(h_i\in H_j\). Hence conjugation by
\(h_i\) preserves \(H_j\) as a subgroup:
\(H_j^{h_i}=H_j.\)
Using \((\ast_i)\), we get
\[
        S_j^{k_{i+1}}
        =
        S_j^{k_i h_i}
        =
        \left(S_j^{k_i}\right)^{h_i}
        =
        H_j^{h_i}
        =
        H_j.
\]
For \(j=i+1\), we have already shown that
\[
        H_{i+1}=S_{i+1}^{k_{i+1}}.
\]
Therefore the compatibility condition \((\ast_{i+1})\) holds.

This completes the inductive construction of a strictly descending chain
\(G=H_0>H_1>H_2>\cdots.\)
By construction,
\(H_i=S_i^{k_i}\)
for every \(i\). 

Note that the intersection of the groups \(H_i\), \(i\ge0\), is trivial.
Indeed, for any \(g\in G\setminus\{1\}\) there exists \(i\ge 0\) with
\(g=g_i\), and by construction
\(g_i\notin H_{i+1}.\)

Finally, the interval property for the chain \((H_i)\) follows from the
compatibility condition and the interval property for the chain \((S_i)\).
Indeed, for each \(i\ge 0\), conjugating by \(k_i\) and using \((\ast_i)\)
gives
\[
        \{K\mid H_i\le K\le G\}
           =
          \{M^{k_i}\mid S_i\le M\le G\}
          =
          \{S_i^{k_i},S_{i-1}^{k_i},\ldots,S_0^{k_i}\}
          =
        \{H_i,H_{i-1},\ldots,H_0\}.
\]
\end{proof}

Now, we can prove the main theorem.

\begin{theorem}[Higman--Thompson groups all the way down]
\label{thm:Fn-all-the-way-down-final}
For every \(n\ge2\), there exists a descending chain
\[
        F_n=H_0>H_1>H_2>\cdots
\]
such that:
\begin{enumerate}
    \item \(H_i\cong F_n\) for every \(i\);
    \item every subgroup of \(F_n\) containing \(H_i\) is one of
    \[
            H_i,H_{i-1},\ldots,H_0=F_n;
    \]
    \item every adjacent inclusion
    \[
            H_{i+1}<H_i
    \]
    is maximal;
    \item
    \[
            \bigcap_{i\ge0}H_i=\{1\}.
    \]
\end{enumerate}
\end{theorem}

\begin{proof}
Let
\(S_i=F_n^{\zeta^i}\)
be the standard chain constructed above.  By
Corollary~\ref{cor:standard-chain}, it is a strictly descending chain of copies of \(F_n\),
and each \(S_i\) is contained only in
\(S_i,S_{i-1},\ldots,S_0\). In addition, for each $i$, the subgroup $S_{i+1}$ does not contain a non-trivial normal subgroup of $S_i$. Indeed, assume that $S_{i+1}$ contains a non-trivial normal subgroup of $S_{i}$. Conjugation by $\zeta^{-i}$ would imply that $S_1$ contains a non-trivial normal subgroup of $F_n$, and as such $S_1$ contains $[F_n,F_n]$. That is impossible, since otherwise, by Corollary \ref{prop:Si-full-abelianization} $S_1=S_1[F_n,F_n]=F_n$, (or simply because the core of $S_1$ has more than $n-1$ distinct inner vertices). Now,
apply Proposition~\ref{prop:compatible-twisting-kills-intersection-general}.  The resulting chain
\(H_i=S_i^{k_i}\)
has trivial intersection and preserves the same overgroup property. 
\end{proof}

\end{document}